%% file: main.tex
\documentclass{mcom-l}
\usepackage[utf8]{inputenc}
\usepackage{amssymb}
\usepackage{amsmath,bbm,mathabx,cases,booktabs,amsthm}
\usepackage{enumitem}   
\usepackage{mathrsfs}
\usepackage{caption,subcaption}
\usepackage{graphicx,multirow}
\usepackage{cite}
\usepackage{hyphenat}

\newtheorem{theorem}{Theorem}[section]

\newtheorem{proposition}{Proposition}

\theoremstyle{definition}

\theoremstyle{remark}
\newtheorem{remark}[theorem]{Remark}

\numberwithin{equation}{section}

\DeclareMathOperator*{\argmax}{arg\,max}
\DeclareMathOperator*{\argmin}{arg\,min}

\newcommand{\bigzero}{\mbox{\normalfont\large 0}}
\newcommand{\rvline}{\hspace*{-\arraycolsep}\vline\hspace*{-\arraycolsep}}
\newcommand{\subscript}[2]{$#1 _ #2$}

\begin{document}

\title[GPR under diff. constraints with applications to the wave eq.]{Stochastic processes under linear differential constraints : Application to Gaussian process regression for the 3 dimensional \\ free space wave equation}

\author{Iain Henderson}
\address{INSA Toulouse, 135 Avenue de Rangueil, 31400 Toulouse, France}
\curraddr{}
\email{henderso@insa-toulouse.fr}
\thanks{}

%    author two information
\author{Pascal Noble}
\address{INSA Toulouse, 135 Avenue de Rangueil, 31400 Toulouse, France}
\curraddr{}
\email{}
\thanks{}
%    author three information
\author{Olivier Roustant}
\address{INSA Toulouse, 135 Avenue de Rangueil, 31400 Toulouse, France}
\curraddr{}
\email{}
\thanks{}

% 35L15		Initial value problems for second-order hyperbolic equations
% 65M32  	Numerical methods for inverse problems for initial value and initial-boundary value problems involving PDEs
% 35Q62  	PDEs in connection with statistics
% 60G12  	General second-order stochastic processes
% 60G17  	Sample path properties
% 60G15  	Gaussian processes
% 62F15  	Bayesian inference
% 46E22  	Hilbert spaces with reproducing kernels (= (proper) functional Hilbert spaces, including de Branges-Rovnyak and other structured spaces) [See also 47B32]

\begin{abstract}
Let $P$ be a linear differential operator over $\mathcal{D} \subset \mathbb{R}^d$ and $U = (U_x)_{x \in \mathcal{D}}$ a second order stochastic process. In the first part of this article, we prove a new necessary and sufficient condition for all the trajectories of $U$ to verify the partial differential equation (PDE) $T(U) = 0$. This condition is formulated in terms of the covariance kernel of $U$. When compared to previous similar results, the novelty lies in that the equality $T(U) = 0$ is understood in the \textit{sense of distributions}, which is a relevant framework for PDEs. This theorem provides precious insights during the second part of this article, devoted to performing ``physically informed" machine learning for the homogeneous 3 dimensional free space wave equation. We perform Gaussian process regression (GPR) on pointwise observations of a solution of this PDE. To do so, we propagate Gaussian processes (GP) priors over its initial conditions through the wave equation. We obtain explicit formulas for the covariance kernel of the propagated GP, which can then be used for GPR. %Our theorem states that this kernel, the trajectories of the corresponding GP and the predictions provided by GPR are all solutions to the wave equation in the sense of distributions.
We then explore the particular cases of radial symmetry and point source. For the former, we derive convolution-free GPR formulas; for the latter, we show a direct link between GPR and the classical triangulation method for point source localization used in GPS systems. Additionally, this Bayesian framework provides a new answer for the ill-posed inverse problem of reconstructing initial conditions for the wave equation with a limited number of sensors, and simultaneously enables the inference of physical parameters from these data. Finally, we illustrate this physically informed GPR on a number of practical examples.
\end{abstract}

\maketitle

\input{1_intro.tex}
\input{2_background.tex}
\input{3_GP_diff.tex}
\input{4_3D_wave.tex}

\input{5_numerical_experiments_new_results}
\input{6_conclusion.tex}
\input{Appendix_num_new_results.tex}

\bibliographystyle{amsplain}
\bibliography{bibliography}

\end{document}

%% file: 1_intro.tex
\section{Introduction}
Machine learning techniques have proved countless times that they can provide efficient solutions to difficult problems when field data are available. One key element to a great part of this success is the incorporation of ``expert knowledge" in the corresponding statistical models. In a good deal of practical applications, this knowledge takes the form of mathematical models which, at times, are already well understood. This is very common when dealing with problems coming from physics such as thermodynamics, continuum mechanics or fluid mechanics to name a few. In these examples,  the mathematical models take the form of Partial Differential Equations (PDEs). The zoology of PDEs is incredibly vast \cite{evans1998} and their applications are ubiquitous. As such, tremendous efforts have been devoted to solving them, both theoretically \cite{evans1998} and numerically \cite{Ames1977NumericalMF}. These equations impose very specific structures on the observed data which can be extremely difficult to capture or mimic with general machine learning models.
Therefore, given the available knowledge of the underlying PDEs, one may try incorporating these structures in machine learning models. How can this be done? 
Gaussian Process Regression (GPR) \cite{gpml2006}, which is a type of Bayesian framework for machine learning, provides a possible answer when the PDE is \textit{linear}. Indeed, Gaussian processes (GPs) are the most ``linear" of all random processes : they are stable under (finite) linear combinations of their elements \cite{gpml2006}. 
Although they are very simple mathematical objects when compared to non linear ones, linear PDEs are central in general PDE theory and remain physically relevant in a number of applications such as acoustics, electromagnetics or quantum mechanics. In this paper, we focus on the 3D wave equation which in fact plays a fundamental role in all three of the aforementioned domains.

\subsection{State of the Art}
GPR is a ``kernel based" machine learning method, which means that it is built around a positive definite function (see equation \eqref{eq:pd kernel}) called kernel. Solving or ``learning" linear ODEs and PDEs thanks to GPR is not a very new idea. The first initiative in that direction may go back to \cite{graepel2003} and has been re-explored ever since. It has been developed in the context of latent forces (\cite{alvarez2009}, \cite{alvarez2013},  \cite{Srkk2019GaussianPL},  \cite{LpezLopera2021PhysicallyInspiredGP}) and was then applied to certain wave equations (\cite{alvarado2014}, \cite{alvarado2016}). Latent forces are interested in linear PDEs of the form 
\begin{equation}
Lu = f \label{eq:f source}
\end{equation}
%\begin{numcases}
%Lu &= f \label{eq:f source} \\ 
%%Bu &= g \label{eq:g boundary} 
%\end{numcases}
% , and $B$ a boundary operator that operates on $\partial\mathcal{D}$, the boundary on $\mathcal{D}$
where both $u$ and $f$ are defined on the same domain $\mathcal{D} \subset \mathbb{R}^d$, and $L$ is a linear differential operator.
Latent forces put a GP prior on the driving source term $f$. Explicit resolution of \eqref{eq:f source}, thanks to Green's function \cite{Duffy2015GreensFW}, translates this prior as a GP distribution over the solution $u$. Conversely, a second approach (\cite{raissi2017}, \cite{raissi2018numericalGP}) rather puts a GP prior on the \textit{solution} $u$ and straightforwardly translates \eqref{eq:f source} as a GP distribution on the driving term $f$, avoiding the need for Green's functions and convolutions. Though both of these approaches are ``physically informed", they may not account for strict linear equality constraints in the interior domain $\mathcal{D}$, i.e. when $f$ is known and should not be seen as random. In this case, this knowledge could be exploited for dimension reduction. Actually, a number of famous PDEs can be studied with no interior source term $(f \equiv 0)$, in which case initial conditions or more general boundary source terms are provided. Such PDEs fall in the category of initial value or boundary problems (IVP/BP). This is frequent in the evolution equation literature, which for instance gives rise to the elegant semi-group theory for PDEs (\cite{Pazy1983SemigroupsOL} or \cite{evans1998}, section 7.4). Equation \eqref{eq:f source} then writes $Lu = 0$ and while $u$ may be modelled as a random object, $0$ should be strictly $0$ and not just a centered stochastic process, as would be the case if using any of the two frameworks described above. % from e.g. \cite{alvarez2009} and \cite{raissi2017}.
A first step towards enforcing strict linear constraints in the approximation space probably dates back to \cite{Narcowich1994GeneralizedHI} where in a deterministic context, divergence-free interpolation spaces were first built. This idea was pursued in \cite{Schaback2009SolvingTL} where an interpolation space comprised of solutions to Laplace's equation was constructed. In both cases, these spaces are built upon positive definite kernels (see equation \eqref{eq:pd kernel}), which makes them easy to transpose in a GPR framework. As such, the kernel from \cite{Schaback2009SolvingTL} was then used in \cite{mendes2012} for performing GPR on Laplace's equation. Likewise for \cite{scheuerer2012}, where GPs are built so that their trajectories are systematically divergence/curl free.
This was then taken a step further in \cite{wahlstrom2013}, \cite{jidling2017},  \cite{Jidling2018ProbabilisticMA} and \cite{hegerman2018} in the more general context of Maxwell's stationary equations. Finally, \cite{albert2020} applies this framework to the 1D heat equation, Laplace's and Helmholtz' 2D equations. The matter of enforcing strict homogeneous boundary conditions in the context of GPR has also been addressed in \cite{gulian2022}, \cite{hegermann2021LinearlyCG}. Enforcing these constraints provides another way of lowering the dimension of the initial problem. Following \cite{Solin2020HilbertSM}, \cite{gulian2022} builds a PDE-tailored covariance kernel thanks to a spectral decomposition of the differential operator $L$. Note that \cite{owhadi_bayes_homog} tackles the matter of defining a unifying GPR framework in which all the above approaches can fit, in the context of numerical Bayesian homogenization. As in many of the references above, the central tool in \cite{owhadi_bayes_homog} is the PDE's Green's function (and the chosen example is elliptic); in applications, the philosophy of \cite{owhadi_bayes_homog} is more or less that of latent forces \cite{alvarez2009}. Actually, \cite{cockayne2017probabilistic} and \cite{cockayne_aip} reformulate the approach exposed in \cite{owhadi_bayes_homog} and term it probabilistic meshless method (PMM). They also consider putting the prior directly over the solution $u$, as in \cite{raissi2017}; they then expose a convergence result of PMM in a probability framework. GP priors can also be used to quantify uncertainty w.r.t. ODE and PDE discretization in numerical solvers (\cite{conrad2017statistical}). Interestingly, \cite{hegermann2021LinearlyCG} raises the question of rigorous proofs and regularity issues regarding the derivations and applications of GPR for PDEs, and resorts to algebraic techniques to justify the different steps of their approach. We raise the same questions here, though we rather make use of a functional analysis framework adapted to PDEs. 
All of the aforementioned approaches as well as the one presented in this article can be formulated using the theory of stochastic partial differential equations (SPDE), which are PDEs whose source term is a random function. The general matter of applying physically informed GPR to linear PDEs thanks to an SPDE formulation is tackled in \cite{sarkka2011}, without addressing regularity questions. \cite{Solin2016} reformulates GPR as an SPDE problem, enabling the use of Kalman filter theory for computational efficiency. In \cite{NGUYEN_peraire}, the variational formulation (see \cite{evans1998}, section 6.1.2 for a definition) of certain linear PDEs has been incorporated into a GPR framework thanks to a SPDE reformulation. This approach requires the use of Gaussian generalized stochastic processes (see \cite{agnan2004}, section 2.2.1.1), or ``functional Gaussian processes" following \cite{NGUYEN_peraire}. In \cite{Nikitin2021NonseparableSG}, covariance kernels on graphs are obtained thanks to an adaptation of SPDEs on graphs. Finally, \cite{vergara2022general} focuses on the study of stationary stochastic processes that are solutions of a wide class of linear SPDEs, outside the context of GPR. In particular, \cite{vergara2022general} provides a description of all the second order stationary stochastic processes that are solutions to the 3D wave equation, a central equation in the present article; this description is done in terms of the  covariance kernel of the corresponding stochastic process. Note that as in this article, \cite{vergara2022general} also makes use of the theory of generalized functions. On a side note, an in-depth theoretical study of a stochastic 3D wave equation (central in this paper) and the regularity of the associated random paths is exposed in \cite{Dalang2009wave_3D}. In a much wider framework, general linearly constrained stochastic processes and GPs in particular are thoroughly explored in \cite{ginsbourger2016}. Though \cite{ginsbourger2016} deals with many different types of linear operators, the application of the corresponding results to linear PDEs are not straightforward, see section \ref{section:sto diff}. Indeed, these results are not phrased using generalized functions, which are especially convenient to study differential equations. This is why we prove in Proposition \ref{prop : diff constraints} a theorem which is reminiscent of \cite{ginsbourger2016}, but specifically adapted to linear PDEs. A recent survey on linearly constrained GPs presents most of these different approaches \cite{Swiler2020ASO}. Finally, a recent article \cite{CHEN_owhadi_2021} extended the use of GPR to nonlinear PDEs by imposing nonlinear interpolation constraints, thus setting the way forward for many possible applications of GPR to nonlinear realistic PDE models, as found e.g. in biology of fluid mechanics.

Very closely related approaches are the kerned-based methods for solving PDEs (\cite{wendland2004scattered}, section 16.3), which make extensive use of the theory of reproducing kernel Hilbert spaces (RKHS, \cite{agnan2004}). These spaces are also built upon positive definite kernels (equation \eqref{eq:pd kernel}) and kerned based methods naturally generalize collocation methods with radial basis functions (RBF, \cite{wendland2004scattered},\cite{fasshauer2007}); those are actually not a new tool for solving PDEs (\cite{buhmann_2003}, section 2.4 and the references therein). The RKHS point of view can be seen as a deterministic counterpart of GPR (section \ref{subsub:rkhs}); in particular the final function approximation formula \eqref{eq:krig mean} is the same. Kernel based methods can also be used to tackle linear SPDEs : in this case, one can go back and forth between the RKHS and GPR points of view (\cite{cialenco2012approximation}). Interestingly, \cite{fasshauer2011reproducing} and \cite{fasshauer2013reproducing} also build positive definite kernels thanks to the PDE's Green's function when it exhibits sufficient Sobolev regularity.

In this article, we focus on the so-called (3 dimensional, free space, time dependent) wave equation. As a time-dependent PDE, we show that performing GPR on wave equation data amounts to reconstructing the corresponding initial conditions of the wave equation from incomplete scattered data. This immediately echoes with questions arising from the inverse problem community (\cite{Tarantola2005}, \cite{ammari2012}); the solution we provide in this article falls in the domain of Photo Acoustic Tomography (PAT), which aims at recovering the initial conditions of wave propagation problems similar to equation \eqref{eq:wave_eq}. We refer \cite{kuchment2010} and the many references therein for PAT. Closely related is \cite{Purisha_2019} (and the references within), where a GPR framework based on linear Radon transforms are described to solve x-ray tomography problems. Also, Bayesian approaches for inverse problems involving PDEs are in fact standard (\cite{Tarantola2005}, \cite{stuart_2010}, \cite{dashti_stuart_pde}, \cite{Dashti2017}) and have been set up a number of times for equations from fluid mechanics in particular (see e.g. \cite{stuart_2010}, \cite{dashti_stuart_pde} and the related references therein). However, these approaches never entirely follow the GPR methodology presented in this article when it comes down to reconstructing some initial conditions, see the contributions section for a more detailed discussion. Additionally, the questions at hand in \cite{stuart_2010}, \cite{dashti_stuart_pde} and \cite{Dashti2017} differ from those that are targeted in the present article, as these works rather study general theoretical questions such as well-posedness and stability in the Bayesian inversion framework. 
The general work-flow of Bayesian inversion is the following : a probability prior is set on the model's parameter space. This prior is then conditioned on field data thanks to Bayes' theorem, in order to estimate the model's true parameters. See \cite{Tarantola2005}, sections 5.6 to 5.8 for practical examples, or \cite{ESMAILZADEH201556} for an inversion of a non uniform propagation speed $c(x)$.

%\subsection{State of the art}
%\begin{itemize}
%\item The standard Inverse problem approach : \cite{Tarantola2004} for general methods inverse problems; Photoacoustic Tomography (PAT),\cite{kuchment2010} and the many references therein.
% and the references therein Photoacoustic Tomography, \cite{ammari2012}, Chapter 3. for a more elaborated bibliography.
%\item Early approach for solving ODEs/PDEs with GP \cite{graepel2003}
%\item Latent Forces approach :  \cite{Srkk2019GaussianPL} %\cite{LpezLopera2021PhysicallyInspiredGP}
%\item  Karniadakis paper approach \cite{raissi2017} \cite{raissi2018numericalGP}.
%\item Div/Curl free Gaussian processes for Maxwell's eqs/ fluid dynamics \cite{scheuerer2012} \cite{jidling2017} \cite{Jidling2018ProbabilisticMA} \cite{wahlstrom2013}
%\item Attempts to deal with some other PDEs (le papier d'Entropy un peu pourri) \cite{albert2020}
%\end{itemize} 

\subsection{Contribution}
We consider the general problem of applying GPR on scattered observations of solutions of the wave equation using ``physically informed" GPs. We explore both the theoretical and applied aspects of this task.

In Proposition \ref{prop : diff constraints}, we prove a general result that provides a simple necessary and sufficient condition for the trajectories of any second order stochastic process to be solutions to a given linear PDE in the distributional sense (section \ref{section:sto diff}). This condition is formulated in terms of its covariance function. The hypotheses are minimal and the displayed result is concise. This theorem is phrased in a functional analysis framework, using the permissive language of generalized functions.

% the singularity of the Green function calls for a study... does not fit in the references...
We describe a general Gaussian process model for the homogeneous 3D wave equation, with the corresponding proofs (section \ref{section:GP_wave}). This model is obtained by putting GP priors over the initial conditions of the wave equation, which comes naturally if considering those to be unknown. In particular we derive the corresponding covariance kernel which we will directly use for GPR. For short, we denote WIGPR the use of this kernel to perform GPR, as in ``Wave Informed GPR". Our approach enforces strict linear homogeneous PDE constraints in the interior domain, similarly to what was observed in \cite{wahlstrom2013}, \cite{jidling2017} and \cite{hegerman2018} among others for different PDEs (see the state of the art section for more details). More precisely, we get from Proposition \ref{prop : diff constraints} that the trajectories of the corresponding GP all verify the wave equation, as well as the predictions provided by WIGPR. Here, these linear constraints are understood in light of our result from section 2, i.e. in the sense of distributions. The key difference with the kernels presented in \cite{vergara2022general} is that here, no stationarity assumptions are made on the underlying stochastic process. In particular the spectral measure provided by Bochner's theorem \cite{gpml2006}, which is the key tool used in \cite{vergara2022general}, is not available anymore. We thus resort to more general integration techniques in the proofs. 

We then provide an inverse problem interpretation of the use of WIGPR on wave equation data, as the prediction from WIGPR evaluated at $t = 0$ provides a finite dimensional reconstruction of the real initial conditions corresponding to the observed data.  This is a natural thing to do because here, the uncertainty (randomness) was initially set on the initial conditions. The resulting posterior distribution then enables the estimation of said initial conditions. Note that this approach differs from the Bayesian approach to PDE inverse problems described e.g. in \cite{stuart_2010} and \cite{dashti_stuart_pde} because here, the forward linear problem is completely solved inside the GPR kernel \textit{before} conditioning the GP prior on the field data. When compared to equation (2.6) from \cite{stuart_2010}, we use a wave equation-tailored GP prior $\pi_0$ for the full space-time solution : the forward model is already solved in the prior $\pi_0$ and doesn't appear in the noise term $\rho$. It is only after choosing this tailored prior that the corresponding initial conditions are estimated by evaluating the Kriging mean at $t=0$. Among other things, the solution to the corresponding inverse problem provided by our approach is implementable in practice.

We investigate the two particular cases of radial symmetry and point source. When the initial conditions exhibit radial symmetry, we derive convolution-free covariance formulas and discuss them when the initial conditions are compactly supported. Indeed, these formulas can be directly linked to the finite speed propagation principle for the 3D wave equation, also known as the strong Huygens principle. In the point source case, we show numerically (Figure \ref{fig : point source}) and theoretically that the parameter fitting step from WIGPR naturally reduces to the classic triangulation approach for point source localization, used for instance in GPS systems. Indeed, as in \cite{raissi2017}, WIGPR can be used to jointly estimate physical parameters such as wave speed, source localization or source size. 

Note that the wave equation differs from most of the PDEs mentioned in the introduction, such as the heat equation or Laplace's equation, because these are either parabolic or elliptic. There are known regularization effects for such PDEs (\cite{evans1998}, sections 6.3 and 7.1.3) which somewhat mitigate the need for overly prudent mathematical argumentation w.r.t. the derivations of GPR for PDEs. Such regularization effects disappear for hyperbolic PDEs such as the wave equation (\cite{evans1998}, section 7.2.4) and careful derivations now become critical (\cite{evans1998}, section 3.4.1). Equivalently, such derivations require special care because the wave equation's Green's function is singular, see equation \eqref{eq:ft ftp in 3D}. This is also why we cannot directly refer to \cite{owhadi_bayes_homog}, \cite{cockayne2017probabilistic} or \cite{cialenco2012approximation} for an adapted theoretical framework. This is the reason why throughout this section, effort is done to formulate the arguments necessary for the derivations that lead to the exposed formulas.

We illustrate WIGPR on a few numerical experiments (section \ref{section : num}). They are performed on simulated wave equation data, with radially symmetric compactly supported initial conditions. This data takes the form of a number of noise polluted time series, each of them corresponding to an ``artificial" sensor placed in the numerical simulation. We thus use the fast-to-compute covariance expressions derived in the previous section. The tackled questions concern $(i)$ the quality of the joint estimation of the wave speed, source position and size and the quality of the resulting initial condition reconstruction and $(ii)$ the sensibility of the reconstruction step w.r.t. the sensor location. We display initial condition reconstruction images, in light of the inverse problem interpretation described in the previous section. In appendix A are presented more complete numerical results, showing for each example the quality of the physical parameter estimation as well as $L^2$, $L^1$ and $L^{\infty}$ relative error estimates in terms of the number of sensors used.

\subsubsection{Organization of the paper}
The paper is organized as follow. For self-contain\-ment, section 2 is dedicated to reminders on GPs, GPR and generalized functions. This section and all the proofs are detailed enough so that this article is accessible both to the analyst and the statistician. In section 3, we state and prove our new necessary and sufficient condition on stochastic processes that are subject to linear differential constraints. Section 4 is dedicated to the study of the wave equation thanks to Gaussian processes and Gaussian process regression. In section 5, we showcase some numerical applications of the previous section on wave equation data. We conclude in section 6.

%Linear differential operators are hard to deal with compared to other, usually continuous, linear operators. They are the prototype for \textit{unbounded} operators \cite{rudin1991}, which more or less means that they are discontinuous for the natural topology of the space on which they are defined. Properly dealing with them usually implies the use of the powerful yet thorny machinery of functional analysis, whether it be generalized functions, weak derivatives, unbounded operator theory or yet other topics.

%\subsection{Statement of the problem}
%\begin{itemize}
%\item The problem of interest is the following : given a small database of observations of a solution to the 3 dimensional wave equation, try to reconstruct it over the whole space for all times. Obviously, this is an ill-posed problem : from finite dimensional data, we may only build finite dimensional approximations.
%\item Limited database of observations of the solution. Known underlying physics. Type of measurements : time series from a small number of sensors.
%\end{itemize}

%\subsection{Contribution and organization of the paper}
%\begin{itemize}
%\item La Proposition 2 avec un cadre fonctionnel; \item Presque toute la section 3 avec le cas particulier de la limite ponctuelle (GPS) et l'aspect problème inverse;
%\item les contraintes linéaires réduisent la dimension du problème (grosse diff avec Karniadakis/Alvarez); 
%\item Les preuves rigoureuses;
%\item Les expériences numériques.
%\end{itemize}

%% file: 2_background.tex
%!TEX root = ../main.tex

\section{Notations and background}
\subsection{Notations}
Let $m$ be a scalar function defined on an open set $\mathcal{D}\subset \mathbb{R}^d$ and $k$ a scalar function defined on $\mathcal{D} \times \mathcal{D}$. Let $X = (x_1,...,x_n)^T$ be a column vector in $\mathcal{D}^n$. Let $(\Omega,\mathcal{F},\mathbb{P})$ be a probability space, over which all the random objects of this article will be defined.
\begin{enumerate}[label=\textbf{N.\arabic*},ref=\textbf{N.\arabic*}]
    \item \label{not : vecmat} Note $m(X)$ the column vector such that $m(X)_i = m(x_i)$, $k(X,X)$ the square matrix such that $k(X,X)_{ij} = k(x_i,x_j)$ and given $x \in \mathcal{D}$, $k(X,x)$ the column vector such that $k(X,x)_i = k(x_i,x)$. %Note the $\delta_{x,x'}$ the Kronecker delta : $\delta_{x,x'} = 1$ if $x = x'$, $0$ else.
    \item \label{not : rkhs} for any positive definite kernel $k$ (see equation \eqref{eq:pd kernel}), $\mathcal{H}_k$ denotes the associated Reproducing Kernel Hilbert Space (RKHS) as defined in \cite{agnan2004}.
    \item \label{not : gp et al} Note $L^2(\mathbb{P})$ the Hilbert space of real-valued random variables defined on $\Omega$ with finite second order moment endowed with the inner product $\langle X, Y\rangle = \mathbb{E}[XY]$. For a GP $(U(x))_{x \in \mathcal{D}}$, the trajectory of $U$ at point $\omega \in \Omega$ is the deterministic function $x \longmapsto U(x)(\omega)$ and is noted $U_{\omega}$. $\mathcal{L}(U) := \overline{\text{Span}(U(x), x \in \mathcal{D})} \subset L^2(\mathbb{P})$ denotes the Hilbert subspace of $L^2(\mathbb{P})$ induced by $U$. Since  $L^2(\mathbb{P})$-limits of Gaussian random variables drawn from the same GP remain Gaussian (\cite{janson_1997}, section 1.3), $\mathcal{L}(U)$ only encompasses Gaussian random variables.
    \item \label{not : as ae} If $X$ is a random variable, then ``$X \in A$ a.s." means ``$X \in A$ almost surely", or equivalently $\mathbb{P}(X \in A) = 1$. Likewise, if $f$ is a function defined on  $\mathcal{D}\subset \mathbb{R}^d$, then ``$f(x) \in A$ a.e." means ``$f(x) \in A$ almost everywhere" or equivalently, $\lambda_d(\{x \in \mathcal{D} : f(x) \notin A\}) = 0$ where $\lambda_d$ is the Lebesgue measure on $\mathbb{R}^d$.
    \item \label{not : function spaces} $L^1_{loc}(\mathcal{D})$ denotes the space of measurable scalar functions $f$ defined on $\mathcal{D}$ that are locally integrable, i.e. such that $\int_K |f| < +\infty$ for all compact sets $K \subset \mathcal{D}$. $\mathscr{D}(\mathcal{D})$ denotes the space of compactly supported infinitely differentiable functions supported on $\mathcal{D}$.    
    \item \label{not : derivatives} for a multi-index $k = (k_1,...,k_d) \in \mathbb{N}^d$, we use the usual notation $|k| = k_1+...+k_d$ and $\partial^k := \partial_{x_1}^{k_1}...\partial_{x_d}^{k_d}$ where $\partial_{x_i}^{k_i}$ is the $k_i^{th}$ derivative w.r.t the $i^{th}$ coordinate $x_i$.    
    \item \label{not : spherical} The variables $(r,\theta,\phi)$ will always denote spherical coordinates; $S(0,1)$ denotes the unit sphere of $\mathbb{R}^3$ and we will always write $d\Omega = \sin \theta d\theta d\phi$ its surface differential element; $\gamma = (\sin \theta \cos \phi,\sin \theta \sin \phi,  \cos \theta)^T \in S(0,1)$ denotes the unit length vector parametrized by $(\theta,\phi)$.
\end{enumerate}

%We only present here the technical elements necessary for understanding the next sections. 
\subsection{Gaussian Processes}\label{sub : GP}
We refer to \cite{gpml2006} for further details on on Gaussian processes. %\ref{sub : GP} and \ref{sub : GPR}.
\subsubsection{Definition} Let $(\Omega, \mathcal{F},\mathbb{P})$ be a probability space and $\mathcal{D} \subset \mathbb{R}^d$ an open set. A Gaussian process $(U(x))_{x \in \mathcal{D}}$ is a collection of normally distributed random variables defined on $\Omega$ and indexed by $\mathcal{D}$ such that for any $(x_1,...,x_n) \in \mathcal{D}^n$, the law of $(U({x_1}),...,U({x_n}))^T$ is a multivariate normal distribution.
The law of a GP is characterized by its mean and covariance functions (\cite{janson_1997}, section 8), defined by 
\begin{itemize}
    \item $m(x) := \mathbb{E}[U(x)]$
    \item $k(x,x') = \text{Cov}(U(x),U({x'})) = \mathbb{E}[(U(x)-m(x))(U({x'})-m(x'))]$
\end{itemize}

We write $(U(x))_{x \in \mathcal{D}} \sim GP(m,k) $. % A famous example is the Wiener measure over $C([0,1])$, using the Brownian motion for $U$.
\subsubsection{Covariance kernels}
The mean function $m$ can be any function, and is actually often set to zero. On the other hand, the function $k$ has to be positive definite (PD) :
\begin{align}\label{eq:pd kernel}
    \sum_{i,j=1}^n a_i a_j k(x_i,x_j) \geq 0 \ \ \ \forall \ (x_1,...,x_n) \in \mathcal{D}^n,\ \ \ \forall \  (a_1,...,a_n) \in \mathbb{R}^n
\end{align}
PD functions verify the Cauchy-Schwarz inequality \cite{gpml2006} :
\begin{align}\label{eq:CS PD}
\forall x,x' \in \mathcal{D}, \ \ \ |k(x,x')| \leq \sqrt{k(x,x)}\sqrt{k(x',x')}
\end{align} 
The covariance kernel $k$ is the core element that encodes the mathematical properties of the GP. Furthermore, there is a one-to-one correspondence between positive definite kernels and (covariance kernels of) centered GPs \cite{dudley2002}. Thus we will focus on the design of positive definite kernels.

Among all covariance kernels, some are said to be \textit{stationary}, in which case the value at $(x,x')$ only depends on the increment $x-x'$ : $k(x,x') = k_S(x-x')$. Common examples are the squared exponential and Matérn kernels \cite{gpml2006}; see equation \eqref{eq:matern 5/2} for an example. 
%Loosely speaking, supposing that a kernel is stationary means that the trajectories of the associated zero mean Gaussian process ``have the same properties" everywhere, such as regularity, variations, and shape in general. In particular, such trajectories cannot be compactly supported if the domain $\mathcal{D}$ is unbounded unless they are identically zero.
\subsubsection{Bayesian inference of functions}
A Gaussian process $U = (U(x))_{x \in \mathcal{D}}$ can also be seen as a random variable that is valued in a space of functions through the mapping $T : \omega \longmapsto U_{\omega}$, i.e. a random function. This in turn defines a probability distribution over functions \cite{gpml2006} through the pushforward measure of $\mathbb{P}$ through $T$. Suppose a function $u$ is unknown, one can model it as a GP $U \sim GP(m,k)$ just as one would model an unknown scalar $y$ as a normally distributed random variable $Y \sim \mathcal{N}(\mu,\sigma^2)$. We then say that we put a Gaussian process prior over $u$. Thanks to Bayes' theorem, the prior GP distribution of $U$ can then be updated through probability conditioning when data on $u$ is available. The resulting conditioned GP distribution is called the posterior. Statistical indicators can then be derived from the posterior, such as expectation and standard deviation, to estimate the unknown function $u$ over which the prior was initially set. Bayesian inference is one way of understanding Gaussian process regression.
\subsection{Gaussian Process Regression}\label{sub : GPR}
We refer to \cite{gpml2006} for further details on GPR.
\subsubsection{Kriging equations} GPs can be used for function interpolation. Let $u$ be  a function defined on $\mathcal{D}$ for which we know a dataset of values $B = \{u(x_1),...,u(x_n)\}$. Conditioning the law of a GP $(U(x))_{x \in \mathcal{D}} \sim GP(m,k)$ on the database $B$ yields a second GP $\Tilde{U}$ with
$ \Tilde{U}(x) := (U(x) | U({x_i}) = u(x_i), i = 1,...,n)$.
The law of $\Tilde{U}$ is known : $(\Tilde{U}(x))_{x \in \mathcal{D}} \sim GP(\Tilde{m},\Tilde{k})$. $\Tilde{m}$ and $\Tilde{k}$ are given by the so-called \textit{Kriging} equations \eqref{eq:krig mean} and \eqref{eq:krig cov}. Note $X = (x_1,...,x_n)^T$ and suppose that $K(X,X)$ is invertible, then \cite{gpml2006}
\begin{numcases}{}
    \Tilde{m}(x) &= \hspace{3mm}$m(x) + k(X,x)^Tk(X,X)^{-1}(u(X) - m(X))$ \label{eq:krig mean} \\
    \Tilde{k}(x,x') &= \hspace{3mm}$k(x,x') - k(X,x)^Tk(X,X)^{-1}k(X,x')$\label{eq:krig cov}
\end{numcases}
%\mathbb{E}[\tilde{U}(x)] = \text{Cov}(\tilde{U}(x),\tilde{U}(x'))
% : $\tilde{m}(x)$ can be expressed as a conditional expectation.
In a Bayesian framework, the initial GP $(U(x))_{x \in \mathcal{D}}$ is the prior and the conditioned GP $(\Tilde{U}(x))_{x \in \mathcal{D}}$ is the posterior; the Kriging mean $\tilde{m}$ and covariance $\tilde{k}$ are simply the mean and covariance of the posterior. At location $x$, $\Tilde{m}(x)$ is the prediction of $u(x)$. By construction, for all $i \in \{1,...,n\}$, we have that $\Tilde{m}(x_i) = u(x_i)$ and $\Tilde{k}(x_i,x_i) = 0$. If observing noisy data ${U_i = U(x_i) + \varepsilon_i}$ with $(\varepsilon_1,...,\varepsilon_n)^T \sim \mathcal{N}(0,\sigma^2I_n)$ independent from $U$, one replaces $K(X,X)$ with $K(X,X) + \sigma^2 I$ above and leaves the terms $k(X,x)$ unchanged. This amounts to applying Tikhonov regularization on $k(X,X)$ and may also be used to approximate \eqref{eq:krig mean} and \eqref{eq:krig cov}  when $k(X,X)$ is ill-conditioned.

%GPR is said to be a Bayesian approach to machine learning because we also have the following identity
%\begin{align}
%\tilde{m}(x)  = \mathbb{E}[U(x)|U(x_1),...,U(x_n)]
%\end{align}
\subsubsection{Tuning covariance kernels}
For discussions on general kernel construction and selection, we refer to \cite{gpml2006}, chapter 5. Usually, a family of kernels $k_{\theta}$ indexed by $\theta \in \Theta \subset \mathbb{R}^q$ is first selected. The elements of $\theta$ are the \textit{hyperparameters} of $k_{\theta}$. One may then try to find the value $\theta^*$ that fits the best the observations, which corresponds to maximizing the \textit{marginal likelihood}. It is the probability density of the Gaussian random vector $(U(x_1),...,U(x_n))^T$ at point $(u(x_1),...,u(x_n))^T$, see equation \eqref{eq:likelihood}. Note $u_{obs} = (u(x_1),...,u(x_n))^T$ the vector of observations at locations $X = (x_1,...,x_n)$ and $p(u_{obs}|\theta)$ the associated marginal likelihood at point $\theta$, we search for $\theta^*$ such that
\begin{align}\label{eq:learning theta}
    \theta^* = \argmax_{\theta \in \Theta} p(u_{obs} | \theta)
\end{align}
Explicitly, assuming that $m \equiv 0$, then $(U(x_1),...,U(x_n))^T \sim \mathcal{N}(0,k_{\theta}(X,X))$ and
\begin{align}\label{eq:likelihood}
p(u_{obs} | \theta) = \frac{1}{(2\pi)^{n/2}\det k_{\theta}(X,X)^{1/2}}e^{-\frac{1}{2} u_{obs}^T k_{\theta}(X,X)^{-1}u_{obs}}
\end{align}
Set $\mathcal{L}(\theta) := -2\log p(u_{obs} | \theta) - n\log 2\pi$, then \eqref{eq:learning theta} is equivalent to
\begin{align}\label{eq:learning theta log lik}
\theta^* = \argmin_{\theta \in \Theta} \mathcal{L}(\theta)
\end{align}
Problem \eqref{eq:learning theta log lik} is better behaved numerically. From now on, we call $\mathcal{L}(\theta)$ the negative log marginal likelihood and we have, for noiseless observations,
\begin{align}\label{eq:loglikdef}
\mathcal{L}(\theta) = u_{\mathrm{obs}}^T k_{\theta}(X,X)^{-1}u_{\mathrm{obs}} + \log \det k_{\theta}(X,X)
\end{align}
and for noisy observations with noise standard deviation $\sigma$,
\begin{align}\label{eq:loglik_noisy}
\mathcal{L}(\theta,\sigma^2) = u_{\mathrm{obs}}^T (k_{\theta}(X,X)+ \sigma^2 I_n)^{-1}u_{\mathrm{obs}} + \log \det (k_{\theta}(X,X) + \sigma^2 I_n)
\end{align}
%In the case of noisy observations with noise standard deviation $\sigma$, we have
%\begin{align}
%\mathcal{L}(\theta,\sigma^2) = u_{\mathrm{obs}}^T (k_{\theta}(X,X)+ \sigma^2 I_n)^{-1}u_{\mathrm{obs}} + \log \det (k_{\theta}(X,X) + \sigma^2 I_n)
%\end{align}
$\sigma$ can be interpreted as an additional hyperparameter and estimated through \eqref{eq:learning theta log lik}.
\subsubsection{Scattered Data interpolation and the RKHS point of view}\label{subsub:rkhs}
Kriging equations \eqref{eq:krig mean} and \eqref{eq:krig cov} can be encountered without resorting to GPs.
Given a positive definite kernel $k$ defined on $\mathcal{D}$, one may build a Reproducing Kernel Hilbert Space (RKHS) of functions defined on $\mathcal{D}$ which we denote by $\mathcal{H}_k$, see \ref{not : rkhs}. The inner product of $\mathcal{H}_k$ verifies the so called reproducing property :
\begin{align}\label{eq:repro_rkhs}
\forall x,x'\in \mathcal{D}, \langle k(x,\cdot),k(x',\cdot) \rangle_{\mathcal{H}_k} = k(x,x')
\end{align}
In the \textit{meshfree interpolation} framework \cite{agnan2004} \cite{fasshauer2007}, one may formulate the following constrained (interpolation) optimization problem
\begin{align}\label{eq:rkhs pov pb}
    \min_{v \in \mathcal{H}_k}||v||_{\mathcal{H}_k} \hspace{5mm} \text{s.t.} \hspace{5mm} v(x_i) = u(x_i) \ \ \ \forall i \in \{1,...,n\}
\end{align}
Solving \eqref{eq:rkhs pov pb} leads to the Kriging equation for $\Tilde{m}$ in \eqref{eq:krig mean}; the second equation \eqref{eq:krig cov} is what is called the power function in \cite{fasshauer2007}. One may also show \cite{agnan2004} that equation \eqref{eq:krig mean} can be summerized as
\begin{align}\label{eq:krig projection}
    \Tilde{m} = m + p_F(u-m)    
\end{align}
with $F$ is the finite dimensional space $F = \text{Span}(k(x_1,\cdot),...,k(x_n,\cdot)) \subset \mathcal{H}_k$ and $p_F$ stands for the orthogonal projection operator on $F$ w.r.t. the inner product of $\mathcal{H}_k$. In particular, when $m \equiv 0$, equation \eqref{eq:krig projection} amounts to $\Tilde{m} = p_F(u)$. Likewise, equation \eqref{eq:krig cov} amounts to 
\begin{align}\label{eq:tilde k proj}
    \Tilde{k}(x,\cdot) &= P_{F^\perp}(k(x,\cdot)) \\
    \text{with} \ \ \ \Tilde{k}(x,x) &= ||P_{F^\perp}k(x,\cdot)||_{\mathcal{H}_k}^2 \leq ||k(x,\cdot)||_{\mathcal{H}_k}^2 = k(x,x)
\end{align}
One perk of this approach is that the Kriging mean can now be understood as an orthogonal projection over a finite dimensional deterministic space, which is reminiscent of Fourier series or Galerkin reconstruction approaches.
%Note that strictly speaking, $\Tilde{m} = p_F(u)$ only makes sense if $u \in \mathcal{H}_k$. %, thus the choice of kernel $k$ has its importance.

\subsection{Generalized functions}\label{sub : distrib}
We refer to \cite{rudin1991} and \cite{treves2006topological} for further details on generalized functions and Radon measures. In this whole subsection, $\mathcal{D}$ is an open set of $\mathbb{R}^d$.
\subsubsection{Definitions and properties} Endow $\mathscr{D}(\mathcal{D})$ with its usual LF-space topology, defined for example in \cite{rudin1991}. We call generalized function any continuous linear form on $\mathscr{D}(\mathcal{D})$, i.e. any element of $\mathscr{D}(\mathcal{D})'$, the topological dual of $\mathscr{D}(\mathcal{D})$. We will rather denote it by $\mathscr{D}'(\mathcal{D})$ as in \cite{rudin1991}. The topology of $\mathscr{D}(\mathcal{D})$ is such that $T \in \mathscr{D}'(\mathcal{D})$ if and only if for all compact set $K \subset \mathcal{D}$, there exists $C_K > 0$ and a non negative integer $n_K$ such that
\begin{align}\label{eq:def distrib continue}
\forall \varphi \in \mathscr{D}(\mathcal{D})\ \text{s.t} \ \text{Supp}(\varphi) \subset K, |T(\varphi)| \leq C_K \sum_{|k|\leq n_K} ||\partial^k \varphi||_{\infty}
\end{align}

Generalized functions are also called ``distributions", a terminology we will only use when there is no risk of confusion with probability distributions. The duality bracket will be denoted $\langle, \rangle$ : for $\varphi \in \mathscr{D}(\mathcal{D})$ and $T \in \mathscr{D}'(\mathcal{D})$, we have $\langle T, \varphi \rangle = T(\varphi)$.

\begin{itemize}
\item Any function $f \in L^1_{loc}(\mathcal{D})$ can be injectively identified to a generalized function $T_f$ \cite{rudin1991} defined as follow %we will use the following
\begin{align}\label{eq:fL1 loc distrib}
\forall \varphi \in \mathscr{D}(\mathcal{D}), \ \ \ \langle T_f, \varphi \rangle := \int_{\mathcal{D}} f(x) \varphi(x) dx
\end{align}
The map $L^1_{loc}(\mathcal{D}) \ni f \longmapsto T_f$ is linear and injective. Throughout this article, we will use the abusive notation $\langle T_f, \varphi \rangle = \langle f, \varphi \rangle$, as if $\langle, \rangle$ were the $L^2$ inner product.
\item Any generalized function $T$ can be indefinitely differentiated \cite{rudin1991} with the following definition (see \ref{not : derivatives})
\begin{align}\label{eq:diff distrib}
\partial^k T : \varphi \longmapsto \langle T, (-1)^{|k|} \partial^k \varphi \rangle
\end{align}
which coincides with the definition of weak derivatives when $T$ is a function that admits the according weak derivatives \cite{rudin1991}.
\end{itemize}
In particular, \eqref{eq:fL1 loc distrib} and \eqref{eq:diff distrib} combined provide a flexible definition for the derivatives of any function $f \in L^1_{loc}(\mathcal{D})$ up to any order.

\subsubsection{Radon measures} In this paper, we call positive Radon measure any positive measure over $\mathcal{D}$ that is Borel regular (\cite{evans2018measure}, Def 1.9) and that has finite mass over any compact subset of $\mathcal{D}$. Borel regularity is a standard regularity hypothesis from measure theory. We call real-valued Radon measure, or simply Radon measure, any linear combination of positive Radon measures. In \cite{lang1993}, Chapter IX, it is proved that
the space of Radon measures over $\mathcal{D}$ is isomorphic to the space of continuous linear forms over $\mathcal{C}_c(\mathcal{D})$, the space of compactly supported continuous functions on $\mathcal{D}$ endowed with its usual LF-space topology described e.g. in \cite{treves2006topological}.
The corresponding isomorphism is given by
\begin{align}\label{eq:mu forme lin}
\mu \longmapsto 
\begin{cases}
\mathcal{C}_c(\mathcal{D}) &\longrightarrow \mathbb{R}\\
\ \ f &\longmapsto \int_{\mathcal{D}}f(x) \mu(dx) 
\end{cases}
\end{align}
We have the following facts :
\begin{itemize}
\item any signed measure that admits a density $f$ w.r.t. the Lebesgue measure such that $f \in L^1_{loc}(\mathcal{D})$ is a Radon measure (\cite{treves2006topological},p.217).
\item mimicking \eqref{eq:fL1 loc distrib} and \eqref{eq:mu forme lin}, any Radon measure can be injectively identified to a generalized function with the following identification \cite{treves2006topological}
\begin{align}\label{eq:mu distrib}
\forall \varphi \in \mathscr{D}(\mathcal{D}), \ \ \ \langle \mu, \varphi \rangle := \int_{\mathcal{D}} \varphi(x)\mu(dx) 
\end{align}
In particular, Radon measures can be differentiated up to any order through equation \eqref{eq:diff distrib}.
\item for any Radon measure $\mu$, there is a unique couple $(\mu^+,\mu^-)$ of positive Radon measures such that $\mu = \mu^+ - \mu^-$  (\cite{lang1993}, Chapter IX). We then define $|\mu|$ as
\begin{align}\label{eq:abs mu}
|\mu| := \mu^+ + \mu^- 
\end{align}
\item If $\mu$ and $\nu$ are two finite Radon measures over $\mathbb{R}^d$ (i.e. $\int_{\mathbb{R}^d}|\mu|(dx) < \infty$ and likewise for $\nu$), their convolution $\mu *\nu$ is defined as follow : let $\mathcal{B}(\mathbb{R}^d)$ be the Borel $\sigma-$algebra of $\mathbb{R}^d$, then
\begin{align}\label{eq:conv mesures}
\forall A \in \mathcal{B}(\mathbb{R}^d), \ \ \ (\mu * \nu)(A) = \int_{\mathbb{R}^d}\int_{\mathbb{R}^d} \mathbbm{1}_A(x+y)\mu(dx)\nu(dy)
\end{align}
and $\mu * \nu$ is also a Radon measure over $\mathbb{R}^d$. When $\mu$ and $\nu$ have densities $f_{\mu}$ and $f_{\nu}$, $\mu * \nu$ has the density $f_{\mu}*f_{\nu}$ defined by %$(f_{\mu}*f_{\nu})(x) = \int_{\mathbb{R}^d}f_{\mu}(y)f_{\nu}(x-y)dy$.
\begin{align*}
(f_{\mu}*f_{\nu})(x) = \int_{\mathbb{R}^d}f_{\mu}(y)f_{\nu}(x-y)dy
\end{align*}
\end{itemize}
\begin{remark}
What is meant behind the terminology of Radon measures varies between authors. \cite{evans2018measure} calls Radon measure what we call \textit{positive} Radon measure in this article. \cite{lang1993} proves that continuous linear forms over $\mathcal{C}_c(\mathcal{D})$ are differences of Radon measures in the sense of the Radon measures defined in \cite{evans2018measure}, but \cite{lang1993} never uses the term of Radon measures, positive of not. Likewise, \cite{treves2006topological} calls \textit{positive} Radon measure any \textit{positive} linear form over $\mathcal{C}_c(\mathcal{D})$ which, thanks to the proof from \cite{lang1993}, reduces to Radon measures in the sense of \cite{evans2018measure}.
\end{remark}

\subsubsection{Finite order generalized functions}
Let $k$ be a non negative integer, we consider $\mathcal{C}_c^k(\mathcal{D})$ the space of compactly supported functions of class $\mathcal{C}^k$ endowed with its usual LF-space topology \cite{rudin1991}. We denote $\mathcal{C}_c^k(\mathcal{D})'$ its topological dual. The topologies of $\mathcal{C}_c^k(\mathcal{D})$ and $\mathscr{D}(\mathcal{D})$ are such that the canonical injection $\mathscr{D}(\mathcal{D}) \rightarrow \mathcal{C}^k(\mathcal{D})$ is continuous \cite{treves2006topological}, which yields that $\mathcal{C}^k(\mathcal{D})' \subset \mathscr{D}'(\mathcal{D})$ : continuous linear forms over $\mathcal{C}^k(\mathcal{D})$, when restricted to $\mathscr{D}(\mathcal{D})$, become continuous linear forms over $\mathscr{D}(\mathcal{D})$, i.e. generalized functions. We then have the following definitions and facts.
\begin{itemize}
\item Generalized functions $T \in \mathscr{D}'(\mathcal{D})$ that are restrictions of continuous linear forms over $\mathcal{C}_c^k(\mathcal{D})$ are called generalized functions of order $k$. If $T$ is of order $k$ for some $k \in \mathbb{N}$, $T$ is said to be of finite order.
\item $T \in \mathscr{D}'(\mathcal{D})$ is at most of order $k$ if in equation \eqref{eq:def distrib continue}, the integer $n_K$ can always be taken to be equal to $n$, whatever the compact set $K$.
\item Let $T$ be a generalized function of order $k$. Then (\cite{treves2006topological}, p 259) there exists a family of Radon measures $\{\mu_p\}_{|p|\leq k}$ over $\mathcal{D}$ such that
\begin{align}\label{eq:finite order radon}
T = \sum_{|p|\leq k}\partial^p \mu_p
\end{align}
where the equality in \eqref{eq:finite order radon} holds in $\mathscr{D}'(\mathcal{D})$ and $\mathcal{C}_c^k(\mathcal{D)}'$. Note that we recover \eqref{eq:mu distrib} when $k = 0$.
\item Among the finite order generalized functions are those that are \textit{compactly supported}, i.e. those for which the measures $\mu_p$ such that $T = \sum_{|p|\leq k}\partial^p \mu_p$ all have compact support. One property is that one can define the Fourier transform of any compactly supported generalized functions \cite{rudin1991}.
\end{itemize}

\subsubsection{Convolution with generalized functions} Let $k$ be a non negative integer. As above, we consider $\mathcal{C}_c^k(\mathbb{R}^d)$ endowed with its usual topology. Let $f \in \mathcal{C}_c^k(\mathbb{R}^d)$ and $T \in \mathcal{C}_c^k(\mathbb{R}^d)'$. Note $\tau_x f$ the function $y \longmapsto f(y-x)$ and $\check{f}$ the function $y \longmapsto f(-y)$. Then \cite{treves2006topological} one may define the convolution between $T$ and $f$ by
\begin{align}\label{eq:conv distrib}
T * f : x \longmapsto \langle T, \tau_{-x} \check{f} \rangle
\end{align}
and $T * f$ is a function in the classical sense, i.e. defined pointwise. When $T$ lies in $L_{loc}^1(\mathcal{D})$=, equation \eqref{eq:conv distrib} reduces to the usual convolution of functions through the identification defined in equation \eqref{eq:fL1 loc distrib} :
\begin{align}
(T * f)(x) = \int_{\mathbb{R}^d} T(y)f(x-y)dy
\end{align}
Similarly if $T$ is in fact a Radon measure $\mu$ :
\begin{align}\label{eq:conv_mesure_func}
(T * f)(x) = \int_{\mathbb{R}^d} f(x-y)\mu(dy)
\end{align} 
More general definitions of generalized function convolution are available \cite{treves2006topological} but this one is sufficient for our use.

\subsubsection{Tensor product of generalized functions} For two generalized functions $T_1 \in \mathscr{D}'(\mathcal{D}_1)$ and $T_2 \in \mathscr{D}'(\mathcal{D}_2)$, $T_1 \otimes T_2 \in \mathscr{D}'(\mathcal{D}_1 \times \mathcal{D}_2)$ denotes their tensor product \cite{treves2006topological}, which is uniquely determined by the following tensor property :
\begin{align}
\forall \varphi_1 \in \mathscr{D}(\mathcal{D}_1), \forall \varphi_2 \in \mathscr{D}(\mathcal{D}_2), \ \langle T_1 \otimes T_2, \varphi_1 \otimes \varphi_2 \rangle = \langle T_1, \varphi_1 \rangle \times \langle T_2, \varphi_2 \rangle 
\end{align}

$T_1 \otimes T_2$ reduces to the tensor product of functions when $T_1$ and $T_2$ are functions through the identification of equation \eqref{eq:fL1 loc distrib}, and the product measure when $T_1$ and $T_2$ are Radon measures through \eqref{eq:mu distrib}.

%over $\mathcal{D}$ any real valued measure that has finite absolute mass over any compact set \cite{lang1993}, additionally to other standard regularity hypotheses from measure theory as in \cite{lang1993}.  \cite{bourbaki2004} shows that the space of Radon measures is isomorphic to the space of continuous linear forms on 

%% file: 3_GP_diff.tex
\section{Stochastic processes under linear differential constraints}\label{section:sto diff}
One may wish to force the trajectories of a stochastic process $U = (U(x))_{x\in \mathcal{D}}$ to verify linear constraints, i.e. to lie in the kernel of some linear operator. This is a priori an ambitious task as the trajectories of $U$ form a vast set of functions. However, if  $U$ is a second order stochastic process (i.e. $\forall x \in \mathcal{D}, \text{Var}\big( U(x)\big) <+ \infty$), then in many cases linear constraints on the trajectories of $U$ can be completely translated as linear constraints on the covariance kernel of $U$. In particular, these new linear constraints are imposed on a much smaller set of accessible ``explicit" functions. Overall, the resulting constraints on the covariance kernel of $U$ are much easier to handle than the constraints on the trajectories of $U$. This idea was thoroughly explored in \cite{ginsbourger2016}, where different general frameworks were studied in order to formulate mathematical results on linearly constrained stochastic processes. In proposition \ref{prop : lin constraints classic}, we recall a particular result from \cite{ginsbourger2016} that was then applied to the stationary heat equation in the same article.

Note $\mathcal{F}(\mathcal{D},\mathbb{R})$ the space of real-valued functions defined on $\mathcal{D}$. Proposition \ref{prop : lin constraints classic} is based on the so called Loève isometry \cite{agnan2004} between $\mathcal{L}(U)$ and $L^2(\mathbb{P})$ (see \ref{not : gp et al} for notations), which in turn leads to the following theorem.
\begin{proposition}[Trajectories of GPs under linear constraints \cite{ginsbourger2016}]\label{prop : lin constraints classic}
Let $\big(U(x)\big)_{x\in \mathcal{D}}  \sim GP(0,k)$ be a centered GP. Note for all $x \in \mathcal{D}$ the function $k_x : y \longmapsto k(x,y)$. Let $E$ be a real vector space of functions defined on $\mathcal{D}$ that contains the trajectories of $U$ almost surely and $T : E \longrightarrow \mathcal{F}(\mathcal{D},\mathbb{R})$ be a linear operator. Suppose that for all $x \in \mathcal{D}, T(U)(x) \in \mathcal{L}(U)$. Then there exists a unique linear operator $\mathcal{T} : \mathcal{H}_k \longrightarrow \mathcal{F}(\mathcal{D},\mathbb{R})$ such that for all $x,x' \in \mathcal{D},$
\begin{align*}
    \mathbb{E}[T(U)(x)U(x')] = \mathcal{T}(k_{x'})(x)
\end{align*}
and $\forall x \in \mathcal{D}, \forall h_n \xrightarrow[]{\mathcal{H}_k} h, \ \mathcal{T}(h_n)(x) \longrightarrow \mathcal{T}(h)(x)$.
Moreover, the following statements are equivalent :
\begin{enumerate}[label=(\roman*)]
    \item $\mathbb{P}(\{\omega  \in \Omega : T(U_{\omega})= 0\}) = 1$
    \item $\forall x \in \mathcal{D}, \mathcal{T}(k_{x}) = 0$
\item $\mathcal{T}(\mathcal{H}_k) = \{0\}$
\end{enumerate}
\end{proposition}

This theorem can be applied when $T$ is a differential operator as discussed in \cite{ginsbourger2016}. However, in Proposition \ref{prop : lin constraints classic}, the differential operator $T$ of order $n$ has to be valued in the space of (classical, pointwise defined) functions $\mathcal{F}(\mathcal{D},\mathbb{R})$; in particular for $u \in E$, the function $T(u)$ has to be defined pointwise in order to use the Loève isometry. To summarize, in all generality the derivatives in $T$ have to be understood in a classical sense and $E$ has to be contained in $\mathscr{D}^n(\mathcal{D})$, the space of $n$ times differentiable functions on $\mathcal{D}$. Requiring that $E \subset \mathscr{D}^n(\mathcal{D})$ is a very strong assumption w.r.t. the trajectories of $U$; furthermore, this is not compliant with the usual way of studying PDEs where derivatives are understood in a weaker sense. We present in Proposition \ref{prop : diff constraints} an adaptation of Proposition \ref{prop : lin constraints classic} where we make use of the distributional definition of derivatives and relax the assumptions made on  $U$ and its trajectories.  In this proposition, $\big(U(x)\big)_{x\in \mathcal{D}}$ is not supposed Gaussian and is only required to be second order. We refer to the notation paragraphs \ref{not : function spaces} and \ref{not : derivatives}.
%\begin{lemma}
%Suppose that $\mathcal{D} \subset \mathbb{R}^d$ is an open set and $T = \sum_{|k| \leq n} a_k(x) \partial^k$ be a linear differential operator with $\mathcal{C}^{\infty}$ coefficients. We set $E = \mathcal{C}^k(\mathcal{D})$. Suppose that the trajectories of $U$ belong to $E$ with probability 1. Then $\mathcal{H}_k \subset \mathcal{C}^k(\mathcal{D})$ and for the $\mathcal{T}$ in Proposition \ref{prop : lin constraints classic}, $\mathcal{T} = T_{|\mathcal{H}_k}$.
%\end{lemma}

\begin{proposition}[Trajectories of stochastic processes under linear differential constraints] \label{prop : diff constraints} Let $\mathcal{D} \subset \mathbb{R}^d$ be an open set and let $T = \sum_{|k| \leq n} a_k(x) \partial^k$ be a linear differential operator with coefficients $a_k(x) \in \mathcal{C}^{|k|}(\mathcal{D})$. Let $U = \big(U(x)\big)_{x\in \mathcal{D}}$ be a second order stochastic process with mean function $m(x)$ and covariance kernel $k(x,x')$. For all $x \in \mathcal{D}$, note ${k_{x} : y \longmapsto k(x,y)}$. Suppose that its mean function $m$ lies in $L^1_{loc}(\mathcal{D})$ as well as its standard deviation function ${\sigma : x \longmapsto \sqrt{k(x,x)}}$. \\
1) Then on a set of probability $1$, the trajectories of $U$ lie in $L^1_{loc}(\mathcal{D})$ as well as the functions $k_{x}$ for all $x \in \mathcal{D}$. \\ %Thus these functions as well as $x \longmapsto m(x)$ can all be interpreted as generalized functions through equation \eqref{eq:fL1 loc distrib}. 
2) Suppose that $T(m) = 0$ in the sense of distributions. Then the following statements are equivalent :
\begin{enumerate}[label=(\roman*)]
    \item $\mathbb{P}(T(U) = 0 \text{ in the sense of distributions}) = 1$
    \item $\forall x \in \mathcal{D}, T(k_{x}) = 0$ in the sense of distributions.
\end{enumerate}
\end{proposition}
Here we write down precisely what we mean by $(i)$ and $(ii)$. Note $T^*$ the formal adjoint of $T$ defined by $T^*u = \sum_{|k|\leq n} (-1)^{|k|}\partial^k(a_k(x)u)$. By $(i)$, we mean that
\begin{align}\label{eq:meaning i}
    \exists A \in \mathcal{F},\ \ \mathbb{P}(A) = 1,\ \ \forall \omega \in A,\ \ \forall \varphi \in \mathscr{D}(\mathcal{D}),\ \ \langle U_{\omega}, T^*\varphi \rangle = \int_{\mathcal{D}}U_{\omega}(x)T^*\varphi(x)dx = 0  
\end{align}
Similarly, $(ii)$ means that
\begin{align}\label{eq:meaning ii}
\forall x \in \mathcal{D}, \ \ \forall \varphi \in \mathscr{D}(\mathcal{D}),\ \ \langle k_x, T^*\varphi \rangle = \int_{\mathcal{D}}k_x(y)T^*\varphi(y)dy = 0  
\end{align}
This definition can be found e.g. in \cite{Hrmander1990TheAO}. The fact that the functions $x \longmapsto U_{\omega}(x)$ and $y \longmapsto k_x(y)$ lie in $L^1_{loc}(\mathcal{D})$ ensure the existence of the integrals in equations \eqref{eq:meaning i} (see point 2 of the proof of Proposition \ref{prop : diff constraints}) as well as the continuity of the associated linear forms over $\mathscr{D}(\mathcal{D})$, following the definition \eqref{eq:fL1 loc distrib}. In every case, the term ``in the sense of distributions" can be replaced by ``in $\mathscr{D}'(\mathcal{D})$" : stating that $T(f) = 0$ in the sense of distributions means that $T(f)$, seen as an element of $\mathscr{D}'(\mathcal{D})$, is equal to the null generalized function $0_{\mathscr{D}'(\mathcal{D})} : \varphi \longmapsto 0$.

\begin{proof}
Suppose first that $U$ is centered, i.e. $m \equiv 0$. \\
1) We begin by showing that the trajectories of $U$ almost surely lie in $L^1_{loc}(\mathcal{D})$. Note first that thanks to the Cauchy-Schwarz inequality, $\mathbb{E}[|U(x)|] \leq \sigma(x)$.
%\sqrt{\mathbb{E}[1]}\sqrt{\mathbb{E}[U(x)^2]} =
Now, let $(K_n)_{n \in \mathbb{N}}$ be an increasing sequence of compact subsets of $\mathcal{D}$ such that $\bigcup_{n \in \mathbb{N}} K_n = \mathcal{D}$. Using Tonelli's theorem, we have that for any $n \in \mathbb{N}$,
\begin{align}
\mathbb{E}\big[\int_{K_n}|U(x)|dx\big] = \int_{K_n}\mathbb{E}[|U(x)|]dx \leq \int_{K_n} \sigma(x)dx < +\infty
\end{align}
since $\sigma \in L^1_{loc}(\mathcal{D})$. Using the property that ``$\mathbb{E}[|X|] < +\infty \implies |X| < + \infty$ almost surely", this yields a set $B_n \subset \Omega$ of probability 1 such that the random variable
$\omega \longmapsto \int_{K_n}|U_{\omega}(x)|dx$ takes finite values over $B_n$. Consider now the set $B = \bigcap_{n \in \mathbb{N}}B_n$ which remains of probability 1. For all compact subset $K \subset \mathcal{D}$, there exists an integer $n_K$ such that $K \subset K_{n_K}$ and thus for all $\omega \in B$,
\begin{align*}
\int_K |U_{\omega}(x)|dx \leq \int_{K_{n_K}} |U_{\omega}(x)|dx < + \infty
\end{align*}
which shows that the trajectories of $U$ lie in $L^1_{loc}(\mathcal{D})$ almost surely.

Now, we check that for all $x \in \mathcal{D},  k_{x} \in L^1_{loc}(\mathcal{D})$ : for any compact set $K$, since $\sigma \in  L^1_{loc}(\mathcal{D})$ and because of \eqref{eq:CS PD},

\begin{align*}
\int_K |k_{x}(y)|dy = \int_K |k(x,y)|dy \leq \sigma(x) \int_K \sigma(y)dy < \infty
\end{align*}

2) Let us check in advance that whatever $f \in L^1_{loc}(\mathcal{D})$, the map $T(f) : \varphi \longmapsto \langle f, T^*\varphi \rangle$ is a continuous linear form over $\mathscr{D}(\mathcal{D})$. Since $a_k \in \mathcal{C}^k(\mathcal{D})$, we can apply Leibniz' rule on $T^*\varphi = \sum_{|k|\leq n} (-1)^{|k|}\partial^k(a_k\varphi)$. This yields a family $\{f_k\}_{|k| \leq n}$ of continuous functions over $\mathcal{D}$ such that
\begin{align}\label{eq:les f_k}
\forall \varphi \in \mathscr{D}(\mathcal{D}), \ \ \forall x \in \mathcal{D},\ \ T^*\varphi(x) = \sum_{|k|\leq n}f_k(x) \partial^k\varphi(x)
\end{align}
For all $f \in L^1_{loc}(\mathcal{D})$, for all compact set $K \subset \mathcal{D}$ and for all $\varphi \in \mathscr{D}(\mathcal{D})$ such that $\text{Supp}(\varphi) \subset K$, \eqref{eq:les f_k} yields
\begin{align}\label{eq:control inf} 
|\langle f, T^*\varphi \rangle| &\leq \int_{\mathcal{D}}|f(x)||T^*\varphi(x)|dx \nonumber \\ 
&\leq \bigg( \int_{K}|f(x)|dx \times \max_{|k|\leq n} \sup_{x \in K} |f_k(x)|\bigg) \sum_{|k| \leq n} ||\partial^k\varphi ||_{\infty} < +\infty
\end{align}
This proves that $T(f) : \varphi \longmapsto \langle f, T^*\varphi \rangle$ is a continuous linear form over $\mathscr{D}(\mathcal{D})$ (see equation \eqref{eq:def distrib continue}). %, which will be useful in the proof of Proposition \ref{prop : diff constraints}.

\underline{$(i) \implies (ii)$} : Suppose $(i)$. Let $\varphi \in \mathscr{D}(\mathcal{D})$. There exists a  set $A\subset \Omega$ such that $\mathbb{P}(A) = 1$ and such that for all $\omega \in A$,
\begin{align*}
    \int_{\mathcal{D}}U_{\omega}(x)T^*\phi(x)dx = 0
\end{align*} 
Multiplying equation above with $U_{\omega}(x')$, taking the expectancy and formally permuting (for now) the integral and the expectancy, we obtain
\begin{align*}
   0 &= \mathbb{E}\Bigg[U(x')\int_{\mathcal{D}}U(x)T^*\varphi(x)dx\Bigg] = \int_{\mathcal{D}}T^*\varphi(x)\mathbb{E}[U(x)U(x')]dx \\
     &= \int_{\mathcal{D}}T^*\varphi(x)k(x,x')dx = \langle k_{x'},T^*\varphi \rangle
\end{align*}
The integral-expectancy permutation is justified by writing down the expectancy as an integral and using Fubini's theorem, checking that the below quantity is finite. We use Tonelli's theorem and the Cauchy-Schwarz inequality :
\begin{align*}
    \mathbb{E}\Bigg[\int_{\mathcal{D}}|U(x')U(x)T^*\varphi(x)|dx\Bigg] &= \int_{\mathcal{D}}|T^*\varphi(x)|\mathbb{E}[|U(x)U(x')|]dx\\ 
    &\leq \int_{\mathcal{D}}|T^*\varphi(x)|\mathbb{E}[U(x)^2]^{1/2}\mathbb{E}[U(x')^2]^{1/2}dx \\
    &\leq \sigma(x')\int_{\mathcal{D}}|T^*\varphi(x)|\sigma(x) dx < +\infty
\end{align*}
Indeed, since $\sigma \in L^1_{loc}(\mathcal{D})$, setting $f = \sigma$ in \eqref{eq:control inf} shows that the last integral is finite.
Thus, $\forall x \in \mathcal{D}, \forall \varphi \in \mathscr{D}(\mathcal{D}), \langle k_{x}, T^*\varphi \rangle = 0$ which proves that $(i) \implies (ii)$. \\
\underline{$(ii) \implies (i)$} : Suppose $(ii)$. Let $\varphi \in \mathscr{D}(\mathcal{D})$, we have $\langle k_{x'}, T^*\varphi \rangle = 0$. Multiplying this with $T^*\varphi(x')$ and integrating w.r.t. $x'$ yields
\begin{align*}
    0 = \int_{\mathcal{D}}T^*\varphi(x') \int_{\mathcal{D}}T^*\varphi(x)k(x,x')dxdx' =  \int_{\mathcal{D}} \int_{\mathcal{D}}T^*\varphi(x)T^*\varphi(x')\mathbb{E}[U(x)U(x')]dxdx'
\end{align*}
Permuting formally the expectancy and the integrals (justified in equation \eqref{eq:justify fubini prop2}) yields
\begin{align*}
   0 &=  \int_{\mathcal{D}} \int_{\mathcal{D}}T^*\varphi(x)T^*\varphi(x')\mathbb{E}[U(x)U(x')]dxdx' \\
    &= \mathbb{E}\Bigg[\Bigg(\int_{\mathcal{D}}T^*\varphi(x)U(x)dx\Big)^2\Bigg] = \mathbb{E}[\langle U,T^*\varphi\rangle ^2]
\end{align*}
and thus $\langle U,T^*\varphi\rangle = 0 $ a.s. : there exists $A_{\varphi} \in \mathcal{F}$ with $\mathbb{P}(A_{\varphi}) = 1$ such that $\forall \omega \in A_{\varphi}, \langle U_{\omega}, T^*\varphi \rangle = 0$. We justify the expectancy-integral permutation with the computation below
\begin{align} \label{eq:justify fubini prop2}
    \int_{\mathcal{D}}\int_{\mathcal{D}}|T^*\varphi(x)&T^*\varphi(x')|\mathbb{E}[|U(x)U(x')|]dxdx' \\
    &\leq \int_{\mathcal{D}}\int_{\mathcal{D}}|T^*\varphi(x)T^*\varphi(x')|\sigma(x)\sigma(x')dxdx' \nonumber \\
     &\leq \Bigg(\int_{\mathcal{D}}|T^*\varphi(x)|\sigma(x)dx \Bigg)^2 < + \infty \nonumber
\end{align}
As previously, setting $f = \sigma$ in \eqref{eq:control inf} shows that the integral above is finite.

This does not finish the proof as we need to find a set $A$ with $\mathbb{P}(A) = 1$, independently from $\varphi$, such that $\forall \omega \in A, \langle U_{\omega}, T^*\varphi \rangle = 0$. %The obvious candidate is $\bigcap_{\varphi \in \mathscr{D}(\mathcal{D})}A_{\varphi}$ but sets of probability 1 are not closed under uncountable intersections. However, they are for countable intersections. 
For this we use the fact that $\mathscr{D}(\mathcal{D})$ is a separable topological space, which we prove at the end of this proof. Let $F \subset \mathscr{D}(\mathcal{D})$ be a countable dense subset of $\mathscr{D}(\mathcal{D})$, let $A := B \cap\big( \bigcap_{\varphi \in F}A_{\varphi}\big)$ and let $\omega \in A$. Since $U_{\omega} \in L^1_{loc}(\mathcal{D})$, \eqref{eq:control inf} shows that the map $ L_{\omega} : \varphi \longmapsto \langle U_{\omega},T^*\varphi\rangle$ is a continuous linear form on $\mathscr{D}(\mathcal{D})$. Even though $\mathscr{D}(\mathcal{D})$ is not a metric space, theorem 2.1.4 from \cite{Hrmander1990TheAO} states that $L_{\omega}$ is in fact \textit{sequentially} continuous. Let $\varphi \in \mathscr{D}(\mathcal{D})$ and a sequence $(\varphi_n) \subset F$ such that $\varphi_n \rightarrow \varphi$ (in the sense of the topology of $\mathscr{D}(\mathcal{D})$). Then $L_{\omega}(\varphi) = \text{lim}_n  L_{\omega}(\varphi_n) = 0$ since $\forall n, L_{\omega}(\varphi_n) = 0$. That is, we have proved that
\begin{align*}
    \forall \omega \in A,\ \ \forall \varphi \in \mathscr{D}(\mathcal{D}),\ \ \langle U_{\omega}, \ \ T^*\varphi \rangle = L_{\omega}(\varphi) = 0
\end{align*}
Since $\mathbb{P}(A) = 1$, this shows that $(ii) \implies (i)$. 

When $U$ is not centered, consider the centered stochastic process $V$ defined by $V(x) = U(x) - m(x)$ for which the above proof can be applied. Since $T$ is linear and $m$ is supposed to verify $T(m) = 0$ in the sense of distributions, the probabilistic events $\{T(U) = 0 \text{ in the sense of distributions}\}$ and $\{T(V) = 0 \text{ in the sense of distributions}\}$ coincide and thus have the same probability measure.
%\begin{align}
%\mathbb{P}(T(U) = 0 \text{ in the sense of distributions}) = \mathbb{P}(T(V) = 0 \text{ in the sense of distributions})
%\end{align}
Finally, $U$ and $V$ have the same covariance kernel $k(x,x')$. Thus, 
\begin{align*}
&\mathbb{P}(T(U) = 0 \text{ in the sense of distributions}) = 1 \\
&\iff \mathbb{P}(T(V) = 0 \text{ in the sense of distributions}) = 1 \iff \forall x \in \mathcal{D}, T(k_{x}) = 0
\end{align*}
which finishes the proof in the general case.
\\

\textit{Proof that $\mathscr{D}(\mathcal{D})$ is separable : } $\mathscr{D}(\mathcal{D})$ is an LF-space as the inductive limit of the Fréchet spaces $\mathscr{D}_{K_i}(\mathcal{D}):= \{\varphi \in C^{\infty}(\mathcal{D}) : \text{Supp}(\varphi) \subset K_i\}, i \in \mathbb{N}$, where $K_1 \subset K_2 \subset ...$ are compact subsets of $\mathcal{D}$ such that $\bigcup_i K_i = \mathcal{D}$ (\cite{treves2006topological}, p.131-133). As such, $\mathscr{D}(\mathcal{D})$ is separable iff $\mathscr{D}_{K_i}(\mathcal{D})$ is separable for all $i \in \mathbb{N}$ \cite{vidossich1968}, which we now show. The Fréchet topology of $\mathscr{D}_{K_i}(\mathcal{D})$ is the one induced by the usual Fréchet topology of $C^{\infty}(\mathcal{D})$ when $\mathscr{D}_{K_i}(\mathcal{D})$ is seen as a subspace of $C^{\infty}(\mathcal{D})$ (\cite{rudin1991}, section 1.46). As a Fréchet space, $C^{\infty}(\mathcal{D})$ is metrizable. But $C^{\infty}(\mathcal{D})$ is also a Montel space (\cite{treves2006topological}, Prop 34.4) : as a metrizable space, it is automatically separable (\cite{schaefer1999}, p.195). Thus $\mathscr{D}_{K_i}(\mathcal{D})$ is also separable as a subspace of the separable metric space $C^{\infty}(\mathcal{D})$.
\end{proof}
%\modif{dire que on retrouve la prop 1 si T agit sur des fonctions assez régulières + discussions}
%\modif{corolliare sur $\tilde{m}$ et $\tilde{k}$}
%However, this may be equivalent to other a priori stronger forms of solutions for linear PDEs. For instance, this is the case for the weak solutions of the Poisson equation with homogeneous Dirichlet boundary conditions, which is the archetype of linear elliptic PDEs. This is not true anymore for non linear PDEs and 

\begin{remark}
Distributional solutions are the weakest types of solutions for PDEs. In general, additional regularity conditions have to be imposed to obtain physically realistic solutions, such as Sobolev regularity or entropy conditions as for non linear hyperbolic PDEs \cite{Serre1999SystemsOC}. However, every step in the above proof remains valid when replacing $\varphi \in \mathscr{D}(\mathcal{D})$ with $\varphi \in \mathcal{C}_c^n(\mathcal{D})$. Although we have not clarified the usual topology of $\mathcal{C}_c^n(\mathcal{D})$ in this article, we state that this is enough to show that the equalities stated in Proposition \ref{prop : diff constraints} also hold in $\mathcal{C}_c^n(\mathcal{D})'$, the space of finite order generalized functions of order $n$, rather than just in $\mathscr{D}'(\mathcal{D})$. $\mathcal{C}_c^n(\mathcal{D})'$ is a smaller space than $\mathscr{D}'(\mathcal{D})$, though less used in PDE theory than $\mathscr{D}'(\mathcal{D})$.
\end{remark}

\begin{remark}
We gave here an elementary proof that
\begin{align}\label{eq:sigma L1 loc traj}
\sigma \in L^1_{loc}(\mathcal{D}) \implies \text{the trajectories of $U$ lie in \ } L^1_{loc}(\mathcal{D}) \text{\ almost surely}
\end{align}
Similar results on Sobolev regularity of the trajectories of second order stochastic processes are scarce in the literature. Some are available in \cite{SCHEUERER2010}, though the result \eqref{eq:sigma L1 loc traj} is actually not covered in \cite{SCHEUERER2010}, where additional continuity hypotheses would be required in the left hand side of \eqref{eq:sigma L1 loc traj} to apply results from \cite{SCHEUERER2010}.
\end{remark}
We partially recover Proposition \ref{prop : lin constraints classic} when the trajectories of $U$ lie in $\mathcal{C}^n(\mathcal{D})$ and $k \in \mathcal{C}^{n,n}(\mathcal{D}\times \mathcal{D})$. Indeed, in that case one can show that if $T = \sum_{|k|\leq n} a_k(x)\partial^k$, then we simply have $\mathcal{T} = T$ in Proposition \ref{prop : lin constraints classic}. Additionally, $T(U_{\omega})$ and $T(k_x)$ both lie in $\mathcal{F}(\mathcal{D},\mathbb{R}) \cap L^1_{loc}(\mathcal{D})$, and for any function $g$ that lies in $L^1_{loc}(\mathcal{D})$, we have (see \ref{not : as ae})
\begin{align}\label{eq:g 0 ae}
g = 0 \text{\ in the sense of distributions} \iff
g = 0 \ a.e.
\end{align}
Equation \eqref{eq:g 0 ae} is just another way of saying that the linear map $f \longmapsto T_f$ given in \eqref{eq:fL1 loc distrib} is injective. In that framework, Proposition \ref{prop : lin constraints classic} states that
\begin{align}\label{eq:res prop 1}
\forall x \in \mathcal{D}, \ T(k_x) = 0 \iff \mathbb{P}(T(U) = 0) = 1
\end{align}
where the function equalities of the form $T(f) = 0$ in \eqref{eq:res prop 1} are valid everywhere on $\mathcal{D}$. Following equation \eqref{eq:g 0 ae}, Proposition \ref{prop : diff constraints} states a slightly weaker result, namely that
\begin{align}
\forall x \in \mathcal{D}, \ \ T(k_x) = 0 \ a.e. \iff \mathbb{P}(T(U) = 0 \ a.e.) = 1
\end{align}

We can now state the following corollary, which draws the consequences of Proposition \ref{prop : diff constraints} when applied to GPR.

\begin{proposition}[Heredity of Proposition \ref{prop : diff constraints} to conditioned GPs]\label{prop : inheritance pde conditioned}
Let $\mathcal{D}$ and $T$ be as defined in Proposition \ref{prop : diff constraints}. Let $(U(x))_{x\in \mathcal{D}} \sim GP(m,k)$ be a Gaussian process that verifies the hypotheses of Proposition \ref{prop : diff constraints}.
Suppose also that
\begin{align}\label{eq:m k_x sol}
T(m) = 0 \text{ \ and \ } \forall x \in \mathcal{D},\ T(k_x) = 0 \text{ \ both in the sense of distributions}
\end{align}
$(i)$ Then whatever the integer $p$, the vector $u = (u_1,...,u_p)^T \in \mathbb{R}^p$ and the vector $X = (x_1,...,x_p)^T \in \mathcal{D}^p$ such that $k(X,X)$ is invertible, the Kriging mean $\tilde{m}(x)$ and the Kriging standard deviation function $\tilde{\sigma}(x) = \sqrt{\tilde{k}(x,x)}$ both lie in $L^1_{loc}(\mathcal{D})$, and we have
\begin{align*}
T(\tilde{m}) = 0 \text{ \ and \ } \forall x \in \mathcal{D}, \ T(\tilde{k}_x) = 0 \text{ \ both in the sense of distributions}
\end{align*}
where $\tilde{m}$ and $\tilde{k}$ are defined in equations \eqref{eq:krig mean} and \eqref{eq:krig cov}. \\
$(ii)$ As such, the trajectories of the conditioned Gaussian process $\big(\tilde{U}(x)\big)_{x\in \mathcal{D}}$ defined by $\tilde{U}(x) = (U(x)|U(x_i) = u_i \ \forall i = 1,...,p)$ are almost surely solutions of the equation $T(f) = 0$ in the sense of distributions :
\begin{align*}
\mathbb{P}(T(\tilde{U}) = 0 \text{ in the sense of distributions}) = 1
\end{align*}
\end{proposition}

\begin{proof}
Note first that for all $x \in \mathcal{D}, \tilde{k}(x,x) \leq k(x,x)$, which is immediate from \eqref{eq:tilde k proj}. Thus the function $\tilde{\sigma} : x \longmapsto \sqrt{\tilde{k}(x,x)}$ also lies in $L^1_{loc}(\mathcal{D})$.
Point $(i)$ is then a direct consequence of the definition of $\tilde{m}$ and $\tilde{k}$ in equations \eqref{eq:krig mean} and \eqref{eq:krig cov}, and the linearity of $T$.  %let $A \subset \Omega$ be such that $\mathbb{P}(A) = 1$ and $\forall \omega \in A, U_{\omega} \in L^1_{loc}(\mathcal{D})$; for all $x \in \mathcal{D}$, define the conditioned variable $U_A(x) = (U(x)|A)$. Then the Gaussian processes $(U(x))_{x\in \mathcal{D}}$ and $(U_A(x))_{x\in \mathcal{D}}$ are equal almost everywhere.
Proposition \ref{prop : diff constraints} can then be applied conjointly with $(i)$, which yields point $(ii)$ since the mean and covariance functions of the GP $\tilde{U}$ are $\tilde{m}$ and $\tilde{k}$ (see section \ref{sub : GP}, equations \eqref{eq:krig mean} and \eqref{eq:krig cov}).
\end{proof}
 %More practically, one may only wish to ensure that whatever the observation points $(x_1,...,x_n)$, the prediction $\Tilde{m}$ from eq.\eqref{eq:krig mean} verify $T(\Tilde{m}) = 0$, but this corresponds to point $(ii)$ of Proposition \ref{prop : diff constraints} : there is no loss of generality between requiring that all paths of $U$ lie in the null space of $T$ and only requiring that the kriging means $\Tilde{m}$ belong to the null space of $T$.
Proposition \ref{prop : inheritance pde conditioned} shows that when $U$ is a GP, the results of Proposition \ref{prop : diff constraints} are inherited on the conditioned posterior process $\tilde{U}$. One weak consequence of Proposition \ref{prop : inheritance pde conditioned} is that if GPR is performed with a kernel $k$ that verifies point $(ii)$ of Proposition \ref{prop : diff constraints}, then the predictions provided by GPR are all solutions of the PDE $T(\tilde{m}) = 0$.

The goal of the next section is to apply this idea  to a special case of the (3 dimensional) wave equation defined in eq.\eqref{eq:wave_eq}, by building an ``explicit" positive definite kernel $k$ such that $\forall x' \in \mathcal{D}, {\Box k_{x'} = 0}$ in the sense of distributions, where the box symbol $\Box$ classically denotes the linear wave operator a.k.a. the d'Alembert operator. With this new kernel, we will perform GPR on observations of a function that is solution to the wave equation and draw a number of related consequences.

%\subsection{Solving Linear PDEs with Gaussian Process Regression}
%

%% file: 4_3D_wave.tex
\section{Gaussian Processes and the 3 Dimensional Wave Equation}\label{section:GP_wave}

\subsection{General Solution to the 3 Dimensional Wave Equation}
Denote the 3D Laplace operator $\Delta = \partial_{xx}^2 + \partial_{yy}^2 + \partial_{zz}^2$ and the d'Alembert operator $\Box = 1/c^2 \partial_{tt}^2 - \Delta$ with constant wave speed $c > 0$.
We focus on the general initial value problem in the free space $\mathbb{R}^3$
\begin{align}\label{eq:wave_eq}
\begin{cases}
    \hfil \Box w &= 0 \hspace{35pt} \forall (x,t) \in \mathbb{R}^3 \times \mathbb{R}_+^*  \\
    \hfil w(x,0) &= u_0(x) \hspace{16pt} \forall x \in \mathbb{R}^3  \\
    (\partial_t w)(x,0) &= v_0(x) \hspace{17pt} \forall x \in \mathbb{R}^3
\end{cases}
\end{align}
Throughout this article, we will refer to $u_0$ as the initial position and $v_0$ as the initial speed. The problem \eqref{eq:wave_eq} is a Cauchy problem with initial conditions (IC) $u_0$ and $v_0$. It admits a unique solution which can be extended to all times $t\in \mathbb{R}$, and is represented as follow (\cite{evans1998}, section 4.3.1.b, example 4)
\begin{align}\label{eq sol wave}
    w(x,t) = (F_t * v_0)(x) + (\Dot{F}_t * u_0)(x) \hspace{10mm} \forall (x,t) \in \mathbb{R}^3 \times \mathbb{R}
\end{align}
where $F_t$ and $\Dot{F}_t$ are known generalized functions. Actually, $F_t$ and $\Dot{F}_t$ are better known through their Fourier transforms \cite{evans1998}, as
\begin{align}\label{eq:fourier of ft ftp}
    \mathcal{F}(F_t)(\xi) = \frac{\sin(ct|\xi|)}{c|\xi|} \ \ \ \text{and} \ \ \ \mathcal{F}(\Dot{F}_t)(\xi) = \cos(ct|\xi|)
\end{align}
where $|\xi|$ is the euclidean norm of $\xi \in \mathbb{R}^3$. Note that the relation $\Dot{F}_t = \partial_t F_t$ can be directly deduced from \eqref{eq:fourier of ft ftp}. Additionally, the representation \eqref{eq sol wave} is valid in any dimension as well as the Fourier formulas \eqref{eq:fourier of ft ftp}, see \cite{evans1998}. Finally, $F_t$ also corresponds to the Green's function of the wave equation \cite{Duffy2015GreensFW}.

In dimension 3, $F_t$ and $\Dot{F}_t$ are compactly supported generalized functions of order $0$ and $1$ respectively. More explicitly, in dimension 3 $F_t$ and $\Dot{F}_t$ are given by
\begin{align}\label{eq:ft ftp in 3D}
    F_t = \frac{\sigma_{c|t|}}{4\pi c^2 t} \hspace{8mm} \text{and} \hspace{8mm} \Dot{F}_t = \partial_t F_t
\end{align} %\hspace{5mm} \text{(in the sense of distributions)}
where $\sigma_R$ is the surface measure of the sphere of center $0$ and radius $R$. $\Dot{F}_t = \partial_t F_t$ means that
\begin{align*}
\forall \mathcal{C}_0^1(\mathbb{R}^3), \langle \dot{F}_t, f\rangle = \partial_t \langle F_t, f\rangle = \partial_t \int_{\mathbb{R}^3}f(x)F_t(dx)
\end{align*}
Suppose that $u_0 \in \mathcal{C}^1(\mathbb{R}^3) $ and $v_0 \in \mathcal{C}^0(\mathbb{R}^3)$, then $w$ as defined in \eqref{eq sol wave} is a function in the classical sense \cite{treves2006topological} and in that case an explicit formula for such convolutions is reminded in equation \eqref{eq:conv distrib} (yet one may actually make sense out of \eqref{eq sol wave} when $u_0$ and $v_0$ are only required to be any generalized functions \cite{treves2006topological}). 
Combining formulas \eqref{eq sol wave} and \eqref{eq:ft ftp in 3D} leads to the Kirschoff formula \cite{evans1998} (see \ref{not : spherical} for spherical coordinates notations) :
\begin{align}\label{eq:kirschoff}
w(x,t) = \int_{S(0,1)}tv_0(x-c|t|\gamma) + u_0(x-c|t|\gamma) -  c|t|\gamma \cdot \nabla u_0(x-c|t|\gamma) \frac{d\Omega}{4\pi}
\end{align}
 %However, equation \eqref{eq:ft ftp in 3D} is specific to dimension 3.

\subsection{Gaussian Process Modelling of the Solution}
%Suppose the initial conditions are centered Gaussian processes, i.e. $u_0 \sim GP(0,k_{\mathrm{u}})$ and $v_0 \sim GP(0,k_{\mathrm{v}})$ and that we are interested in solving problem \eqref{eq:wave_eq} where the initial conditions are trajectories drawn from $u_0$ an $v_0$. Consider now the Gaussian process obtained through the linear transformation
Suppose now that $u_0$ and $v_0$ are unknown, and only pointwise values of $w$ are observed. In a Bayesian approach, we model $u_0$ and $v_0$ as random functions and put a Gaussian process prior over $u_0$ and $v_0$. More precisely, we make the following assumptions.
\begin{enumerate} [label=(\subscript{A}{{\arabic*}})]
    \item Suppose that the initial conditions $u_0$ and $v_0$ of Problem \eqref{eq:wave_eq} are trajectories drawn from two independent Gaussian processes $U^0 \sim GP(0,k_{\mathrm{u}})$ and $V^0 \sim GP(0,k_{\mathrm{v}})$ : $\exists \omega \in \Omega, \forall x \in \mathbb{R}^3, u_0(x) = U^0_{\omega}(x)$ and $v_0(x) = V^0_{\omega}(x)$. %By compatible we mean that $(T^0(x))_{x \in \mathcal{D},T \in \{U,V\}}$ is a Gaussian process.
    \item Suppose that all trajectories of $U^0$ lie in $\mathcal{C}^1(\mathbb{R}^3)$ and that those of $V^0$ lie in $\mathcal{C}^0(\mathbb{R}^3)$ almost surely. A sufficient condition for this is given in \cite{adler2007}, Thm 1.4.2. This theorem states that under mild technical assumptions, the paths of $(U(x))_{x \in \mathcal{D}} \sim GP(0,k)$ lie in $\mathcal{C}^l$ a.s. as soon as $ k \in \mathcal{C}^{2l}(\mathcal{D} \times \mathcal{D})$, \textit{which we assume from now on}.
\end{enumerate}
We now analyse the consequence of these two assumptions. First, they imply that by solving \eqref{eq:wave_eq}, one obtains a time-space stochastic process $W(x,t)$ defined by 
\begin{align}\label{eq:W sol}
    W(x,t) : \Omega \ni \omega \longmapsto (F_t * V^0_{\omega})(x) +  (\Dot{F}_t * U^0_{\omega})(x)    
\end{align}
Here again, $V^0_{\omega}$ denotes the trajectory of $V^0$ at $\omega \in \Omega$ and likewise for $U_{\omega}^0$. In particular, thanks to assumption $(A_2)$, \eqref{eq:W sol} defines a random variable for all $(x,t)$. 
Note the space-time variable $z = (x,t)$ and note the random variables
\begin{align}\label{eq:U_and_V}
V(z) : \omega \longmapsto (F_t * V^0_{\omega})(x) \ \text{ and } \ U(z) :  \omega \longmapsto (\Dot{F}_t * U^0_{\omega})(x)    
\end{align}
 that is, $W(z) = U(z) + V(z)$. We show in the next proposition that the stochastic processes $U,V$ and $W$ are GPs as well. In particular we describe their covariance kernels.
%eq:2 wave kernels
\begin{proposition}\label{prop : wave kernel}Define the two functions
\begin{align}
{k_{\mathrm{v}}^{\mathrm{wave}}(z,z') =[(F_t \otimes F_{t'}) * k_{\mathrm{v}}](x,x')}\label{eq: kv wave} \\
{k_{\mathrm{u}}^{\mathrm{wave}}(z,z') = [(\Dot{F}_t \otimes \Dot{F}_{t'}) * k_{\mathrm{u}}](x,x')} \label{eq: ku wave}
\end{align}
(i) Then $U = (U(z))_{z \in \mathbb{R}^3 \times \mathbb{R}}$ and $V =(V(z))_{z \in \mathbb{R}^3 \times \mathbb{R}}$ as defined in \eqref{eq:U_and_V} are two independent centered GPs with covariance kernels $k_{\mathrm{u}}^{\mathrm{wave}}$ and $k_{\mathrm{v}}^{\mathrm{wave}}$ respectively. Consequently, $(W(z))_{z \in \mathbb{R}^3 \times \mathbb{R}}$ is a centered GP whose covariance kernel is given by
\begin{align}\label{eq:wave kernel}
k_{W}(z,z') = k_{\mathrm{v}}^{\mathrm{wave}}(z,z') + k_{\mathrm{u}}^{\mathrm{wave}}(z,z')
\end{align}
(ii) Conversely, any centered second order stochastic process with covariance kernel $k_{W}$ has its sample paths solution of the 3 dimensional wave equation \eqref{eq:wave_eq} almost surely.
%&= [(\Dot{F}_t \otimes \Dot{F}_{t'}) * k_{\mathrm{u}}](x,x') + [(F_t \otimes F_{t'}) * k_{\mathrm{v}}](x,x') \
\end{proposition}
\begin{proof}
$(i)$ : first we prove that $U$ and $V$ are GPs. Since $U^0$ and $V^0$ are GPs, $\mathcal{L}(U^0)$ and $\mathcal{L}(V^0)$ are only comprised of Gaussian random variables (see \ref{not : gp et al}).

To prove that $U$ and $V$ are Gaussian processes, we rely on the Kirschoff formula \eqref{eq:kirschoff}, writing the integrals as limits of Riemann sums.
We start with $V$, that is, we focus on the first term in Kirschoff's formula \eqref{eq:kirschoff}. To show that $V$ is a Gaussian process, we only need to show that for any $z$, $V(z) \in \mathcal{L}(V^0)$ as this will ensure the Gaussian process property. Since the trajectories of $V^0$ are continuous almost surely, there exists a sequence of numbers $a_k^n$ and points $y_k^n$ such that for almost any $\omega \in \Omega$,
\begin{align*}
V(z)(\omega) &= (F_t * V_{\omega}^0)(x) = t\int_{S(0,1)} V^0(x-c|t|\gamma)(\omega)\frac{d\Omega}{4\pi} \\
= &\frac{t}{4\pi}\int_{0}^{2\pi}\int_{0}^{\pi} V^0(x-c|t|\gamma(\theta,\phi))(\omega) \sin(\theta)d\theta d\phi =  \lim_{n \rightarrow \infty} \sum_{k=1}^n a_k^n V^0(x-y_k^n)(\omega)
\end{align*}
This shows that $V(z)$ is an a.s. limit of a sequence of centered Gaussian random variables $(Y_n)\subset \mathcal{L}(V^0)$; a.s. convergence implies convergence in law. From \cite{legall2013}, Proposition 1.1, $V(z)$ is normally distributed and the convergence also takes place in $L^2(\mathbb{P})$. Therefore, $V(z) \in \mathcal{L}(V^0)$ and $V$ is a Gaussian process. From the same proposition, $V(z)$ is centered because the variables $Y_n$ are centered. Note that since $F_t$ is supported on the compact set $S(0,1)$, we only required the trajectories of $V^0$ to be continuous rather than continuous and compactly supported.

We apply the same reasoning to $U$, by applying the above steps to the second part of Kirschoff's formula \eqref{eq:kirschoff}. One's ability to write out the integrals as a limit of Riemann sums is ensured when the trajectories of $U^0$ lie in $\mathcal{C}^1(\mathbb{R}^3)$.

Finally, since $U^0$ and $V^0$ are independent, $\mathcal{L}(U^0)$ and $\mathcal{L}(V^0)$ are orthogonal in $L^2(\mathbb{P})$. Since $\mathcal{L}(U) \subset \mathcal{L}(U^0)$ and likewise for $V$, $U$ and $V$ are independent Gaussian processes (for Gaussian random variables, independence is equivalent to null covariance). Finally, the sum of independent Gaussian random variables is a Gaussian random variable. Therefore $\mathcal{L}(W) \subset \mathcal{L}(U) + \mathcal{L}(V)$ is only comprised of Gaussian random variables and $W$ is a Gaussian process. Now, we prove that
\begin{align}\label{eq:ftp ftp fubini}
\mathbb{E}[U(z)U(z')] = [(\dot{F}_t \otimes \dot{F}_{t'})*k_{\mathrm{u}}](x,x')
\end{align}
The main argument is Fubini's theorem. For this we use the fact that $\dot{F}_t$ is a distribution of order $1$ and can be identified to the a sum of derivatives of measures (see equation \eqref{eq:finite order radon}) : for all $t \in \mathbb{R}$, there exists $\{\mu^{t}_k\}_{k \in \mathbb{N}^3,|k|\leq 1}$ a family of Radon measures such that
\begin{align}
    \Dot{F}_t &= \sum_{|k| \leq 1} \partial^k \mu^{t}_k \ \ \ \text{in the sense of distributions}
\end{align}
Moreover, $\dot{F}_t$ is compactly supported, therefore all the measures $\mu^{t}_k$ are also compactly supported. 
First, we write $U_{\omega}(z)$ in integral form :
\begin{align}
U_{\omega}(z) &= \big(\Dot{F}_t * U_{\omega}^0\big)(x) = \langle \Dot{F}_t, \tau_{-x}\widecheck{U}_{\omega}^0 \rangle = \Big \langle \sum_{|k| \leq 1} \partial^k \mu^{t}_k, \tau_{-x}\widecheck{U}_{\omega}^0 \Big \rangle \\
    &= \sum_{|k|\leq 1} \langle \mu^{t}_k, (-1)^{|k|}\partial^k \tau_{-x}\widecheck{U}_{\omega}^0 \rangle = \sum_{|k|\leq 1} \int_{\mathbb{R}^3} (-1)^{|k|}\partial^k U_{\omega}^0(x-y) \mu^{t}_k(dy) 
\end{align}
Before applying Fubini's theorem, we need to check an integrability condition. Let $k \in \mathbb{N}^3$ such that $|k| \leq 1$. Recall that $|\mu_t^k|$ is defined in \eqref{eq:abs mu}; denote also $\sigma_{\partial^k U^0}(x) = \sqrt{\text{Var}(\partial^k U_0(x))}$. Since the trajectories of $U^0$ lie in $\mathcal{C}^1(\mathcal{D})$ a.s, those of $\partial^k U^0$ lie in $\mathcal{C}^0(\mathcal{D})$ and thus the function $x \longmapsto \text{Var}(\partial^k U_0(x))$ also lies in $\mathcal{C}^0(\mathcal{D})$ (\cite{azais_level_2009}, chapter 1, section 4.3). Therefore the function $x \longmapsto \sigma_{\partial^k U^0}(x)$ also lies in $\mathcal{C}^0(\mathcal{D})$. We now check that the integral $I$ below is finite. We use Tonelli's theorem and the Cauchy-Schwarz inequality:
\begin{align*}
    I :=& \int_{\Omega} \sum_{|k|\leq 1}\int_{\mathbb{R}^3}\Big|\partial^k U_{\omega}^0(x-y) \Big| |\mu^{t}_k|(dy) \sum_{|k'|\leq 1} \int_{\mathbb{R}^3} \Big|\partial^{k'} U_{\omega}^0(x'-y') \Big| |\mu^{t'}_{k'}|(dy')  \mathbb{P}(d\omega) \\
    =& \sum_{|k|,|k'|\leq 1} \int_{\mathbb{R}^3}\int_{\mathbb{R}^3} \int_{\Omega}  \Big|\partial^k U_{\omega}^0(x-y)\partial^{k'} U_{\omega}^0(x'-y')\Big| \mathbb{P}(d\omega) |\mu^{t}_k|(dy)|\mu^{t'}_{k'}|(dy') \\
    =& \sum_{|k|,|k'|\leq 1} \int_{\mathbb{R}^3}\int_{\mathbb{R}^3} \mathbb{E}\big[|\partial^kU^0(x-y)\partial^{k'}U^0(x'-y')|\big] |\mu^{t}_k|(dy)|\mu^{t'}_{k'}|(dy') \\
    \leq& \sum_{|k|,|k'|\leq 1} \int_{\mathbb{R}^3}\int_{\mathbb{R}^3} \bigg(\mathbb{E}\big[\partial^kU^0(x-y)^2\big]\mathbb{E}\big[\partial^{k'}U^0(x'-y')^2\big]\bigg)^{1/2}|\mu^{t}_k|(dy)|\mu^{t'}_{k'}|(dy') \\
    \leq& \Bigg( \sum_{|k|\leq 1} \int_{\mathbb{R}^3} \bigg(\mathbb{E}\big[\partial^kU^0(x-y)^2\big]\bigg)^{1/2} |\mu^{t}_k|(dy) \Bigg)\times \Bigg( \sum_{|k|\leq 1} \int_{\mathbb{R}^3} \bigg(\mathbb{E}\big[\partial^kU^0(x-y)^2\big]\bigg)^{1/2} |\mu^{t'}_k|(dy) \Bigg) 
    \\ \leq &\Big( \sum_{|k|\leq 1} (|\mu^{t}_k| * \sigma_{\partial^k U^0})(x) \Big) \times \Big( \sum_{|k|\leq 1} (|\mu^{t'}_k| * \sigma_{\partial^k U^0})(x') \Big) < +\infty
\end{align*}
For all multi-index $k$, the quantity $(|\mu^{t}_k| * \sigma_{\partial^k U^0})(x)$ is finite because $x \longmapsto \sigma_{\partial^k U^0}(x)$ is continuous and $|\mu^{t}_k|$ is compactly supported.
Note also that from hypothesis $(A_2)$, the GP $U^0$ is \textit{mean square differentiable} up to order 1, which implies (\cite{ritter2007average}, section III.1.4) that we have, for all multi-indexes $k,k'$ such that $|k|,|k'| \leq 1$, $x$ and $x'$ :
\begin{align}
\mathbb{E}\big[\partial^k U^0(x)\partial^{k'}U^0(x')\big] = \partial_1^k \partial_2^{k'}k_{\mathrm{u}}(x,x')
\end{align}
where $\partial_1$ (resp. $\partial_2$) denotes derivatives w.r.t. the first (resp. second) argument of $k_{\mathrm{u}}$.
We may thus permute integrals and differential operators in $\mathbb{E}\big[U(z)U(z')\big]$ : 
\begin{align}
    \mathbb{E}\big[U(z)U(z')\big] &= \mathbb{E}\Bigg[\sum_{|k|\leq 1} \int_{\mathbb{R}^3} (-1)^{|k|}\partial^k U^0(x-y) \mu^{t}_k(dy))\sum_{|k'|\leq 1} \int_{\mathbb{R}^3} (-1)^{|k'|}\partial^{k'} U^0(x-y) \mu^{t'}_{k'}(dy')\Bigg] \nonumber \\
    &= \sum_{|k|,|k'|\leq 1}\int_{\mathbb{R}^3} \int_{\mathbb{R}^3}(-1)^{|k|}(-1)^{|k'|}\partial_1^k \partial_2^{k'}\mathbb{E}\big[U^0(x-y)U^0(x'-y')\big]\mu^{t}_k(dy)\mu^{t'}_{k'}(dy') \nonumber \\
    &= \sum_{|k|,|k'|\leq 1}\int_{\mathbb{R}^3} \int_{\mathbb{R}^3}(-1)^{|k|}(-1)^{|k'|}\partial_1^k \partial_2^{k'}k_{\mathrm{u}}(x-y,x'-y')\mu^{t}_k(dy)\mu^{t'}_{k'}(dy') \nonumber \\
    &= \bigg[\Big(\sum_{|k|\leq 1} \partial^k\mu^{t}_k  \otimes \sum_{|k'|\leq 1} \partial^{k'}\mu^{t'}_{k'}\Big) * k_{\mathrm{u}}\bigg](x,x') = [(\Dot{F}_t \otimes \dot{F}_{t'}) * k_{\mathrm{u}}](x,x') \nonumber
\end{align}
%The permutation of the expectation and the differential operators in \eqref{eq:permute k k'} is justified, because the hypothesis $(A_2)$ implies that $U^0$ is \textit{mean square differentiable} \cite{gpml2006} \cite{adler2007} : the covariance kernel of.
which proves \eqref{eq:ftp ftp fubini}. 

One proves that $\mathbb{E}\big[V(z)V(z')\big] = [({F}_t \otimes {F}_{t'}) * k_{\mathrm{v}}](x,x')$ the exact same way, which is actually simpler as $F_t$ is directly a measure. To conclude,

\begin{align}
    k_{W}(z,z') &= \text{Cov}(W(z),W(z')) \nonumber \\
    &= \mathbb{E}[(W(z)W(z')] = \mathbb{E}\Big[\big(U(z) + V(z)\big)\big(U(z') + V(z') \big)\Big] \nonumber\\
    &= \mathbb{E}\Big[U(z)U(z')\Big] +  \mathbb{E}\Big[U(z)V(z')\Big] +  \mathbb{E}\Big[V(z)U(z')\Big] +  \mathbb{E}\Big[V(z)V(z')\Big] \nonumber \\
    %\mathbb{E}\Big[(F_t * V^0_{\omega})(x)(F_{t'} * V^0_{\omega})(x')\Big] + \mathbb{E}\Big[(F_t * V^0_{\omega})(x)(\Dot{F}_{t'} * U^0_{\omega})(x')\Big] \ \ + \nonumber \\
    %&\ \ \ \ \mathbb{E}\Big[(\Dot{F}_t * U^0_{\omega})(x)(\Dot{F}_{t'} * U^0_{\omega})(x')\Big]  + \mathbb{E}\Big[(\Dot{F}_t * U^0_{\omega})(x)(F_{t'} * V^0_{\omega})(x')\Big] \nonumber\\
    &= [(\Dot{F}_t \otimes \Dot{F}_{t'}) * k_{\mathrm{u}}](x,x') + [(F_t \otimes F_{t'}) * k_{\mathrm{v}}](x,x')
\end{align}
The cross terms are null because $U(z)$ and $V(z')$ are independent as well as $U(z')$ and $V(z)$.

$(ii)$ : with expression \eqref{eq:wave kernel}, one checks that for any fixed $z'$, the function $z \longmapsto k_{W}(z,z')$ is of the form \eqref{eq sol wave} and thus verifies $\Box k_{x'} = 0$ in the sense of distributions. $(ii)$ is then a direct consequence of Proposition \ref{prop : diff constraints}.

\end{proof}
\begin{remark}
If $U$ and $V$ are not independent, then the two terms $[(\Dot{F}_t \otimes F_{t'}) * k_{uv}](x,x')$ and $[(F_t \otimes \Dot{F}_{t'}) * k_{vu}](x,x')$ must be added to equation \eqref{eq:wave kernel}, where $k_{uv}(x,x')$ denotes the cross covariance between $U$ and $V$ : $k_{uv}(x,x') = \text{Cov}(U(x),V(x'))$ and $k_{vu}(x,x') = \text{Cov}(V(x),U(x')) = k_{uv}(x',x)$.
\end{remark}
More explicitly, we have the following Kirschoff-like integral formulas for $k_{\mathrm{v}}^{\mathrm{wave}}$ and $k_{\mathrm{u}}^{\mathrm{wave}}$ :
\begin{align}
[(F_t \otimes F_{t'}) * k_{\mathrm{v}}]&(x,x') = tt'\int_{S(0,1) \times S(0,1)}k_{v}(x-c|t|\gamma,x'-c|t'|\gamma')\frac{d\Omega d\Omega'}{(4\pi)^2} \label{eq:explicit Ft Ftp}\\
[(\Dot{F}_t \otimes \Dot{F}_{t'}) * k_{\mathrm{u}}](x,x') &= \int_{S(0,1) \times S(0,1)}\Big( k_{u}(x-c|t|\gamma,x'-c|t'|\gamma') \nonumber \\ 
& \hspace{30pt}-ct \nabla_1 k_{u}(x-c|t|\gamma,x'-c|t'|\gamma')\cdot \gamma \nonumber \\
& \hspace{30pt}-c|t'| \nabla_2 k_{u}(x-c|t|\gamma,x'-c|t'|\gamma')\cdot \gamma' \nonumber\\
&\hspace{30pt}+ c^2tt' \gamma^T \nabla_1 \nabla_2 k_{u}(x-c|t|\gamma,x'-c|t'|\gamma')\gamma'\Big) \frac{d\Omega d\Omega'}{(4\pi)^2}\label{eq:explicit Ft Ftp point}
\end{align}
where $\nabla_1 k_{u}(x,x')$ is the gradient vector of $k_{u}$ w.r.t. $x$,  $\nabla_2 k_{u}(x,x')$ is the gradient vector of $k_{u}$ w.r.t. $x'$ and $\nabla_1 \nabla_2 k_{u}(x,x')$ is the matrix whose entry $(i,j)$ is given by $\partial_{x_i^1}\partial_{x_j^2}k_{u}(x,x')$ ($\partial_{x_i^1}$ is the partial derivative w.r.t. the $i^{th}$ coordinate of $x$, $\partial_{x_j^2}$ is the partial derivative w.r.t. the $j^{th}$ coordinate of $x'$).
From now on, we call WIGPR the act of performing GPR with a covariance kernel of the form \eqref{eq:wave kernel}.
%\begin{remark}
%It is not at all obvious that $U$ and $V$ would still be GPs if $F_t$ and/or $\dot{F}_t$ were not related to some form of Lebesgue measure as the proof of point $(i)$ of the above Proposition relies on Riemann sums.
%\end{remark} 
%\begin{remark}
%Assumption $(A_2)$ is the minimal assumption to make sense of \eqref{eq:W sol} pointwise. However, the trajectories of $W$ in \eqref{eq:W sol} only verify Problem \eqref{eq:wave_eq} in the sense of distributions.
%\end{remark}

%\begin{remark}
%To perform physically informed scattered data interpolation, it is possible to come up with kernel \eqref{eq:wave kernel} with the RKHS point of view, as it appears to be a practical way of ensuring that $\Box \Tilde{m} = 0$ in the sense of distributions. The stochastic process point of view shows that it is in fact a natural way of building such a kernel $k_{W}$.
%\end{remark}

\subsubsection{Stationary kernel for the initial conditions} Many standard covariance kernels used for GPR are stationary \cite{gpml2006}. For this reason, we study equation \eqref{eq: kv wave} when $k_{\mathrm{v}}$ is stationary. For conciseness, we restrict ourselves to the case where $u_0 = 0$, i.e. $k_{\mathrm{u}} = 0$ (and we discard \eqref{eq: ku wave}).

%\begin{lemma}\label{lemma : k_{\mathrm{v}} stat}

%\end{lemma}
%\begin{proof}
%See appendix B
%\end{proof}

\begin{proposition}\label{prop : F_t * F_tp} Suppose that $k_{\mathrm{v}}$ is stationary : $k_{\mathrm{v}}(x,x') = k_S(x-x')$. \\
(i) Then $k_{\mathrm{v}}^{\mathrm{wave}}$ is stationary in space and
\begin{align}
[(F_t \otimes F_{t'}) * k_{\mathrm{v}}](x,x') = (F_t * F_{t'} * k_S)(x-x')
\end{align}
(ii) Moreover, the measure $F_t * F_{t'}$ is absolutely continuous over $\mathbb{R}^3$ and identifying it to its density, we have
\begin{align} \label{f_t star f_t'}
    (F_t * F_{t'})(h) = \frac{\text{sgn}(t)\text{sgn}(t')}{8\pi c^2|h|}\mathbbm{1}_{\big[c\big||t|-|t'|\big|, c(|t| + |t'|)\big]}(|h|)
\end{align}
where $|h|$ is the euclidean norm of $h \in \mathbb{R}^3$.
\end{proposition}
\begin{proof}
$(i) :$ Suppose for simplicity that $c=1$. Using \eqref{eq:conv_mesure_func},
\begin{align*}
[(F_t \otimes F_{t'}) * k_{\mathrm{v}}](x,x')&= 
\int_{\mathbb{R}^3 \times \mathbb{R}^3} k(x-s_1,x'-s_2)dF_t(s_1) dF_{t'}(s_2)\\
&= \int_{\mathbb{R}^3 \times \mathbb{R}^3} k_S(x-x'-s_1+s_2)dF_t(s_1) dF_{t'}(s_2)
\end{align*}
But $S(0,1)$ is invariant under the change of variable $\gamma \longmapsto - \gamma$ and thus for any function $f$, $\int_{\mathbb{R}^3} f(s_2)dF_{t'}(s_2) = \int_{\mathbb{R}^3} f(-s_2)dF_{t'}(s_2)$. This yields
\begin{align*}
[(F_t \otimes F_{t'}) * k_{\mathrm{v}}](x,x')= \int_{\mathbb{R}^3 \times \mathbb{R}^3} k_S(x-x'-s_1-s_2)dF_t(s_1) dF_{t'}(s_2)
\end{align*}
Then, using the definition of $F_t * F_{t'}$ in \eqref{eq:conv mesures},
\begin{align*}
(F_t * F_{t'} * k_S)(h) &= \int_{\mathbb{R}^3} k_S(h-s)d(F_t * F_{t'})(s) = \int_{\mathbb{R}^3}\int_{\mathbb{R}^3} k_S(h-s_1-s_2)dF_t(s_1)dF_{t'}(s_2)
\end{align*}
Setting $h=x-x'$ finishes the proof. \\

$(ii) :$ Without loss of generality we suppose that $c=1$. The computation is derived in the Fourier domain. We use here the following convention for the Fourier transform of functions and measures over $\mathbb{R}^d$, along with its inverse transform
\begin{align}
\mathcal{F}(f)(\xi) &= \int_{\mathbb{R}^d}f(x)e^{-i\langle x,\xi \rangle}dx \label{eq:fourier func} \\
f(x) &= \frac{1}{(2\pi)^d}\int_{\mathbb{R}^d}\mathcal{F}(f)(\xi)e^{i\langle x,\xi\rangle}d\xi \label{eq:inverse fourier} \\
\mathcal{F}(\mu)(\xi) &= \int_{\mathbb{R}^d}f(x)e^{-i\langle x,\xi \rangle}\mu(dx) \label{eq:fourier mesure}
\end{align}
Above, $\langle,\rangle$ denotes the usual dot product of $\mathbb{R}^d$.
Equation \eqref{eq:fourier mesure} for an absolutely continuous measure $\mu(dx)=f(x)dx$ reduces to \eqref{eq:fourier func}, i.e. the Fourier transform of its density $f$. This convention is the one that yields the expressions \eqref{eq:fourier of ft ftp} and it is such that $\mathcal{F}(f*g)(\xi) = \mathcal{F}(f)(\xi)\mathcal{F}(g)(\xi)$ \cite{rudin1991} whatever $f$ and $g$, functions or measures. Thus,
\begin{align}\label{eq:facto fourier}
     \mathcal{F}(F_t * F_{t'})(\xi) = \mathcal{F}(F_t)(\xi)\mathcal{F}(F_{t'})(\xi) = \frac{\sin(t|\xi|)\sin(t'|\xi|)}{|\xi|^2} = \frac{\cos(a|\xi|) - \cos(b|\xi|)}{2|\xi|^2}
 \end{align}
 with $a = t-t', b = t+t$. We then compute the inverse Fourier transform of the quantity above. Let $h \in \mathbb{R}^3$. In spherical coordinates (see \ref{not : spherical}), noting the unit vectors $\gamma_h = \frac{h}{|h|}$ and $\gamma = \frac{\xi}{|\xi|} = \frac{\xi}{r}$, we define $f_a$ by
 \begin{align}
     f_a(h) &= \int_{\mathbb{R}^3}e^{i \langle h,\xi \rangle} \frac{\cos(a|\xi|)}{|\xi|^2}d\xi = 
    \int_0^{+\infty}\int_0^{2\pi} \int_0^{\pi}e^{i r\langle h,\gamma \rangle} \frac{\cos(ar)}{r^2}r^2\sin \theta d\theta d\phi dr \label{eq:simplify r2}\\
    &=\int_0^{+\infty} \cos(ar) \int_{S(0,1)}e^{i r|h|\langle \gamma,\gamma_h \rangle}d\Omega dr \nonumber
 \end{align}
%\int_0^{+\infty}\int_{S(0,1)}e^{i r|h|\langle \gamma,\gamma_h \rangle} \frac{\cos(ar)}{r^2} r^2 d\Omega dr =
In the lines above, we used the spherical coordinate change $\xi = r\gamma, d\xi = r^2\sin \theta d\theta d\phi dr$. We then make use of radial symmetry in the interior integral. Note $e_3$ the third vector of the canonical basis of $\mathbb{R}^3$ and let $M$ be an orthogonal matrix such that $M\gamma_h = e_3$. We perform the change of variable $\gamma' = M\gamma$, using that $MS(0,1) = S(0,1)$ and that the corresponding Jacobian is equal to $1$:
 \begin{align}
     \int_{S(0,1)}e^{i r|h|\langle \gamma,\gamma_h \rangle}d\Omega &= \int_{M S(0,1)}e^{i r|h|\langle M^T\gamma', \gamma_h \rangle}d\Omega' \label{eq:change M} \\
      &= \int_{S(0,1)}e^{i r|h|\langle \gamma', M\gamma_h \rangle}d\Omega'\\
     &= \int_{S(0,1)}e^{i r|h|\langle \gamma, e_3 \rangle}d\Omega = 2\pi \int_0^\pi e^{i r|h|\cos(\theta)}\sin(\theta)d\theta \nonumber \\
     &= 2\pi \left[-\frac{e^{i r|h|\cos(\theta)}}{ir|h|}\right]_0^\pi = 2\pi \frac{e^{ir|h|}-e^{-ir|h|}}{ir|h|} = 4\pi\frac{\sin(r|h|)}{r|h|} \label{eq:integ_exact}
 \end{align}
 and thus
 \begin{align}\label{eq:sin a sin b}
     f_a(h) &= 4\pi \int_0^{\infty} \frac{\cos(ar)\sin(|h|r)}{r|h|}dr = 4\pi \int_0^{\infty} \frac{\sin((|h|+a)r) + \sin((|h|-a)r)}{2r|h|}dr \nonumber \\
     &= \frac{2\pi}{|h|}\int_0^{\infty}\frac{\sin(\alpha r)}{r} + \frac{\sin(\beta r)}{r}dr
 \end{align}
 with $\alpha = |h|+a, \beta = |h|-a$. Finally, we have the Dirichlet integral
 \begin{align}\label{eq:dirichlet}
     \int_0^{\infty}\frac{\sin(\alpha r)}{r}dr = \text{sgn}(\alpha)\frac{\pi}{2}
 \end{align}
We define the function $f_b$ exactly as $f_a$, and compute it by replacing $a$ by $b$ in every step above. Putting \eqref{eq:facto fourier}, \eqref{eq:sin a sin b} and \eqref{eq:dirichlet} together, the inverse Fourier transform of $F_t*F_{t'}$ computed as in equation \eqref{eq:inverse fourier} is an absolutely continuous measure with density
 \begin{align}
     f(h) &= \frac{1}{(2\pi)^3}\frac{1}{2}\big(f_a(h) - f_b(h)\big) \nonumber \\ 
     &= \frac{1}{16\pi |h|}\Big(\text{sgn}(|h| + t-t') + \text{sgn}(|h| - t + t') \nonumber\\
     & \hspace{45pt} - \text{sgn}(|h| + t + t') - \text{sgn}(|h| -t-t')\Big) \label{ft sign}\\
     &=: \frac{1}{16\pi |h|}K(|h|,t,t') \nonumber
 \end{align}
%= \frac{1}{16\pi^3}\frac{2\pi}{|h|}\frac{\pi}{2}(\text{sgn}(|h| + t-t') + \text{sgn}(|h| - t + t') - \text{sgn}(|h| + t + t') - \text{sgn}(|h| -t-t')) \label{ft sign}
$K(|h|,t,t')$ is defined in the line above. Note that $K(|h|,-t,t') = -K(|h|,t,t')$ and likewise with $t'$, thus $K(|h|,t,t') = \text{sgn}(t)\text{sgn}(t')K(|h|,T,T')$ with $T = |t|, T' = |t'|$. Using the symmetries in $t$ and $t'$ in \eqref{ft sign} and the fact that $\text{sgn}(s) = 1$ if $s > 0$, we obtain 
 \begin{align*}
     K(|h|,T,T') &= \hspace{4pt} \text{sgn}(|h| + |T-T'|) + \text{sgn}(|h| - |T-T'|)\\
     &\hspace{9pt}- \text{sgn}(|h| + T + T') - \text{sgn}(|h| -T-T') \\
     &= 1 + \text{sgn}(|h| - |T-T'|) - 1 - \text{sgn}(|h| -T-T') \\
     &= \text{sgn}(|h| - |T-T'|) - \text{sgn}(|h| -T-T')
 \end{align*}
 Therefore,
 \begin{align*}
      K(|h|,T,T') = &\begin{cases}
                    &0 \text{  if  } |h| < |T-T'| \\
                    &2 \text{  if  } |T-T'| < |h| < T + T' \\
                    &0 \text{  if  } |h| > T+T'
                    \end{cases}
                    \ \ \ = \ 2 \times \mathbbm{1}_{[|T-T'|, T + T']}(|h|)
 \end{align*}
 Identifying the measure $F_t*F_{t'}$ with its density, we have
 \begin{align}
     (F_t*F_{t'})(h) = \frac{\text{sgn}(t)\text{sgn}(t')}{8\pi|h|}\mathbbm{1}_{\Big[\big||t|-|t'|\big|, |t| + |t|'\Big]}(|h|)
 \end{align}
\end{proof}
The proof above makes extensive use of the specificities of the dimension $3$. First in equation \eqref{eq:simplify r2}, where the scalars $r^2$ cancel each other out; second in \eqref{eq:integ_exact} where an exact antiderivative of the integrated function can be computed. None of these two simplifications hold in higher dimension or in dimension 2.
\begin{remark}
Note that expression \eqref{f_t star f_t'} with $h=x-x'$ is the covariance kernel of the solution process $U$ with initial condition the ``formal" white noise process $V^0$ with the stationary Dirac delta covariance kernel $k_{\mathrm{v}}(x,x') = \delta_0(x-x')$ : 
\begin{align}
{[(F_t \otimes F_{t'}) * k_{\mathrm{v}}](x,x') = (F_t * F_{t'} * \delta_0)(x-x') = (F_t * F_{t'})(x-x')}
\end{align}
Somewhat surprisingly, although formula \eqref{f_t star f_t'} yields a summable function over $\mathbb{R}^3$ when $t$ and $t'$ are fixed, it is not directly usable for practical computations as the diagonal terms of the related covariance matrices are all singularities : $(F_t * F_t)(0) = +\infty$... Yet, formula \eqref{f_t star f_t'} may be used together with explicit kernels $k_S$ to yield usable expressions. For instance, if $k_{\mathrm{v}}(x,x') = k_S(x-x') = C\exp\Big(-\frac{|x-x'|^2}{2L^2}\Big)$, we state without proof that :
\begin{align}
    (F_t * F_{t'}* k_S)(h) = \text{sgn}(tt') \frac{C'L^3}{c^2}\Bigg(\frac{\Phi\big( \frac{R_1 + |h|}{L}\big)-\Phi\big( \frac{R_1 - |h|}{L}\big)}{2|h|}- \frac{\Phi\big( \frac{R_2 + |h|}{L}\big)-\Phi\big( \frac{R_2 - |h|}{L}\big)}{2|h|}\Bigg)
\end{align}
where $h = x-x'$, $\ \Phi(s) = \int_{-\infty}^s \frac{exp(-t^2/2)}{\sqrt{2\pi}}dt$, $R_1 = c\big||t|-|t'|\big|$, $R_2 = c(|t| + |t'|) $. Such a kernel always takes finite values : when $h$ goes to $0$, the above formula reduces to computable finite derivatives.
\end{remark}

Although these formulas are interesting in themselves, the study of propagation phenomena is usually done thanks to compactly supported initial conditions, which can never be modelled with a stationary GP. Compactly supported initial conditions are dealt with in subsection \ref{sub : radial} within the context of radial symmetry.

\subsection{Initial Condition Reconstruction and Error control}
\subsubsection{Initial condition reconstruction} 
Consider a set of locations $(x_i)_{1\leq i \leq q}$ and moments $(t_j)_{1 \leq j \leq N}$ (imagine $q$ sensors each collecting an observation at time $t_j$ for all $j$). Consider now the following inverse problem 
\begin{equation}\label{eq:inverse_pb}
\begin{aligned}
\text{Recover } &u_0 \text{ and } v_0 \text{ from the set of observations} \{w(x_i,t_j)\}_{i,j} \\
& \text{ where } (w, u_0, v_0) \text{ are subject to \eqref{eq:wave_eq}}
\end{aligned}
\end{equation}
We will see that WIGPR provides an answer to the problem \eqref{eq:inverse_pb}. This is not surprising because the model described in the previous section was derived under the hypothesis that $u_0$ and $v_0$ were unknown.

%we used a Bayesian approach on the problem \eqref{eq:wave_eq} by putting a Gaussian process prior distribution on the initial conditions $u_0$ and $v_0$, and propagated the priors through the wave equation in Proposition \ref{prop : wave kernel}. 

We first detail some background from the inverse problem community for the problem \eqref{eq:inverse_pb}. The general problem of reconstructing the initial conditions $u_0$ and $v_0$ from some observations of the solution $w$ of \eqref{eq:wave_eq} is known as Photoacoustic Tomography (PAT)(\cite{kuchment2010},\cite{ammari2012}). In the usual inverse problem approach, the solution is accessible over a whole surface $S \subset \mathbb{R}^3$ or an open subset of $\mathbb{R}^3$. The initial conditions are recovered from the observations over $S$ thanks to an integral transform such as the Radon transform (\cite{kuchment2010},\cite{ammari2012}). Under certain geometric conditions, the corresponding inverse problem is known to be well-posed (\cite{kuchment2010}). In real life applications though, only discrete finite dimensional data is available : one then has to discretize the integrals of the inversion step (\cite{ammari2012}, chapter 3). This means that a sufficiently large amount of sensors is needed for the numerical integration step to yield satisfying results; additionally, the sensors have to be located at specific locations dictated by the numerical integration method used.
General Bayesian inversion approaches are also standard for this type of; see the state of the art section for a discussion and references.

We now return to WIGPR. From Proposition \ref{prop : inheritance pde conditioned}, we observe that performing GPR on any data with kernel \eqref{eq:wave kernel} automatically produces a prediction $\tilde{m}$ that verifies $\Box \tilde{m} = 0$ in the sense of distributions. Therefore this function $\tilde{m}$ is the solution of problem \eqref{eq:wave_eq} for some initial conditions $\Tilde{u}_0$ and $\Tilde{v}_0$ :
\begin{align}
    \Tilde{m}(x,t) &= (F_t * \Tilde{v}_0)(x) + (\Dot{F}_t * \Tilde{u}_0)(x)
\end{align}
These initial conditions are simply given by $\Tilde{u}_0(x) = \tilde{m}(x,0) $ and $\Tilde{v}_0(x) = \partial_t\tilde{m}(x,0)$. If the data $(w_1,...,w_n)^T$ on which GPR is performed is comprised of observations of a function $w$ that is another solution of problem \eqref{eq:wave_eq}, the initial conditions $(\Tilde{u}_0,\Tilde{v}_0)$ can be understood as approximations of the initial conditions $(u_0,v_0)$ corresponding to $w$. More precisely, following equation \eqref{eq:krig projection}, we have
\begin{align}
	\tilde{m}(x,t) &= p_F(w)(x,t) &&\forall (x,t) \in \mathbb{R}^3 \times \mathbb{R}^+ \\
    \Tilde{u}_0(x) &= \Tilde{m}(x,0) = p_F(w)(x,0) &&\forall x \in \mathbb{R}^3 \label{eq:U^0 tilde}\\
    \Tilde{v}_0(x) &= \partial_t\Tilde{m}(x,0) = \partial_t p_F(w)(x,0) = p_F(\partial_t w)(x,0)  &&\forall x \in \mathbb{R}^3 \label{eq:V^0 tilde}
\end{align}
where $F$ denotes the finite dimensional space $\text{Span}(k_{W}(z_1,\cdot),...,k_{W}(z_n,\cdot))$ and $p_F$ is the orthogonal projection on $F$ defined right after equation \eqref{eq:krig projection}. Here, $z_m$ is of the form $z_m = (x_i,t_k) \in \mathbb{R}^4$).
We see here that this use of WIGPR provides a flexible framework for tackling the problem \eqref{eq:inverse_pb} : the sensors are not constrained in number or location by any integration scheme. Taking a look at equations \eqref{eq:U^0 tilde} and \eqref{eq:V^0 tilde}, we can quickly discuss the matter of optimal sensor locations for WIGPR. First, we expect that $\tilde{m}$ will provide a better approximation of $w$ when the functions $k_{W}(z_i,\cdot)_{i=1,...,n}$ are as orthogonal as possible in $\mathcal{H}_{k_W}$ since $\tilde{m}$ is an orthogonal projection on $F$ w.r.t. the $\mathcal{H}_{k_W}$ metric. That is, the following inequality should be required
\begin{align}\label{eq:ortho_F}
\langle k_{W}((x_i,t_k),\cdot),k_{W}((x_j,t_l),\cdot) \rangle_{\mathcal{H}_{k_{W}}} = k_{W}((x_i,t_k),(x_j,t_l)) \ll 1
\end{align}
A close inspection of expressions \eqref{eq:explicit Ft Ftp} and \eqref{eq:explicit Ft Ftp point} shows that the property \eqref{eq:ortho_F} can be obtained for most times $t_k$ and $t_l$ when the sensors are far apart from each other, as soon as the kernels $k_{\mathrm{u}}$ and $k_{\mathrm{v}}$ are such that $k(x,x') \longrightarrow 0$ when $|x-x'| \longrightarrow +\infty$ (which is common, see e.g. the kernel \eqref{eq:matern 5/2}).

\subsubsection{Error Control} Now that we have showed that WIGPR can provide approximations of the initial conditions of \eqref{eq:wave_eq}, we underline the fact that these initial condition reconstructions are  enough to control the spatial error between the target function $u$ and the Kriging mean $\tilde{m}$, at all times. Indeed, we have the following $L^p$ control in terms of the initial condition reconstruction error. For $p \in [1,+\infty]$, denote the Sobolev space ${W^{1,p}(\mathbb{R}^3) = \{f \in L^p(\mathbb{R}^3) : \forall |k| \leq 1, \partial^k f \in L^p(\mathbb{R}^3) \}}$.
  
\begin{proposition}\label{lemma : Lp control}
For any $p \in [1,+\infty]$ and any pair $v_0 \in L^p(\mathbb{R}^3)$, $u_0 \in W^{1,p}(\mathbb{R}^3)$ we have the $L^p$ estimates for all $t \in \mathbb{R}$
\begin{align}
    || F_t * v_0 ||_p &\leq |t| \ ||v_0||_p \label{eq:estim Lp v_0} \\
    || \Dot{F}_t * u_0 ||_p &\leq ||u_0||_p + C_p c|t| \ ||\nabla u_0||_p \label{eq:estim Lp u_0} 
\end{align}
with $C_p = \Big(\int_S |\gamma|_q^p \frac{d\Omega}{4\pi}\Big)^{1/p} \leq 3^{1/q} \leq 3, 1/p + 1/q = 1$. Supposing that the correct speed $c$ is known and plugged in $k_{W}$, equations \eqref{eq:estim Lp v_0} and \eqref{eq:estim Lp u_0} lead to the following $L^p$ error estimate between the target $w$ and its approximation $\tilde{m}$
\begin{align}\label{eq:diff w m tilde}
    ||w(\cdot,t) - \Tilde{m}(\cdot,t)||_p \leq |t| \ ||v_0 - \Tilde{v}_0||_p + ||u_0 - \Tilde{u}_0||_p + C_pc|t| \ ||\nabla(u_0 - \Tilde{u}_0)||_p
\end{align}
where $\tilde{u_0}$ and $\tilde{v}_0$ are defined in \eqref{eq:U^0 tilde} and \eqref{eq:V^0 tilde}.
\end{proposition}
Equations \eqref{eq:estim Lp v_0} and  \eqref{eq:estim Lp u_0} are simple stability estimates for the 3D wave equation, though they do not seem to have been documented in that form in the literature (notably the explicit control constants $|t|$ and $C_p c|t|$). They fall in the category of Strichartz estimates with $L^p$ control for the space variable and $L^{\infty}$ control for the time variable. We thus provide a proof of the Proposition \ref{lemma : Lp control}.
%\modif{The authors could not find a reference where the above proposition is proved, though it is a simple stability result for the 3 dimensional wave equation. Therefore the authors have taken it upon themselves to prove it.}

\begin{proof}
We have 
\begin{align}
(F_t*v_0)(x) = t\int_{S(0,1)}v_0(x-c|t|\gamma)\frac{d\Omega}{4\pi}
\end{align}
%\frac{1}{4\pi c^2t}\int_{S(0,1)}v_0(x-c|t|\gamma)c^2t^2d\Omega =
%\begin{align*}
%(F_t*v_0)(x) = \frac{1}{4\pi c^2t}\int_{S(0,1)}v_0(x-c|t|\gamma)c^2t^2d\Omega = t\int_{S(0,1)}v_0(x-c|t|\gamma)\frac{d\Omega}{4\pi}
%\end{align*}
$\frac{d\Omega}{4\pi}$ is the normalized Lebesgue measure on $S(0,1)$. For $p\in [1,+\infty [$, since $t \longmapsto t^p$ is convex on $\mathbb{R}^+$, Jensen's inequality yields
%\begin{align*}
%|(F_t*v_0)(x)|^p \leq \int_{S(0,1)} |v_0(x-c|t|\gamma)|^p  \frac{d\Omega}{4\pi}
%\end{align*}
%and
\begin{align*}
||F_t * v_0||_p^p &= t^p \int_{\mathbb{R}^3}|(F_t*v_0)(x)|^p dx = |t|^p\int_{\mathbb{R}^3}\bigg|\int_{S(0,1)}v_0(x-c|t|\gamma)\frac{d\Omega}{4\pi}\bigg|^p dx \\
 &\leq |t|^p\int_{\mathbb{R}^3} \int_{S(0,1)}|v_0(x-c|t|\gamma)|^p \frac{d\Omega}{4\pi} dx = |t|^p \int_{S(0,1)} \int_{\mathbb{R}^3} |v_0(x-c|t|\gamma)|^p dx \frac{d\Omega}{4\pi} \\
 &\leq \int_{S(0,1)} ||v_0||_p^p \frac{d\Omega}{4\pi} = |t|^p||v_0||_p^p
\end{align*}
Which yields the first estimate. Next,
\begin{align*}
(\dot{F}_t *u_0)(x) &= \partial_t (F_t * u_0)(x) = \partial_t \Bigg( t\int_{S(0,1)}u_0(x-c|t|\gamma)\frac{d\Omega}{4\pi} \Bigg) \\
&= \int_{S(0,1)}u_0(x-c|t|\gamma)\frac{d\Omega}{4\pi} + t\int_{S(0,1)}- c \gamma \cdot \nabla u_0(x-c|t|\gamma)\frac{d\Omega}{4\pi} =: I_1(x) + I_2(x)
\end{align*}
Functions $I_1$ and $I_2$ are defined in the line above. We have $||\dot{F}_t *u_0||_p = ||I_1 + I_2||_p \leq ||I_1||_p + ||I_2||_p$.
As above, $||I_1||_p \leq ||u_0||_p$. Using again Jensen's inequality,
\begin{align*}
||I_2||_p^p = c^p |t|^p\int_{\mathbb{R}^3} \Bigg|\int_{S(0,1)} \gamma \cdot \nabla u_0(x-c|t|\gamma)\frac{d\Omega}{4\pi} \Bigg|^p dx \leq c^p |t|^p \int_{\mathbb{R}^3}\int_{S(0,1)} |\gamma \cdot \nabla u_0(x-c|t|\gamma)|^p \frac{d\Omega}{4\pi}dx
\end{align*}
Next, we use Hölder's inequality in $\mathbb{R}^3$ : $|\gamma \cdot \nabla u_0| \leq |\nabla u_0|_p \times |\gamma|_q$ with $\frac{1}{p} + \frac{1}{q} = 1$, where $|v|_p = (|v_1|^p + |v_2|^p + |v_3|^p)^{1/p}$ is the $p$-norm of $v \in \mathbb{R}^3$ and likewise for $|v|_q$. Thus,
\begin{align*}
||I_2||_p^p &\leq c^p |t|^p \int_{\mathbb{R}^3}\int_{S(0,1)} |\nabla u_0(x-c|t|\gamma)|_p^p \times |\gamma|_q^p \frac{d\Omega}{4\pi} dx \\
&\leq c^p |t|^p\int_{S(0,1)}|\gamma|_q^p \int_{\mathbb{R}^3} |\nabla u_0(x-c|t|\gamma)|_p^p dx \frac{d\Omega}{4\pi} = c^p |t|^p\Bigg(\int_{S(0,1)}|\gamma|_q^p \frac{d\Omega}{4\pi}\Bigg) ||\nabla u_0||_p^p
\end{align*}
which yields the second estimate. Finally, the case $p = +\infty$ is trivial.

Equation \eqref{eq:diff w m tilde} is then the result of equations \eqref{eq:estim Lp v_0} and  \eqref{eq:estim Lp u_0} applied to the function 
\begin{align*}
w(x,t) - \tilde{m}(x,t) = [F_t *(v_0 - \tilde{v}_0)](x) + [\dot{F} *(u_0 - \tilde{u}_0)](x) 
\end{align*}
%\modif{justifier la dernière équation}
\end{proof}

Equation \eqref{eq:diff w m tilde} shows that $L^p$ approximations of the initial conditions provide $L^p$ control between the solution $w$ and the approximation $\tilde{m}$ for any time $t$. This is why in our numerical applications (Section \ref{section : num}), we only focus on initial condition reconstruction. 

When $c$ is unknown and estimated by $\hat{c}$ through maximizing the log marginal likelihood, we have (highlighting the $c$ dependence $F_t^c = \frac{\sigma_{c|t|}}{4\pi c^2 t}$)
\begin{align*}
    ||w(\cdot,t) - &\Tilde{m}(\cdot,t)||_p = ||F_t^c * u_0 - F_t^{\hat{c}} * \Tilde{u}_0 + \Dot{F}_t^c *v_0 - \Dot{F}_t^{\hat{c}} *\Tilde{v}_0||_p \nonumber \\
    &= ||F_t^c * (u_0 - \Tilde{u}_0) + (F_t^c-F_t^{\hat{c}}) * \Tilde{u}_0 + \Dot{F}_t^c *(v_0 - \Tilde{v}_0) + (\Dot{F}_t^c- \Dot{F}_t^{\hat{c}}) *\Tilde{v}_0||_p \nonumber 
\end{align*}    
and thus
 \begin{align}
	\begin{split}    
   ||w(\cdot,t) - \Tilde{m}(\cdot,t)||_p &\leq |t| \ ||v_0 - \Tilde{v}_0||_p + ||u_0 - \Tilde{u}_0||_p + C_p c|t| \ ||\nabla(u_0 - \Tilde{u}_0)||_p\\
    &\hspace{20pt} + ||(F_t^c-F_t^{\hat{c}}) * \Tilde{u}_0
||_p + ||(\Dot{F}_t^c- \Dot{F}_t^{\hat{c}}) *\Tilde{v}_0||_p
\end{split}
\end{align}
The terms with $F_t^c-F_t^{\hat{c}}$ and $\Dot{F}_t^c-\Dot{F}_t^{\hat{c}}$ may be further controlled in terms of $|c - \hat{c}|$ with additional hypotheses such as Lipschitz continuity of $u_0$ and $v_0$. Likewise, the quantity $||w(\cdot,t) - \Tilde{m}(\cdot,t)||_p$ may be further controlled if additional hypotheses are made on $u_0$ and/or $v_0$, such as the ones presented in the next section. %We do not show the related computations and leave them to the interested reader as they are not a priority in this text. % one can use radial symmetry/compactly supported hypotheses (see next section) on $u_0$ and $v_0$ to further control the difference $||w(\cdot,t) - \Tilde{m}(\cdot,t)||_p$ when $R$ and $x_0$ are unknown and estimated. We do not show the related computations and leave them to the interested reader as they are not a priority in this text.

\subsection{The special case of radial symmetry}\label{sub : radial}
Suppose that the trajectories of $V^0$ have radial symmetry centered on some $x_0 \in \mathbb{R}^3$, which can be expressed in terms of linear operators with the spherical coordinates $(r,\theta,\phi)$ around $x_0$:
\begin{align}
    \mathbb{P}(\{\omega \in \Omega : \partial_{\theta}V_{\omega}^0 = 0\}) = 1 \ \ \ \text{ and } \ \ \
    \mathbb{P}(\{\omega \in \Omega : \partial_{\phi}V_{\omega}^0 = 0\}) = 1
\end{align}
Then by Proposition \ref{prop : diff constraints}, it is equivalent to  $k_{\mathrm{v}}$ verifying (in the sense of distributions)
\begin{align*}
    \forall x \in \mathcal{D},\ \ \ \partial_{\theta}(k_{\mathrm{v}}(x,\cdot)) = 0 \ \ \ \text{ and } \ \ \ \partial_{\phi}(k_{\mathrm{v}}(x,\cdot)) = 0
\end{align*}
This means that there exists a kernel $k_{\mathrm{v}}^0$ on $\mathbb{R}_+$ such that $k_{\mathrm{v}}(x,x') = k_{\mathrm{v}}^0(r^2,r'^2)$. Suppose also that the trajectories of $U^0$ exhibit radial symmetry : there exists $k_{\mathrm{u}}^0$ such that $k_{\mathrm{u}}(x,x') = k_{\mathrm{u}}^0(r^2,r'^2)$. Then we have the following ``explicit" formulas :

\begin{proposition}\label{prop : radial sym kernel}
Set $K_{\mathrm{v}}(r,r') = \int_0^{\mathrm{R}} \int_0^{r'}k_{\mathrm{v}}^0(s,s')dsds'$. Then $\forall z = (x,t) \in \mathbb{R}^3 \times \mathbb{R}$, $\forall z' = (x',t') \in \mathbb{R}^3 \times \mathbb{R}$,
\begin{align}
k_{\mathrm{v}}^{\mathrm{wave}}(z,z') &= \frac{\text{sgn}(tt')}{16c^2 r r'} \sum_{\varepsilon,\varepsilon' \in \{-1,1\}}\varepsilon \varepsilon' K_{\mathrm{v}}\big( (r + \varepsilon c|t|)^2,(r' + \varepsilon' c|t'|)^2\big) \label{eq:ft ft' gen}
\end{align}
\begin{align}
k_{\mathrm{u}}^{\mathrm{wave}}(z,z') &= \frac{1}{4rr'}\sum_{\varepsilon,\varepsilon' \in \{-1,1\}}(r+\varepsilon c|t|)(r'+\varepsilon' c|t'|)k_{\mathrm{u}}^0\big((r + \varepsilon c|t|)^2,(r' + \varepsilon' c|t'|)^2\big) \label{eq:ftp ft'p gen}
\end{align}
%\begin{equation}\label{eq:ft ft' gen}

%\begin{split}
% k_{\mathrm{v}}^{\mathrm{wave}}(z,z') & = \frac{\text{sgn}(tt')}{16c^2 r r'} \times
%     \Big[ K_{\mathrm{v}}\big((r+c|t|)^2,(r'+c|t'|)^2\big) + K_{\mathrm{v}}\big((r-c|t|)^2,(r'-c|t'|)^2\big) \\
%     & \hspace{18.4mm} - K_{\mathrm{v}}\big((r+c|t|)^2,(r'-c|t'|)^2\big) - K_{\mathrm{v}}\big((r-c|t|)^2,(r'+c|t'|)^2\big) \Big]
%\end{split}
%\end{equation}
%\begin{equation}
%\begin{split}
%        k_{\mathrm{u}}^{\mathrm{wave}}(z,z') = \frac{1}{4rr'} \times \Big[ &(c|t|+r)(c|t'|+r')k_{\mathrm{u}}^0\big((r+c|t|)^2,(r'+c|t'|)^2\big) \\
%    + &(c|t|-r)(c|t'|-r')k_{\mathrm{u}}^0\big((r-c|t|)^2,(r'-c|t'|)^2\big) \\
%    - &(c|t|+r)(c|t'|-r')k_{\mathrm{u}}^0\big((r+c|t|)^2,(r'-c|t'|)^2\big)\\
%    - &(c|t|-r)(c|t'|+r')k_{\mathrm{u}}^0\big((r-c|t|)^2,(r'+c|t'|)^2\big)\Big]
%\end{split}
%\end{equation}

\end{proposition}
\begin{proof}
Without loss of generality, we suppose that $c = 1$ and $x_0 = 0$.
We first derive expression \eqref{eq:ft ft' gen}. Let $f$ be a function defined on $\mathbb{R}_+$ and $g$ the function defined on $\mathbb{R}^3$ by $g(x) = f(|x|^2)$. Let $F$ be an antiderivative of $f$. Recall $\gamma$ and $d\Omega$ defined in  \ref{not : spherical}. As in \eqref{eq:change M}, let $M$ be an orthogonal matrix such that $M(x/|x|) = e_3$ and use the change of variable $\gamma' = M\gamma$. Note that $M S(0,1) = S(0,1)$:
\begingroup
\allowdisplaybreaks
\begin{align}
    (F_t * g)(x) &= \frac{1}{4\pi t}\int_{S(0,1)}g(x-t\gamma)t^2d\Omega = \frac{t}{4\pi} \int_{S(0,1)}f(|x-t\gamma|^2)d\Omega \nonumber\\
    &=  \frac{t}{4\pi} \int_{S(0,1)}f(|x|^2 + t^2 - 2|t|\langle x,\gamma \rangle )d\Omega \nonumber \\
    &= \frac{t}{4\pi} \int_{M S(0,1)}f(|x|^2 + t^2 - 2|t||x|\Big\langle \frac{x}{|x|},M^T\gamma' \Big\rangle )d\Omega' \nonumber \\
    &=  \frac{t}{4\pi} \int_{S(0,1)}f(|x|^2 + t^2 - 2|t||x|\langle e_3,\gamma \rangle )d\Omega \nonumber \\
    &= \frac{t}{4\pi} \int_{\phi = 0}^{2\pi}\int_{\theta = 0}^{\pi}f(|x|^2 + t^2 - 2|t||x|\cos(\theta))\sin(\theta)d\theta d\phi \nonumber \\
    &= \frac{t}{2} \int_{\theta = 0}^{\pi}f(|x|^2 + t^2 - 2|t||x|\cos(\theta))\sin(\theta)d\theta  \nonumber\\
    %&= \frac{t}{4 |t||x|} \int_{\theta = 0}^{\pi}f(|x|^2 + t^2 - 2|t||x|\cos(\theta))\times 2|t||x|\sin(\theta)d\theta \\
    &= \frac{t}{4|x||t|}\Big[F(|x|^2 + t^2 - 2|t||x|\cos(\theta))\Big]_{\theta = 0}^{\theta = \pi} \nonumber \\
    &= \frac{\text{sgn}(t)}{4|x|}\Big(F((|x|+|t|)^2) - F((|x|-|t|)^2) \Big)\nonumber \\
   &= \frac{\text{sgn}(t)}{4|x|}\sum_{\varepsilon \in \{-1,1\}}\varepsilon F((|x|+\varepsilon|t|)^2) \label{eq:ft on f}
\end{align}
\endgroup
%\modif{l'écire avec une somme sur epsilon}
%If $f$ is supported on $[0,R]$ for some $R > 0$, $F$ will be constant outside of $[0,R]$, yielding the expression 
%\begin{align}
%   (F_t * g)(x) =  \frac{\text{sgn}(t)}{4|x|}\Big(F(s_1(x,t)) - F(s_2(x,t)) \Big)
%\end{align}
%with $s_1(x,t) := \min(R,|x|+|t|)^2$ and $s_1(x,t) := \min(R,\big||x|-|t|\big|)^2$. We now apply result \eqref{eq:ft on f} twice on the function  $k_{\mathrm{v}}(x,x') = k_0(|x|^2,|x'|^2)$. 
First, we note
\begin{align}
    &k_0^1(r,r') := \int_0^{r'} k_0(r,s)ds \hspace{5mm} \text{ and } \hspace{5mm} K_{\mathrm{v}}(r,r') :=  \int_0^{\mathrm{R}}\int_0^{r'} k(s,s')ds'ds \label{def Kint}
\end{align}
We apply twice result \eqref{eq:ft on f} on $k_{\mathrm{v}}$ : first by setting $g(x') = k_0(|x-t\gamma|^2,|x'|^2)$ where $x-t\gamma$ is fixed, which integrates to $F(s) = k_0^1(|x-t\gamma|^2,s)$. Second, by setting $g(x) = k_0^1(|x|^2,(|x'|+\varepsilon|t'|)^2)$ where $|x'|+\varepsilon'|t'|$ is fixed, which integrates to $F(s) = K_{\mathrm{v}}(s,(|x'|+\varepsilon'|t'|)^2)$ :

 %which yields \modif{dire qui est g ici}
\begin{align} \label{eq:ft ft proof}
    [(F_t \otimes F_{t'})*k_{\mathrm{v}}](x,x')& = \frac{1}{4\pi t}\frac{1}{4 \pi t'} \int_{S(0,1)} \int_{S(0,1)} k_0(|x-t\gamma|^2,|x'-t'\gamma'|^2){t'}^2d\Omega' t^2 d\Omega \nonumber\\
    & = \frac{1}{4 \pi t} \frac{\text{sgn}(t')}{4|x'|} \int_{S(0,1)} \sum_{\varepsilon' \in \{-1,1\}}\varepsilon' k_0^1(|x-t\gamma|^2,(|x'|+\varepsilon'|t'|)^2) t^2 d\Omega \nonumber\\
     & = \frac{\text{sgn}(tt')}{16 r r'} \sum_{\varepsilon,\varepsilon' \in \{-1,1\}}\varepsilon \varepsilon' K_{\mathrm{v}}\big( (r + \varepsilon |t|)^2,(r' + \varepsilon' |t'|)^2\big)
\end{align}
 %& = \frac{\text{sgn}(tt')}{16|x||x'|} \times
 %    \Big[ K_{\mathrm{v}}\big((|x|+|t|)^2,(|x'|+|t'|)^2\big) + K_{\mathrm{v}}\big((|x|-|t|)^2,(|x'|-|t'|)^2\big) \nonumber\\
%     & \hspace{19.5mm} - K_{\mathrm{v}}\big((|x|+|t|)^2,(|x'|-|t'|)^2\big) - K_{\mathrm{v}}\big((|x|-|t|)^2,(|x'|+|t'|)^2\big) \nonumber \Big]\nonumber \\ 
     
We can then use this result to compute 
\begin{align}
[(\Dot{F}_t \otimes \Dot{F}_{t'})*k_{\mathrm{u}}](x,x') = \partial_t \partial_{t'} [(F_t \otimes F_{t'})*k_{\mathrm{u}}](x,x')
\end{align}
First, we compute it for $t \neq 0$ and $t' \neq 0$ by derivating \eqref{eq:ft ft proof} w.r.t. $t$ and $t'$, using that for $t \neq 0$, $d|t|/dt = \text{sgn}(t)$ and $d \text{sgn}(t)/dt = 0$. It immediately yields
\begin{align} \label{eq:ftp ftp}
    [(\Dot{F}_t \otimes \Dot{F}_{t'})*k_{\mathrm{u}}](x,x') = \frac{1}{4rr'}\sum_{\varepsilon,\varepsilon' \in \{-1,1\}}(r+\varepsilon |t|)(r'+\varepsilon' |t'|)k_{\mathrm{u}}^0\big((r + \varepsilon |t|)^2,(r' + \varepsilon' |t'|)^2\big)
\end{align}
For the case where $t = 0$ or $t' = 0$, note from \eqref{eq:fourier of ft ftp} that $\mathcal{F}(\dot{F}_0)(\xi) = 1$ and thus $\dot{F}_0 = \delta_0$, the Dirac mass at $0$, which is the neutral element for the convolution. Therefore, when we have both $t = 0$ and $t' = 0$ :
\begin{align*}
[(\Dot{F}_0 \otimes \Dot{F}_{0})*k_{\mathrm{u}}](x,x') = [(\delta_0 \otimes \delta_{0})*k_{\mathrm{u}}](x,x') = [\delta_{(0,0)}*k_{\mathrm{u}}](x,x') = k_{\mathrm{u}}(x,x')
\end{align*}
which is also the result provided by \eqref{eq:ftp ftp} evaluated at $t = t' = 0$. When $t'=0$ and $t \neq 0$, then we still have $d|t|/dt = \text{sgn}(t)$ and $d\text{sgn}(t)/dt = 0$ :
\begin{align*}
[(\Dot{F}_t \otimes \Dot{F}_{0})*k_{\mathrm{u}}](x,x') &= [(\Dot{F}_t \otimes \delta_{0})*k_{\mathrm{u}}](x,x') = \partial_t [(F_t \otimes \delta_{0})*k_{\mathrm{u}}](x,x') \\
 &= \partial_t \int_{\mathbb{R}^3}\int_{\mathbb{R}^3}k_{\mathrm{u}}^0(|x-y|^2,|x'-y'|^2)F_t(dy)\delta_0(dy') \\
 &= \partial_t \frac{1}{4\pi t}\int_{S(0,1)} k_{\mathrm{u}}^0(|x-t\gamma|^2,|x'|^2)t^2d\Omega \\
 &= \partial_t \frac{\text{sgn}(t)}{4r} \sum_{\varepsilon \in \{-1,1\}}\varepsilon k_0^1((r+\varepsilon |t|)^2,|x'|^2) \\
 &= \frac{1}{2r}\sum_{\varepsilon \in \{-1,1\}}(r+\varepsilon |t|) k_{\mathrm{u}}^0((r+\varepsilon |t|)^2,|x'|^2)
\end{align*}
which is also the result provided by \eqref{eq:ftp ftp} evaluated at $t' = 0$. The same arguments apply to show that expression \eqref{eq:ftp ftp} is valid when $t = 0$ and $t' \neq 0$. Therefore the expression \eqref{eq:ftp ftp} is valid whatever the value of $t,t' \in \mathbb{R}$.
\end{proof}

\begin{remark}
Note that expressions \eqref{eq:ft ft' gen} and \eqref{eq:ftp ft'p gen} are much faster to compute numerically than \eqref{eq: ku wave} and \eqref{eq: kv wave}, which require to compute convolutions. 
\end{remark}
\subsubsection{Compactly supported Initial Conditions and Finite Speed Propagation} Suppose that $v_0$ is compactly supported on a ball $B(x_0,R)$. The Strong Huygens Principle for the 3 dimensional wave equation (\cite{evans1998}, p.80) states that $F_t *v_0$ is supported on the ring $B(x_0,R + c|t|) \setminus B(x_0,(R - c|t|)_+)$, where $x_+ := \max(0,x)$.
From a Gaussian process modelling point of view, supposing that $\text{Supp}(V^0) \subset B(x_0,R)$ amounts to imposing that if $x \notin B(x_0,R)$, then $V^0(x) = 0 \ a.s.$. This is equivalent to $\text{Var}(V^0(x)) = k_{\mathrm{v}}(x,x) = 0$ since $V^0$ is supposed centered. The same reasoning in terms of compactly supported initial conditions can be applied to $u_0$ and $U(z)$. In the next proposition, we explore this possibility by forcing the covariance kernels of $U^0$ and $V^0$ to be zero outside of a ball centered on $x_0$ and identifying the consequences on the formulas \eqref{eq:ft ft' gen} and \eqref{eq:ftp ft'p gen}. These formulas are the ones used in section \ref{section : num}.

%A quick way to ensure this compact support constraint is to use a truncated kernel $k_{\mathrm{v}}^{\mathrm{R}}(x,x')$ defined by

%.  From a practical point of view, in equation \eqref{eq:ft ft' gen}, this amounts to replacing $(r+ct)^2$ by $\min((r+ct)^2,R^2)$ and $(r-ct)^2$ by $\min((r-ct)^2,R^2)$ in equation \eqref{eq:ftp ft'p gen}, 

\begin{proposition}\label{prop : radial sym kernel compact}
Let $R_{\mathrm{v}} > 0$ and $R_{\mathrm{u}} > 0$. Let $\alpha \in (0,1)$, $\varphi_{\alpha} : \mathbb{R}_+ \rightarrow [0,1]$ be a $C^{\infty}$ decreasing function such that
\begin{align*}
\varphi_{\alpha}(s) = 
\begin{cases}
1 \ \text{ if } s < \alpha \\
0 \ \text{ if } s \geq 1
\end{cases}
\end{align*}
Set the truncated kernels
\begin{align}
k_{\mathrm{v}}^{R_{\mathrm{v}}}(x,x') &= k_{\mathrm{v}}^{0,R_{\mathrm{v}}}(r^2,r'^2) =  k_{\mathrm{v}}^0(r^2,r'^2)\mathbbm{1}_{[0,R_{\mathrm{v}}]}(r)\mathbbm{1}_{[0,R_{\mathrm{v}}]}(r') \label{eq:truncated kernel v}  \\
k_{\mathrm{u}}^{R_{\mathrm{u}}}(x,x') &= k_{\mathrm{u}}^{0,R_{\mathrm{u}}}(r^2,r'^2) = k_{\mathrm{u}}^0(r^2,r'^2)\varphi\big(r/R_{\mathrm{u}}\big)\varphi\big(r'/R_{\mathrm{u}}\big) \label{eq:truncated kernel u}
\end{align}
Suppose now that $V^0 \sim GP(0,k_{\mathrm{v}}^{R_{\mathrm{v}}})$ and $U^0 \sim GP(0,k_{\mathrm{u}}^{R_{\mathrm{u}}})$.
Then, with $K_{\mathrm{v}}(r,r') = \int_0^{\mathrm{R}} \int_0^{r'}k_{\mathrm{v}}^0(s,s')dsds'$,
\begin{align}
k_{\mathrm{v}}^{\mathrm{wave}}(z,z') &= \frac{\text{sgn}(tt')}{16c^2 r r'} \sum_{\varepsilon,\varepsilon' \in \{-1,1\}}\varepsilon \varepsilon' K_{\mathrm{v}}\Big(\min\big( (r + \varepsilon c|t|)^2,R_{\mathrm{v}}^2\big),\min\big((r' + \varepsilon' c|t'|)^2,R_{\mathrm{v}}^2\big)\Big) \label{eq:ft ft' compact} \\
k_{\mathrm{u}}^{\mathrm{wave}}(z,z') &= \frac{1}{4rr'}\sum_{\varepsilon,\varepsilon' \in \{-1,1\}}(r+\varepsilon c|t|)(r'+\varepsilon' c|t'|)k_{\mathrm{u}}^{0,R_{\mathrm{u}}}\big((r + \varepsilon c|t|)^2,(r' + \varepsilon' c|t'|)^2\big) \label{eq:ftp ft'p compact}
\end{align}
\end{proposition}

\begin{proof}
When using kernel $k_{\mathrm{v}}^{0,R_{\mathrm{v}}}$, we can directly use equation \eqref{eq:ft ft' gen} by substituting $K_{\mathrm{v}}$ with $K_{\mathrm{v}}^{R_{\mathrm{v}}}(r,r') := \int_0^{r} \int_0^{r'}k_{\mathrm{v}}^{0,R_{\mathrm{v}}}(s,s')dsds'$ and realizing that for all $r,r' \geq 0$,
\begin{align*}
K_{\mathrm{v}}^{R_{\mathrm{v}}}(r^2,r'^2) := \int_0^{r^2} \int_0^{r'^2}k_{\mathrm{v}}^{0,R_{\mathrm{v}}}(s,s')dsds' = K_{\mathrm{v}}\Big(\min\big( r^2,R_{\mathrm{v}}^2\big),\min\big(r'^2,R_{\mathrm{v}}^2\big)\Big)
\end{align*}
which directly proves \eqref{eq:ft ft' compact}. Additionally, \eqref{eq:ftp ft'p compact} is only a substitution of $k_{\mathrm{u}}^0$ with $k_{\mathrm{u}}^{0,R_{\mathrm{u}}}$ in \eqref{eq:ftp ft'p gen} : all the mathematical steps are justified as $\varphi \in C^{\infty}(\mathbb{R}_+)$.
\end{proof}
\begin{remark}
$k_{\mathrm{v}}^{R_{\mathrm{v}}}$ and $k_{\mathrm{u}}^{R_{\mathrm{u}}}$ are the covariance kernels of the truncated processes $V_{\mathrm{trunc}}^0(x) = \mathbbm{1}_{[0,R_{\mathrm{v}}]}(|x-x_0|)V^0(x)$ and $U_{\mathrm{trunc}}^0(x) = \varphi\big(|x-x_0|/R_{\mathrm{u}}\big)U^0(x)$ respectively. These kernels can never be stationary as their trajectories are compactly supported.
\end{remark}
Using equation \eqref{eq:ft ft' compact}, one can then check that $k_{\mathrm{v}}^{\mathrm{wave}}(z,z) = \text{Var}(V(z)) = 0$ as soon as $(r - c|t|)^2 > R_{\mathrm{v}}^2$, ie $ V(z) = 0 \ a.s. $ and likewise for $k_{\mathrm{u}}^{\mathrm{wave}}$ : this is the expression of the strong Huygens principle on the kernels $k_{\mathrm{v}}^{\mathrm{wave}}$ and $k_{\mathrm{u}}^{\mathrm{wave}}$.

Such compactly supported kernels may lead to sparse covariance matrices which may then be used for computational speedups, see section 3.6.

\subsubsection{Estimation of physical parameters}
The wave kernel \eqref{eq:wave kernel}, using for $k_{\mathrm{u}}$ and $k_{\mathrm{v}}$ radially symmetric kernels supported in $B(x_0^u,R_{\mathrm{u}})$ and $B(x_0^v,R_{\mathrm{v}})$ respectively, has for hyperparameters
\begin{align*}
    \theta = (c,x_0^u,R_{\mathrm{u}},\theta_{k_{\mathrm{u}}^0},x_0^v,R_{\mathrm{v}},\theta_{k_{\mathrm{v}}^0})
\end{align*}
Among those, $(c,x_0^u,R_{\mathrm{u}},x_0^v,R_{\mathrm{v}})$ all correspond to physical parameters that can be estimated via likelihood maximisation. This is numerically investigated in Section \ref{section : num}.
\begin{remark}\label{rk:R_difficult}
Note that finding the correct $R_{\mathrm{u}}$ and $R_{\mathrm{v}}$ is not a well posed problem : if $\text{Supp} (U^0) \subset B(x_0^u,R_{\mathrm{u}})$ then $\text{Supp}(U^0) \subset B(x_0^u,\alpha R_{\mathrm{u}})$ for any $\alpha \geq 1$ and $\alpha R_{\mathrm{u}}$ is also a correct solution. This is discussed in Section \ref{section : num}.
\end{remark}
%in section (), we show that when $R_{\mathrm{u}}$ and/or $R_{\mathrm{v}}$ are fixed to their correct or overestimated values, the hyperparameter estimation step behaves better.

\subsubsection{Gaussian process regression, radial symmetry and the 1D wave equation} Note that supposing that the initial conditions exhibit radial symmetry reduces Problem \eqref{eq:wave_eq} to the 1D wave equation since for a radially symmetric function $f$ defined on $\mathbb{R}^3$, we have $(\Delta_{3D} f)(r) = \frac{1}{r^2}[\partial_{rr}^2 (r^2f)](r)$ and thus if $w$ in \eqref{eq:wave_eq} exhibits radial symmetry,
\begin{align*}
    \Box_{3D} w = 0 &\iff \frac{1}{c^2}\partial_{tt}^2 w - \Delta_{3D} w = 0 \iff \frac{1}{c^2}\partial_{tt}^2 w - \frac{1}{r^2}\partial_{rr}^2 (r^2w) = 0 \\
    &\iff \frac{1}{c^2}\partial_{tt}^2 \Tilde{w} - \partial_{rr}^2 \Tilde{w} = 0
\end{align*}
with $\Tilde{w}(r,t) = r^2w(r,t)$. However, the joint problem of approximating the solution $w$ of Problem \eqref{eq:wave_eq} with GPR in presence of radial symmetry \textit{and} searching for the real source location parameters $(x_0^u,R_{\mathrm{u}},x_0^v,R_{\mathrm{v}})$ cannot be reduced to the one dimensional case, as the centers $x_0^u$ and $x_0^v$ both belong to $\mathbb{R}^3$. 

\subsection{The Point Source Limit}
When studying linear PDEs, the case of the point source is very important. From an application point of view, modelling the source term as a point source (i.e. a Dirac mass) is relevant in a number of real life cases : a localized detonation in acoustics, an electric point source in electromagnetics, a mass point in mechanics, and so forth. From a theoretical point of view, the solution of a linear PDE when the source is a Dirac mass also plays a special role. This solution is called the PDE's Green's function and encodes all of the mathematical properties of the corresponding PDE \cite{Duffy2015GreensFW}. In this section, we will not make use of the Kriging equations \eqref{eq:krig mean} and \eqref{eq:krig cov} as reconstructing an initial condition that is a point source is actually of little interest. Also, reconstructing the wave equation's Green function thanks to GPR is expected to yield very poor results because this Green's function in particular is extremely irregular; it is not even a function, it is a family of singular measures, see equation \eqref{eq:ft ftp in 3D}. However, estimating the physical parameters attached to it, essentially the position parameter $x_0$, is a relevant question and an attainable goal. This is what we do in this section, by studying the behaviour of the log marginal likelihood that comes with WIGPR when the initial condition reduces to point source. Here, we restrict ourselves to the case $u_0 = 0$ in \eqref{eq:wave_eq} and thus focus on the kernel $k_{\mathrm{v}}^{\mathrm{wave}}(z,z')$. We begin by clarifying the setting in which we will work.

\subsubsection{Setting} 
\begin{itemize}
\item Note $x_1,...,x_q$ the $q$ sensor locations and suppose we have $N$ time measurements in $[0,T]$ corresponding to times $0 = t_1<... < t_N = T$ for each sensor ; we have overall $n = Nq$ pointwise observations of a function $w$ that is a solution of the problem \eqref{eq:wave_eq}. These observations are stored in a vector $w_{\mathrm{obs}} \in \mathbb{R}^{Nq}$.
\item We suppose that the initial condition $v_0$ that corresponds to $w$ is almost a point source : in particular it is supported on a small ball $B(x_0^*,R^*)$ where $R^* \ll 1$.
\item We are interested in finding $x_0^*$, the correct source location. To do so, we study the log marginal likelihood that corresponds to the observations $w_{\mathrm{obs}}$ using the following truncated kernel
\begin{align}
k_{\mathrm{nor,x_0}}^{\mathrm{R}}((x,t),(x',t')) &= \bigg(\frac{1}{\frac{4}{3}\pi R^3}\bigg)^2[(F_t \otimes F_{t'})*k_{x_0}^{\mathrm{R}}](x,x') \label{eq:trunc_point} \\
\text{where}\ \ \ k_{x_0}^{\mathrm{R}}(x,x') &= k_{\mathrm{v}}(x,x')\mathbbm{1}_{B(x_0,R)}(x)\mathbbm{1}_{B(x_0,R)}(x') \nonumber
\end{align}
The renormalization by the constant $(\frac{4}{3}\pi R^3)^2$ is an anticipation of the upcoming Proposition \ref{prop : point source}. Note that the kernel $k_{x_0}^{\mathrm{R}}$ corresponds to the kernel \eqref{eq:truncated kernel v} in Proposition \ref{prop : radial sym kernel compact}, though we do not make any radial symmetry assumption here (in the point source limit, radial symmetry appears by itself anyway).
Here, $(x_0^*,R^*)$ denotes the real source position and source size corresponding to the observations $w_{\mathrm{obs}}$ while $(x_0,R)$ denotes the position and source size hyperparameter of $k_{\mathrm{nor,x_0}}^{\mathrm{R}}$. We suppose that except for $x_0$, all the other hyperparameters $\theta$ of $k_{\mathrm{nor,x_0}}^{\mathrm{R}}$ are fixed. In particular, we suppose that the kernel speed $c$ is fixed to the correct speed value $c^*$ ; we also suppose that $R = R^*$ for simplicity.
\end{itemize}

In that framework, the log-marginal likelihood $p(w_{\mathrm{obs}}|\theta)$ only depends on $x_0$. We thus write $K_{x_0} = k_{x_0}^{\mathrm{R}}(X,X)$ and $\mathcal{L}(\theta,\lambda) = \mathcal{L}(x_0,\lambda)$. The log-marginal likelihood then writes
\begin{align}\label{eq:log lkhd x_0}
\mathcal{L}(\theta,\lambda) = \mathcal{L}(x_0,\lambda) = w_{\mathrm{obs}}^T(K_{x_0} + \lambda I_n)^{-1}w_ {\mathrm{obs}} + \log \det(K_{x_0} + \lambda I_n)
\end{align}

In Figure \ref{fig : point source}, we provide a 3 dimensional image that shows some numerical values of \eqref{eq:log lkhd x_0} as a function of $x_0$ on a test case. This figure constitutes visual evidence that in the limit $R \rightarrow 0$, trying to recover a point source location from minimizing the log marginal likelihood provided by the kernel \eqref{eq:trunc_point} reduces to the classic triangulation method used for example in GPS systems. More precisely, we observe three facts. $(i)$ As a function of $x_0$, $\mathcal{L}(x_0,\lambda)$ reaches local minima over the whole surface of spheres centered on each sensor. $(ii)$ At the intersection of two of those spheres, the local minima are smaller. $(iii)$ The spheres all intersect at a single point $x_0^*$, which is the global minima of $\mathcal{L}(x_0,\lambda)$ and the real source location.

\begin{figure}[!ht]
	\hspace{-20pt}\centerline{\includegraphics[width=1.8\textwidth]{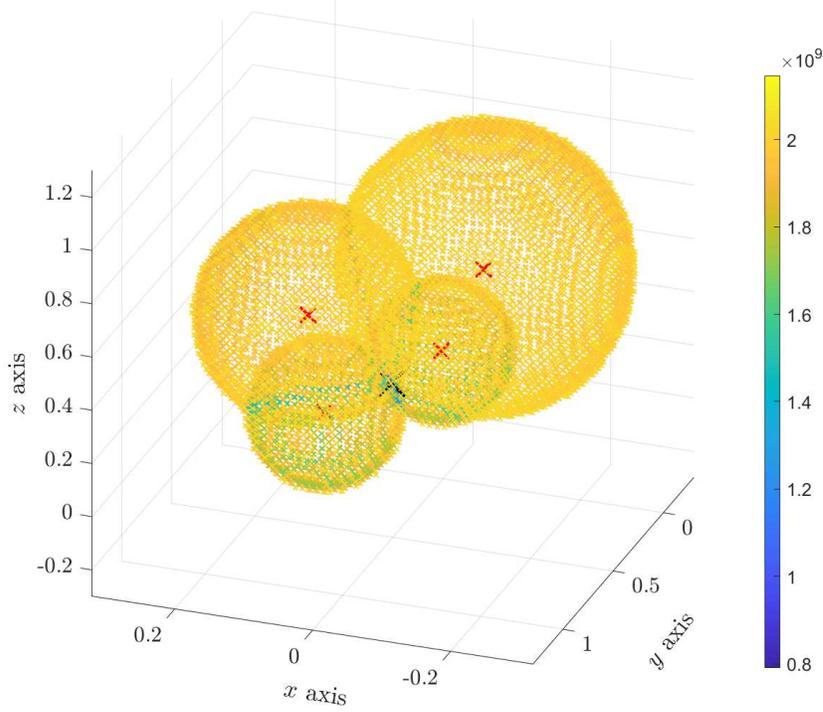}}
	\caption{Negative log marginal likelihood as a function of $x_0 \in \mathbb{R}^3$. Are only represented values of the negative log marginal likelihood that are below $2.035 \times 10^9$. There only remains thin spherical shells. Red crosses : sensor locations. Black cross : source position. The source is located at the intersection of thin spherical shells centered at the sensor locations.}
	\label{fig : point source}
\end{figure}

Our task now is to explain why we observe such patterns on the Figure \ref{fig : point source}. In that sense, we begin with a convergence statement to understand what happens in the point source limit from a Gaussian process point of view. 
%. Suppose that $k$ lies in $L_{loc}^1(\mathbb{R}^3 \times \mathbb{R}^3)$ and is continuous in the vicinity of $(x_0,x_0)$.
\begin{proposition}\label{prop : point source}
Let $k$ be a continuous PD kernel defined on $\mathbb{R}^3 \times \mathbb{R}^3$ and $x_0 \in \mathbb{R}^3$. For $R > 0$, define $k_{x_0}^{\mathrm{R}}$ its truncation around $x_0$ by
\begin{align*}
    k_{x_0}^{\mathrm{R}}(x,x') = k(x,x')\mathbbm{1}_{B(x_0,R)}(x)\mathbbm{1}_{B(x_0,R)}(x')
\end{align*} 
Let $t,t' \in \mathbb{R}$. Then $(F_t \otimes F_{t'})*k_{x_0}^{\mathrm{R}}$ defines an absolutely continuous Radon measure over $\mathbb{R}^3 \times \mathbb{R}^3$. Furthermore we have the following weak-$\star$ convergence in the space of Radon measures :
\begin{align}\label{eq:cv point source}
    \frac{1}{\big(\frac{4}{3}\pi R^3\big)^2}[(F_t \otimes F_{t'})*k_{x_0}^{\mathrm{R}}] \xrightarrow[R \rightarrow 0]{\mathcal{C}_c(\mathbb{R}^3 \times \mathbb{R}^3)'}  k(x_0,x_0) \times (\tau_{x_0} F_t) \otimes (\tau_{x_0}F_{t'})
\end{align}
where $\tau_{x}\mu$, the translation of a measure $\mu$ by $x$, is defined by 
\begin{align*}
\int f(y)\tau_{x}\mu(dy) := \int f(x+y)\mu(dy)
\end{align*}
\end{proposition}
 
\begin{proof} First, using equation \eqref{eq:explicit Ft Ftp}, we have
\begin{align}\label{eq:piecewise}
[(F_t \otimes F_{t'}) * k_{x_0}^{\mathrm{R}}](x,x') = tt'\int_{S(0,1) \times S(0,1)}k_{x_0}^{\mathrm{R}}(x-c|t|\gamma,x'-c|t'|\gamma')\frac{d\Omega d\Omega'}{(4\pi)^2}
\end{align}
The integrated function in equation \eqref{eq:piecewise} is piecewise continuous over $\mathbb{R}^3 \times \mathbb{R}^3$ and the integral in \eqref{eq:piecewise} is well defined, whatever the values of $x$ and $x'$. Let $f$ be a continuous compactly supported function on $\mathbb{R}^3 \times \mathbb{R}^3$. We define
\begin{align*}
I_R :&= \frac{1}{(\frac{4}{3}\pi R^3)^2}\langle (F_t \otimes F_{t'}) * k_{x_0}^{\mathrm{R}},f \rangle
\end{align*}
and wish to show that $I_R \rightarrow k(x_0,x_0)\langle \tau_{x_0}F_t \otimes \tau_{x_0}F_{t'},f \rangle$ when $R \rightarrow 0$. Using equation \eqref{eq:explicit Ft Ftp}, we have
\begin{align*}
I_R &= \frac{1}{(\frac{4}{3}\pi R^3)^2}\ \int_{\mathbb{R}^3 \times \mathbb{R}^3} f(x,x')[(F_t \otimes F_{t'}) * k_{x_0}^{\mathrm{R}}](x,x')dxdx' \\
&= \frac{1}{(\frac{4}{3}\pi R^3)^2} \int_{\mathbb{R}^3 \times \mathbb{R}^3} f(x,x')tt'\int_{S(0,1) \times S(0,1)}k_{x_0}^{\mathrm{R}}(x-c|t|\gamma,x'-c|t'|\gamma')\frac{d\Omega d\Omega'}{(4\pi)^2} dxdx' \\
&= \frac{1}{(\frac{4}{3}\pi R^3)^2}\ tt'\int_{S(0,1) \times S(0,1)}\int_{\mathbb{R}^3 \times \mathbb{R}^3} \Bigg( f(x,x')k_{x_0}(x-c|t|\gamma,x'-c|t'|\gamma')\\
& \hspace{2cm} \times \mathbbm{1}_{[0,R]}(|x-c|t|\gamma - x_0|)\mathbbm{1}_{[0,R]}(|x'-c|t'|\gamma' - x_0|) \Bigg) dxdx'\frac{d\Omega d\Omega'}{(4\pi)^2}
\end{align*}
The first indicator function restricts the integration domain of $x$ to $B(x_0 + ct\gamma,R)$, and likewise for the second indicator function and $x'$.
In spherical coordinates, write $x = x_0 + ct\gamma + R \rho \gamma_x$ with $\rho \in [0,1]$ and $\gamma_x \in S(0,1)$ and likewise for $x'$.
\begin{align*}
I_R &= tt'\int_{S(0,1) \times S(0,1)} \int_{S(0,1) \times S(0,1)} \int_0^1 \int_0^1 \Bigg(f(x_0 + ct\gamma + R \rho \gamma_x,x_0 + c|t'|\gamma' + R \rho' \gamma_{x'}) \\
& \hspace{3cm} \times  k(x_0 + R \rho \gamma_x,x_0 + R \rho' \gamma_{x'}) \Bigg) \times 9 \rho^2 d\rho \rho'^2 d\rho' \frac{d\Omega_x d\Omega_{x'}}{(4\pi)^2}\frac{d\Omega d\Omega'}{(4\pi)^2}
\end{align*}
The integration domain above is a compact subset of $\mathbb{R}^{10}$. Since $f$ is continuous and $k$ is supposed continuous in the vicinity of $(x_0,x_0)$, Lebesgue's dominated convergence theorem can be applied when $R \rightarrow 0$, which yields
\begin{align*}
I_R \xrightarrow[R \rightarrow 0]{} tt' k(x_0,x_0) \int_{S(0,1) \times S(0,1)} &f(x_0 + ct\gamma,x_0 + c|t'|\gamma')\frac{d\Omega d\Omega'}{(4\pi)^2} \\
& = k(x_0,x_0)\langle \tau_{x_0}F_t \otimes \tau_{x_0}F_{t'},f \rangle
\end{align*}
\end{proof} 
In the  Proposition \ref{prop : point source} and as in the propositions \ref{prop : radial sym kernel compact}, the kernel $k_{x_0}^{\mathrm{R}}$ is the covariance kernel of the truncated process $V_{\mathrm{trunc}}^0(x) = \mathbbm{1}_{B(x_0,R)}(x)V^0(x)$.

\begin{remark}
We provided above an ad hoc proof of Proposition \ref{prop : point source}, tailored to the measures $F_t$ and $F_{t'}$. Proposition \ref{prop : point source} is actually a consequence of a more general fact, where $F_t$ and $F_{t'}$ are replaced by any Radon measure $\mu$ and $\nu$. However, the corresponding proof is much more tedious (though not much more difficult) and the truncation function $\mathbbm{1}_{B(x_0,R)}$ should be replaced by a continuous counterpart, to ensure minimal regularity for integrability conditions. Finally, Proposition \ref{prop : point source} can also be seen as a weak-$*$ convergence statement of mollifications of $F_t \otimes F_{t'}$ when convolved with $k_{\mathrm{nor,x_0}}^{\mathrm{R}}$, although the authors could not find references in the literature that state such weak-$*$ convergence theorems in the space of Radon measures.
\end{remark}

%\begin{align}
%k_{x_0}^{\mathrm{R}}(x,x') = \mathbbm{1}_{B(x_0,R)}(x)\mathbbm{1}_{B(x_0,R)}(x')\text{Cov}(V^0(x),V^0(x')) = \text{Cov}(V_{\mathrm{trunc}}^0(x),V_{\mathrm{trunc}}^0(x'))
%\end{align}

%Proposition \ref{prop : point source} is a consequence of a more general result presented in Proposition \ref{prop : point source general radon}.

The limit object we obtain by inspecting equation \eqref{eq:cv point source} is not a function, it is only a measure; and thus it is not a covariance kernel. This means that we do not obtain a Gaussian process in the point source limit. Rather, we obtain a generalized stochastic process in the sense of \cite{agnan2004}, paragraph 2.2.1.1. This is actually not surprising as the Green function $F_t$ of the wave equation is only a measure. This is another manifestation of the fact that the wave equation exhibits no regularization effect, because of its hyperbolicity. Yet, Proposition \ref{prop : point source} provides an entry point for studying the log marginal likelihood \eqref{eq:log lkhd x_0} associated with the kernel \eqref{eq:trunc_point} when $R$ is small. Indeed, Proposition \ref{prop : point source} states that for small values of $R$, the kernel \eqref{eq:trunc_point} behaves like a rank one kernel, i.e. a kernel of the form $k(z,z') = f(z)f(z')$ for some particular function $f$. We will see that this observation is actually enough to provide mathematical understanding of the patterns observed in Figure \ref{fig : point source}. %This is done in the propositions \ref{prop : log lik reg} and \ref{prop : log lik reg N}. %This is the starting point of the proofs of propositions \ref{prop : log lik reg} and \ref{prop : log lik reg N} where we provide theoretical justifications to the patterns observed in Figure \ref{fig : point source}.

Rather than proving that we can use the tensor shape of the limit object in Proposition \ref{prop : point source} even though it is not a function, we prefer to introduce regularized versions of both this limit object and of $\mathcal{L}(x_0,\lambda)$, and study these regularized terms instead. This will be enough to exhibit a limit object that explains the patterns visible on Figure \ref{fig : point source}. This is done in propositions \ref{prop : log lik reg} and \ref{prop : log lik reg N}, which are statements on the regularized log marginal likelihood $\mathcal{L}_{\mathrm{reg}}(x_0,\lambda)$ defined in equation \eqref{eq:log lik reg} rather than on $\mathcal{L}(x_0,\lambda)$. Note however that proving a complete mathematical statement linking the behaviours of $\mathcal{L}(x_0,\lambda)$ and $\mathcal{L}_{\mathrm{reg}}(x_0,\lambda)$ is lacking for the moment.

%, which is what we do in the propositions \ref{prop : log lik reg} and \ref{prop : log lik reg N}. 
%In order to make use of Proposition \ref{prop : point source} in a flexible framework, we regularize the mathematical objects involved in \eqref{eq:log lkhd x_0}.
\subsubsection{Regularizing the point source and corresponding approximations}
First, we regularize $F_t$ thanks to a mollifier $\varphi(x)$ on $\mathbb{R}^3$ which we choose to be radially symmetric as in \cite{evans2018measure}, section 4.2.1. %\cite{Hrmander1990TheAO}. 
Define $\varphi_R(y) = \varphi(y/R)/R^3$, then a $\mathcal{C}^{\infty}_c$ regularization of $F_t$ is obtained by
\begin{align*}
f_t^{\mathrm{R}}(x) := (F_t * \varphi_R)(x) \ \ \ \forall x \in \mathbb{R}^3
\end{align*}

As $F_t$, $f_t^{\mathrm{R}}$ exhibits radial symmetry. We will next use the following approximated objects :
\begin{itemize}
\item Note $k_{x_0}^{\mathrm{reg}}((x,t),(x',t')) := f_t^{\mathrm{R}}(x-x_0)f_{t'}^{\mathrm{R}}(x'-x_0)$ a regularized version of the limit object in  Proposition \ref{prop : point source}. Proposition \ref{prop : point source} states that in some sense, when $R$ approaches $0$, $k_{\mathrm{nor,x_0}}^{\mathrm{R}}$ is close to $k_{x_0}^{\mathrm{reg}}$.
\item We also assume that $w(x_i,t_j)$ may be approximated by $\tilde{w}(x_i,t_j) = f_{t_j} ^{\mathrm{R}}(x_i-x_0^*)$ as in the point source limit, $v_0 = \delta_{x_0^*}$ and in that case we would have $w(x_i,t_j) = (F_{t_j} * v_0)(x_i) = F_{t_j}(x_i-x_0^*)$ (forgetting for a second that $F_t$ is not defined pointwise). We thus introduce the column vector of ``approximated observations" $W = \big(\tilde{w}(x_i,t_j)\big)_{i,j}$ and we suppose that $W$ is ordered in the following way :
\begin{align*}
W = \begin{pmatrix}
    W_1 \\
     \cmidrule(lr){1-1}
     \vdots \\
     \cmidrule(lr){1-1}
	W_q
\end{pmatrix}
\end{align*}
where $W_i$ corresponds to the sensor $i$ : $W_i = (\tilde{w}(x_i,t_1),...,\tilde{w}(x_i,t_N))^T$.
\end{itemize}
These approximated objects enable us to introduce the ``regularized" log marginal likelihood built by replacing $k$ with $k_{x_0}^{\mathrm{reg}}$ and $w_{\mathrm{obs}}$ by $W$ :
\begin{align}\label{eq:log lik reg}
\mathcal{L}_{\mathrm{reg}}(x_0,\lambda) := W^T(K_{x_0}^{\mathrm{reg}} + \lambda I_n)^{-1}W + \log \det(K_{x_0}^{\mathrm{reg}} + \lambda I_n)
\end{align}
with $K_{x_0}^{\mathrm{reg}} = k_{x_0}^{\mathrm{reg}}(X,X)$. Instead of studying $\mathcal{L}(x_0,\lambda)$, we now rather analyse $\mathcal{L}_{\mathrm{reg}}(x_0,\lambda)$; we expect $\mathcal{L}_{\mathrm{reg}}(x_0,\lambda)$ to be close to $\mathcal{L}(x_0,\lambda)$ in some sense, although here we do not state any approximation theorem w.r.t. the approximation of $\mathcal{L}(x_0,\lambda)$ with $\mathcal{L}_{\mathrm{reg}}(x_0,\lambda)$. %The point here is to avoid drowning the forthcoming analysis as well as the two following propositions in technicalities as they are, along with their proofs, simple results.
%: the goal here is only to reveal a limit object that provides an explanation of the shapes visible on Figure \ref{fig : point source}, which is what we do in the propositions \ref{prop : log lik reg} and \ref{prop : log lik reg N}. 

Because of the tensor product shape of $k_{x_0}^{\mathrm{reg}}$, $K_{x_0}^{\mathrm{reg}}$ has rank one : $K_{x_0}^{\mathrm{reg}} = F_{x_0}F_{x_0}^T$ with 
\begin{align*}
F_{x_0} = \begin{pmatrix}
    F_{x_0}^1 \\
     \cmidrule(lr){1-1}
     \vdots \\
     \cmidrule(lr){1-1}
	F_{x_0}^q
\end{pmatrix}
\end{align*}
and $F_{x_0}^i = (f_{t_1}^{\mathrm{R}}(x_i-x_0),...,f_{t_N}^{\mathrm{R}}(x_i-x_0))^T$. Below, we write a proposition which describes the asymptotic behaviour of $\mathcal{L}_{\mathrm{reg}}(x_0,\lambda)$ in the limit of $\lambda \rightarrow 0$. This limit corresponds to noiseless observations; we will see that the limit object in Proposition \ref{prop : log lik reg} provides an explanation of the patterns of Figure \ref{fig : point source}. We note $||\cdot||$ and $\langle \cdot,\cdot \rangle$ the usual Euclidean norm and dot product in $\mathbb{R}^n$.

%We note $||W||^2 = \sum_{i=1}^q ||W_i||^2 = \sum_{i=1}^q \sum_{j=1}^N (W_i)_j^2$ and likewise for $F_{x_0}$; we also note $\langle F_{x_0},W \rangle = \sum_{i=1}^q \langle F_{x_0}^i,W_i \rangle = \sum_{i=1}^q \sum_{j=1}^n (F_{x_0}^i)_j(W_i)_j$.

\begin{proposition}[Asymptotic behaviour of $\mathcal{L}_{\mathrm{reg}}(x_0,\lambda)$ when $\lambda \rightarrow 0$]\label{prop : log lik reg}
Let $\varepsilon > 0$ and $E_{\varepsilon} := \{x_0 \in \mathbb{R}^3 : \rho(K_{x_0}^{\mathrm{reg}}) > \varepsilon \}$ where $\rho(A)$ denotes the spectral radius of $A$. Define the correlation coefficient between $F_{x_0}$ and $W$ by $r(x_0) = \text{Corr}(F_{x_0},W) = \langle F_{x_0},W \rangle /(||W|| \times ||F_{x_0}||) \in [-1,1]$. We set $r(x_0) = 0$ if $F_{x_0} = 0$.

Then we have the pointwise convergence
\begin{align*}
\forall x_0 \in \mathbb{R}^3, \ \ \ \big|\lambda \mathcal{L}_{\mathrm{reg}}(x_0,\lambda) - ||W||^2\big(1-r(x_0)^2\big)\big| = O_{\lambda \rightarrow 0}(\lambda \log \lambda)
\end{align*}
and the uniform convergence on $E_{\varepsilon}$
\begin{align*}
\sup_{x_0 \in E_{\varepsilon}} \big|\lambda \mathcal{L}_{\mathrm{reg}}(x_0,\lambda) - ||W||^2\big(1-r(x_0)^2\big)\big| = O_{\lambda \rightarrow 0}(\lambda \log \lambda)
\end{align*}
\end{proposition}
\begin{proof}
Suppose first that $||F_{x_0}||^2 = 0$. Then obviously, $r(x_0) = 0$ and $\mathcal{L}_{\mathrm{reg}}(x_0,\lambda) =||W||^2/\lambda + n \log \lambda$ which shows that 
\begin{align}\label{eq:pointwise}
\big|\lambda\mathcal{L}_{\mathrm{reg}}(x_0,\lambda) - ||W||^2 \big| = O_{\lambda \rightarrow 0}(\lambda \log \lambda)
\end{align}

Now, suppose that $||F_{x_0}||^2 > 0$. Let $\varepsilon > 0$ such that $||F_{x_0}||^2 \geq \varepsilon$.

We first deal with the first term in \eqref{eq:log lik reg}. Using the Sherman–Morrison formula (\cite{2007numerical_recipes}, section 2.7.1), we may invert $(K_{x_0}^{\mathrm{reg}} + \lambda I_n)$ explicitly :
\begin{align*}
	(K_{x_0}^{\mathrm{reg}} + \lambda I_n)^{-1} = \frac{1}{\lambda}I_n - \frac{1}{\lambda^2} \frac{F_{x_0}F_{x_0}^T}{1+ \frac{1}{\lambda}F_{x_0}^TF_{x_0}} = \frac{1}{\lambda}\Big(I_n - \frac{F_{x_0}F_{x_0}^T}{\lambda + ||F_{x_0}||^2}\Big)
\end{align*}
One also easily derives the determinant term as $F_{x_0}F_{x_0}^T$ only has one non zero eigenvalue equal to $||F_{x_0}||^2$, since $(F_{x_0}F_{x_0}^T)F_{x_0} = F_{x_0}(F_{x_0}^TF_{x_0}) = ||F_{x_0}||^2F_{x_0}$ :
\begin{align}
\log \det(K_{x_0}^{\mathrm{reg}} + \lambda I_n) = (n-1) \log \lambda + \log \lambda + ||F_{x_0}||^2)
\end{align} 
(The same argument shows that $\rho(K_{x_0}^{\mathrm{reg}}) = ||F_{x_0}||^2$.) Thus,
\begin{align}\label{eq:L reg corr}
\mathcal{L}_{\mathrm{reg}}(x_0,\lambda) &= W^T (K_{x_0}^{\mathrm{reg}} + \lambda I_n)^{-1} W + \log \det(K_{x_0}^{\mathrm{reg}} + \lambda I_n) \nonumber \\
&= \frac{1}{\lambda}\bigg(||W||^2 - \frac{\langle F_{x_0},W\rangle^2}{\lambda + ||F_{x_0}||^2}\bigg) + (n-1) \log \lambda + \log (\lambda + ||F_{x_0}||^2) \nonumber \\
&= \frac{||W||^2}{\lambda}\bigg(1 - \frac{\langle F_{x_0},W\rangle^2}{||W||^2(\lambda + ||F_{x_0}||^2)}\bigg) + (n-1) \log \lambda + \log (\lambda + ||F_{x_0}||^2) 
\end{align}
and
\begin{align}
\label{eq:log reg diff}
\lambda\mathcal{L}_{\mathrm{reg}}(x_0,\lambda)&-||W||^2(1-r(x_0)^2) \nonumber \\
&= ||W||^2\bigg(\frac{\langle F_{x_0},W \rangle^2}{||W||^2 ||F_{x_0}||^2} - \frac{\langle F_{x_0},W \rangle^2}{||W||^2 (\lambda + ||F_{x_0}||^2)} \bigg) \\ &\hspace{15pt}+ (n-1)\lambda \log \lambda + \lambda \log (\lambda + ||F_{x_0}||^2) \nonumber
\end{align}
But for the term in \eqref{eq:log reg diff} that is multiplied by $||W||^2$,
\begin{align}\label{eq:bound 1 log reg}
\frac{\langle F_{x_0},W \rangle^2}{||W||^2 ||F_{x_0}||^2} - \frac{\langle F_{x_0},W \rangle^2}{||W||^2 (\lambda + ||F_{x_0}||^2)} &= \frac{\langle F_{x_0},W \rangle^2}{||W||^2}\bigg(\frac{1}{||F_{x_0}||^2} -\frac{1}{\lambda + ||F_{x_0}||^2}  \bigg) \nonumber \\
&= \frac{\langle F_{x_0},W \rangle^2}{||W||^2} \frac{\lambda}{||F_{x_0}||^2(\lambda + ||F_{x_0}||^2)} \nonumber \\
&\leq r(x_0)^2  \frac{\lambda}{\lambda + ||F_{x_0}||^2} \leq \frac{\lambda}{||F_{x_0}||^2} \leq \frac{\lambda}{\varepsilon}
\end{align}
and obviously, since $\lambda > 0$, 
\begin{align}\label{eq:truc_positif}
\frac{\langle F_{x_0},W \rangle^2}{||W||^2 ||F_{x_0}||^2} - \frac{\langle F_{x_0},W \rangle^2}{||W||^2 (\lambda + ||F_{x_0}||^2)} \geq 0
\end{align}
Also, one sees that $F_{x_0} = 0$ as soon as $\sup_i |x_0-x_i| > cT + R$, ie $x_0$ is too far from the receivers for them to capture non zero signal during the time interval $[0,T]$. Thus the function $x_0 \longmapsto ||F_{x_0}||^2$ is zero outside of a compact set. It is obviously continuous on $\mathbb{R}^3$ and is thus bounded on $\mathbb{R}^3$ by some constant $M > 0$. Using this together with \eqref{eq:bound 1 log reg} and \eqref{eq:truc_positif} inside \eqref{eq:log reg diff} and assuming that $\lambda \leq 1$ yields
\begin{align}
\big|\lambda\mathcal{L}_{\mathrm{reg}}(x_0,\lambda)-||W||^2(1-r(x_0)^2)\big| \leq \frac{\lambda}{\varepsilon}||W||^2 + (n-1)|\lambda \log \lambda| + \lambda \log (M+1)
\end{align}
which shows the uniform convergence statement as well as the pointwise one (together with \eqref{eq:pointwise}).
\end{proof}

The set $E_{\varepsilon}$ is a set on which all the matrices $K_{x_0}^{\mathrm{reg}}$ are all uniformly large enough in the spectral radius sense.

\begin{remark}
In the proof of Proposition \ref{prop : log lik reg}, the second term in \eqref{eq:log lik reg} has no influence in the limit object and only pollutes the rate of convergence. Discarding it leads to a $O_{\lambda \rightarrow 0}(\lambda)$ rate of convergence.
\end{remark}

It also makes sense to inspect the limiting case $N \rightarrow \infty$, which corresponds to having the sampling frequency of the captors go to infinity. In this case, the discrete objects in Proposition \ref{prop : log lik reg} behave as Riemann sums if the time steps $t_k$ are equally spaced; as a result,  we obtain integrals in the limit  $N \rightarrow \infty$. We state a weak asymptotic behaviour result of $\mathcal{L}_{\mathrm{reg}}(x_0,\lambda)$ in that limit, with a very similar limit object. Highlight the dependence in $N$ in $\mathcal{L}_{\mathrm{reg}}(x_0,\lambda)$ by noting it instead $\mathcal{L}_{\mathrm{reg}}^N(x_0,\lambda)$.

\begin{proposition}[Asymptotic behaviour of $\mathcal{L}_{\mathrm{reg}}^N(x_0,\lambda)$ when $N \rightarrow \infty$]\label{prop : log lik reg N}
Define the following vector valued functions $L^2([0,T],\mathbb{R}^q)$: 
\begin{align}
\forall t\in [0,T], \hspace{11pt} I_{\mathrm{u}}(t) &:= \big(\tilde{u}(x_1,t),...,\tilde{u}(x_q,t)\big)^T \nonumber \\
\forall t \in [0,T], \hspace{7pt} I_{x_0}(t) &:= \big(f_t^{\mathrm{R}}(x_1-x_0),...,f_t^{\mathrm{R}}(x_q-x_0)\big)^T \nonumber
\end{align}
and the scalar valued function
\begin{align}\label{eq:r infinity}
r_{\infty}(x_0) &:= \frac{\langle I_{\mathrm{u}},I_{x_0} \rangle }{||I_{\mathrm{u}}|| \ ||I_{x_0}||}
\end{align}
whenever $||I_{x_0}|| \neq 0$; in \eqref{eq:r infinity}, the norms and the dot product are those of the usual euclidean structure of $L^2([0,T],\mathbb{R}^q)$. Suppose that for all $k \in \{1,...,N\}$, $t_k = T (k-1)/(N-1)$, i.e. the $t_k$ are equally spaced in $[0,T]$.

Then whenever $x_0$ is such that $||I_{x_0}|| \neq 0$, we have the following pointwise convergence at $x_0$
\begin{align}
\frac{\lambda}{N}\mathcal{L}_{\mathrm{reg}}^N(x_0,\lambda) \xrightarrow[N \rightarrow \infty]{} ||I_{\mathrm{u}}||_2^2\big(1-r_{\infty}(x_0)^2\big) + q \lambda\log \lambda
\end{align}
\end{proposition}
\begin{proof}
 In all concerned mathematical objects, we highlight the $N$ dependency with an exponent, i.e. $W^N$, $F_{x_0}^N$, etc.
We use the exact same tool as in the previous proof, namely that we have
\begin{align*}
\mathcal{L}_{\mathrm{reg}}^N(x_0,\lambda) &= \frac{||W^N||^2}{\lambda}\bigg(1 - \frac{\langle F_{x_0}^N,W^N\rangle^2}{||W^N||^2\big(\lambda + ||F_{x_0}^N||^2\big)}\bigg) + (n-1) \log \lambda + \log (\lambda + ||F_{x_0}^N||^2)
\end{align*}
But
\begin{itemize}
\item $||W^N||^2 = \sum_{i=1}^q \sum_{k=1}^N \tilde{u}(x_i,t_k)^2$
\item $||F_{x_0}^N||^2 = \sum_{i=1}^q \sum_{k=1}^N f_{t_k}^{\mathrm{R}}(x_i-x_0)^2$
\item $\langle F_{x_0}^N,W^N\rangle = \sum_{i=1}^q \sum_{k=1}^N f_{t_k}^{\mathrm{R}}(x_i-x_0)\tilde{u}(x_i,t_k)$
\end{itemize}
Since the time steps are equally spaced, we can study the limit $N \rightarrow \infty$ of the above objects thanks to Riemann sums. When $N \rightarrow \infty$,
\begin{align}
\frac{1}{N}||W^N||^2 &\longrightarrow \sum_{i=1}^q\int_0^T \tilde{u}(x_i,t)^2dt = ||I_{\mathrm{u}}||^2 \\
\frac{1}{N}||F_{x_0}^N||^2 &\longrightarrow \sum_{i=1}^q\int_0^T f_t(x_i-x_0)^2dt = ||I_{x_0}||^2 \label{eq:I_x0_lim} \\
\frac{1}{N}
\langle W^N,F_{x_0}^N \rangle &\longrightarrow \sum_{i=1}^q\int_0^T \tilde{u}(x_i,t) f_t(x_i-x_0)dt = \langle I_{\mathrm{u}},I_{x_0} \rangle 
\end{align}
Suppose that $x_0$ is such that $||I_{x_0}|| \neq 0$, then because of \eqref{eq:I_x0_lim}, the quantity $||F_{x_0}^N||$ is bounded from below by $||I_{x_0}||/2$ for sufficiently large $N$. We then have the following convergence
\begin{align}
\frac{\langle F_{x_0}^N,W^N\rangle^2}{||W^N||^2(\lambda + ||F_{x_0}^N||^2)} = \frac{(\frac{1}{N}\langle F_{x_0}^N,W^N\rangle)^2}{\frac{1}{N}||W^N||^2(\frac{\lambda}{N} + \frac{1}{N} ||F_{x_0}^N||^2)}  \xrightarrow[N \rightarrow \infty]{} r_{\infty}(x_0)
\end{align}
Likewise, since $n = qN$, when $N \rightarrow \infty$ :
\begin{align}
\frac{(n-1)\log \lambda}{N} &+ \frac{1}{N}\log (\lambda + ||F_{x_0}||^2) \nonumber \\
&= \frac{(Nq-1)\log \lambda}{N} + \frac{\log N}{N} + \frac{1}{N}\log \Big(\frac{\lambda}{N} + \frac{1}{N}||F_{x_0}||^2\Big)  \xrightarrow[N \rightarrow \infty]{} q \log \lambda
\end{align}
which shows the announced result.
\end{proof}

\begin{remark}
One could surely obtain better rates of convergence in Proposition \ref{prop : log lik reg N} using basic convergence results for Riemann sums and regularity assumptions, but we restrict ourselves to a $o_{N \rightarrow \infty}(1)$ convergence to avoid further technicalities.
\end{remark}

\subsubsection{Discussion : why do we observe that the actual source position is located at the intersection of spheres centered on receivers in Figure \ref{fig : point source} ?}\mbox{}\\

The answer can be found by analysing the limit term in proposition  \ref{prop : log lik reg} (or the one in Proposition \ref{prop : log lik reg N}). Note $L(x_0)$ the said limit object from Proposition \ref{prop : log lik reg} :
\begin{align}
L(x_0) = ||W||^2\big(1-r(x_0)^2\big) = ||W||^2\bigg(1 - \frac{_Big(\sum_{i=1}^q \langle F_{x_0}^i,W_i\rangle\Big)^2}{||W||^2 ||F_{x_0}||^2}\bigg)
\end{align}
Note $T_i$ the time of arrival of the point source wave at sensor $i$ : $|x_i - x_0^*| = c^*T_i$. Define also the sphere $S_i$ centered on $x_i$ and the thin spherical shell $A_i$ of thickness $2R$ that surrounds it
\begin{align}
S_i &:= S(x_i,cT_i) \\
A_i &:= A(x_i,cT_i-R,cT_i+R) = \overline{B(x_0,cT_i + R) \setminus B(x_0,cT_i- R)}
\end{align}
\textit{(i) $L(x_0)$ reaches a local minima over the whole sphere $S_i$.} 
When $x_0$ is located inside $A_i$, the subvectors $W_i$ and $F_{x_0}^i$ of $W$ and $F_{x_0}$ respectively become almost colinear because $f_t^{\mathrm{R}}$ is radially symmetric. They become exactly colinear when $x_0 \in S_i$. This maximizes the term $\langle F_i,W_i\rangle$ in virtue of the Cauchy-Schwarz inequality. When $x_0$ lies on one and only one of those spherical shells $A_i$, the other terms $\langle F_j,W_j \rangle$ are all zero. \\

\textit{(ii) The local minima of $L(x_0)$ located at the intersection of $2$ or more spheres $S_i$ are smaller.}
More generally, when $I$ is a subset of $\{1,...,n\}$ and when $x_0 \in \bigcap_{i \in I} A_i \setminus \bigcap_{j \notin I} A_j$, the term $\sum_{i \in I} \langle F_i,W_i \rangle$ is (almost) maximized while $\sum_{j \notin I} \langle F_j,W_j \rangle = 0$, which explains why the intersection of spheres are darker coloured than the other parts of the spheres in Figure \ref{fig : point source}. \\

\textit{(iii) The spheres $S_i$ intersect at a single point, which is exactly $x_0^*$ as well as the global minima of $L(x_0)$.} The quantity $r(x_0)$ reaches a global maximum when all subvectors $W_i$ and $F_{x_0}^i$ are colinear, which is the case only when $x_0 \in \bigcap_i S_i$. When there are at least 4 sensors, the intersection of all the spheres $\bigcap_i S_i$ is reduced to at most one point. Recall that we have supposed that $c = c^*$ : this implies that $x_0^* \in \bigcap_i S_i$, and thus the minimum of $L(x_0)$ is located at $x_0 = x_0^*$. \\

Note that if the speed $c$ in $k_{x_0}^{\mathrm{R}}$ does not correspond to the real speed $c^*$, the intersection $\bigcap_i S_i$ will be empty. Additionally, from an optimization point of view, numerically solving $\inf_{x_0} \mathcal{L}(x_0,\lambda)$ is obviously highly non convex and none of our numerical experiments lead to the correct solution.
%\begin{itemize} 
%\item In the above limit $R \rightarrow 0$, one should also increase the number of time measurements $N$ in order to capture non zero signal $W$, as the size of the support of the solution $u$ tends to zero when $R \rightarrow 0$.
%\item 
%\item 

%\item In Figure\ref{fig : point source}, we observe a union of thin spherical shells rather than spheres. This is because in the numerical experiments tied to Figure\ref{fig : point source}, the parameter $R$ is small but non-zero. in fact, $R$ corresponds to the thickness of the spherical shells.
%\end{itemize}

\subsection{Computational Speedups}
In the Kriging formulas \eqref{eq:krig mean} and \eqref{eq:krig cov} as well as in the log marginal likelihood equation \eqref{eq:loglik_noisy}, we are interested in inverting matrices of the form $K + \lambda I$ where $K$ in positive semidefinite. Using basic algebra, we show an exact computational shortcut to computing $(K + \lambda I)^{-1}$ when $K$ has null lines, which is typically the case when using kernels \eqref{eq:ft ft' compact} because of finite speed propagation. It may provide huge computational speedups in the limit of small source size as is the case in the previous subsection.

\subsubsection{Inversion shortcut for Tichonov-regularized PSD matrices with null columns } \mbox{}\\
Let $K$ be a symmetric semi-positive definite matrix of size $n$. Suppose it has $p$ non zero columns $i_1,...,i_p$ and $q = n-p$ zero columns $i_{p+1},...,i_n$. Since $K$ is symmetric, its lines $i_1,...,i_p$ are also zero. Finally, a simple necessary and sufficient condition for its column $j$ to be zero is $K_{jj} = 0$ because of the Cauchy-Schwarz inequality for PSD matrices : $|a_{ij}| \leq \sqrt{a_{ii}}\sqrt{a_{jj}}$.
Define the matrix $\Tilde{K}$ of the extracted non-zero columns of $K$, of size $p \times p$ by $\Tilde{K}_{k,l} = K_{i_k,i_l} \forall k,l \in \{1,...,p\}$. Note $\mathcal{E} = (e_1,...,e_n)$ the canonical basis of $\mathbb{R}^n$ and $\mathcal{F} = (e_{i_1},...,e_{i_n})$ a permutation according to the indexes $i_1,...,i_p$ and $i_{p+1},...,i_n$ defined above (that permutation is unique up to additional permutations of the null columns of $K$). Then for the orthogonal permutation matrix $P$ such that
\begin{align}\label{eq:P permutation K}
\forall k \in \{1,...,n\}, \ \ \ Pe_k = e_{i_k}
\end{align}
we have
\begin{align}\label{eq:recomp K}
    K = P^T \times    
    \begin{pmatrix}
                    \begin{matrix}
                      \Tilde{K}
                    \end{matrix}
                    & \rvline & \bigzero_{p,q} \\
                    \hline
                    \bigzero_{q,p} & \rvline &
                    \begin{matrix}
                      \bigzero_{q,q}
                    \end{matrix}
    \end{pmatrix} \times P
\end{align}
Then we have blockwise, 
\begin{align} \label{eq:fast inv}
    (K + \lambda I_n)^{-1} = P^T \times \begin{pmatrix}
                    \begin{matrix}
                      (\Tilde{K} + \lambda I_p)^{-1}
                    \end{matrix}
                    & \rvline & \bigzero_{p,q} \\
                    \hline
                    \bigzero_{q,p} & \rvline &
                    \begin{matrix}
                      \frac{1}{\lambda}I_q
                    \end{matrix}
    \end{pmatrix} \times P
\end{align}
Note that matrix $P$ is easy to find as one only has to compute the diagonal elements of $K$ in order to find the null lines/columns of $K$, since $K$ is PSD. 
That is, we only need to invert a $p \times p$ matrix. For $a$ in $\mathbb{R}^n$, define $a_{in} \in \mathbb{R}^p$ the first $p$ components of $Pa$ and $a_{out} \in \mathbb{R}^{n-p}$ the last $n-p$ components of $Pa$, ie 
\begin{align}\label{aq : a_in a_out}
Pa = \begin{pmatrix}
    (Pa)_1 \\
    \vdots \\
    (Pa)_p \\
     \cmidrule(lr){1-1}
     (Pa)_{p+1} \\
     \vdots \\
    (Pa)_n
\end{pmatrix}  =
\begin{pmatrix}
    {a}_{in} \\
     \cmidrule(lr){1-1}
    a_{out}
\end{pmatrix}   
\end{align}
Using \eqref{eq:fast inv} and \eqref{aq : a_in a_out}, we have immediately that 
\begin{align}\label{eq_speedup}
a^T (K+ \lambda I_n)^{-1}a = a_{in}^T(\tilde{K}+ \lambda I_p)^{-1}a_{in} + \frac{1}{\lambda}||a_{out}||^2
\end{align}
In the term $a^T (K+ \lambda I_n)^{-1}a$, $a_{in}$ corresponds to the coordinates of $a$ that are "kept" by the non zero elements of $K$ and $a_{out}$ corresponds to those that are left out by the non zero elements of $K$. 

\subsubsection{Computational shortcuts for Kriging}
We may now use the previous section to speed up GPR computations when using a covariance kernel $k_{\theta}$. This is summarized in Proposition \ref{prop : fast GPR} in which we use the following notations.

\begin{itemize}
\item Note $X = (x_1,...,x_n)^T \in \mathbb{R}^n$ the observation locations, $Y \in \mathbb{R}^n$ the vector of observations ($Y_i = u(x_i)$ for noiseless observations).
\item Note $K_{\theta} := k_{\theta}(X,X)$ the associated covariance matrix, $P_{\theta}$ and $\tilde{K}_{\theta}$ the matrices $P$ and $\tilde{K}$ defined by \eqref{eq:P permutation K} for $K = K_{\theta}$, $p_{\theta}$ the number of non zero columns of $K_{\theta}$ and $q_{\theta} = n - p_{\theta}$.
\item As in \eqref{aq : a_in a_out}, note $X_{in}^{\theta}$ the first $p_{\theta}$ coordinates of $P_{\theta}X$ and $X_{out}^{\theta}$ the $q_{\theta}$ last, and likewise for $Y_{in}^{\theta}$ and $Y_{out}^{\theta}$.
\end{itemize}

Plugging equation \eqref{eq:fast inv} in the log-marginal likelihood and Kriging formulas, we obtain Proposition \ref{prop : fast GPR}.
\begin{proposition}\label{prop : fast GPR}
If using a covariance kernel $k_{\theta}$ with hyperparameter $\theta$ for GPR, we have the following exact computational shortcuts. \\
$(i)$ For the log-marginal likelihood,
\begin{align}\label{eq:fast likelihood} 
    \mathcal{L}(\theta,\lambda)
    &=  (Y_{in}^{\theta})^T(\tilde{K}_{\theta} + \lambda I_{p_{\theta}})^{-1}Y_{in}^{\theta} \nonumber \\ 
    &+ \frac{1}{\lambda}||Y_{out}^{\theta}||^2 + \log \det(\Tilde{K}_{\theta} + \lambda I_{p_{\theta}}) + q_{\theta} \log \lambda
\end{align}
$(ii)$ For the Kriging formulas,
\begin{numcases}{}
    \Tilde{m}(x) &= \hspace{3mm}$m(x) +  k_{\theta}(X_{in}^{\theta},x)^T(\Tilde{K}_{\theta} + \lambda I_{p_{\theta}})^{-1}(Y_{in}^{\theta} - m(X_{in}^{\theta}))$ \label{eq:krig mean fast} \\
    \Tilde{k}(x,x') &= \hspace{3mm}$k_{\theta}(x,x') - k_{\theta}(X_{in}^{\theta},x)^T(\tilde{K}_{\theta} + \lambda I_{p_{\theta}})^{-1}k_{\theta}(X_{in}^{\theta},x')$\label{eq:krig cov fast}
\end{numcases}
$(iii)$ if $k_{\theta}(x,x) = 0$ then $\Tilde{m}(x) = m(x)$ and $\tilde{k}(x,x') = 0$.
\end{proposition}
\begin{proof}
As in equation \eqref{eq_speedup}, we have immediately that
\begin{align}
Y^T (K_{\theta}+ \lambda I_n)^{-1}Y = (Y_{in}^{\theta})^T(\tilde{K}_{\theta} + \lambda I_{p_{\theta}})^{-1}Y_{in}^{\theta} + \frac{1}{\lambda}||Y_{out}^{\theta}||^2
\end{align}
It is also obvious from \eqref{eq:recomp K} that $\log \det (K_{\theta} + \lambda I_n) = \log \det(\Tilde{K}_{\theta} + \lambda I_{p_{\theta}}) + q_{\theta} \log \lambda$, which proves \eqref{eq:fast likelihood}.

For the Kriging equations, suppose without loss of generality that $m \equiv 0$. The Kriging mean is, as in \eqref{eq_speedup},
\begin{align}\label{eq:m_fast}
\tilde{m}(x) &= k(X,x)^T(K + \lambda I_n)^{-1}Y \nonumber \\
 &= k(X_{in}^{\theta},x)^T(\Tilde{K}_{\theta} + \lambda I_{p_{\theta}})^{-1}Y_{in}^{\theta} + \frac{1}{\lambda}k(X_{out}^{\theta},x)^T Y_{out}^{\theta}
\end{align}
Now, suppose that the observation location $x_i$ is such that $k_{\theta}(x_i,x_i) = (K_{\theta})_{ii} = 0$, which is exactly the case of all the coordinates of $X_{out}^{\theta}$ (by definition). Then from the Cauchy-Schwarz inequality \eqref{eq:CS PD}, the function $x \longmapsto k(x,x_i)$ is identically zero. Thus for all $x$, we have $k(X_{out}^{\theta},x)^T Y_{out}^{\theta} = 0$ which, when plugged in \eqref{eq:m_fast}, proves \eqref{eq:krig mean fast}. The exact same reasoning leads to \eqref{eq:krig cov fast}. Finally, for the same reason, if $k_{\theta}(x,x) = 0$ then the vector $k_{\theta}(X_{in},x)$ is null, which implies that $\tilde{m}(x) = 0$.
\end{proof}

We remind that formulas \eqref{eq:fast likelihood}, \eqref{eq:krig mean fast} and \eqref{eq:krig cov fast} are exact. There are actually three shortcuts for the Kriging equations \eqref{eq:krig mean fast} and \eqref{eq:krig cov fast}. First, the fast matrix inversion, which is especially useful during the log marginal likelihood optimization step (hyperparameter calibration). Second, the selection of the pertinent data points in $k_{\theta}(X_{in},x)$ by using $X_{in}$ instead of $X$. Third, the selection of the pertinent evaluation points thanks to the point $(iii)$ of Proposition \ref{prop : fast GPR}. These last two are useful when $\tilde{m}$ and/or $\tilde{k}$ have to be evaluated many times. Note also that all the mathematical objects on the right-hand side of equations \eqref{eq:fast likelihood}, \eqref{eq:krig mean fast} and \eqref{eq:krig cov fast} depend on $\theta$ and have to be recomputed for each new value of $\theta$. However this is fast since checking which columns of $K_{\theta}$ are zero only requires to compute the diagonal of $K_{\theta}$ : this requires $n$ calls to $k_{\theta}$. 

\begin{remark}
In equation \eqref{eq:fast likelihood}, the term $\frac{1}{\lambda}||Y_{out}^{\theta}||^2$ can be understood as a very heavy penalization term that sanctions non zero observations that do not lie in the support of $k_{\theta}$. When applying this to kernel \eqref{eq:wave kernel}, the term $\frac{1}{\lambda}||Y_{out}^{\theta}||^2$ penalizes the non zero observations that do not lie in the space time domain (aka light cone) allowed by Huygens' principle.
\end{remark}

%Note that in this framework, the change of basis matrix $P$ depends on $\theta$ and has to be recomputed for each new value of $\theta$.

\subsubsection{Expected numerical gain when using the truncated wave kernel \eqref{eq:wave kernel}} \mbox{}\\

In this case, to examine if a point $(x,t)$ is such that $k_{\theta}((x,t),(x,t)) = 0$, one only needs to check that $(x,t)$ lies in the light cone that corresponds to $\theta$ (among others, $\theta$ contains $(x_0,R,c)$). Explicitly, the condition to check is whether or not $x \in B(x_0,c|t|+R) \setminus B(x_0,(c|t|-R)_+)$ or equivalently, $c|t|-R\leq |x-x_0| \leq c|t| + R$. This is actually much cheaper than evaluating $k_{\theta}((x,t),(x,t))$. \\

\textit{For the log-marginal likelihood} : using Cholesky decomposition to invert $\tilde{K}_{\theta} + \lambda I_{p_{\theta}}$,  only $p_{\theta}^3/3$ flops are required versus $n^3/3$ previously (in both cases, the computation of the determinant term in \eqref{eq:fast likelihood} is immediate from the Cholesky decomposition). \\

\textit{For the Kriging formulas} : here we suppose that we are performing GPR using the truncated wave kernel \eqref{eq:wave kernel}. Suppose that the hyperparameter vector $\theta$ has been previously estimated : in particular, it provides estimated wave speed $c$, radius $R$ and source position $x_0$. Now, equation \eqref{eq:krig mean fast} states that 
\begin{align}\label{eq:krig p theta terms}
\tilde{m}(x,t) = \sum_{i=1}^{p_{\theta}}a_i k_{\theta}\big((x,t),(X_{in}^{\theta})_i \big)
\end{align}
for some fixed real numbers $a_1,...,a_{p_{\theta}}$. Suppose that we want to compute $\tilde{m}(x,t)$ over a cubic spatial grid $[0,L]^3$ of side length $L$ at a fixed time $t$. Suppose that the space step $\Delta x$ of this grid verifies $\Delta x \ll L$ : then the number of grid points inside a volume $V$ is approximately $V/\Delta x^3$ (each grid point occupies a cubic volume of $\Delta x^3$). In compliance with the strong Huygens principle, computing $\tilde{m}(x,t)$ over the whole grid using Proposition \ref{prop : fast GPR} results in $N_{\theta}$ calls to $k$, with
\begin{align}\label{eq:ntheta_vol}
N_{\theta} \simeq p_{\theta} \times \frac{\text{Vol}\Big([B(x_0,ct+R) \setminus B(x_0,(ct-R)_+)]\cap [0,L]^3 \Big)}{\Delta x ^3}
\end{align}
The set $B(x_0,ct+R) \setminus B(x_0,(ct-R)_+)$ corresponds to the admissible domain allowed by the strong Huygens principle at time $t$. This is to be compared to $N_{\mathrm{full}} = n \times L^3/\Delta x^3$, the number of calls to $k$ needed without using Proposition \ref{prop : fast GPR}; obviously, $N_{\theta} < N_{\mathrm{full}}$ and actually, $N_{\theta} \ll N_{\mathrm{full}}$ when $R$ is small. More precise estimations of the number of grid points inside the volume in equation \eqref{eq:ntheta_vol} can be derived, though this is a difficult research topic in itself : see for instance \cite{Arkhipova2008NumberOL} or the Gauss circle problem for its 2D counterpart. Also, the volume appearing in \eqref{eq:ntheta_vol} can be explicitly derived but we consider the additional insights this provides to be negligible.

%% file: 5_numerical_experiments_new_results.tex
\section{Numerical experiments}\label{section : num}
In this section, we showcase WIGPR on functions $w$ that are solutions of problem \eqref{eq:wave_eq}, using kernels \eqref{eq:ft ft' gen} and \eqref{eq:ftp ft'p gen} separately as well as together as in \eqref{eq:wave kernel}. The goal is twofold : reconstructing the target function $w$, which amounts to reconstructing its initial conditions, and estimating the physical parameters attached. Note that with badly estimated physical parameters, the reconstruction step is more or less bound to fail, especially with inaccurate wave speed $c$ and/or source centers $x_0^u$ and $x_0^v$. Throughout this section, we only consider compactly supported and radially symmetric initial conditions, for which we can use the formulas \eqref{eq:ft ft' compact} and \eqref{eq:ftp ft'p compact}.

At the end of this section, we also briefly show that perfect transparent boundary conditions are automatically enforced on the numerical reconstructions of $w$ provided by WIGPR. This is true simply because the Kriging means provided by WIGPR are built to be exact solutions of the free space problem \eqref{eq:wave_eq}. Transparent boundary conditions are a well known research topic in the numerical analysis community (\cite{ref_PML},\cite{Engquist1977AbsorbingBC}, \cite{fu_tbc},\cite{noble_kazakova}), which arises from the need to numerically simulate free space (unbounded) PDE problems on numerical simulation domains that are inherently finite.

\subsection{General test protocol for WIGPR} 

\subsubsection{Test cases} Three general test cases for WIGPR are considered here. They correspond to different initial conditions $(u_0,v_0)$ :
\begin{enumerate}
\item A test case for $k_{\mathrm{u}}^{\mathrm{wave}}$ described in Subsection \ref{sub:GP_pos}, for which $v_0 = 0$.
\item A test case for $k_{\mathrm{v}}^{\mathrm{wave}}$ described in Subsection \ref{sub:GP_spd}, for which $u_0 = 0$.
\item A test case for $k_{\mathrm{u}}^{\mathrm{wave}} + k_{\mathrm{v}}^{\mathrm{wave}}$ described in Subsection \ref{sub:GP_mix} using both of the non zero initial conditions of the two previous test cases.
\end{enumerate}
For each test case, the complete procedure described below is performed.

\subsubsection{Numerical simulation} Given initial conditions $u_0$ and $v_0$, we numerically simulate the solution $w$ over a given time period. We use a basic two step explicit finite difference time domain (FDTD) numerical scheme for the wave equation as described in \cite{bilbao2004}, equation A.24, over the cube $[0,1]^3$. The space unit is the meter (each side is $L = 1$ meter), the time unit is the second. We also use second order Engquist-Majda transparent boundary conditions \cite{Engquist1977AbsorbingBC}.
The simulation parameters are as follows. We use a sample rate $SR = 200 \ Hz$ (time step $\Delta t = 1/200$ s), a space step $\Delta x = 43 \ mm$, a wave speed $c = 0.5 \ m/s$. The simulation duration is $T = 1.5 \ s = 0.75 \ L/c$.

%\begin{itemize}
%\item Sample rate : $SR = 200 \ Hz$ (time step $\Delta t = 1/200$ s) ; space step $\Delta x = 43 \ mm$.
%\item wave speed $c = 0.5 \ m/s$; simulation duration : $T = 1.5 \ s = 0.75 \ L/c$.
%\end{itemize}

%\cite{bilbao2004}
\subsubsection{Database generation} 30 sensors are scattered in the cube $[0.2,0.8]^3$ using a Latin hypercube repartition and a minimax space filling algorithm. Signal outputs correspond to time series for each sensor, with a sample rate of $50 \ Hz$, so $75$ data points altogether equally spanned over the interval $[0, 1.5] \ s$ for each sensor, which leads to $30 \times 75 = 2250$ observations. Each signal is then artificially polluted by a centered Gaussian white noise with standard deviation $\sigma_{noise} = 0.09$. This value corresponds to around $10\%$ of the maximal amplitude of the signals, see Figures \ref{fig_cos pos sig}, \ref{fig_ring cos spd sig} and \ref{fig_mix sig}.

\subsubsection{WIGPR} We perform WIGPR on (portions of) the dataset obtained above, using the kergp package \cite{citekergp} from R \cite{citeR}. For that we use kernels \eqref{eq:ft ft' compact} and/or \eqref{eq:ftp ft'p compact} which are ``fast" to evaluate, with $K_v$ and $k_u^0$ both 1D $5/2-$Matérn kernels. This Matérn kernel is stationary and writes, in term of the increment $h = x-x'$,
\begin{align}\label{eq:matern 5/2}
k_{5/2}(h) = \sigma^2 \bigg(1 + \frac{|h|}{\rho} + \frac{h^2}{3\rho^2} \bigg)\exp\bigg(-\frac{|h|}{\rho} \bigg)
\end{align}
It has two additional hyperparameters : $\rho$ and $\sigma^2$. $\rho$ is the length scale of the kernel \eqref{eq:matern 5/2} and should correspond to the typical variation length scale of the function approximated with GPR. $\sigma^2$ is the variance of the kernel. We tackle two different questions related to WIGPR.
\begin{enumerate}[label=(\subscript{Q}{{\arabic*}})]
\item \label{num:multistart} We study how well the physical parameters (propagation speed $c$, source position $x_0$ and size $R$) can be estimated with WIGPR. For this, we first select $N_s$ time series corresponding to the first $N_s$ sensors with $N_s \in \{1,2,3,4,5,6,8,10,15,20,25,30\}$. The corresponding database encompasses $75 \times N_s$ data points. For this database, we perform negative log marginal likelihood minimization to estimate the corresponding hyperparameters, which are either
\begin{align*}
\theta = 
\begin{cases}
(x_0^a,R_a,\theta_{k_a^0},c,\lambda) \in \mathbb{R}^{8}, a \in \{u,v\}, &\text{ if } v_0 = 0 \hspace{1pt} \text{ or } \hspace{1pt} u_0 = 0 \\
(x_0^u,R_u,\theta_{k_u^0},x_0^v,R_v,\theta_{k_v^0},c,\lambda) \in \mathbb{R}^{14} &\text{ if } v_0 \neq 0 \text{ and } u_0 \neq 0 
\end{cases}
\end{align*}
$\lambda$ corresponds to $\sigma^2$ in \eqref{eq:loglik_noisy}, and is viewed as an additional hyperparameter in the log marginal likelihood. We use a COBYLA optimization algorithm to solve \eqref{eq:learning theta} and a multistart procedure with $n_{mult} = 100$ different starting points. That is, $100$ different values of $\theta_0$ are scattered over an hypercube $H \subset \mathbb{R}^8$ or $H \subset \mathbb{R}^{14}$, and the COBYLA log marginal likelihood optimization procedure is run using each value of $\theta_0$ as a starting point. Among all the obtained results, the hyperparameter value that provides the maximal log marginal likelihood is selected. The multistart procedure mitigates the risk of getting stuck in local maxima. COBYLA is a gradient-free optimization method used in kergp and available in the \texttt{nloptr} package from R. We then reconstruct the initial conditions using WIGPR.
\item \label{num:fig_sensibility} We study the sensibility of the reconstruction step w.r.t. the sensor locations. Consider $40$ different Latin hypercube layouts of the $30$ sensors, each obtained with a minimax space filling algorithm. For each layout, we provide the correct hyperparameter values to the model; these valued are described in each test case. We then reconstruct the initial conditions using GPR and $N_s$ sensors, with $N_s \in \{1,5,10,15,20,25,30\}$. $L^p$ relative errors (see equation \eqref{eq:LP_rel_err}) are computed between the reconstructed initial condition and the real initial condition, and for each number of sensors $N_s$, statistics over the 40 different datasets for these $L^p$ errors are summarized in boxplots.
%Then perform Kriging on the corresponding dataset.
\end{enumerate}
In both cases, the approximated initial position $\tilde{u}_0$ is recovered by evaluating the WIGPR Kriging mean at $t = 0$ over a 3D grid and the initial speed $\tilde{v}_0$ is recovered by evaluating the Kriging mean at $t=0$ and $t = dt = 10^{-7}$ over the same 3D grid: $\tilde{v}_0 \simeq (\tilde{m}(\cdot,dt) - \tilde{m}(\cdot,0))/dt$ with $dt = 10^{-7}$.

\subsubsection{Result visualization} For each test case are provided an example of captured data (observation points), for one particular design of experiment (out of 40) for 6 sensors (out of 30). This displayed data takes the form of a sequence of 6 time series (Figures \ref{fig_cos pos sig}, \ref{fig_ring cos spd sig} and \ref{fig_mix sig}). For each test case, an example of initial condition reconstruction achieved thanks to WIGPR is also showcased. These examples take the form of a slice of the corresponding 3D functions evaluated on some plane $z = Cst$. Figures are obtained with MATLAB \cite{MATLAB:2020a}.

\subsubsection{Numerical indicators} 
Precise numerical results are displayed in appendix A; they are as follows. \\

For \ref{num:multistart} are plotted the distances between the correct physical parameters and the estimated ones, depending on the number of sensors used. Additionally, for every $p \in \{1,2,\infty\}$ are plotted relative $L^p$ errors $e_{p,rel}$ defined below depending on the number of sensors used.
\begin{align}\label{eq:LP_rel_err}
e_{p,rel}^u = \frac{||u_0-\tilde{u}_0||_p}{||u_0||_p}\ \ \text{ and } \ \ e_{p,rel}^v = \frac{||v_0-\tilde{v}_0||_p}{||v_0||_p}
\end{align}
A relative error of $50\%$ means that $||u_0-\tilde{u}_0||_p = 0.5 ||u_0||_p$. A relative error of over $100\%$ means that $||u_0-\tilde{u}_0||_p \geq ||u_0||_p$, in which case the null estimator $\hat{u}_0 = 0$ performs better than the estimator $\tilde{u}_0$ provided by WIGPR. Note that obtaining a small $L^p$ relative error is harder when the target function has a small $L^p$ norm. \\

For \ref{num:fig_sensibility} are plotted boxplots of the relative $L^p$ errors over the 40 different sensor layouts, depending on the number of sensors used.

%For each simulation and every $p \in \{1,2,\infty\}$, we compute the following  numerical indicators :
%\begin{itemize}
%    \item $L^p$ errors : $e_p^u = ||u_0 - \Tilde{u}_0||_p$ and $e_p^v = ||v_0 - \Tilde{v}_0||_p$
%    \item $L^p$ relative errors : $e_{p,rel}^u = \frac{e_p^u}{||u_0||_p}$ and $e_{p,rel}^v = \frac{e_p^v}{||v_0||_p}$
%\end{itemize}
Integrals for the $L^p$ error plots are approximated using Riemann sums over a 3D grid with space step $dx = 0.01$. %Note that this is numerically costly as we need to compute $3D$ integrals. %The numerical results of these indicators are summarized in figures in Appendix A.

\subsection{Test case for $k_{\mathrm{u}}^{\mathrm{wave}}$}\label{sub:GP_pos}
In this test case, $v_0$ is assumed null and thus we set $k_{\mathrm{v}} = 0$, which yields $k_{\mathrm{v}}^{\mathrm{wave}} = 0$. We thus use $k_{\mathrm{u}}^{\mathrm{wave}}$ defined in \eqref{eq:ftp ft'p compact} for GPR. 
We use the 1D Matérn kernel \eqref{eq:matern 5/2} for $k_u^0$ in equation \eqref{eq:ftp ft'p compact}. The initial condition $u_0$ is a radial raised cosine described as follows. Set $x_0 = (0.65,0.3,0.5)^T$, $R = 0.25$ and $A = 5$, the corresponding initial conditions (IC) are given by :
\begin{align*}
\begin{cases}
    u_0(x) &= A\mathbbm{1}_{[0,R]}(|x-x_0|)\Bigg(1 + \cos\bigg(\frac{\pi|x-x_0|}{R}\bigg) \Bigg) \\
    v_0(x) &= 0
    \end{cases}
\end{align*}
See Figure \ref{fig_cos pos slice}, left column, for a graphical representation. For problem \ref{num:multistart}, the optimization domain is chosen to be the following hypercube
\begin{align}\label{eq:dom_optim_1}
\theta &= (x_0,R,\rho,\sigma^2,c,\lambda) \nonumber \\
 &\in [0,1]^3 \times [0.03,0.5] \times [0.02,2] \times [0.1,5]\times [0.2,0.8]\times [10^{-8},10^{-2}]
\end{align}
For problem \ref{num:fig_sensibility}, the hyperparameter $\theta_0$ provided to the model is
\begin{align}
\theta_0 = (x_0,R,(\rho,\sigma^2),c,\lambda) = ((0.65,0.3,0.5),0.3,(0.2,3),0.5,\sigma_{noise}^2)
\end{align}
with $\sigma_{noise}^2 = 0.0081$. The value of $0.2$ provided for $\rho$ is a visual estimation of the length scale of $u_0$ based on Figure \ref{fig_cos pos slice}. Additionally, the couple $(0.2,3)$ provided for $(\rho,\sigma^2)$ is more or less confirmed by the estimations of $(\rho,\sigma^2)$ performed for problem \ref{num:multistart}, test case\#1, visible in Table \ref{table:estim_matern}.
\begin{figure}[tb]
\includegraphics[width=0.8\textwidth]{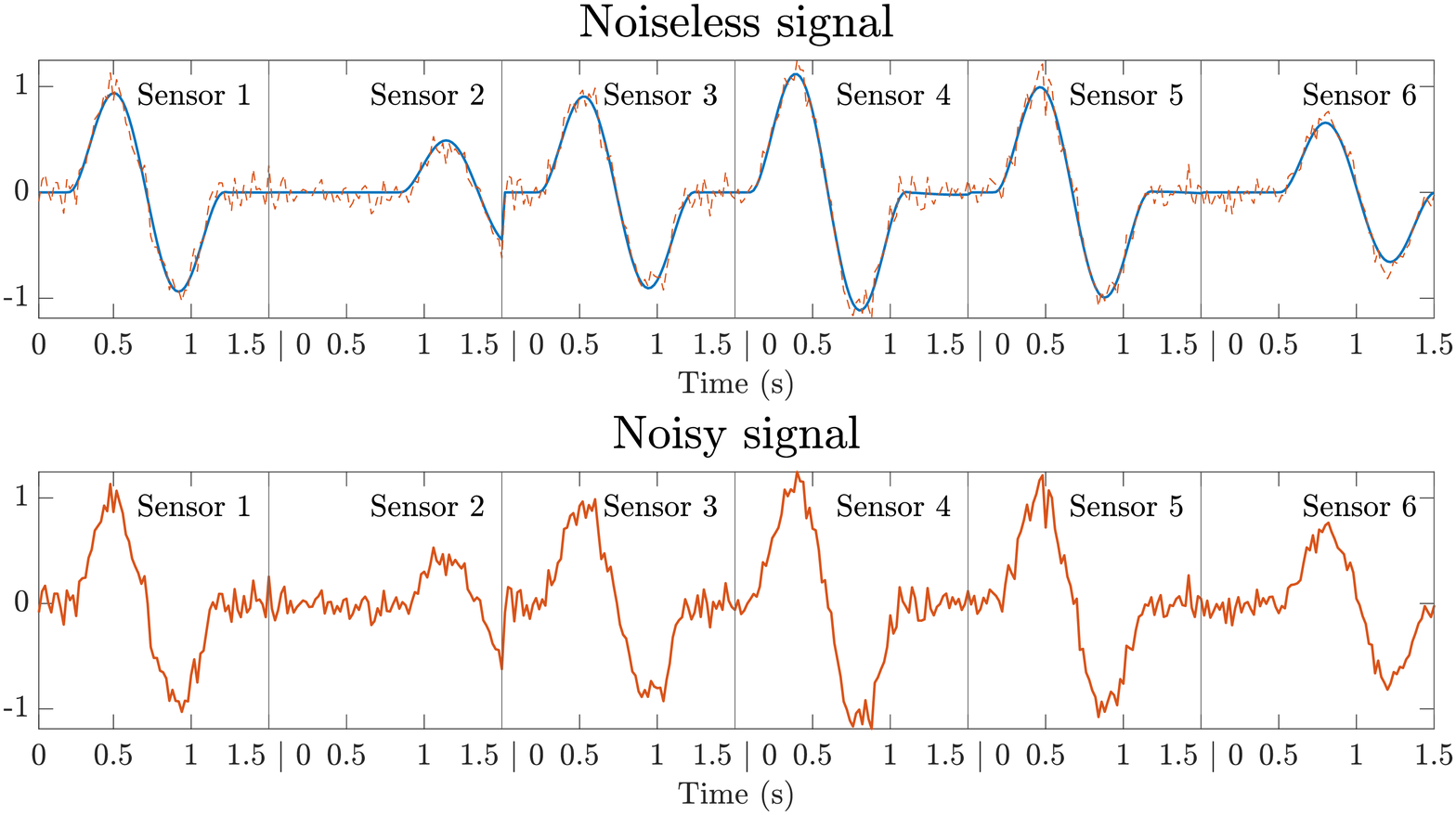}	
\caption{Example of captured signals for the test case \#1}
\label{fig_cos pos sig}
\end{figure}
\begin{figure}[tb]
\includegraphics[width=\textwidth]{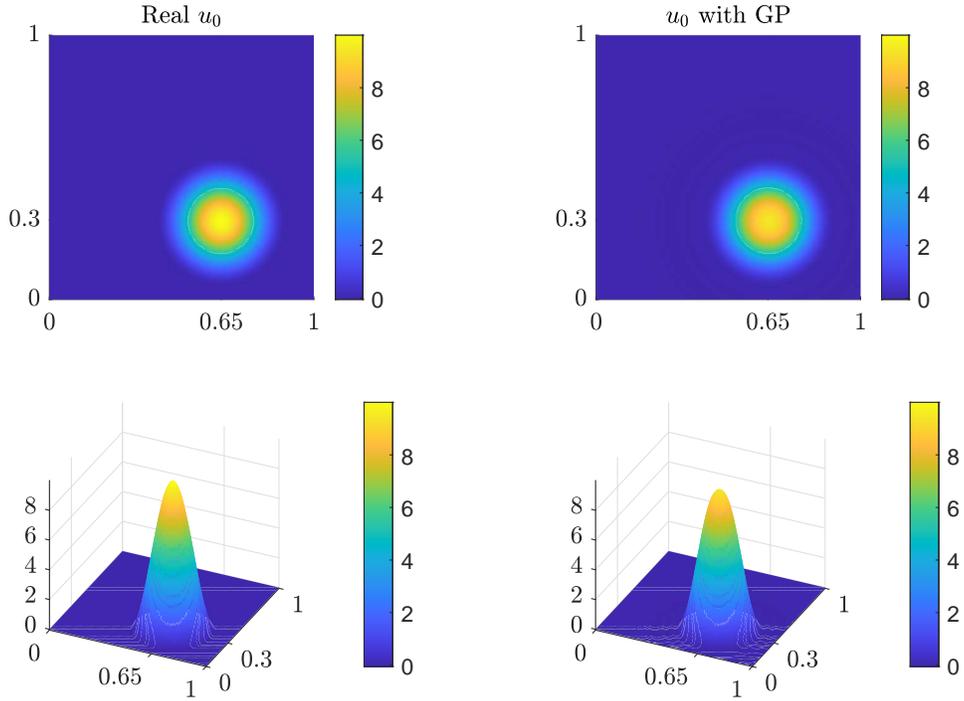}	
\caption{Test case \#1 : top and lateral view of the reconstruction of $u_0$ provided by WIGPR in comparison with $u_0$. Left : $u_0$. Right : reconstruction of $u_0$ using WIGPR. 20 sensors were used. The images correspond to the 3D solutions evaluated at $z = 0.5$.}
\label{fig_cos pos slice}
\end{figure}
\subsubsection{Discussion on the numerical results} For problem \ref{num:multistart}, Figures \ref{fig_cos_pos_x0} and \ref{fig_cos_pos_c} show that the physical parameters $x_0$ and $c$ are well estimated. Figure \ref{fig_cos_pos_R} shows that the source size parameter $R$ is overestimated, as could be expected from Remark \ref{rk:R_difficult}. The relative errors in Figures \ref{fig_cos_pos_L2}, \ref{fig_cos_pos_Linf} and \ref{fig_cos_pos_L1} show that the overall function reconstruction is satisfying, with around $10 \%$ relative error. For problem \ref{num:fig_sensibility} (figures \ref{fig_sensib_u_u_L2}, \ref{fig_sensib_u_u_Linf} and \ref{fig_sensib_u_u_L1}), the relative errors stagnate at around $2\%$. The data summarized by the boxplots are not very scattered, as is visible from the interquartile range (IQR) of the boxplots. The interquartile range is the difference between the $3^{rd}$ and the $1^{st}$ quartiles, and remains between $0.5\%$ and $2\%$ in the three figures aforementioned. This means that for this test case, the reconstruction step is not very sensitive to the sensors layout when they are scattered as a Latin hypercube.
For problems \ref{num:multistart} and \ref{num:fig_sensibility}, the results are very good overall. Note though that this is a simple test case : the function to be reconstructed is differentiable and exhibits rather small variations compared to the next test case.

\subsection{Test case for $k_{\mathrm{v}}^{\mathrm{wave}}$}\label{sub:GP_spd}
We use the kernel \eqref{eq:ft ft' compact} and the 1D Matérn kernel \eqref{eq:matern 5/2} for $K_v$ in equation \eqref{eq:ft ft' compact}. The initial condition $v_0$ is a radial ring cosine described as follows. Set $x_0 = (0.3,0.6,0.7)^T$, $R_1 = 0.05$, $R_2 = 0.15$, $A = 50$, the corresponding IC are given by :
\begin{align*}
\begin{cases}
	u_0(x) &= 0 \\
    v_0(x) &= A\mathbbm{1}_{[R_1,R_2]}(|x-x_0|)\Bigg(1 + \cos\bigg(\frac{2\pi(|x-x_0|-\frac{R_1+R_2}{2})}{R_2-R_1}\bigg) \Bigg)
    \end{cases}
\end{align*}
See Figure \ref{fig_cos ring spd slice}, left column, for a graphical representation. For problem \ref{num:multistart}, the optimization domain is chosen to be the one from the previous test case defined in \eqref{eq:dom_optim_1}. For problem \ref{num:fig_sensibility}, the hyperparameter $\theta_0$ provided to the model is
\begin{align}
\theta_0 = (x_0,R,(\rho,\sigma^2),c,\lambda) = ((0.3,0.6,0.7),0.15,(0.03,3),0.5,\sigma_{noise}^2)
\end{align}
with $\sigma_{noise}^2 = 0.0081$. As in the test case \#1, the value of $0.03$ provided for $\rho$ is a visual estimation of the length scale of $v_0$ based on Figure \ref{fig_cos ring spd slice}. Additionally, the couple $(0.03,3)$ provided for $(\rho,\sigma^2)$ is more or less confirmed by the estimations of $(\rho,\sigma^2)$ performed for problem \ref{num:multistart}, test case \#2, visible in Table \ref{table:estim_matern}.
\begin{figure}[tb]
\includegraphics[width=0.8\textwidth]{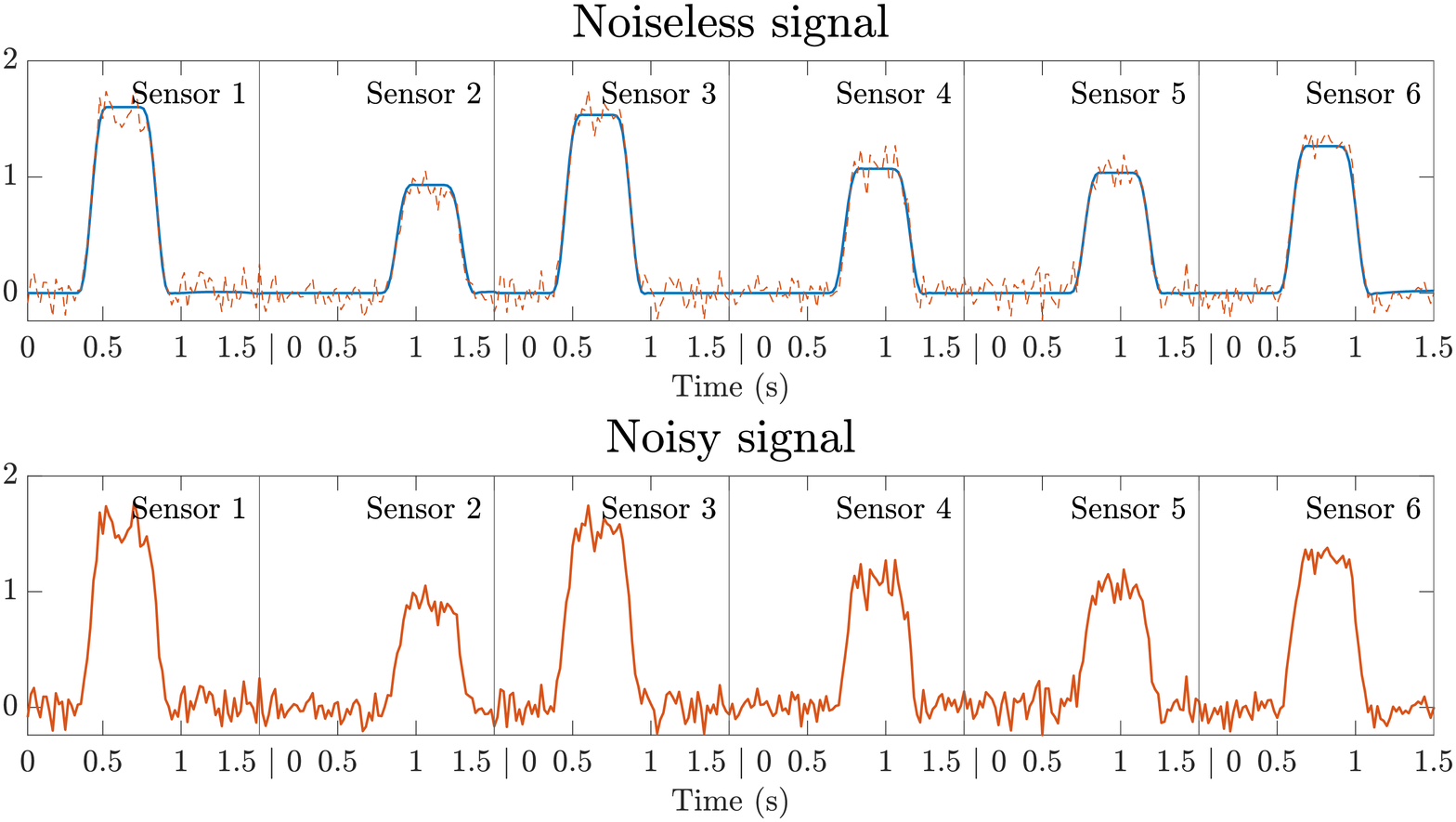}	
\caption{Example of captured signals for the test case \#2}
\label{fig_ring cos spd sig}
\end{figure}
\begin{figure}[tb]
\includegraphics[width=\textwidth]{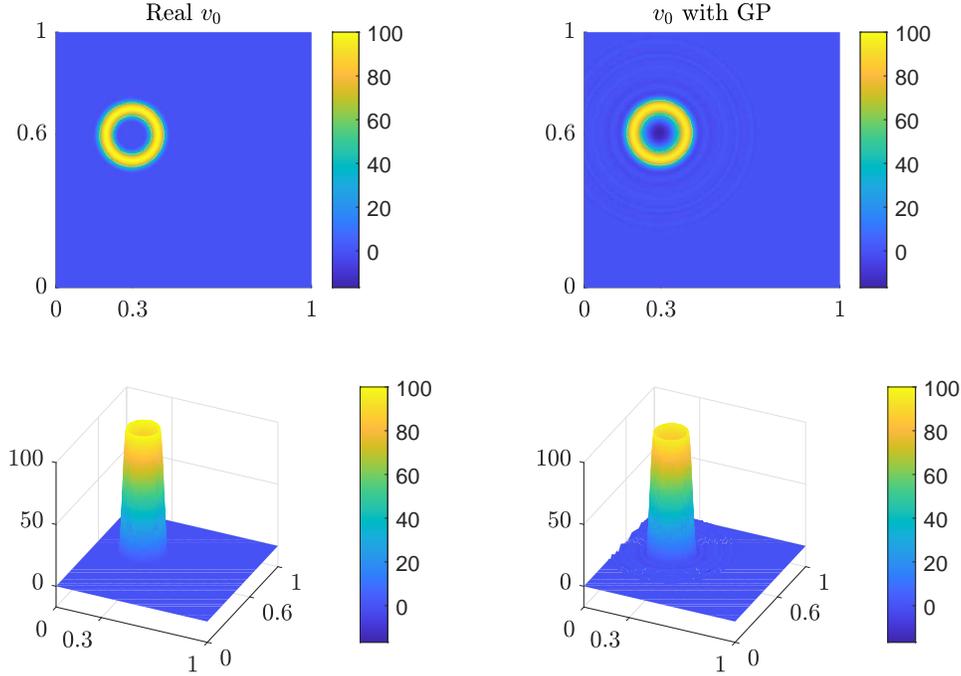}	
\caption{Test case \#2 : top and lateral view of the reconstruction of $v_0$ provided by WIGPR in comparison with $v_0$. Left : $v_0$. Right : reconstruction of $v_0$ using WIGPR. 25 sensors were used. The images correspond to the 3D solutions evaluated at $z = 0.7$.}
\label{fig_cos ring spd slice}
\end{figure}
\subsubsection{Discussion of the numerical results} For problem \ref{num:multistart}, Figures \ref{fig_cos_ring_speed_x0} and \ref{fig_cos_ring_speed_c} show that the physical parameters $x_0$ and $c$ are well estimated. As in the first test case, Figure \ref{fig_cos_ring_speed_R} shows that the source size parameter $R$ is overestimated, see Remark \ref{rk:R_difficult}. However, the relative error plots are not as good as in the previous example. The poor $L^1$ relative error between $60\%$ and $100\%$ in Figure \ref{fig_cos_ring_speed_L1} is an indicator that the reconstructed function (WIGPR Kriging mean) is supported on a set significantly larger than the original function, as the $L^1$ norm is minimized for sparse functions. This is indeed visible on Figure \ref{fig_cos ring spd slice}, and is partly due to the fact that $R$ was overestimated. The $L^2$ and $L^{\infty}$ plots in Figures \ref{fig_cos_ring_speed_L2} and \ref{fig_cos_ring_speed_Linf} are not too bad, with $20\%$ to $30\%$ relative error. Although these errors are large when compared to the test case \#1, WIGPR still managed to capture the structure of the initial condition, as is visible on Figure \ref{fig_cos ring spd slice}. % show that the Kriging means have captured the structure of the data. For these pictures, the $L^{\infty}$ and $L^2$ relative errors are of around $100\%$ (20 sensors in Figures \ref{fig_cos_ring_speed_L2} and \ref{fig_cos_ring_speed_Linf}). 
For problem \ref{num:fig_sensibility}, the relative error plots are better than for \ref{num:multistart}. The $L^2$ and $L^{1}$ relative errors stagnate at around $6\%$  with an IQR of around $2\%$ (figures \ref{fig_sensib_v_v_L2} and \ref{fig_sensib_v_v_L1}). The $L^{\infty}$ relative error stagnates at around $10\%$ with an IQR of around $5\%$ (figure \ref{fig_sensib_v_v_Linf}). Note finally that this test case is more difficult than the previous one : the initial condition to reconstruct has a smaller support and exhibits much larger variations.

\subsection{Test case for $k_{\mathrm{u}}^{\mathrm{wave}} + k_{\mathrm{v}}^{\mathrm{wave}}$}\label{sub:GP_mix}
Here, we combine the initial conditions from the two previous test cases. Set $x_0^u = (0.65,0.3,0.5)^T$, $R_u = 0.25$, $A_u = 2.5$, $x_0^v = (0.3,0.6,0.7)^T$, $R_1^v = 0.05$, $R_2^v = 0.15$ and $A_v = 30$. The corresponding IC are given by :
    \begin{align*}
    \begin{cases}
    u_0(x) &= A_u\mathbbm{1}_{[0,R_u]}(|x-x_0^u|)\Bigg(1 + \cos\bigg(\frac{\pi|x-x_0^u|}{R_u}\bigg) \Bigg) \\
    v_0(x) &= A_v\mathbbm{1}_{[R_1^v,R_2^v]}(|x-x_0^v|)\Bigg(1 + \cos\bigg(\frac{2\pi(|x-x_0^v|-\frac{R_1^v+R_2^v}{2})}{R_2^v-R_1^v}\bigg) \Bigg)
    \end{cases}
\end{align*}
See Figures \ref{fig_cos pos slice} and \ref{fig_cos pos slice}, left columns, for graphical representations. For problem \ref{num:multistart}, the optimization domain is chosen to be the following hypercube
\begin{align}\label{eq:dom_optim_3}
\theta = &(x_0^u,R_u,(\rho_u,\sigma_u^2),x_0^v,R_v,(\rho_v,\sigma_v^2),c,\lambda)\nonumber \\
 \in &[0,1]^3 \times [0.05,0.4]\times [0.02,2] \times [0.1,5] \nonumber \\
 \times &[0,1]^3\times [0.05,0.4]\times [0.02,2] \times [0.1,5]\times [0.2,0.8] \times [10^{-8},10^{-2}]
\end{align}
For problem \ref{num:fig_sensibility}, the hyperparameter value $\theta_0$ provided to the model is
\begin{align}
\theta_0 = ((0.65,0.3,0.5),0.3,(0.2,0.3),(0.3,0.6,0.7),0.15,(0.03,3),0.5,\sigma_{noise}^2)
\end{align}
with $\sigma_{noise}^2 = 0.0081$. The provided values for $(\rho_u,\sigma_u^2)$ and $(\rho_v,\sigma_v^2)$ are those of the two previous test cases; see Subsections \ref{sub:GP_pos} and \ref{sub:GP_spd}. %for explanations on the choice of these values.
\begin{figure}[tb]
\includegraphics[width=0.8\textwidth]{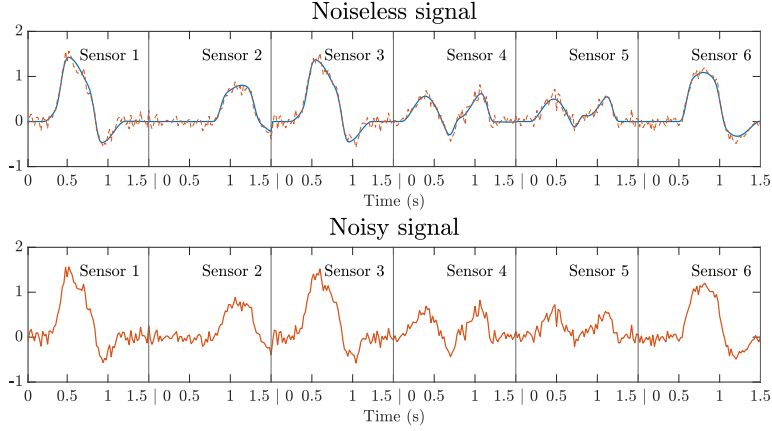}	
\caption{Example of captured signals for the test case \#3}
\label{fig_mix sig}
\end{figure}
\begin{figure}[tb]
\includegraphics[width=\textwidth]{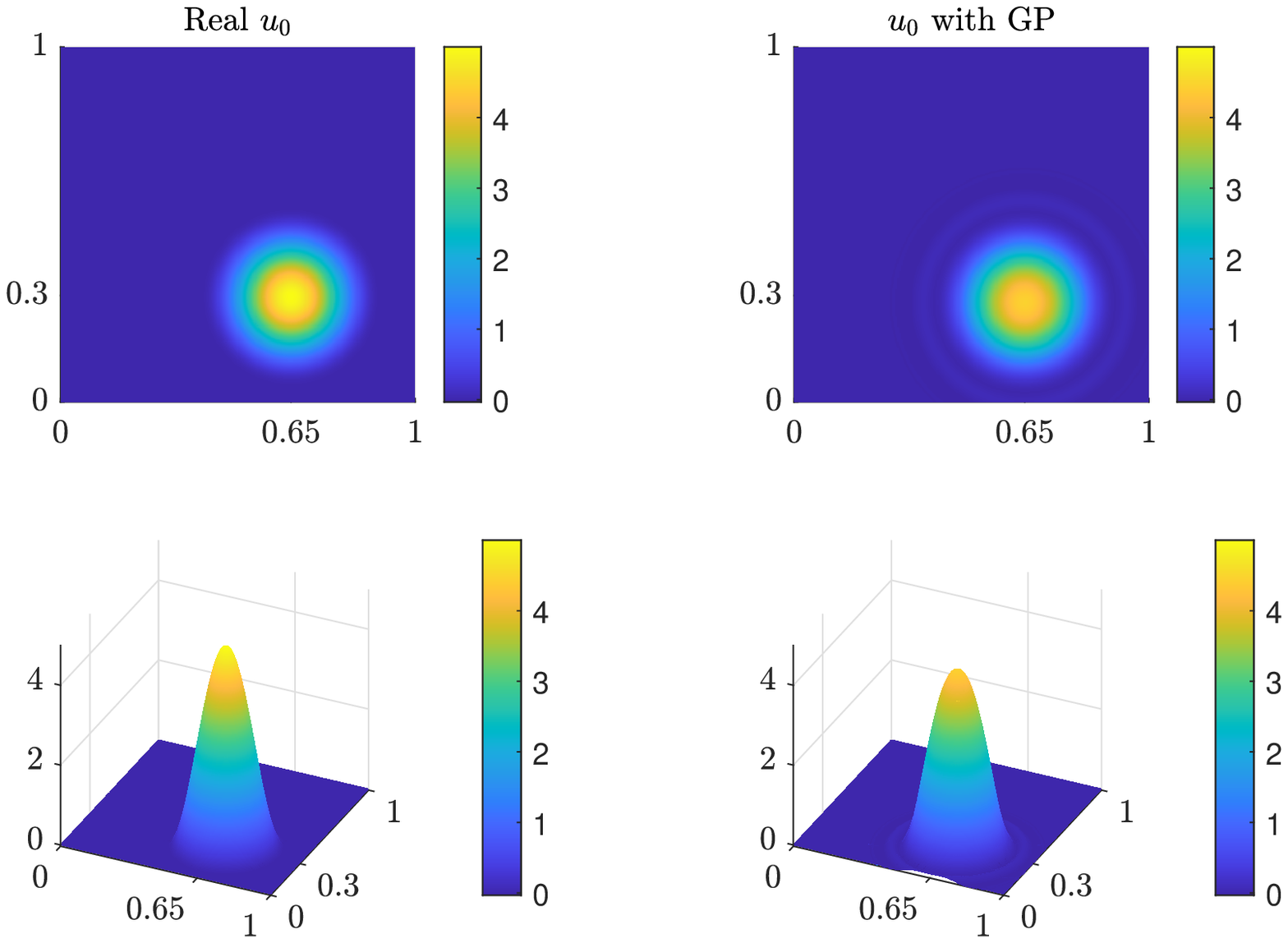}	
\caption{Test case \#3 : top and lateral view of the reconstruction of $u_0$ provided by WIGPR in comparison with $u_0$. Left : $u_0$. Right : reconstruction of $u_0$ using WIGPR. 30 sensors were used. The images correspond to the 3D solutions evaluated at $z = 0.5$.}
\label{fig_mix pos slice}
\end{figure}
\begin{figure}[tb]
\includegraphics[width=\textwidth]{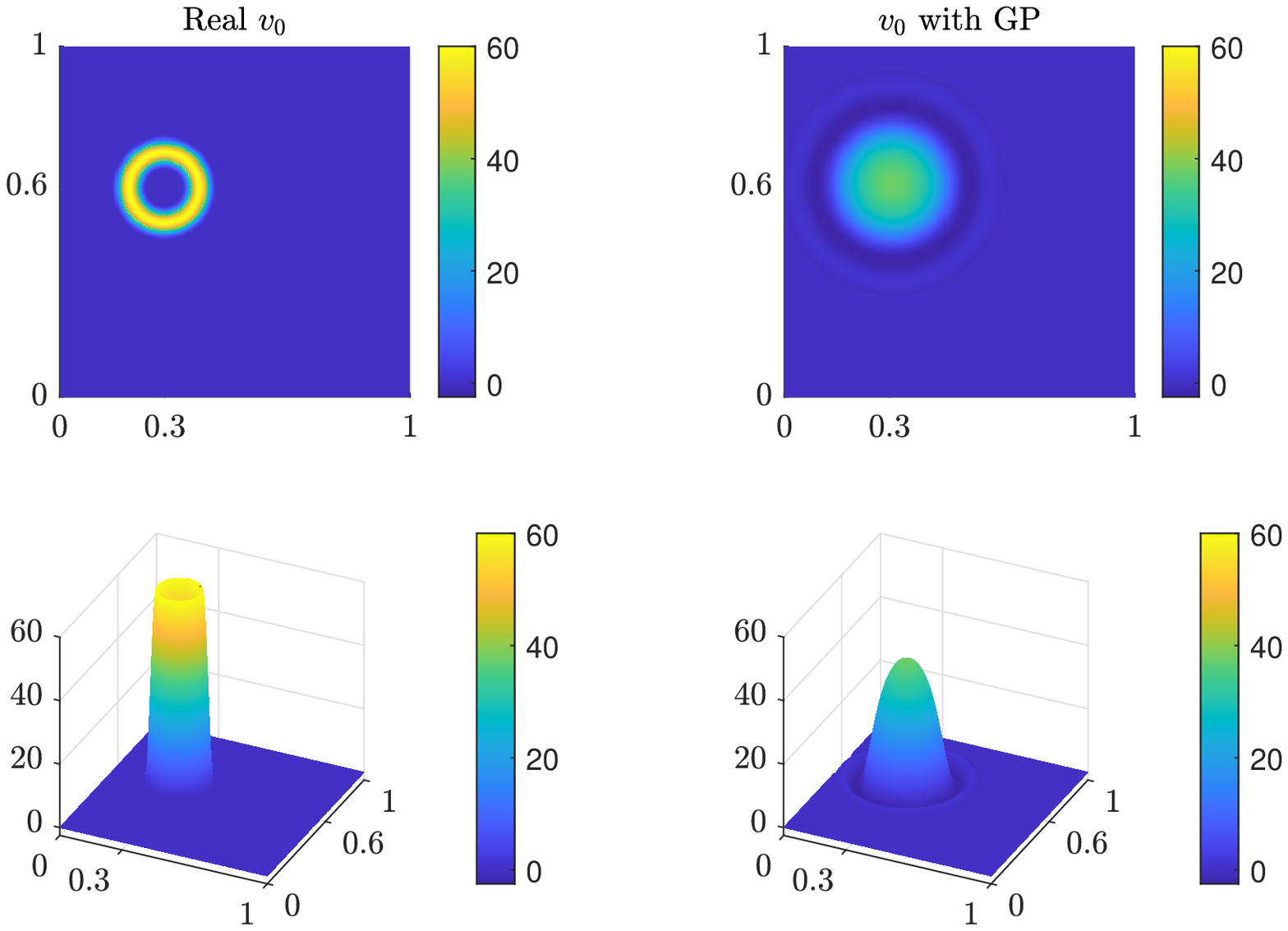}	
\caption{Test case \#3 : top and lateral view of the reconstruction of $v_0$ provided by WIGPR in comparison with $v_0$. Left : $v_0$. Right : reconstruction of $v_0$ using WIGPR. 30 sensors were used. The images correspond to the 3D solutions evaluated at $z = 0.7$.}
\label{fig_mix spd slice}
\end{figure}
\subsubsection{Discussion of the numerical results} Figures \ref{fig_mix_x0_u}, \ref{fig_mix_x0_v} and \ref{fig_mix_c} show that the physical parameters $x_0^u$, $x_0^v$ and $c$ are well estimated. Figures \ref{fig_mix_R_u} and \ref{fig_mix_R_v} show that the source sizes $R_u$ and $R_v$ are overestimated, see Remark \ref{rk:R_difficult}. For the initial position $u_0$, Figures \ref{fig_mix_L2_u} and \ref{fig_mix_Linf_u} show that with $N_s = 30$ sensors, WIGPR yielded good results with around $10\%$ $L^2$ and $L^{\infty}$ of relative errors. However, for the other values of $N_s$, WIGPR did not perform as well with around $50\%$ of $L^2$ and $L^{\infty}$ relative error. As in the previous test case, the $L^1$ relative error plot in Figure \ref{fig_mix_L1_u} shows that the reconstructed function $\tilde{u}_0$ is usually supported on larger sets than the true function $u_0$. With $30$ sensors, WIGPR yielded $30\%$ relative error and between $60\%$ and $90\%$ of relative $L^1$ error for the other values of $N_s$. The reconstruction of $u_0$ in the case $N_s = 30$ is displayed in Figure \ref{fig_mix pos slice}, showing a satisfactory result. For the reconstruction of the initial speed $v_0$, all the numerical attempts failed. The best result is attained for $N_s = 30$ sensors with $65\%$ of $L^2$ and $L^{\infty}$ relative errors (\ref{fig_mix_L2_v} and \ref{fig_mix_Linf_v}). But for this case, Figure \ref{fig_mix spd slice} shows that WIGPR failed to capture the structure of $v_0$. One explanation of this is that for the test case \#3, the value of the length scale hyperparameter $\rho_v$ was overestimated by a factor of at least 10 (see the $4^{th}$ row of Table \ref{table:estim_matern} versus the $3^{rd}$ row). Therefore the WIGPR model was never able to capture the fast variations of $v_0$.
 %Figures \ref{fig_mix_L2_v} and \ref{fig_mix_Linf_v} show that most of the reconstructed functions $\tilde{v}_0$ yield relative $L^2$ and $L^{\infty}$ errors of around $100\%$, with 3 huge overshoots for 4, 10 and 20 sensors. Only one numerical experiment succeeded for the noiseless datasets, using 15 sensors. This experiment yielded relative $L^2$ and $L^{\infty}$ errors of $47\%$ and $16\%$ respectively. This experiment is the one showed in Figure \ref{fig_mix spd slice}. All the reconstructions failed for the noisy datasets. 
For problem \ref{num:fig_sensibility}, the results are better. For the initial position $u_0$, the results are very similar to those of the test case \#1. Figures \ref{fig_sensib_mix_L2_u}, \ref{fig_sensib_mix_Linf_u} and \ref{fig_sensib_mix_L1_u}  show that relative errors stagnate at around $2\%$ to $4\%$ for the three relative errors. The corresponding IQR are of the same order of magnitude : $1\%$ to $2\%$. For the initial speed $v_0$ (figures \ref{fig_sensib_mix_L2_v}, \ref{fig_sensib_mix_Linf_v} and \ref{fig_sensib_mix_L1_v}), the results are very similar to those of the test case \#2 but multiplied by a factor between 3 and 6. The $L^2$ and $L^1$ relative errors now stagnate below $30\%$ and the $L^{\infty}$ relative stagnates below $45\%$. The corresponding IQR are of around $10\%$ for the $L^2$ and $L^1$ errors and $15\%$ for the $L^{\infty}$ error. %.  for the noiseless datasets, remaining below $25\%$ for the $L^2$ and $L^1$ relative errors, and $35\%$ for the $L^{\infty}$ relative error. For the noisy datasets, the curves are basically the same than those of the test case \#2, multiplied by a factor between 3 and 6. In log scale, this pushes up the corresponding curves by a constant. The $L^2$ and $L^1$ relative errors now stagnate below $30\%$ and the $L^{\infty}$ relative stagnates below $45\%$.
Note finally that this test case combines the difficulties of the first and second test cases.

\begin{table}[tb]
\caption{Estimated Matérn hyperparams. and noise var. for \ref{num:multistart}}

\begin{tabular}{|c|c|c|c|c|c|c|}
 \hline & Test case & $N_s = 10$ & $N_s = 15$ & $N_s = 20$ & $N_s = 25$ & $N_s = 30$ \\ 
\hline \multirow{2}{*}{$\rho_u$} & 1 & 0.117 & 0.11 & 0.043 & 0.184 & 0.144 \\ 
& 3 & 1.145 & 0.357 & 0.891 & 0.529 & 0.088 \\ 
\hline \multirow{2}{*}{$\rho_v$} & 2 & 0.105 & 0.149 & 0.023 & 0.068 & 0.046 \\ 
 & 3 & 1.453 & 1.703 & 0.595 & 0.482 & 0.353 \\ 
\hline \multirow{2}{*}{$\sigma_u^2$} & 1 & 4.97 & 3.88 & 4.23 & 3.09 & 4.89 \\ 
& 3 & 3.9 & 4.97 & 3.84 & 4.64 & 4.89 \\ 
\hline \multirow{2}{*}{$\sigma_v^2$} & 2 & 4.25 & 5 & 1.61 & 4.17 & 4.89 \\ 
& 3 & 3.54 & 3.05 & 3.5 & 3.27 & 3.26 \\ 
\hline 
\multirow{3}{*}{$\lambda$} &
1 & 0.01 & 0.0099 & 0.0088 & 0.0086 & 0.0096 \\ 
& 2 & 0.0097 & 0.0096 & 0.0036 & 0.0094 & 0.0095 \\ 
& 3 & 0.0098 & 0.01 & 0.0099 & 0.0083 & 0.0094 \\ 
\hline 
\end{tabular}
\label{table:estim_matern}
\end{table}
\subsubsection{Noise level estimation} The estimated noise variance is displayed in the last three rows of Table \ref{table:estim_matern}. In all the test cases, the noise variance $\lambda$ was rather well estimated for $N_s = 25$ or $N_s = 30$, relatively to the target value of $\sigma_{noise}^2= 0.0081$.
%, approximately reaching $10^{-2}$ which is the upper bound for the noise variance search domain.

\subsection{Transparent boundary conditions}
The search for perfect or approximated transparent boundary conditions for numerical schemes is an active research topic. They became a popular research topic in 1970' \cite{Engquist1977AbsorbingBC} and several different strategies have since been set up to deal with them : see e.g. \cite{ref_PML},\cite{Engquist1977AbsorbingBC}, \cite{fu_tbc}, or \cite{noble_kazakova} for an overview of these different methods. The need for them rises from the fact that numerical simulation domains are necessarily finite and thus bounded; or at the very least, discretized thanks to a finite grid. To numerically solve a free space PDE such as \eqref{eq:wave_eq}, one first uses an adapted numerical scheme (finite differences, finite elements, etc) for the desired PDE applied in the interior domain, i.e. at the interior points of the discretization grid. At the boundaries of the numerical domain, one is then forced to encode boundary conditions : common examples are the Dirichlet and Von Neumann boundary conditions, which cause reflections on the boundaries of the domain. For the free space problem, the question is then the following : can we numerically implement boundary conditions that produce no reflections?
\begin{figure}[h!]

\includegraphics[width=\textwidth]{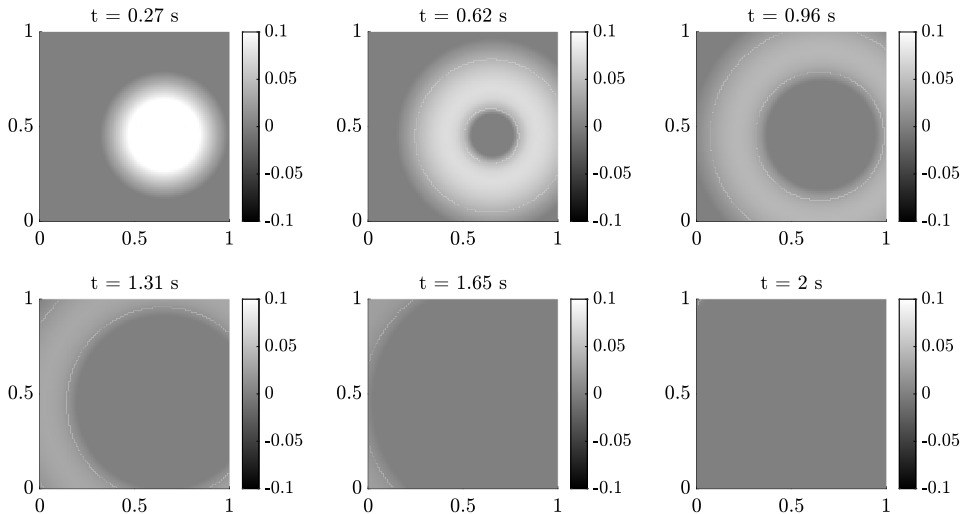}	
\caption{Numerical example showing that the outgoing waves produced by WIGPR do not suffer from numerical reflection issues.}
\label{fig_transparent}
\end{figure}

Figure \ref{fig_transparent} shows that WIGPR provides numerical solutions for problem \eqref{eq:wave_eq} that do not suffer from numerical reflection issues. The reason is simple : these numerical reconstructions are evaluations on a grid of functions that are exact solutions of the free space problem \eqref{eq:wave_eq}. While being a promising start for new types of transparent boundary conditions, the discussion above is not enough to provide viable transparent boundary conditions in the sense of numerical schemes. Given the covariance kernels from \eqref{eq: ku wave} and \eqref{eq: kv wave}, two questions should now be investigated. 1) How can WIGPR be associated to classical numerical schemes such as finite difference schemes for the wave equation? 2) Is the corresponding numerical transparent boundary condition numerically stable?
%\begin{enumerate}
%\item How can WIGPR be associated to classical numerical schemes such as finite difference schemes for the wave equation?
%\item Is the corresponding numerical transparent boundary condition numerically stable?
%\end{enumerate}
%For the first question, one could perform WIGPR on a boundary point by performing Kriging only using the neighbouring points, for instance from the previous time step. This is compatible with classical finite difference methods but this has yet to be tested.
%\subsection{Numerical Costs}

%% file: 6_conclusion.tex
\section{Conclusion and perspectives}
In Section 3, we have presented a new result that provides a simple characterization of the second order stochastic processes whose trajectories verify linear differential constraints within the framework of generalized functions. 
This result is one of the starting points of Section 4, in which a Gaussian process model of the wave equation was described and thoroughly studied. This model was derived by putting a GP prior on the wave equation's initial conditions, following a Bayesian approach when considering these initial conditions to be unknown. We then showed that WIGPR naturally provides an approximation of the initial conditions of the wave equation, as expected with a Bayesian approach w.r.t. the initial conditions. The radial symmetry WIGPR formulas from Section 4 were then showcased in Section 5, where two practical questions were tackled. First, WIGPR can correctly estimate certain physical parameters attached to the corresponding wave equation, such as wave speed or source position. When these parameters are well estimated, the initial condition reconstruction step is rather satisfactory in terms of $L^1$, $L^2$ and $L^{\infty}$ relative errors. Second, we observed that the WIGPR reconstruction step is not very sensitive to the layout of sensors when they are spread according to a Latin hypercube, assuming that the correct set of hyperparameters is provided to the model. 

Future possible investigations may concern the practical use of the more general formula \eqref{eq:wave kernel} without any symmetry assumptions. To compute the convolution efficiently, one may then resort to Fast Fourier transforms (FFT). The case of the two dimensional wave equation is also of practical interest, notably in oceanography \cite{lannes2D}, and presents many different properties than its 3D counterpart (\cite{evans1998}, p.80). It would thus deserve a theoretical and practical study in its own right when coupled with GPR. Moreover, proposition \ref{prop : diff constraints} constitutes a first step towards understanding PDE constrained stochastic processes in an integrated sense; different functional analysis frameworks can now be considered, for example by assuming Sobolev regularity of the coefficients of $T$ instead of $\mathcal{C}^k$ regularity or by trying to work in smaller function spaces than $\mathscr{D}'(\mathcal{D})$ by transferring only a part of the derivatives of the PDE to the test function. A well known particular case is when the differential operator can be ``split in two" by this procedure : it is then alternatively viewed as a bilinear form over some Hilbert space. This leads to the variational formulation of the PDE and is the canonical approach to studying elliptic PDEs.

\subsubsection*{Acknowledgements} We would like to thank the Service Hydrographique et Océano\-graphique de la Marine (SHOM) for funding this work, and Rémy Baraille in particular for his personal involvement in the project. We thank the GMM laboratory of INSA Toulouse for providing us access to their servers on which we ran our numerical computations.

\clearpage

%% file: Appendix_num_new_results.tex
\begin{center}
\section*{Appendix A : Numerical Results}
\end{center}
\section{Hyperparameter estimation and related error plots}

\subsection{Test case \#1 : $k_{\mathrm{u}} ^{\mathrm{wave}}$}
Here, $\theta = (x_0^u,R_u,(\rho_u,\sigma_u^2),c,\lambda) \in \mathbb{R}^8$. The target value is $\theta^* = \big((0.65,0.3,0.5),0.25,(\sim 0.15,?),0.5,0.0081\big)$. The values $(\rho_u,\sigma_u^2)$ are the hyperparameters of the underlying Matérn kernel $k_u^0$. The a priori target value of $0.15$ for $\rho_u$ is a visual estimation of the characteristic length scale of $u_0$ from the figure \ref{fig_cos pos slice}. The corresponding variance $\sigma_u^2$ is difficult to estimate a priori and its target value is left at ``?".

\begin{figure}[b]

\begin{subfigure}[b]{0.45\textwidth}
\includegraphics[width=\textwidth]{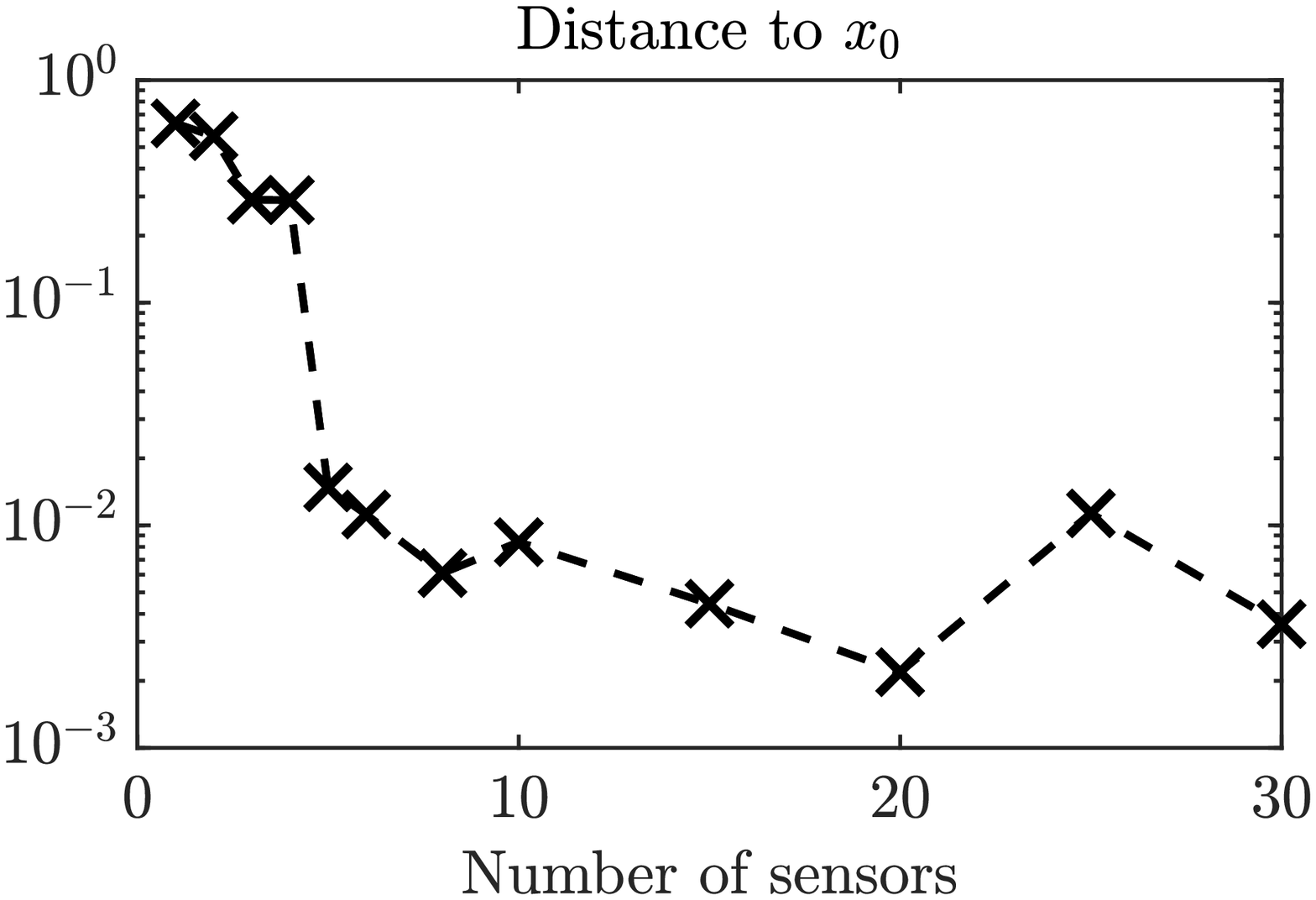}	
\caption{3D distance $|x_0-\hat{x}_0|$}
\label{fig_cos_pos_x0}
\end{subfigure}
\hfill
\begin{subfigure}[b]{0.45\textwidth}
\includegraphics[width=\textwidth,height=3.95cm,keepaspectratio]{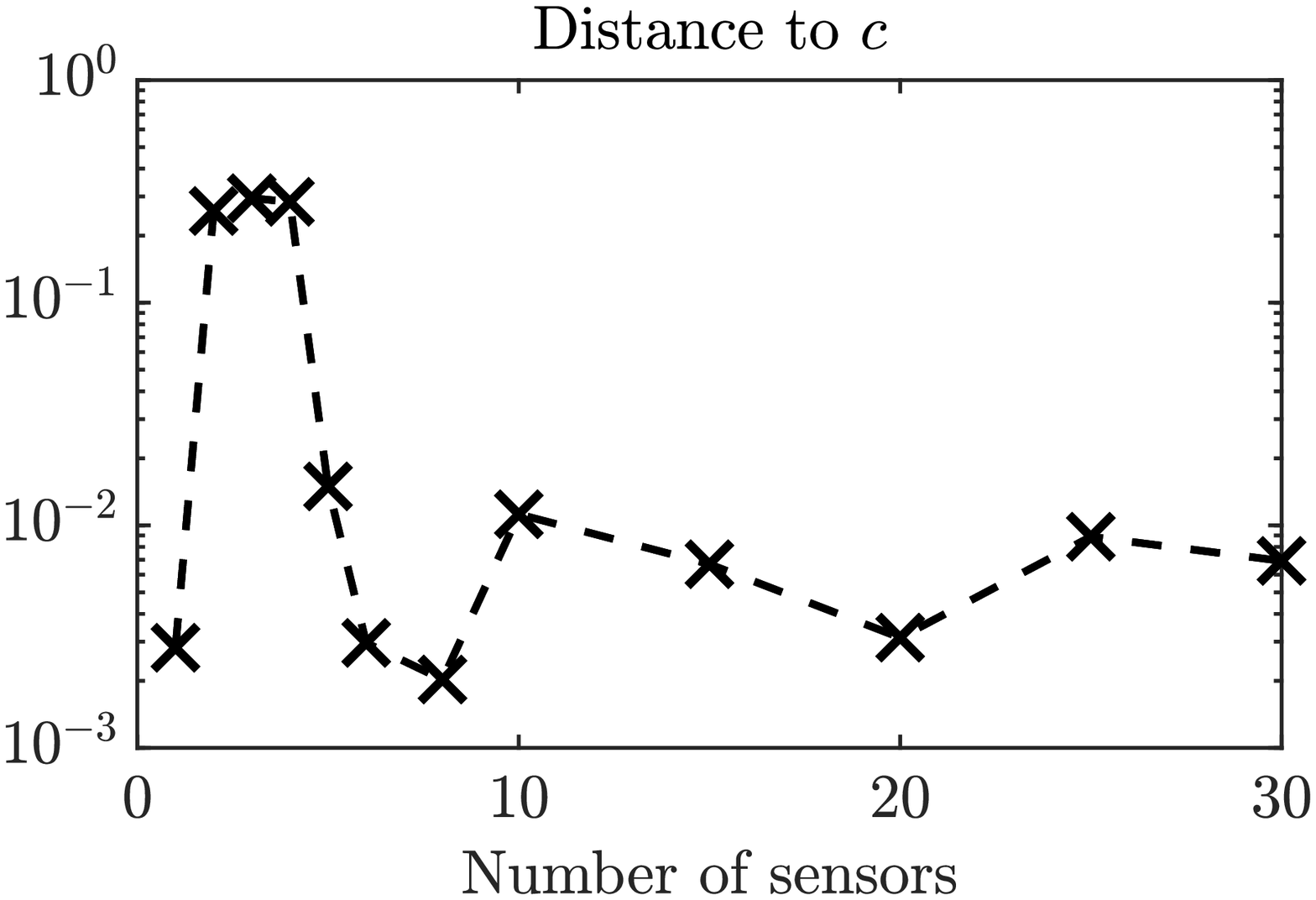}	
\caption{Distance $|c-\hat{c}|$}
\label{fig_cos_pos_c}
\end{subfigure} %

\vspace{5mm}

\begin{subfigure}[b]{0.45\textwidth}
\includegraphics[width=\textwidth,height=3.95cm,keepaspectratio]{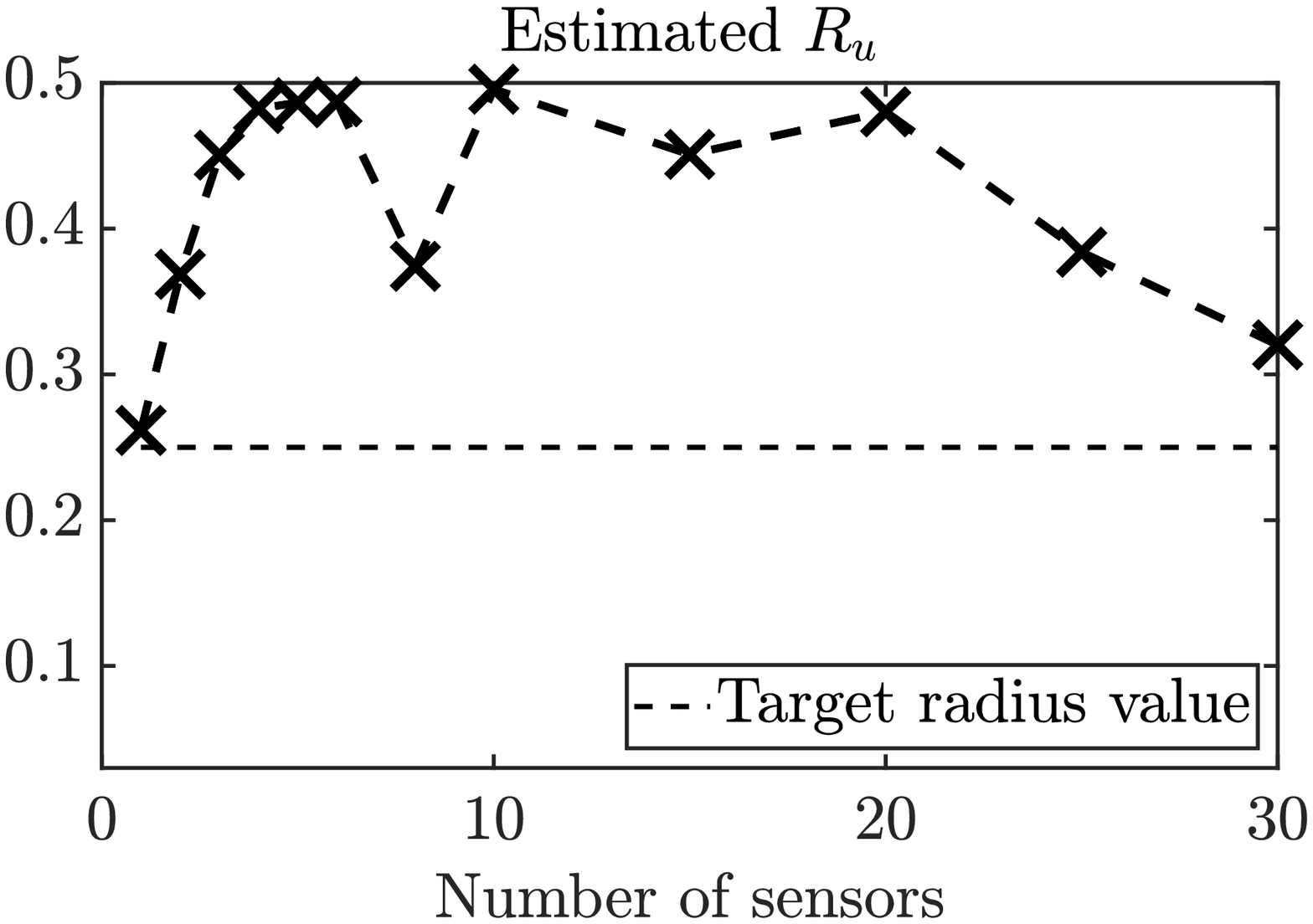}	
\caption{Estimated source radius}
\label{fig_cos_pos_R}
\end{subfigure}
\hfill
\begin{subfigure}[b]{0.45\textwidth}
\includegraphics[width=\textwidth,height=3.95cm,keepaspectratio]{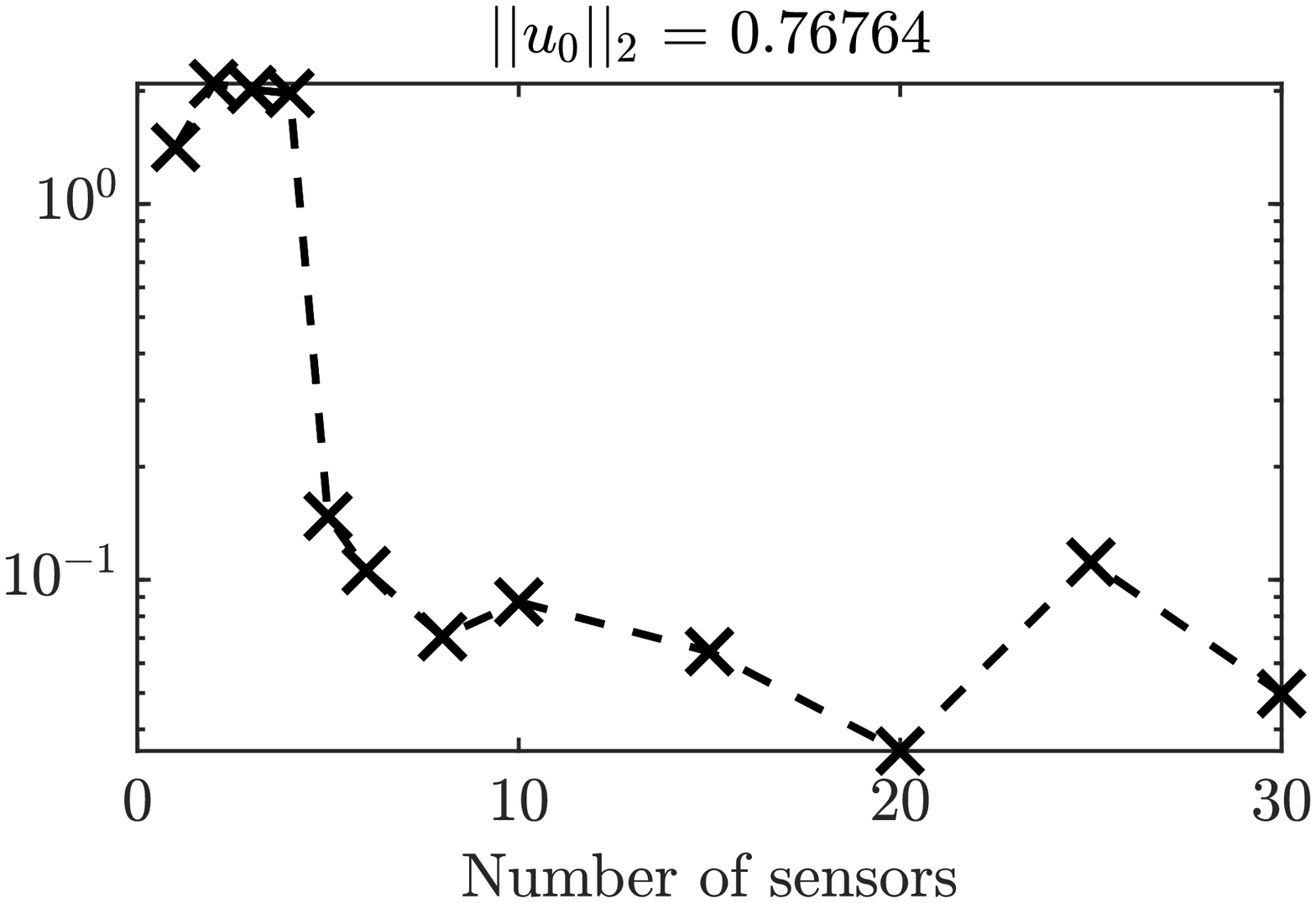}
\caption{$L^2$ rel. error}	
\label{fig_cos_pos_L2}
\end{subfigure} %

\vspace{5mm}

\begin{subfigure}[b]{0.45\textwidth}
\includegraphics[width=\textwidth,height=3.95cm,keepaspectratio]{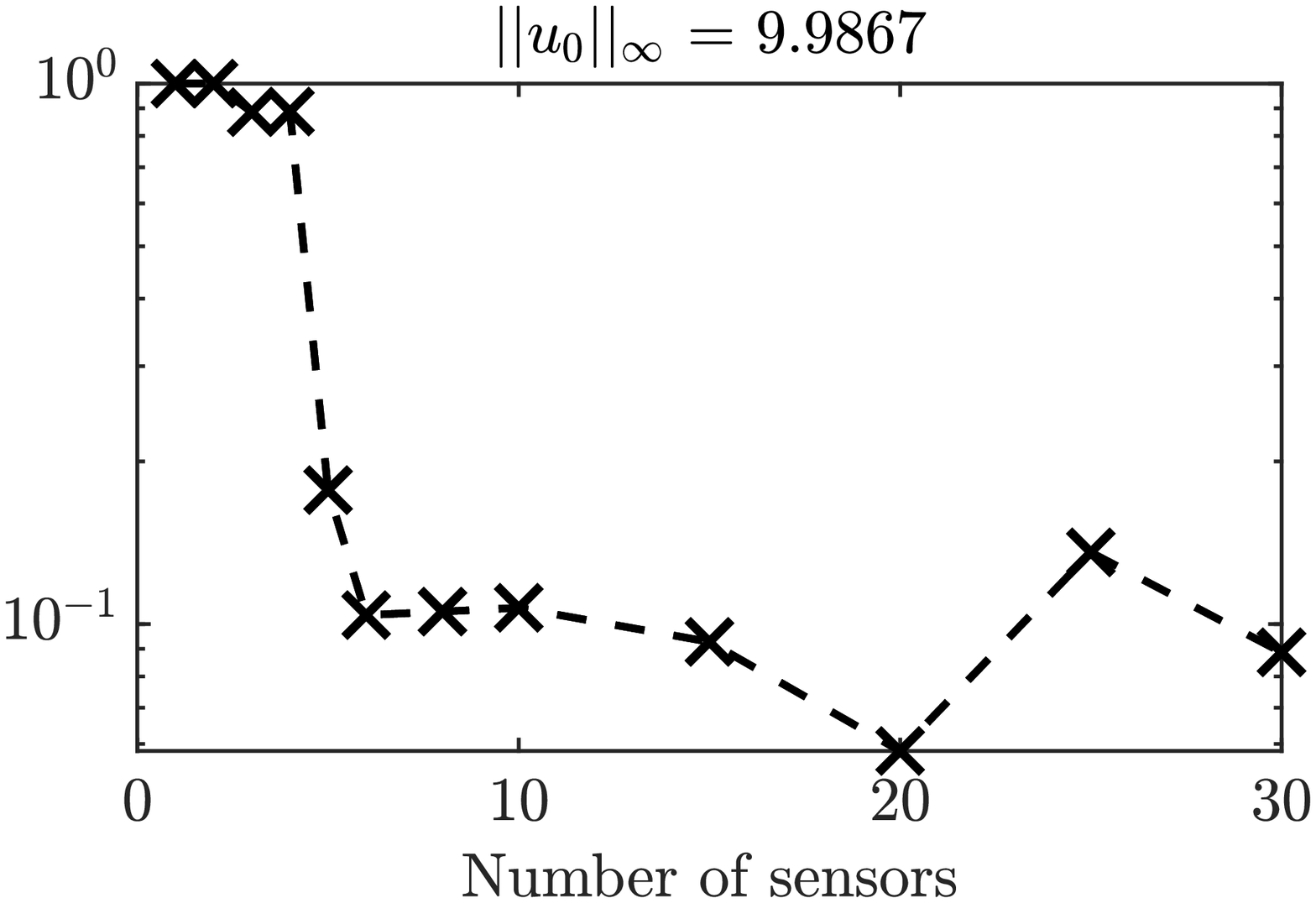}
\caption{$L^{\infty}$ rel. error}	
\label{fig_cos_pos_Linf}
\end{subfigure} 
\hfill
\begin{subfigure}[b]{0.45\textwidth}
\includegraphics[width=\textwidth,height=3.95cm,keepaspectratio]{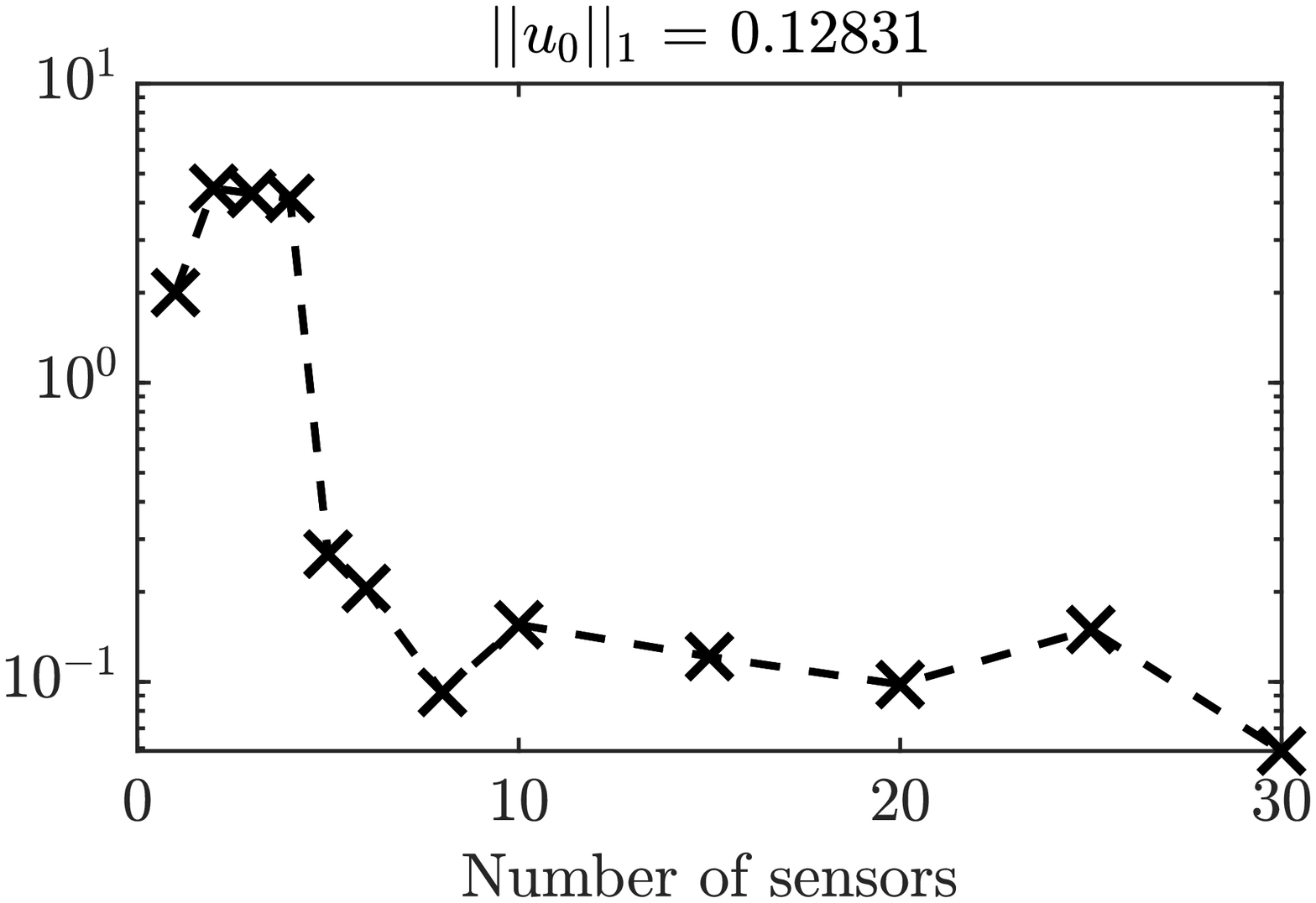}	
\caption{$L^1$ rel. error}
\label{fig_cos_pos_L1}
\end{subfigure} %

\caption{Hyperparameter estimation and error plots for the test case \#1}
\label{fig_cos_pos_multistart}
\end{figure}

%\begin{figure}
%\begin{subfigure}[b]{0.45\textwidth}
%\includegraphics[width=\textwidth,height=3.95cm,keepaspectratio]{u_Ru.eps}	
%\caption{test 1}
%%\label{fig_cos_pos_R}
%\end{subfigure}
%\hfill
%\begin{subfigure}[b]{0.45\textwidth}
%\includegraphics[width=\textwidth,height=4.5cm,keepaspectratio]{u_u_L1.eps}
%\caption{test 2}	
%%\label{fig_cos_pos_L2}
%\end{subfigure}
%\end{figure}

%\begin{figure}
%
%\includegraphics[width=0.3\textwidth]{u_Ru.eps}	
%\caption{test 1} 
%\hfill
%\includegraphics[width=0.3\textwidth]{u_u_L1.eps}
%\caption{test 2}	
%\hfill
%\includegraphics[width=0.3\textwidth]{u_x0_u.eps}
%\caption{test 3}	
%
%\end{figure}

\subsection{Test case \#2 : $k_{\mathrm{v}} ^{\mathrm{wave}}$}
Here, $\theta = (x_0^v,R_v,(\rho_v,\sigma_v^2),c,\lambda) \in \mathbb{R}^8$. The target value is $\theta^* = \big((0.3,0.6,0.7),0.15,(\sim 0.05,?),0.5,0.0081\big)$. The values $(\rho_v,\sigma_v^2)$ are the hyperparameters of the underlying Matérn kernel $k_v^0$. The a priori target value of $0.05$ for $\rho_v$ is a visual estimation of the characteristic length scale of $v_0$ from the figure \ref{fig_cos ring spd slice}. The corresponding variance $\sigma_v^2$ is difficult to estimate a priori and its target value is left at ``?".

\begin{figure}[b]

\begin{subfigure}[b]{0.45\textwidth}
\includegraphics[width=\textwidth,height=3.95cm,keepaspectratio]{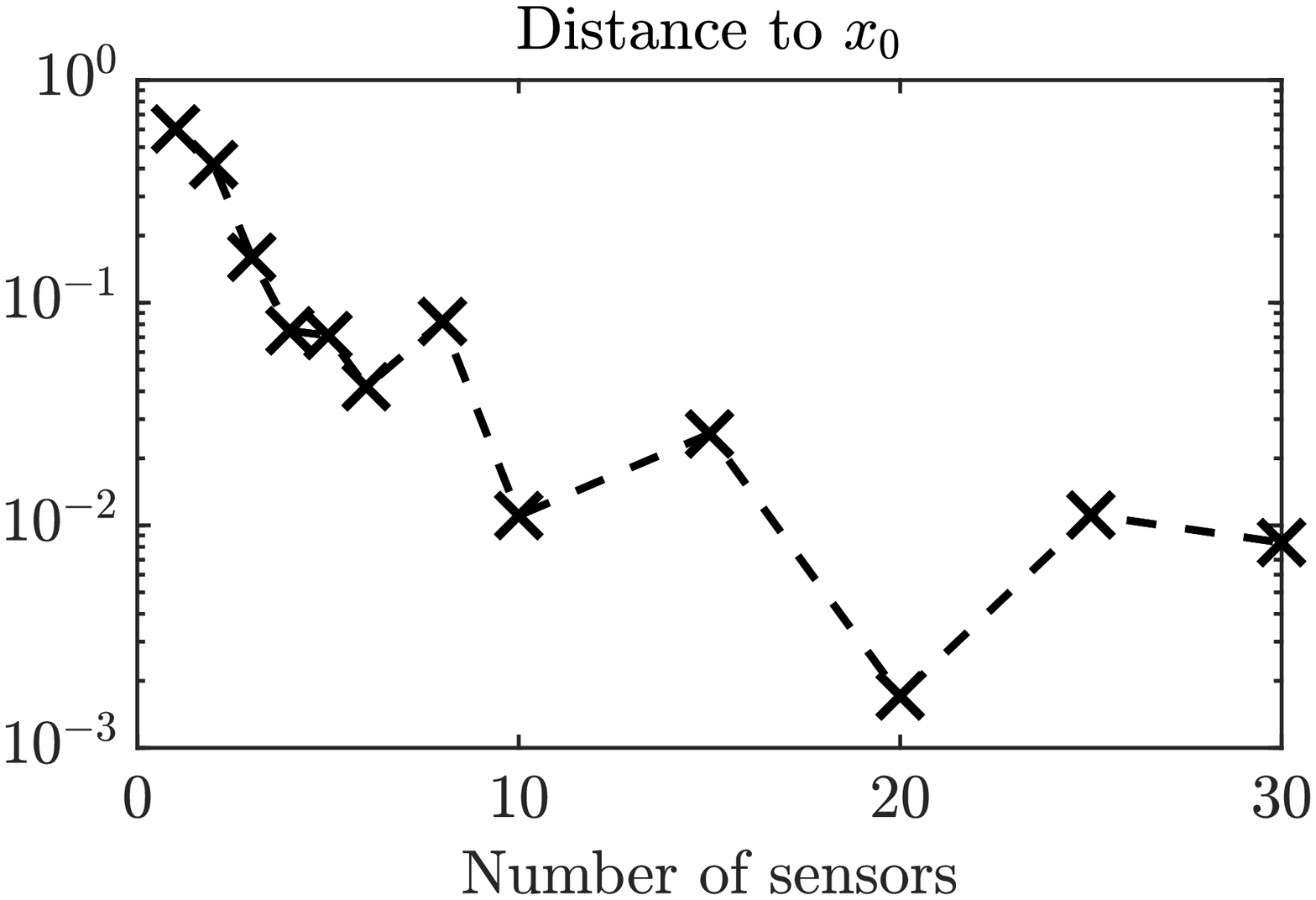}	
\caption{3D distance $|x_0-\hat{x}_0|$}
\label{fig_cos_ring_speed_x0}
\end{subfigure}
\hfill
\begin{subfigure}[b]{0.45\textwidth}
\includegraphics[width=\textwidth,height=3.95cm,keepaspectratio]{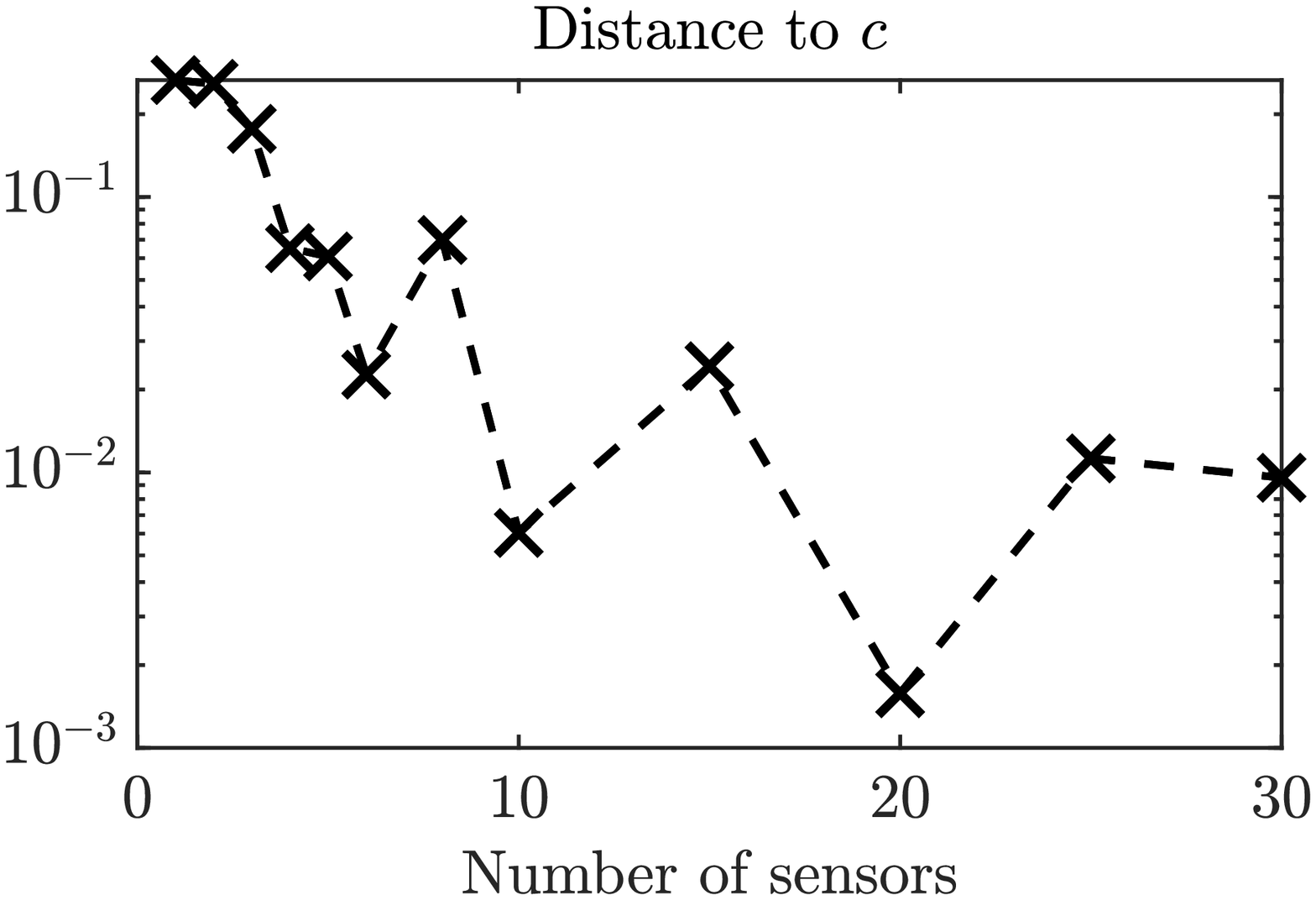}	
\caption{Distance $|c-\hat{c}|$}
\label{fig_cos_ring_speed_c}
\end{subfigure} %
\vspace{5mm}

\begin{subfigure}[b]{0.45\textwidth}
\includegraphics[width=\textwidth,height=3.95cm,keepaspectratio]{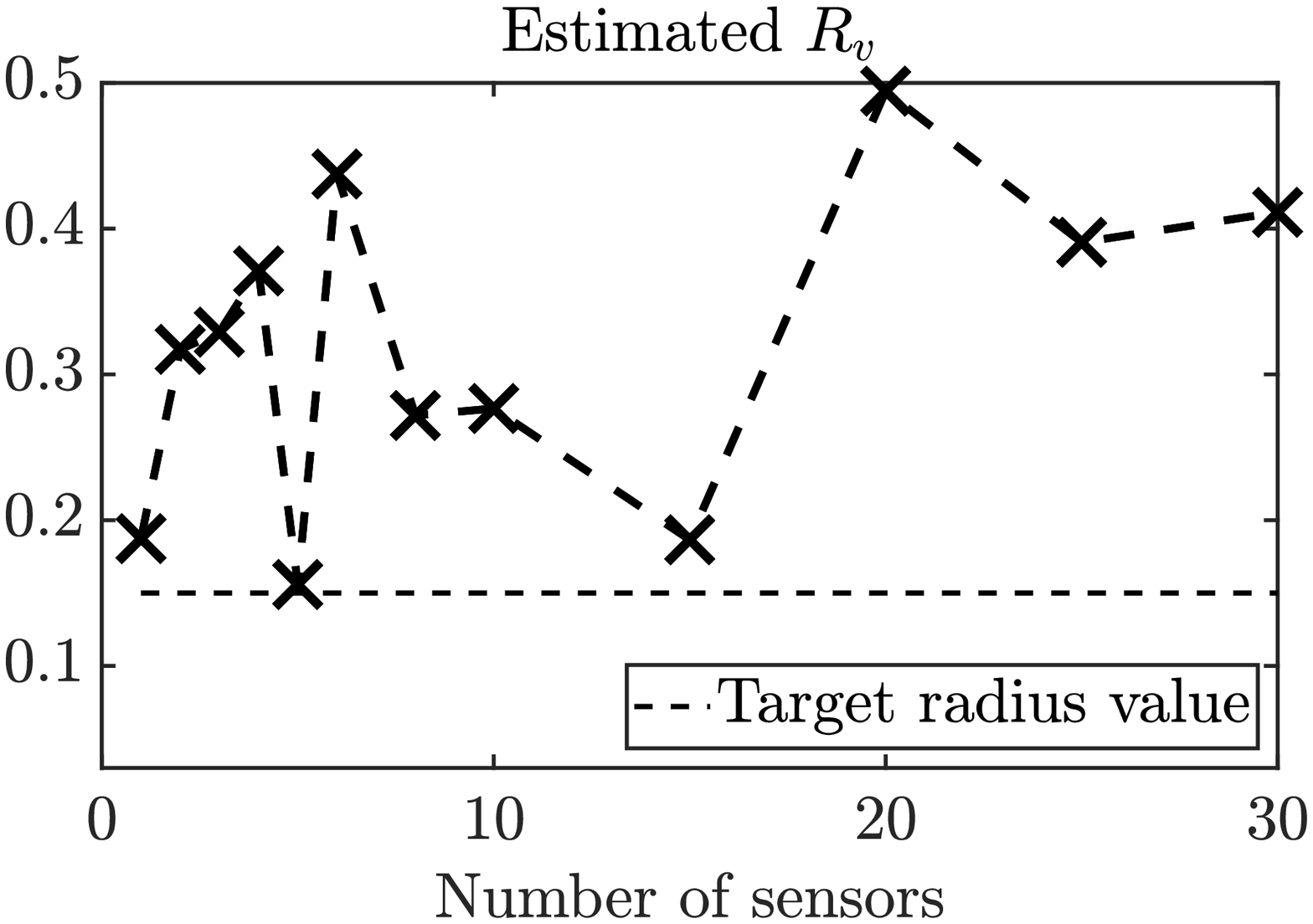}	
\caption{Estimated source radius}
\label{fig_cos_ring_speed_R}
\end{subfigure}
\hfill
\begin{subfigure}[b]{0.45\textwidth}
\includegraphics[width=\textwidth,height=3.95cm,keepaspectratio]{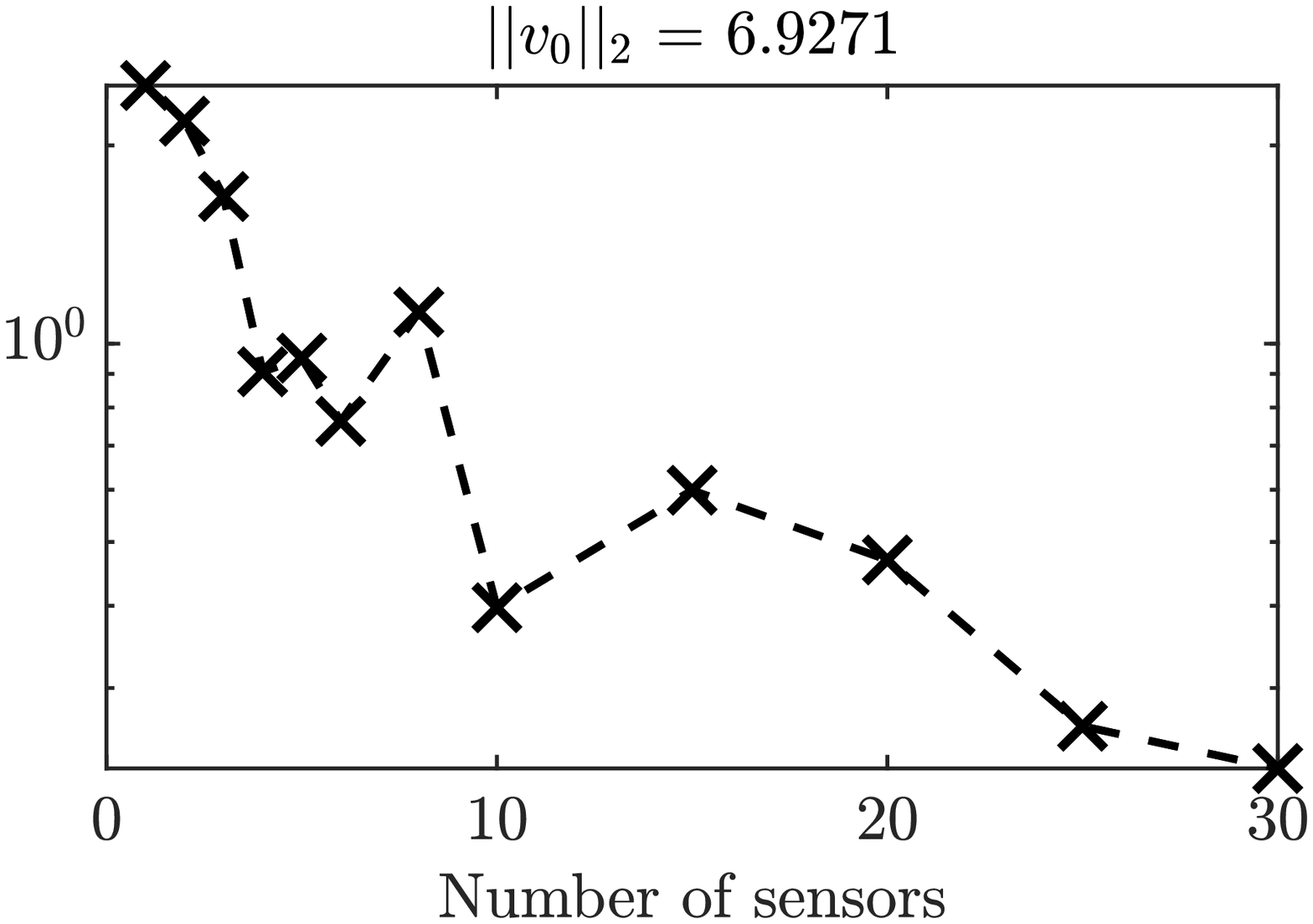}	
\caption{$L^2$ rel. error}
\label{fig_cos_ring_speed_L2}
\end{subfigure} %
\vspace{5mm}

\begin{subfigure}[b]{0.45\textwidth}
\includegraphics[width=\textwidth,height=3.95cm,keepaspectratio]{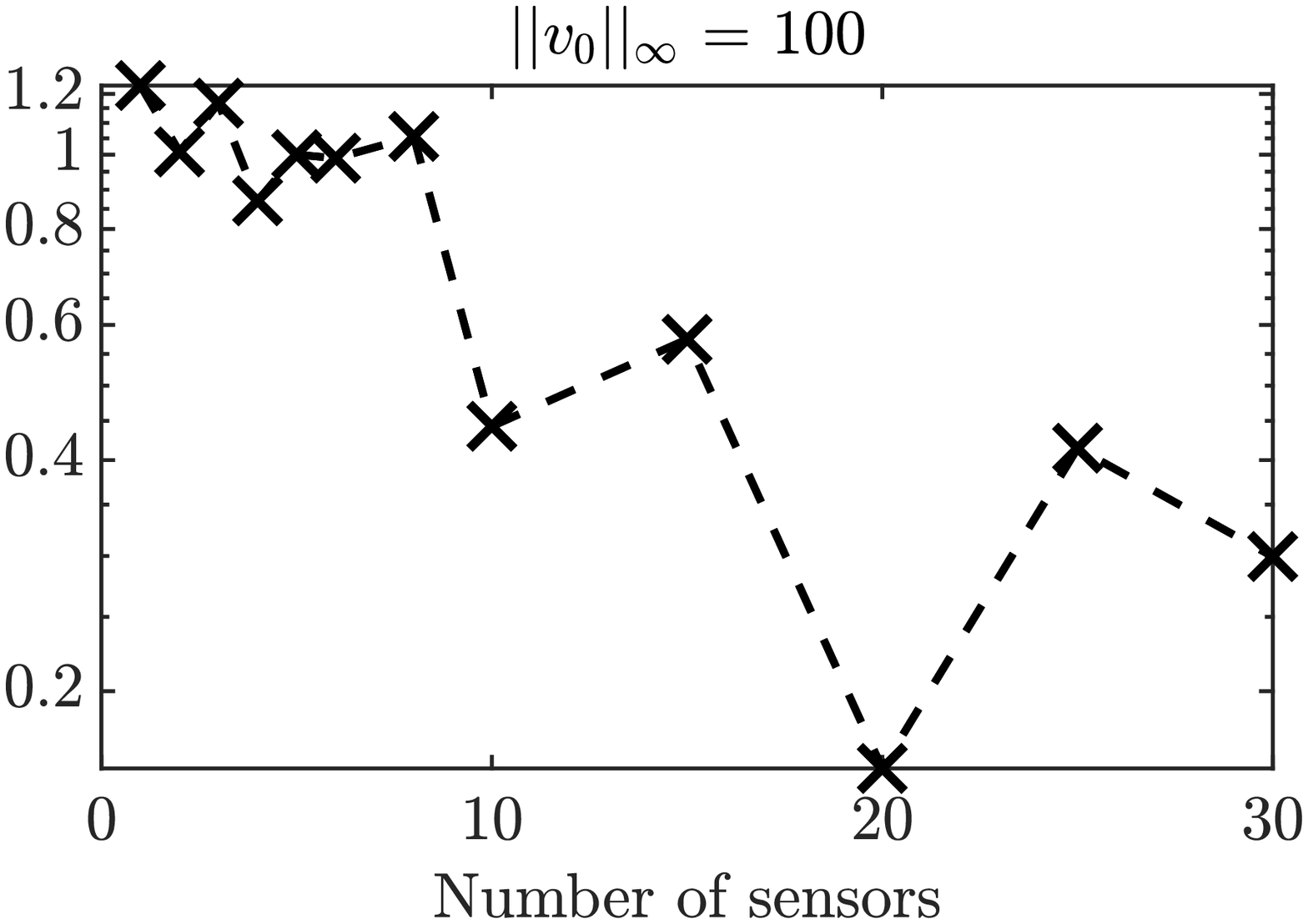}	
\caption{$L^{\infty}$ rel. error}
\label{fig_cos_ring_speed_Linf}
\end{subfigure}
\hfill
\begin{subfigure}[b]{0.45\textwidth}
\includegraphics[width=\textwidth,height=3.95cm,keepaspectratio]{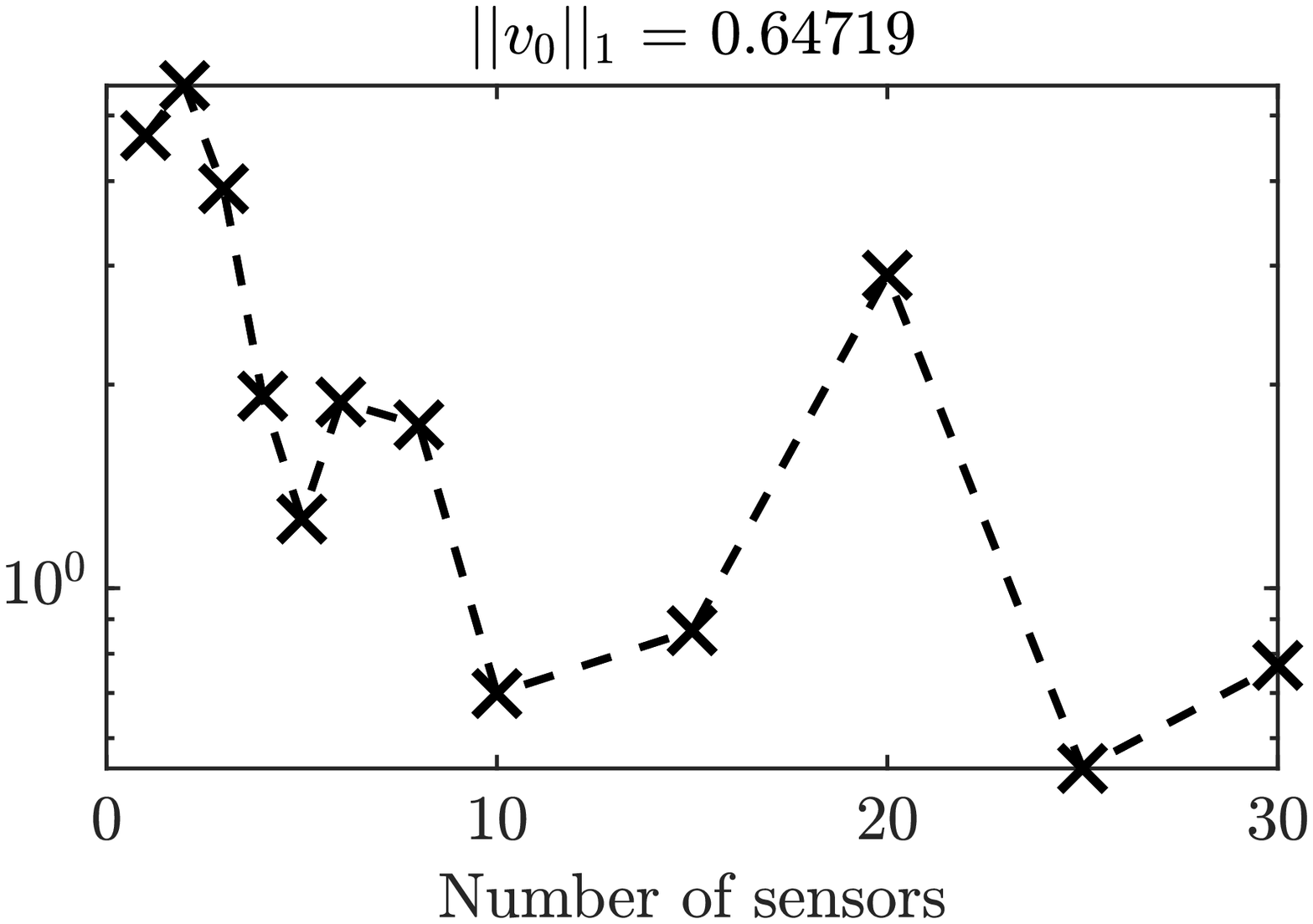}	
\caption{$L^1$ rel. error}
\label{fig_cos_ring_speed_L1}
\end{subfigure}
\caption{Hyperparameter estimation and error plots for the test case \#2}
\end{figure}
\newpage

\subsection{Test case \#3 : $k_{\mathrm{u}} ^{\mathrm{wave}} + k_{\mathrm{v}}^{\mathrm{wave}}$}
Here, $\theta = (x_0^u,R_u,(\rho_u,\sigma_u^2),x_0^v,R_v,(\rho_v,\sigma_v^2),c,\lambda)  \in \mathbb{R}^{14}$. The target value is $\theta^* = \big((0.65,0.3,0.5),0.3,(\sim 0.15,?),(0.3,0.6,0.7),0.15,(\sim 0.05,?),0.5,\lambda\big)$. The values $(\rho_u,\sigma_u^2)$ and $(\rho_v,\sigma_v^2)$ are the hyperparameters of the underlying Matérn kernels $k_u^0$ and $k_v^0$. The a priori target values of $0.15$ for $\rho_u$ and $0.05$ for $\rho_v$ are  visual estimations of the characteristic length scale of $u_0$ and $v_0$ from the figures \ref{fig_cos pos slice} and \ref{fig_cos ring spd slice}. The corresponding variances $\sigma_u^2$ and $\sigma_v^2$ are difficult to estimate a priori and the corresponding target values are left at ``?".
\begin{figure}[b]

\begin{subfigure}[b]{0.45\textwidth}
\includegraphics[width=\textwidth,height=3.95cm,keepaspectratio]{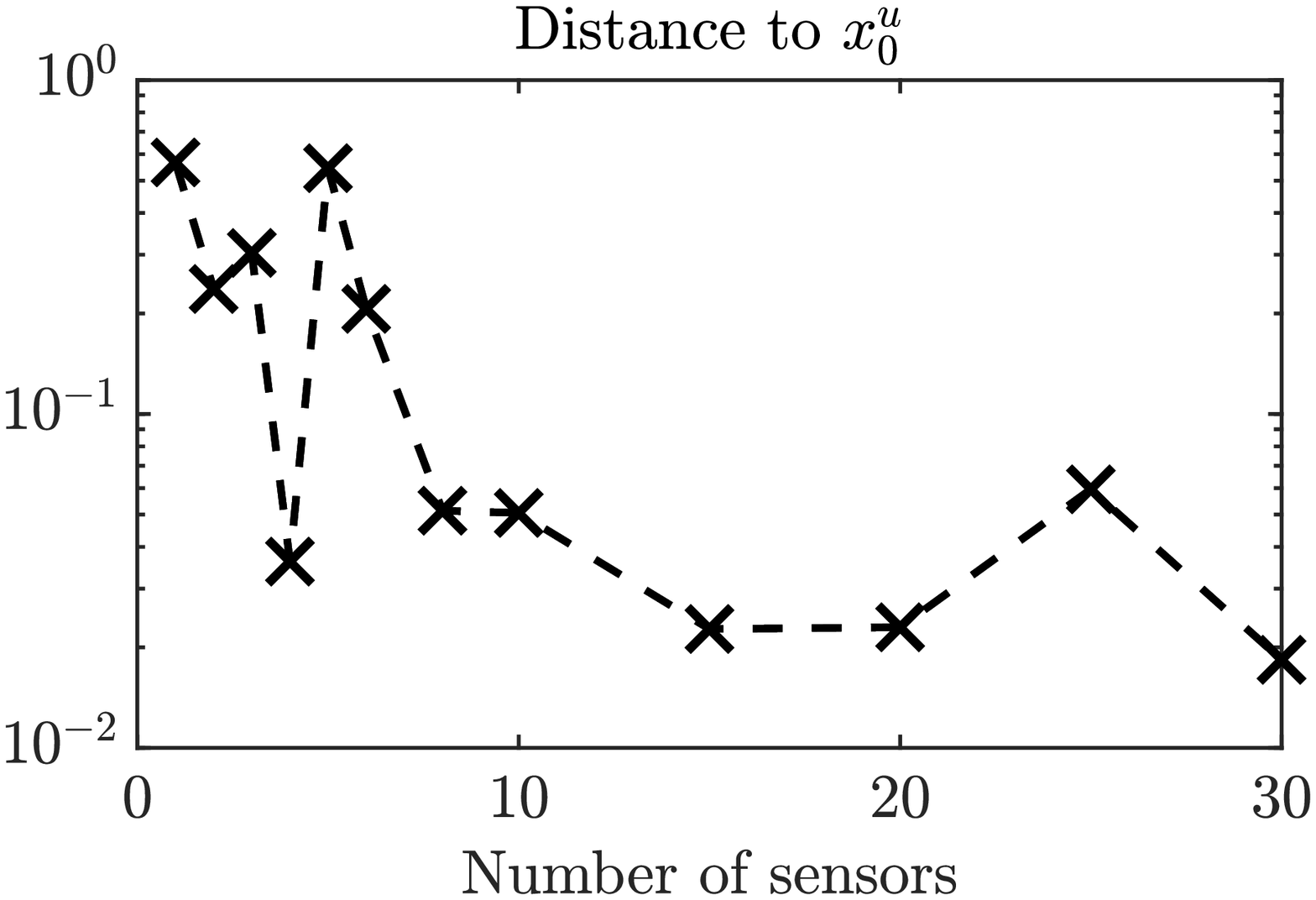}	
\caption{3D distance $|x_0^u-\hat{x}_0^u|$}
\label{fig_mix_x0_u}
\end{subfigure}
\hfill
\begin{subfigure}[b]{0.45\textwidth}
\includegraphics[width=\textwidth,height=3.95cm,keepaspectratio]{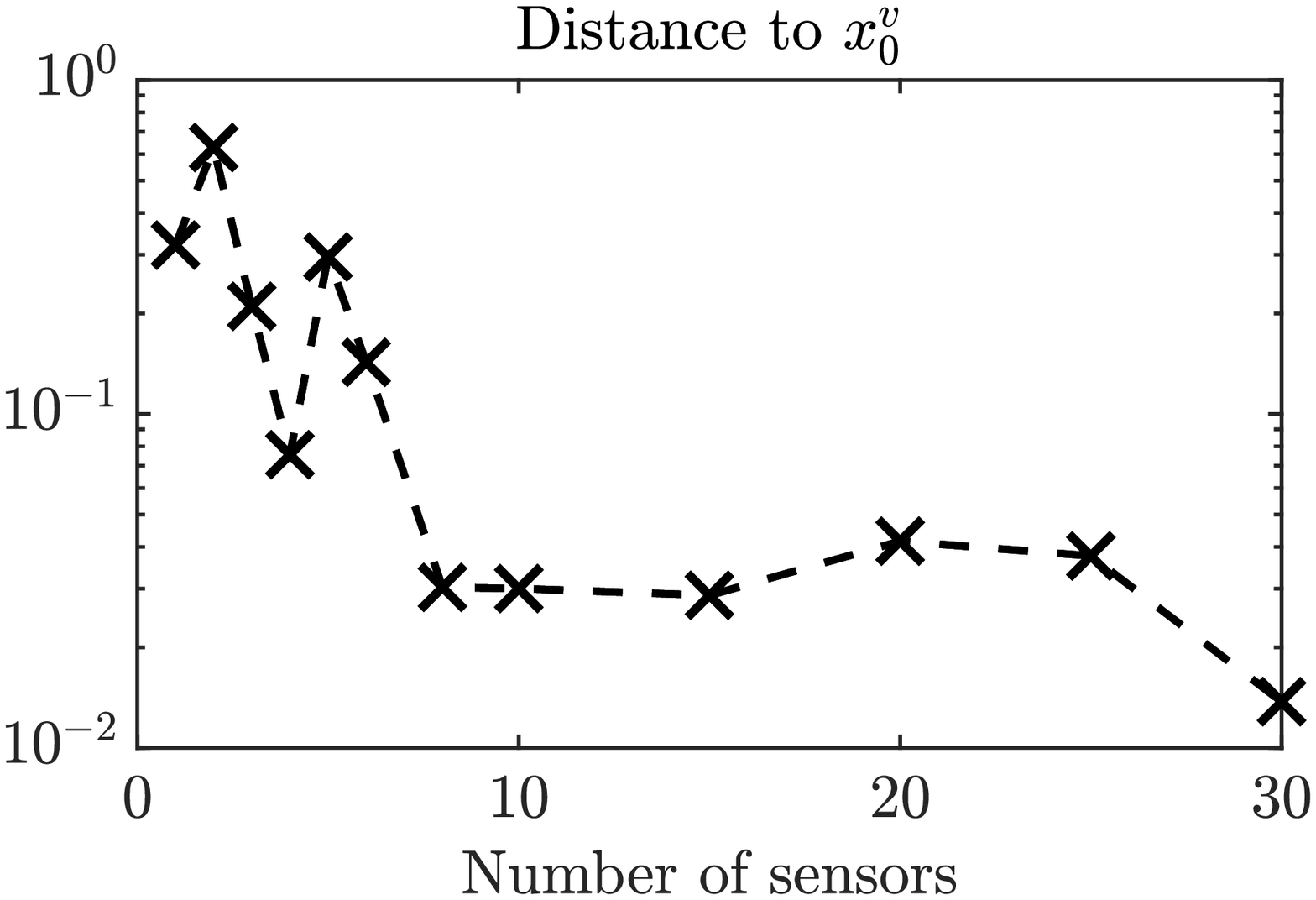}	
\caption{3D distance $|x_0^v-\hat{x}_0^v|$}
\label{fig_mix_x0_v}
\end{subfigure}

\vspace{5mm}

\begin{subfigure}[b]{0.45\textwidth}
\includegraphics[width=\textwidth,height=3.95cm,keepaspectratio]{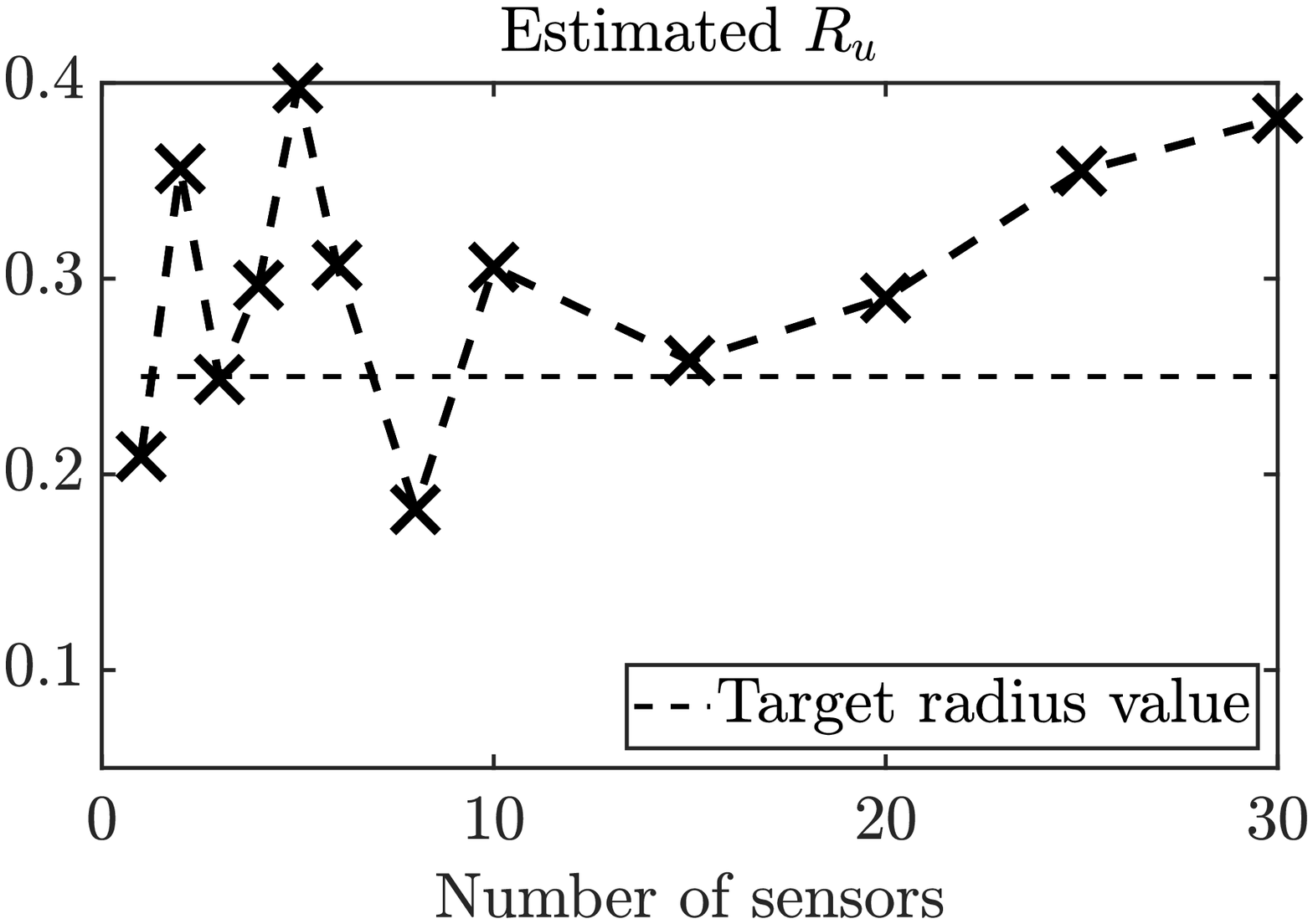}	
\caption{Estimated source radius for $u_0$}
\label{fig_mix_R_u}
\end{subfigure}
\hfill
\begin{subfigure}[b]{0.45\textwidth}
\includegraphics[width=\textwidth,height=3.95cm,keepaspectratio]{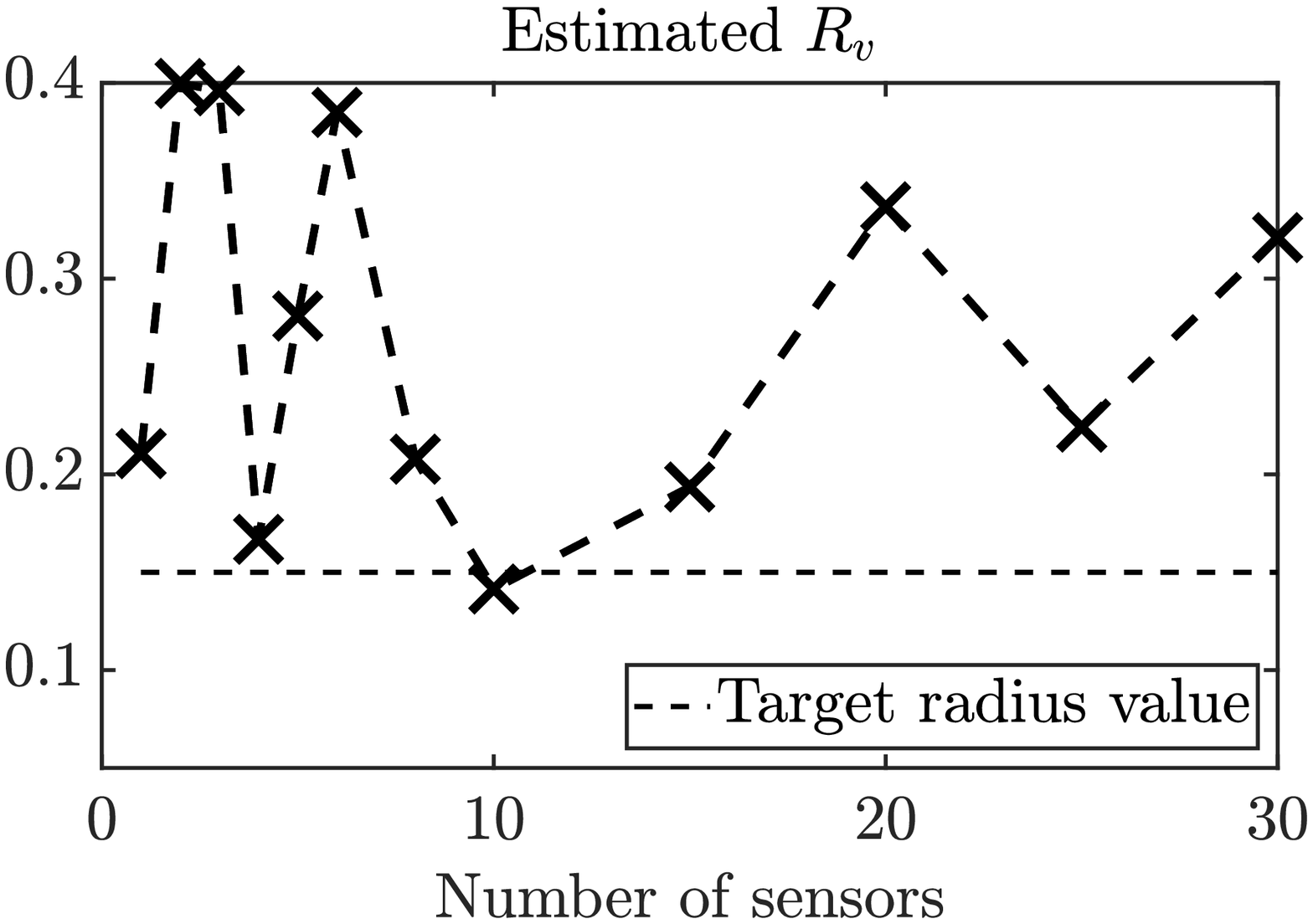}
\caption{Estimated source radius for $v_0$}	
\label{fig_mix_R_v}
\end{subfigure}

\vspace{5mm}

\begin{subfigure}[b]{0.45\textwidth}
\includegraphics[width=\textwidth,height=3.95cm,keepaspectratio]{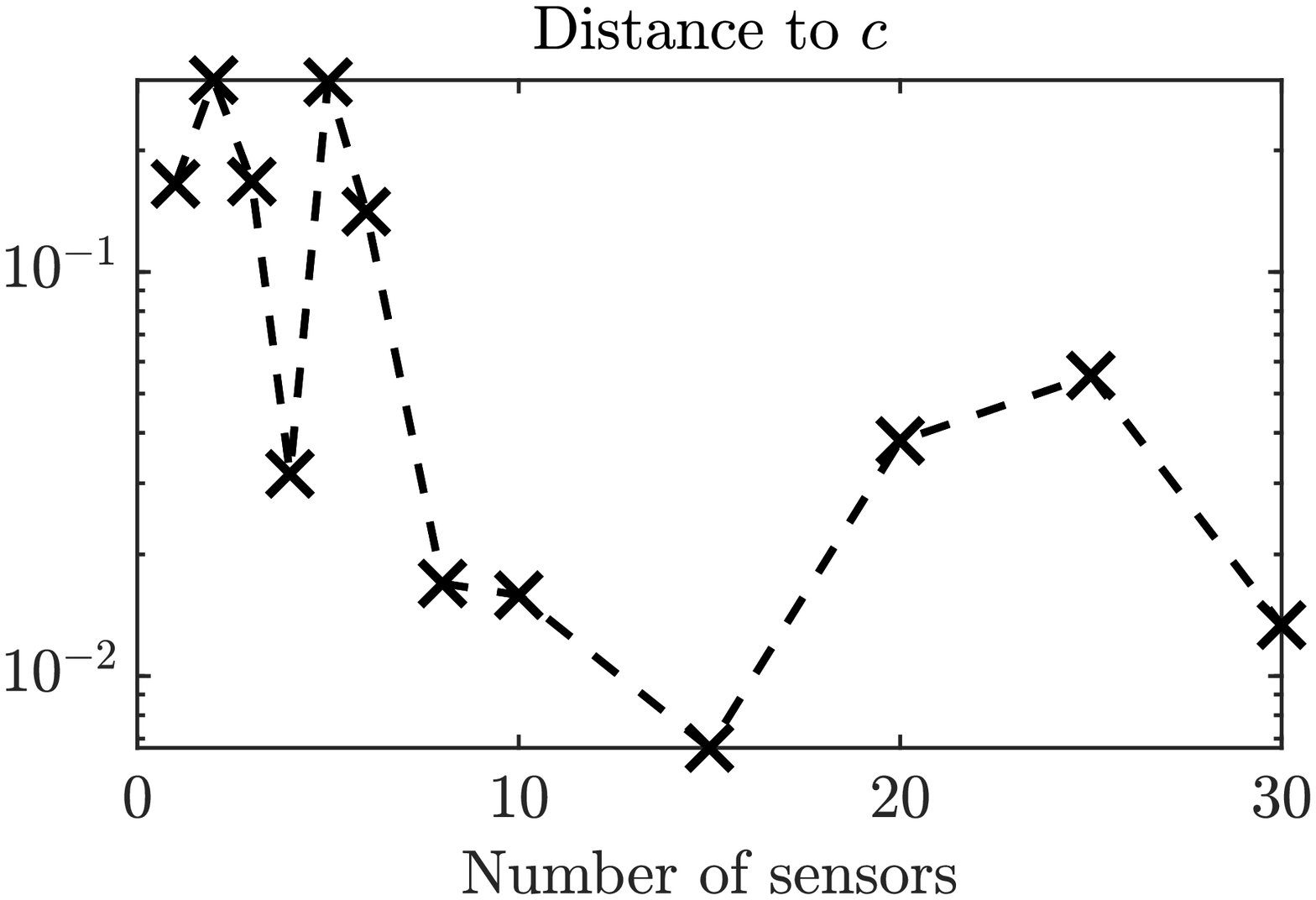}	
\caption{Distance $|c-\hat{c}|$}
\label{fig_mix_c}
\end{subfigure}%

\caption{Hyperparameter estimation for the test case \#3}
\end{figure}

\begin{figure}
\begin{subfigure}[b]{0.45\textwidth}
\includegraphics[width=\textwidth,height=3.95cm,keepaspectratio]{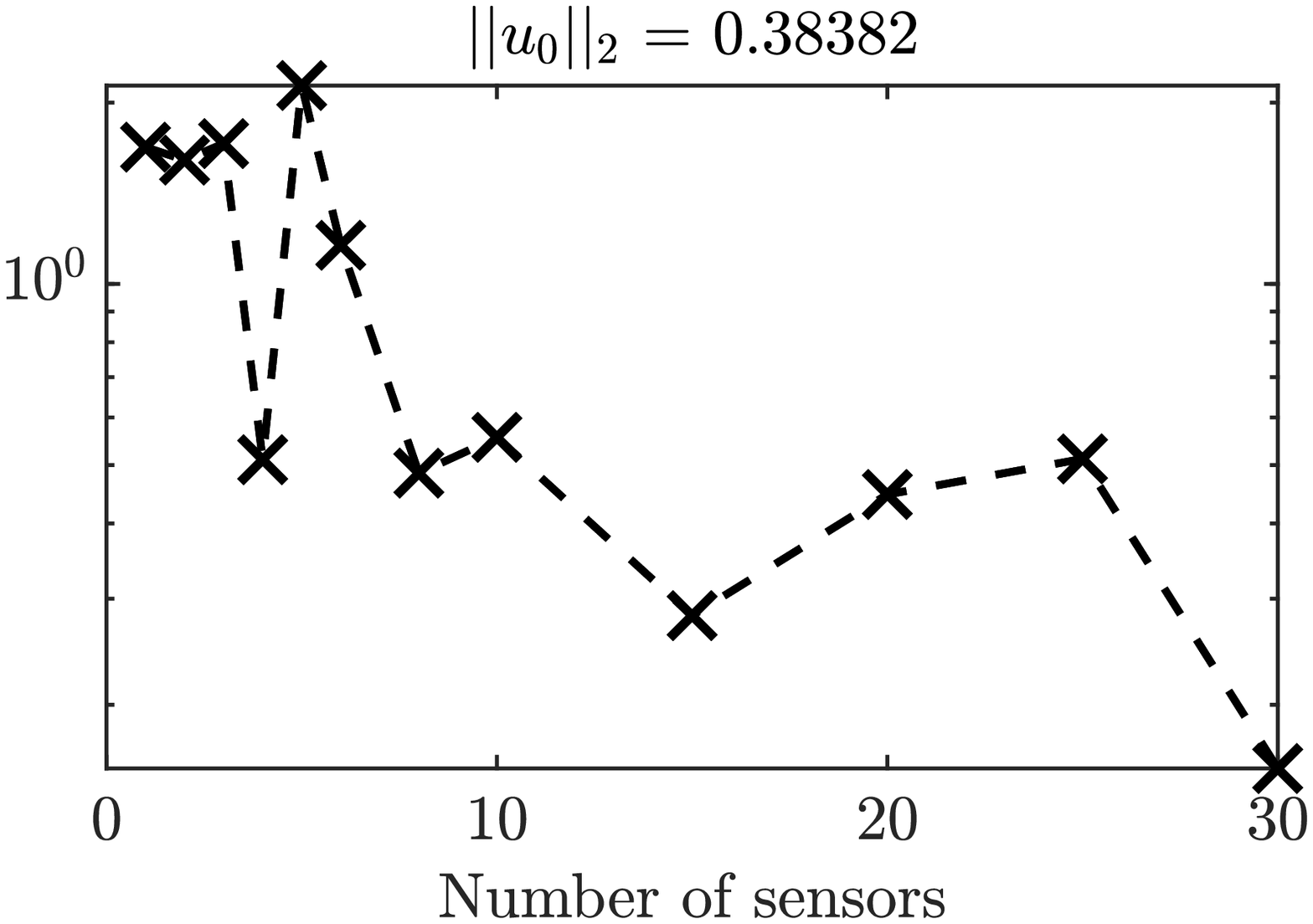}	
\caption{$L^2$ rel. error for $u_0$}
\label{fig_mix_L2_u}
\end{subfigure}
\hfill
\begin{subfigure}[b]{0.45\textwidth}

\includegraphics[width=\textwidth,height=3.95cm,keepaspectratio]{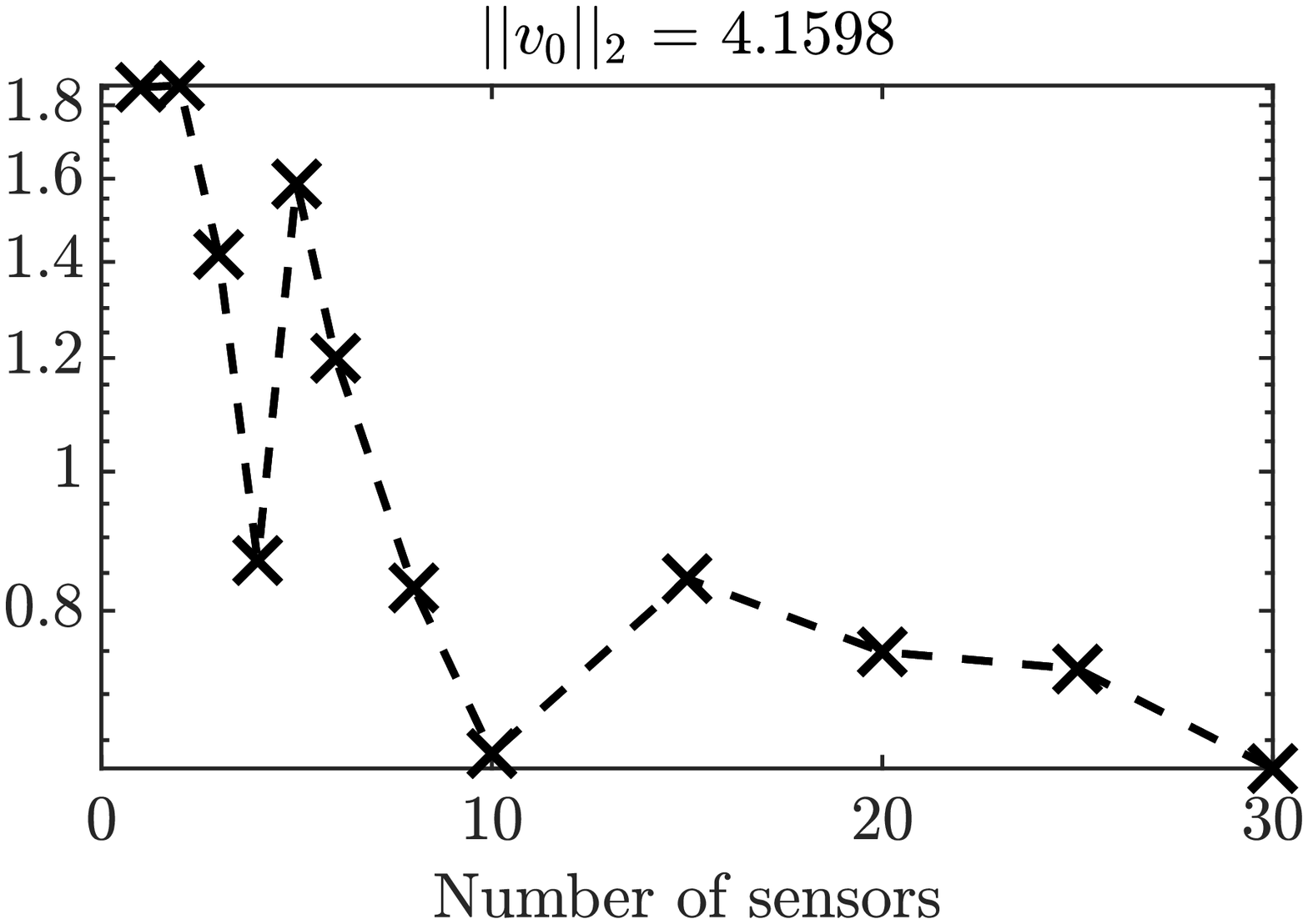}
\caption{$L^2$ rel. error for $v_0$}	
\label{fig_mix_L2_v}
\end{subfigure} %

\vspace{5mm}

\begin{subfigure}[b]{0.45\textwidth}

\includegraphics[width=\textwidth,height=3.95cm,keepaspectratio]{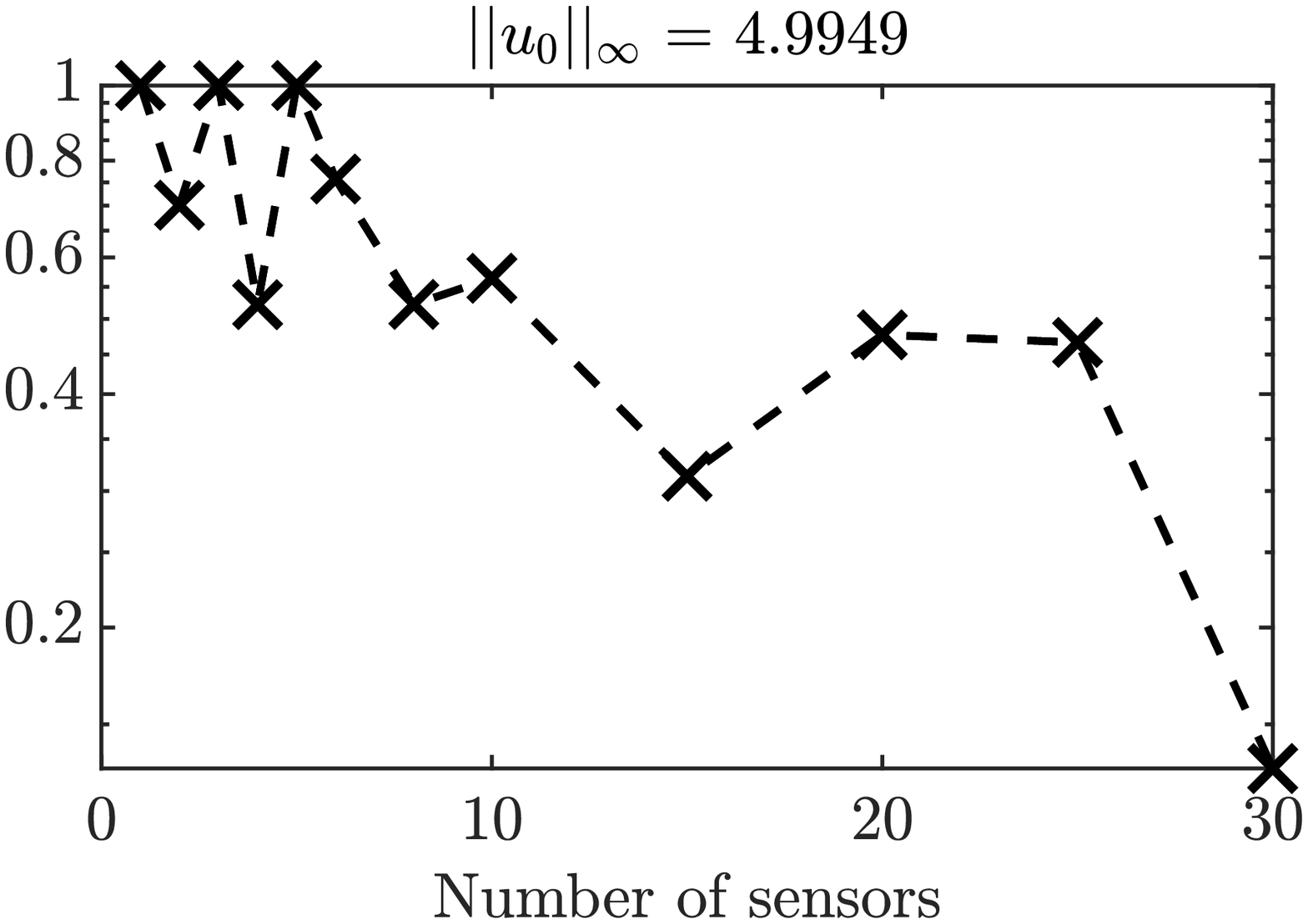}	
\caption{$L^{\infty}$ rel. error for $u_0$}
\label{fig_mix_Linf_u}
\end{subfigure}
\hfill
\begin{subfigure}[b]{0.45\textwidth}
\includegraphics[width=\textwidth,height=3.95cm,keepaspectratio]{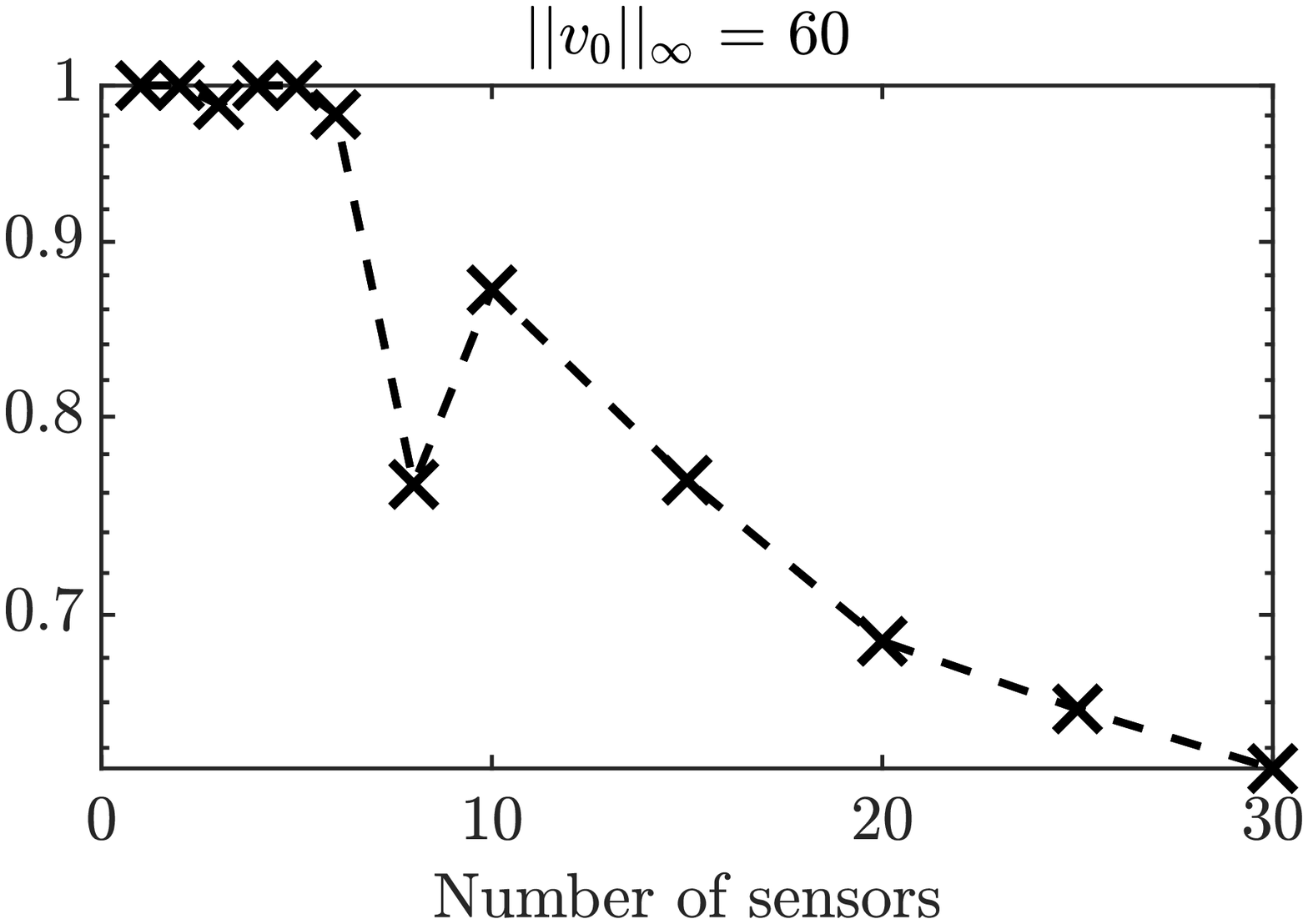}	
\caption{$L^{\infty}$ rel. error for $v_0$}
\label{fig_mix_Linf_v}
\end{subfigure} %

\vspace{5mm}

\begin{subfigure}[b]{0.45\textwidth}

\includegraphics[width=\textwidth,height=3.95cm,keepaspectratio]{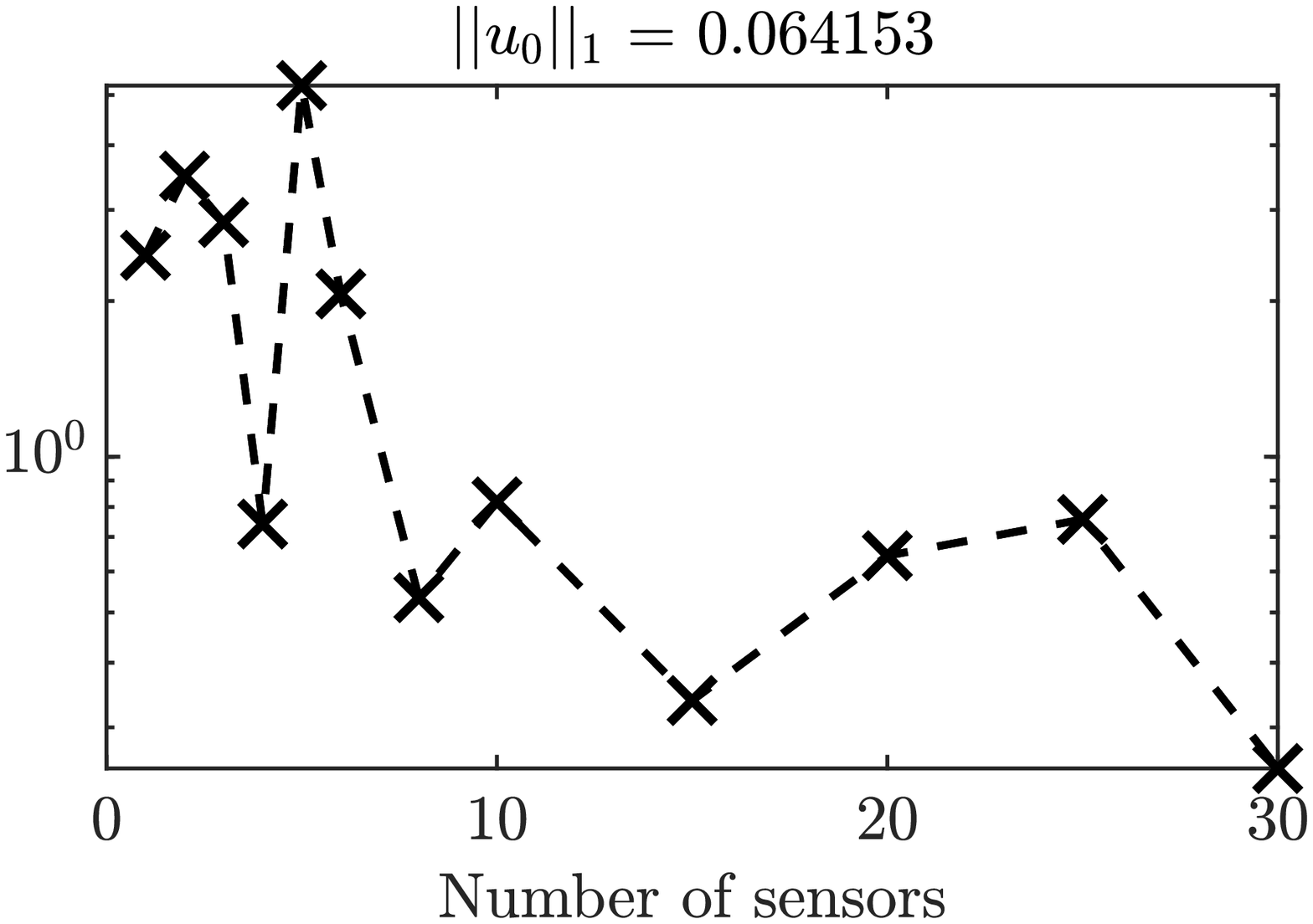}	
\caption{$L^1$ rel. error for $u_0$}
\label{fig_mix_L1_u}
\end{subfigure}
\hfill
\begin{subfigure}[b]{0.45\textwidth}

\includegraphics[width=\textwidth,height=3.95cm,keepaspectratio]{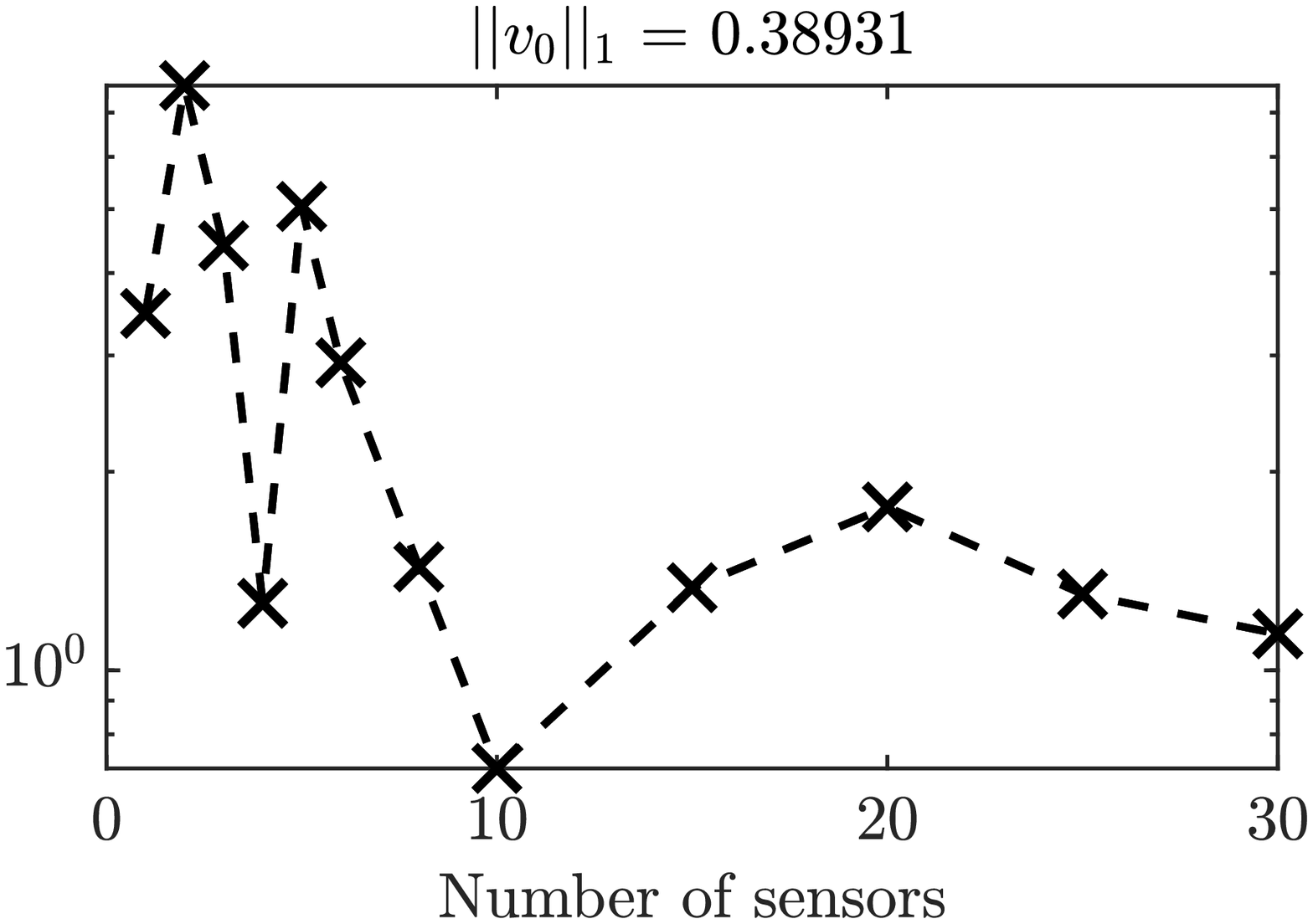}	
\caption{$L^1$ rel. error for $v_0$}
\label{fig_mix_L1_v}
\end{subfigure} %

\caption{Error plots for the test case \#3}
\end{figure}

\clearpage
\section{Experience plan sensibility results}
In all the presented figures are represented the following objects. Each box plot shows the median, the first and the third quartiles of a dataset corresponding to results obtained on the 40 different receiver dispositions. The ``targets" (dot inside a circle) correspond to the median of each boxplot. The dashed line links the mean of each box plot. The black crosses are the mean of each box plot, through which each dashed curve pass.

%\subsection{Test cases \#1 ($k_{\mathrm{u}} ^{\mathrm{wave}}$) and \#2 ($k_{\mathrm{v}} ^{\mathrm{wave}}$)}

\begin{figure}[b]

\begin{subfigure}[b]{0.45\textwidth}
\includegraphics[width=\textwidth,height=3.95cm,keepaspectratio]{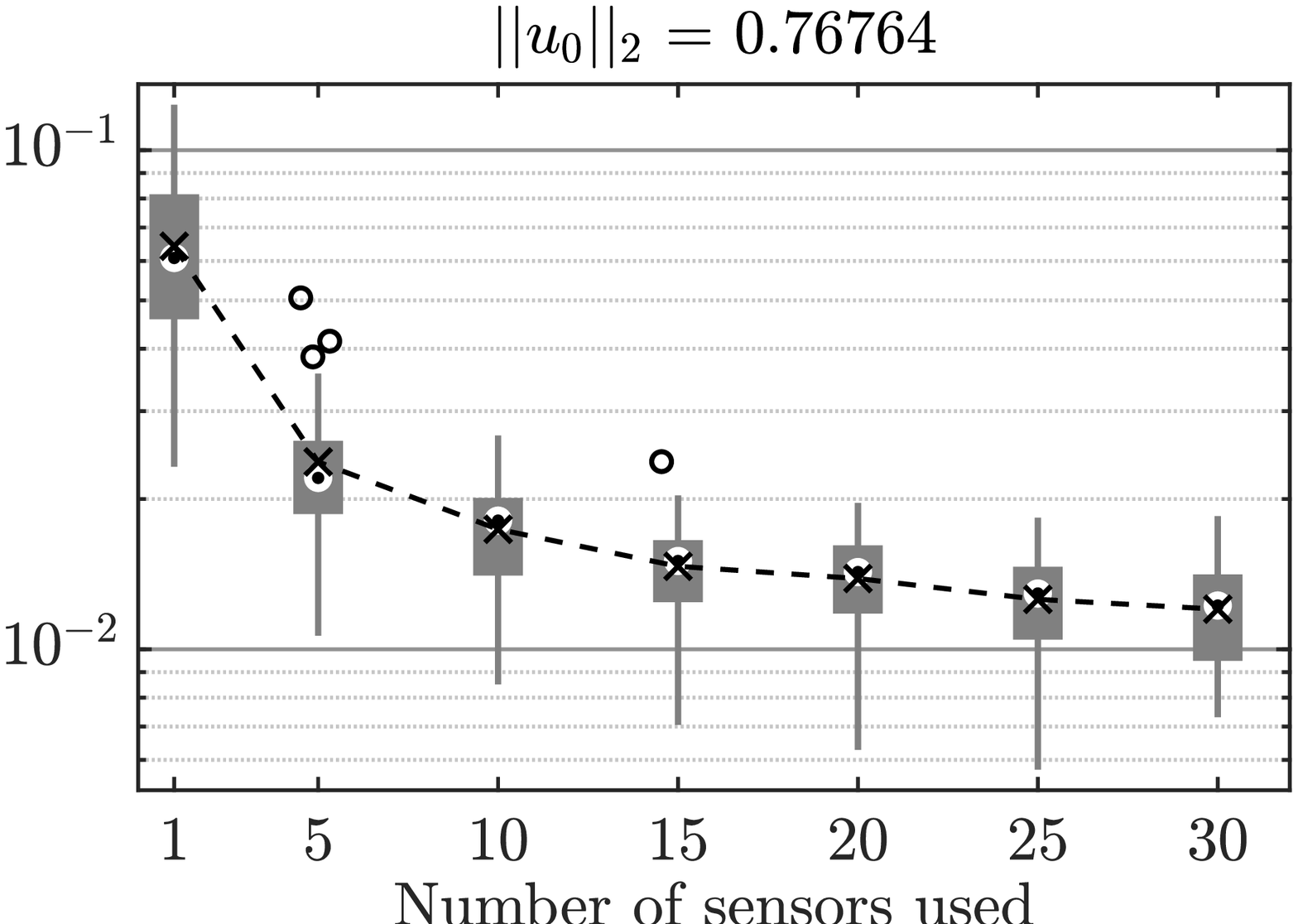}	
\caption{$L^2$ rel. error for $u_0$ (test case \#1)}
\label{fig_sensib_u_u_L2}
\end{subfigure}
\hfill
\begin{subfigure}[b]{0.45\textwidth}
\includegraphics[width=\textwidth,height=3.95cm,keepaspectratio]{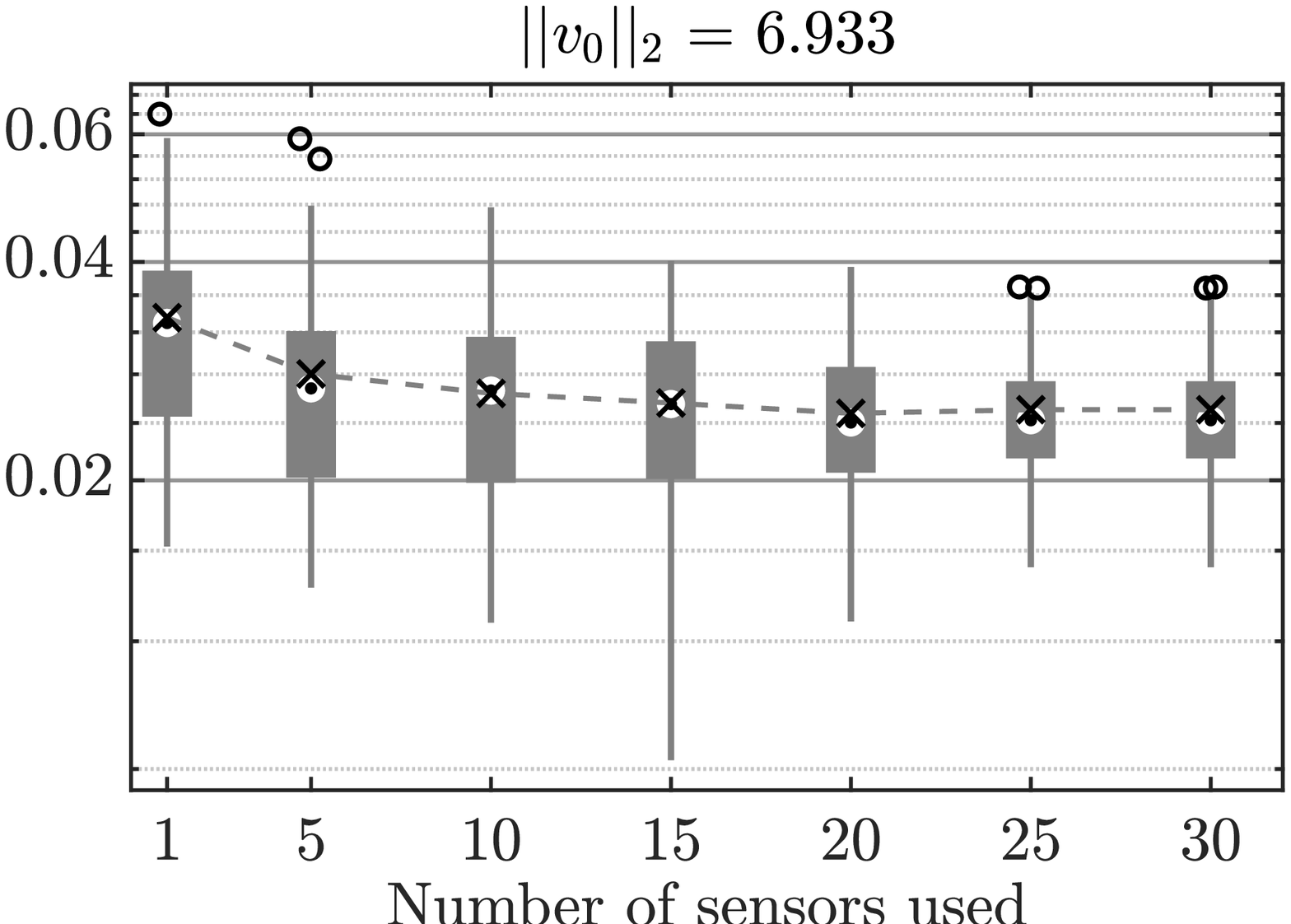}	
\caption{$L^2$ rel. error for $v_0$ (test case \#2)}
\label{fig_sensib_v_v_L2}
\end{subfigure}

\vspace{5mm}

\begin{subfigure}[b]{0.45\textwidth}
\includegraphics[width=\textwidth,height=3.95cm,keepaspectratio]{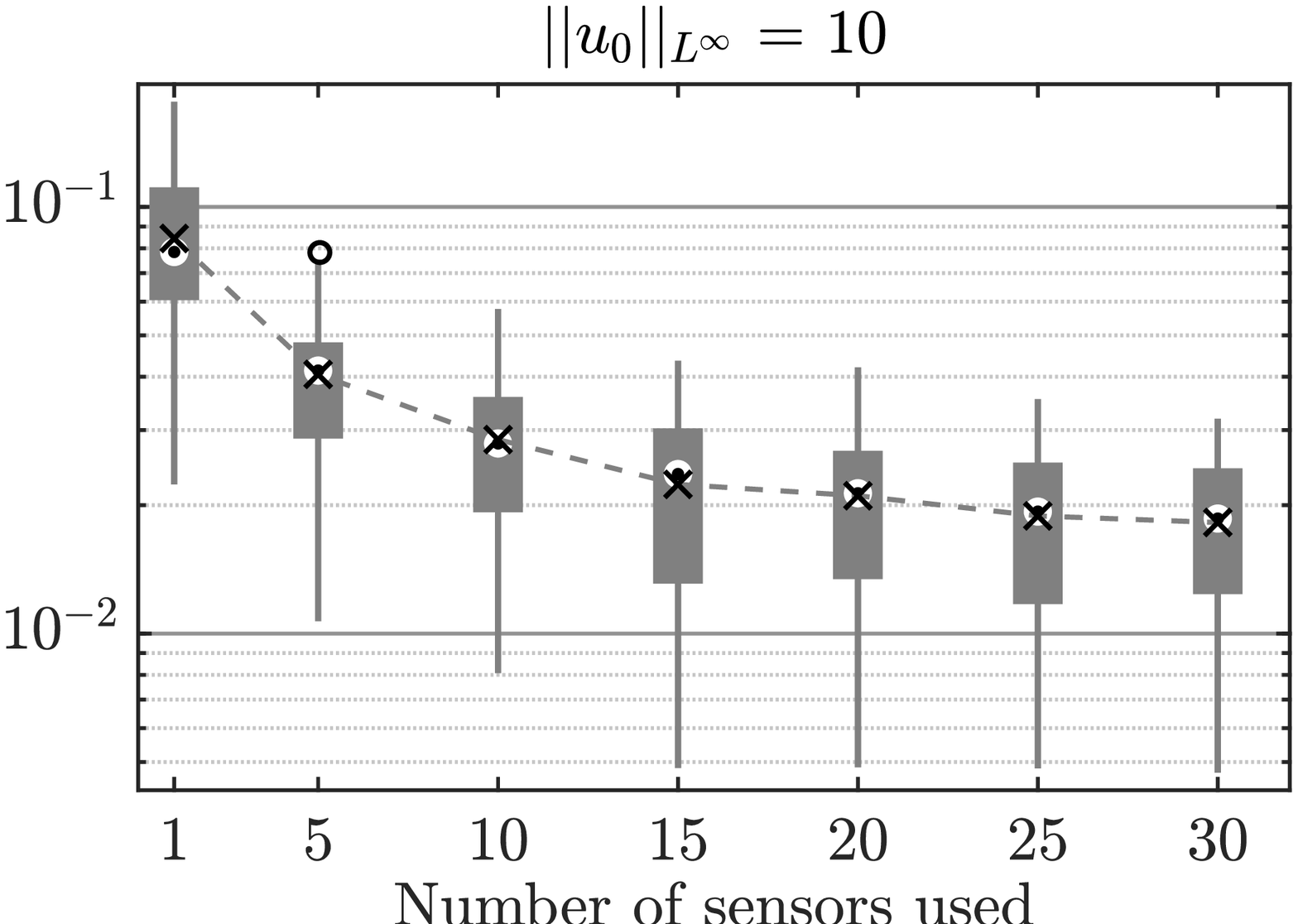}	
\caption{$L^{\infty}$ rel. error for $u_0$ (test case \#1)}
\label{fig_sensib_u_u_Linf}
\end{subfigure}
\hfill
\begin{subfigure}[b]{0.45\textwidth}
\includegraphics[width=\textwidth,height=3.95cm,keepaspectratio]{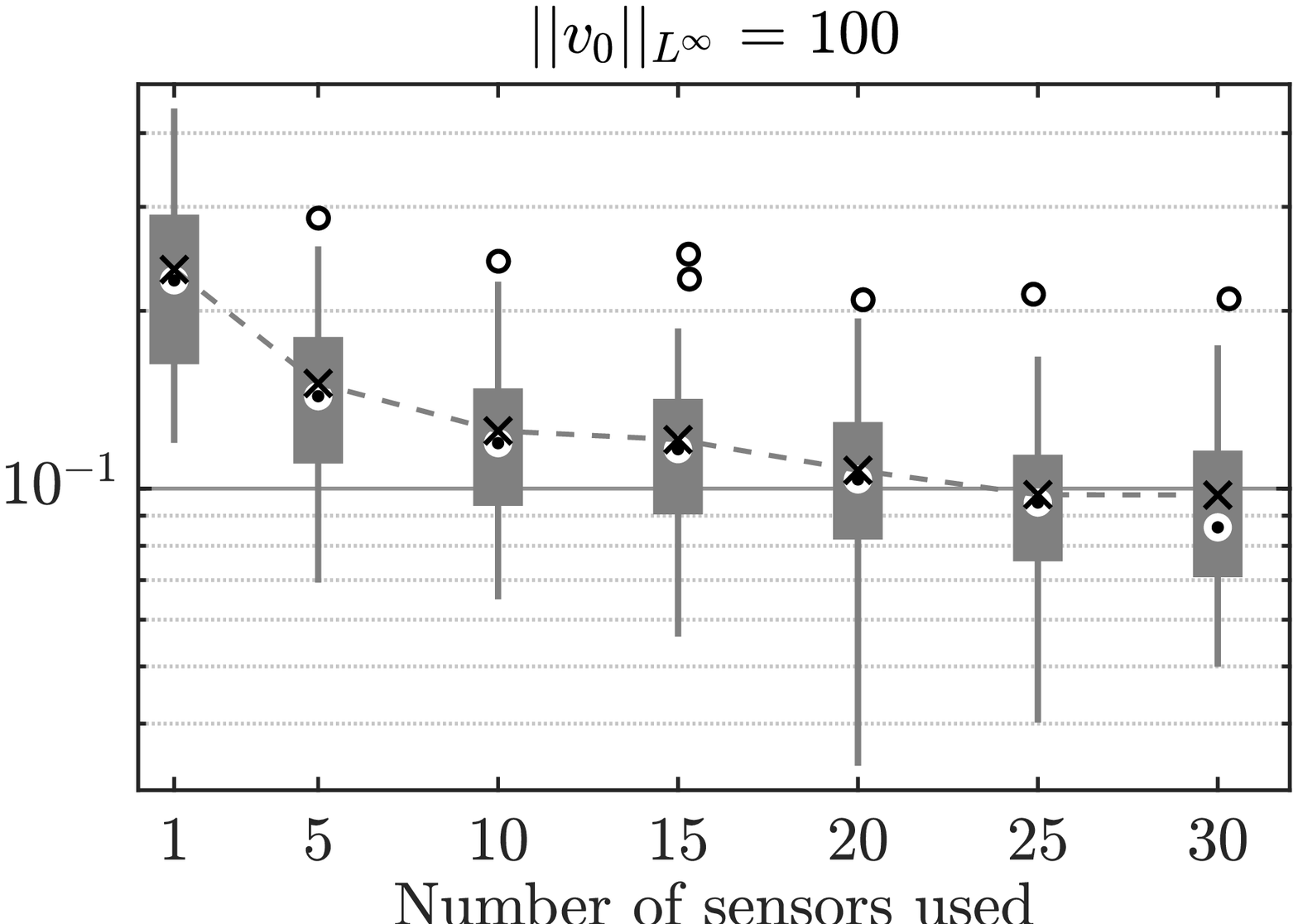}	
\caption{$L^{\infty}$ rel. error for $v_0$ (test case \#2)}
\label{fig_sensib_v_v_Linf}
\end{subfigure}

\vspace{5mm}

\begin{subfigure}[b]{0.45\textwidth}
\includegraphics[width=\textwidth,height=3.95cm,keepaspectratio]{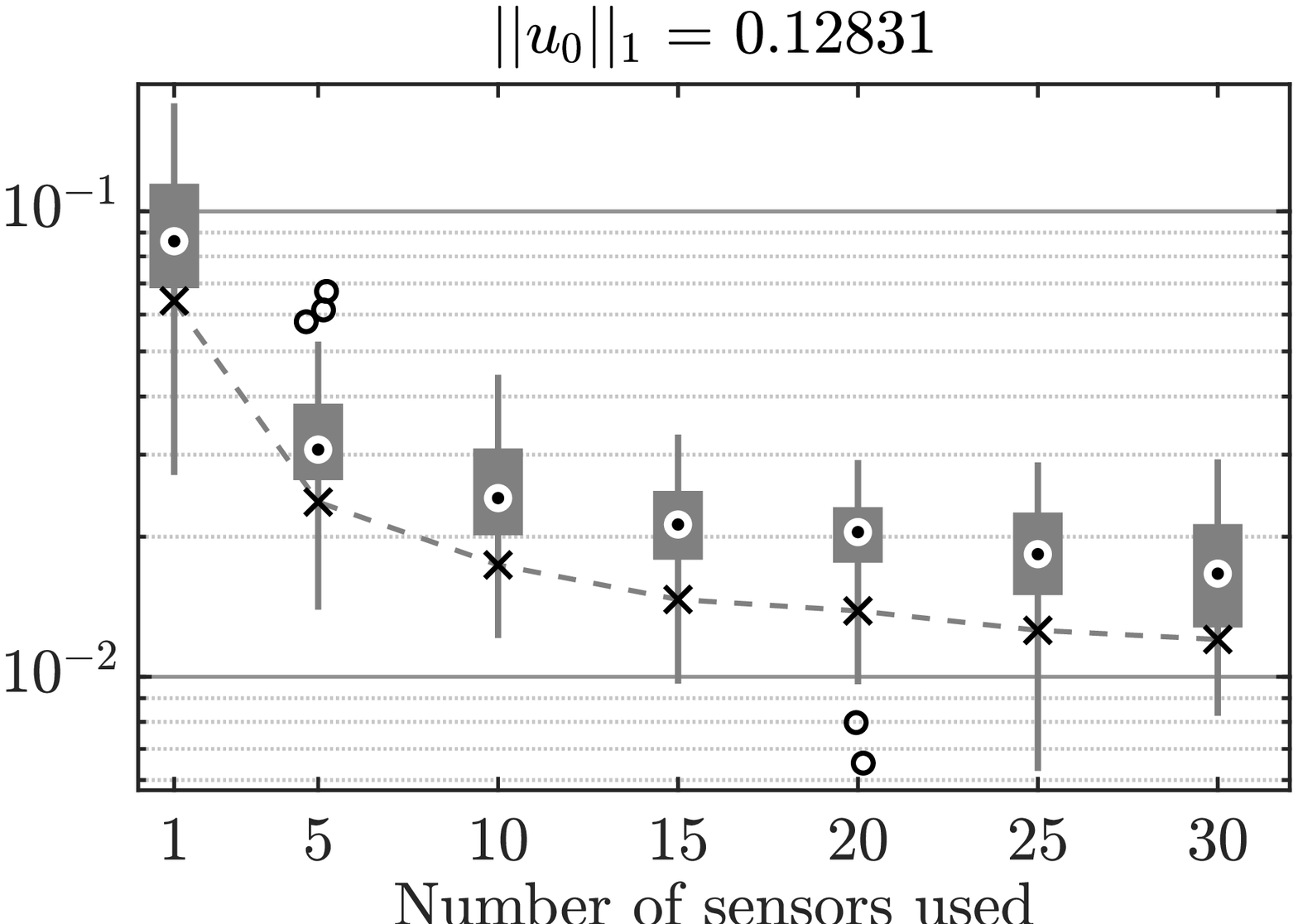}	
\caption{$L^{1}$ rel. error for $u_0$ (test case \#1)}
\label{fig_sensib_u_u_L1}
\end{subfigure}
\hfill
\begin{subfigure}[b]{0.45\textwidth}
\includegraphics[width=\textwidth,height=3.95cm,keepaspectratio]{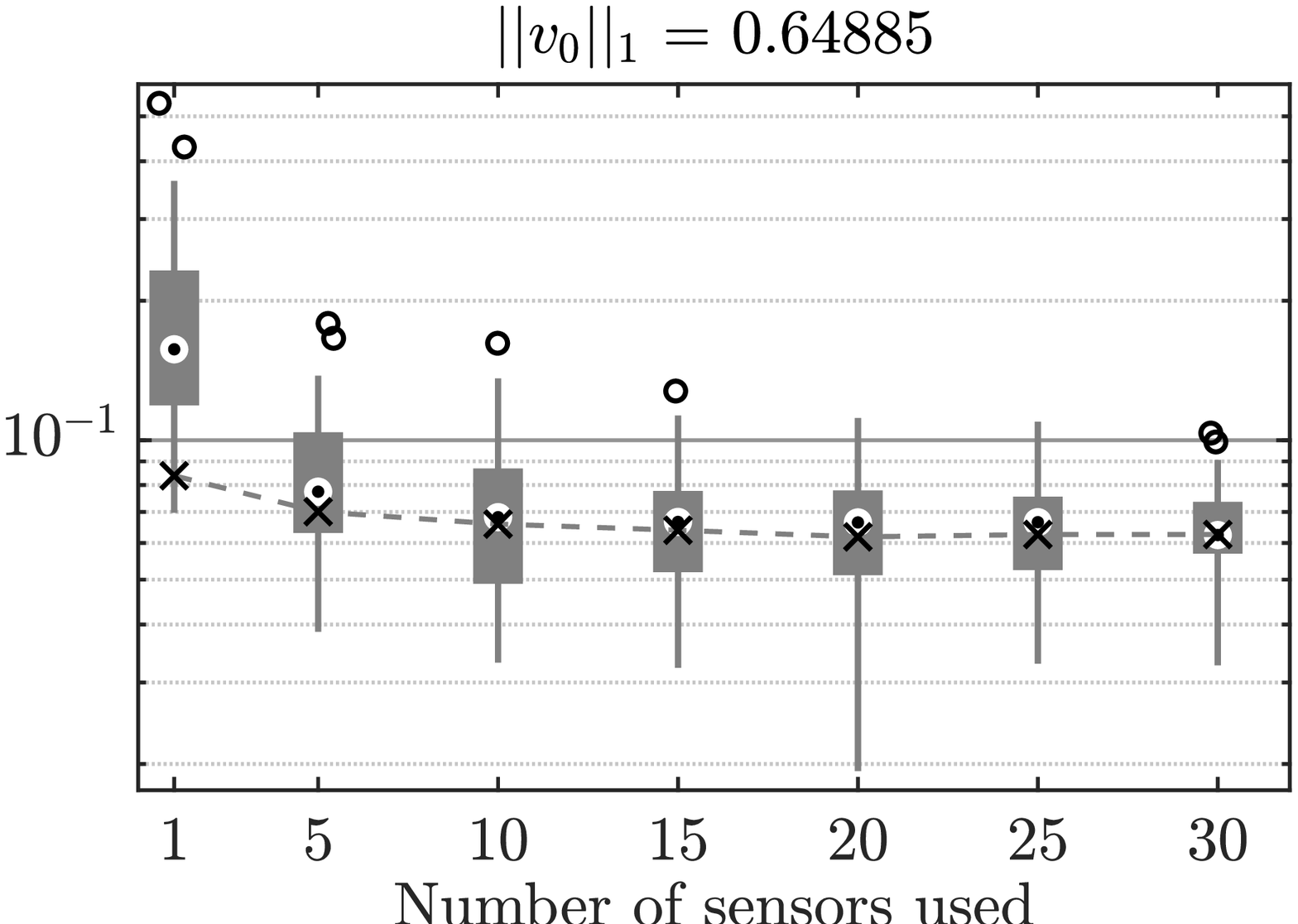}	
\caption{$L^{1}$ rel. error for $v_0$ (test case \#2)}
\label{fig_sensib_v_v_L1}
\end{subfigure}

\caption{Error plots for the test cases \#1 and \#2}
\label{sensib_1_and_2}
\end{figure}

%\subsection{Test case \#3 : $k_{\mathrm{u}} ^{\mathrm{wave}} + k_{\mathrm{v}} ^{\mathrm{wave}}$}

\begin{figure}

\begin{subfigure}[b]{0.45\textwidth}
\includegraphics[width=\textwidth,height=3.95cm,keepaspectratio]{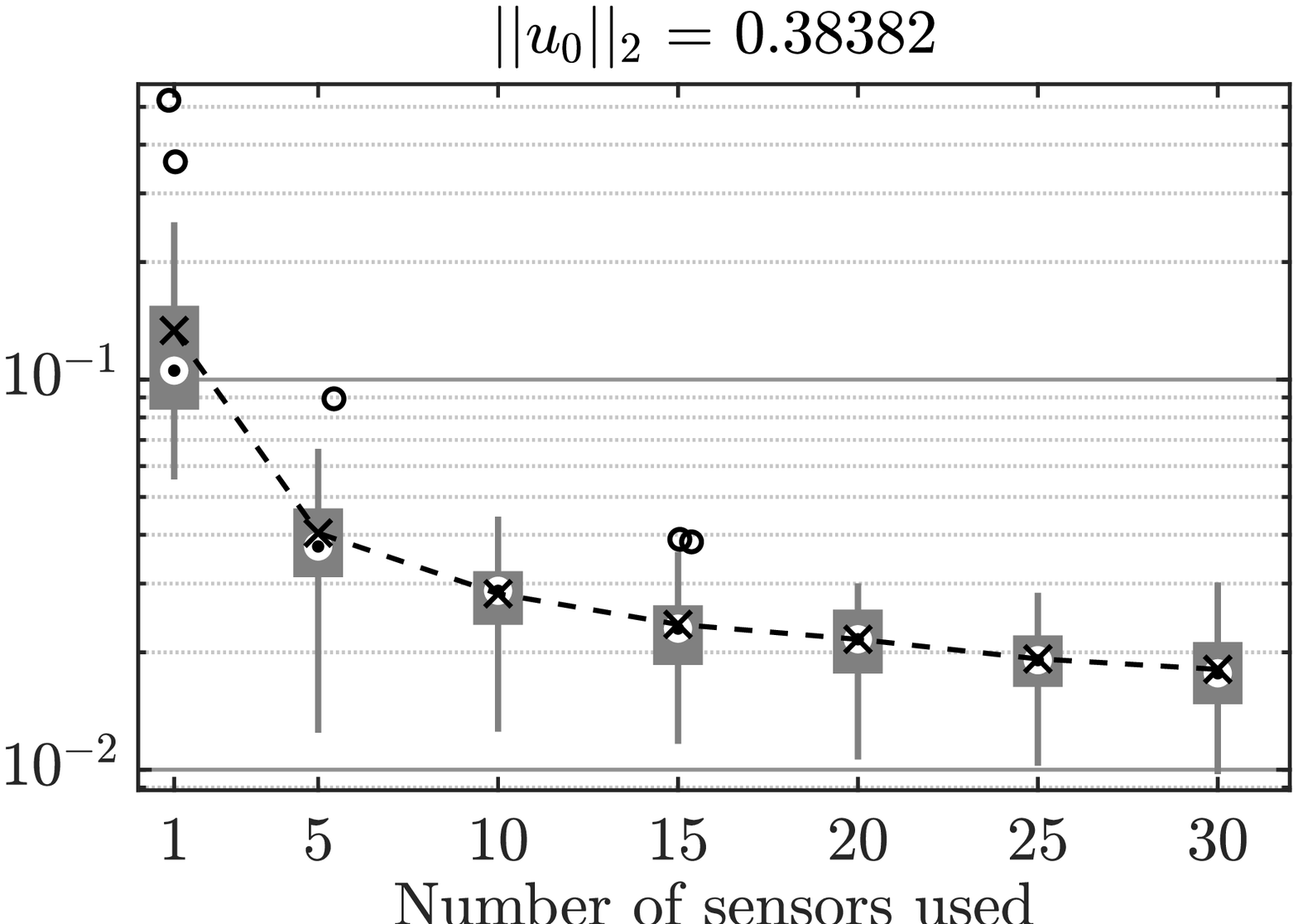}	
\caption{$L^2$ rel. error for $u_0$ (test case \#3)}
\label{fig_sensib_mix_L2_u}
\end{subfigure}
\hfill
\begin{subfigure}[b]{0.45\textwidth}
\includegraphics[width=\textwidth,height=3.95cm,keepaspectratio]{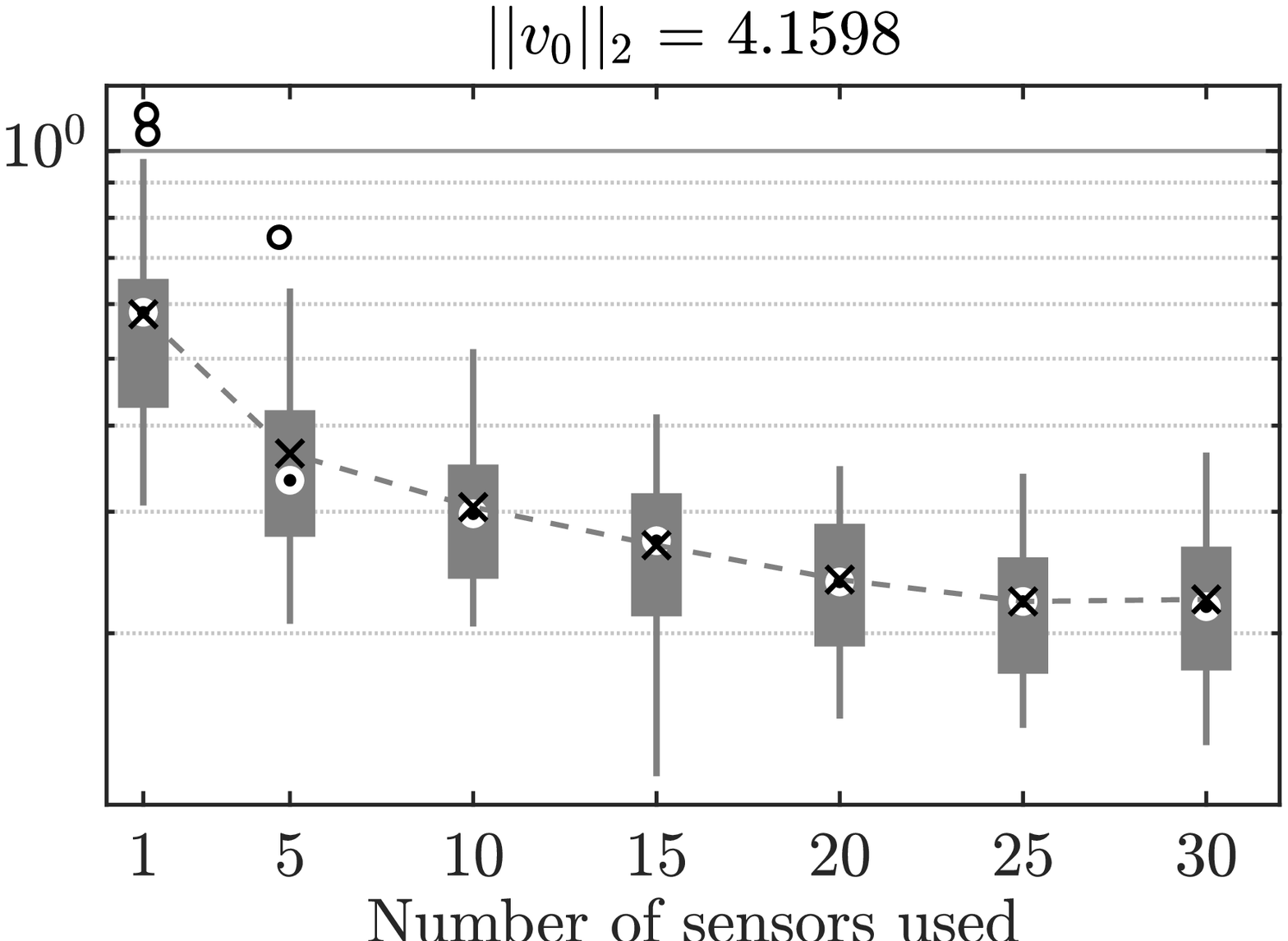}
\caption{$L^2$ rel. error for $v_0$ (test case \#3)}	
\label{fig_sensib_mix_L2_v}
\end{subfigure}\\
\vspace{5mm}
\begin{subfigure}[b]{0.45\textwidth}
\includegraphics[width=\textwidth,height=3.95cm,keepaspectratio]{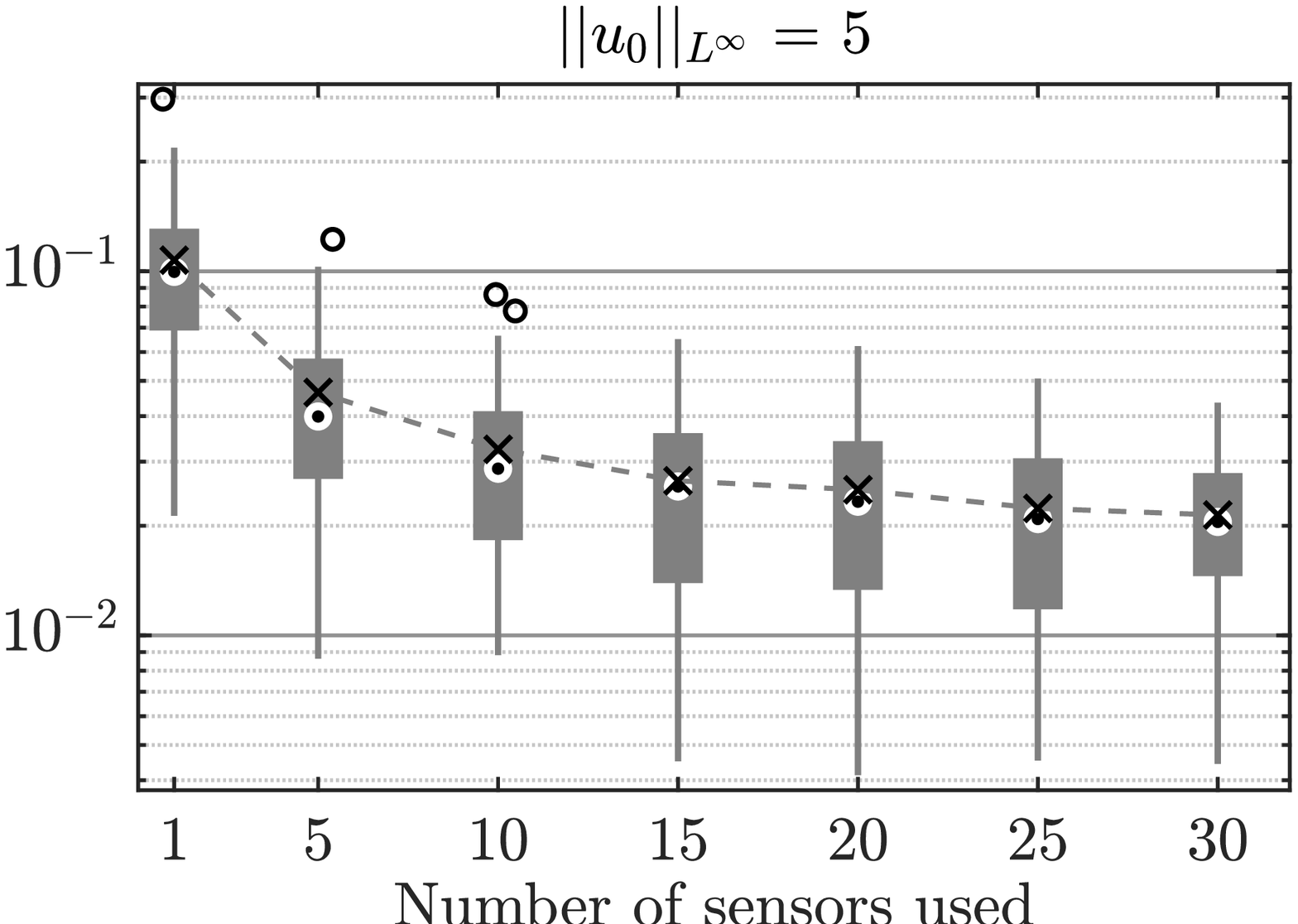}
\caption{$L^{\infty}$ rel. error for $u_0$ (test case \#3)}	
\label{fig_sensib_mix_Linf_u}
\end{subfigure} %
\hfill
\begin{subfigure}[b]{0.45\textwidth}
\includegraphics[width=\textwidth,height=3.95cm,keepaspectratio]{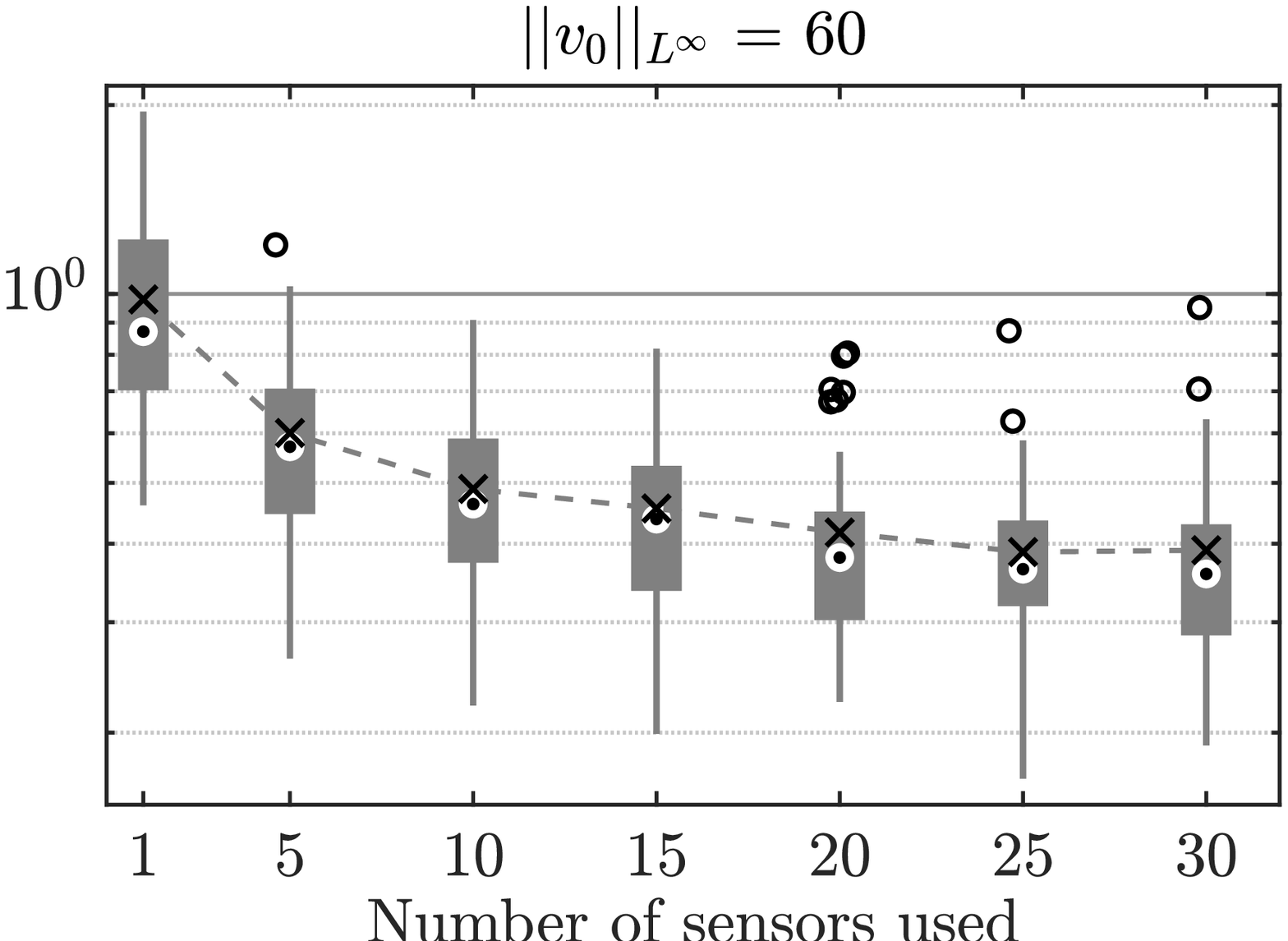}	
\caption{$L^{\infty}$ rel. error for $v_0$ (test case \#3)}	
\label{fig_sensib_mix_Linf_v}
\end{subfigure} %

\vspace{5mm}

\begin{subfigure}[b]{0.45\textwidth}
\includegraphics[width=\textwidth,height=3.95cm,keepaspectratio]{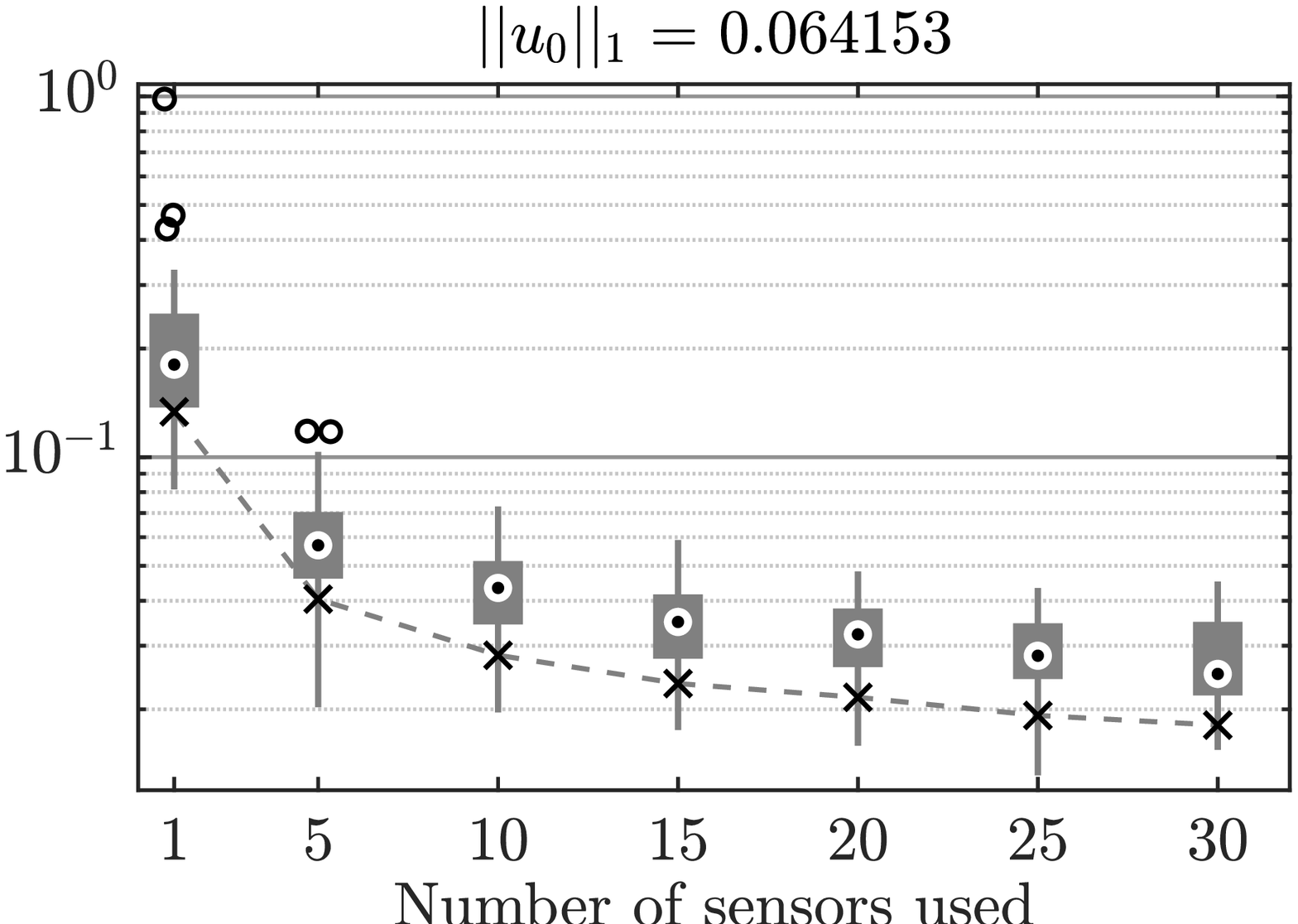}	
\caption{$L^1$ rel. error for $u_0$ (test case \#3)}
\label{fig_sensib_mix_L1_u}
\end{subfigure}
\hfill
\begin{subfigure}[b]{0.45\textwidth}
\includegraphics[width=\textwidth,height=3.95cm,keepaspectratio]{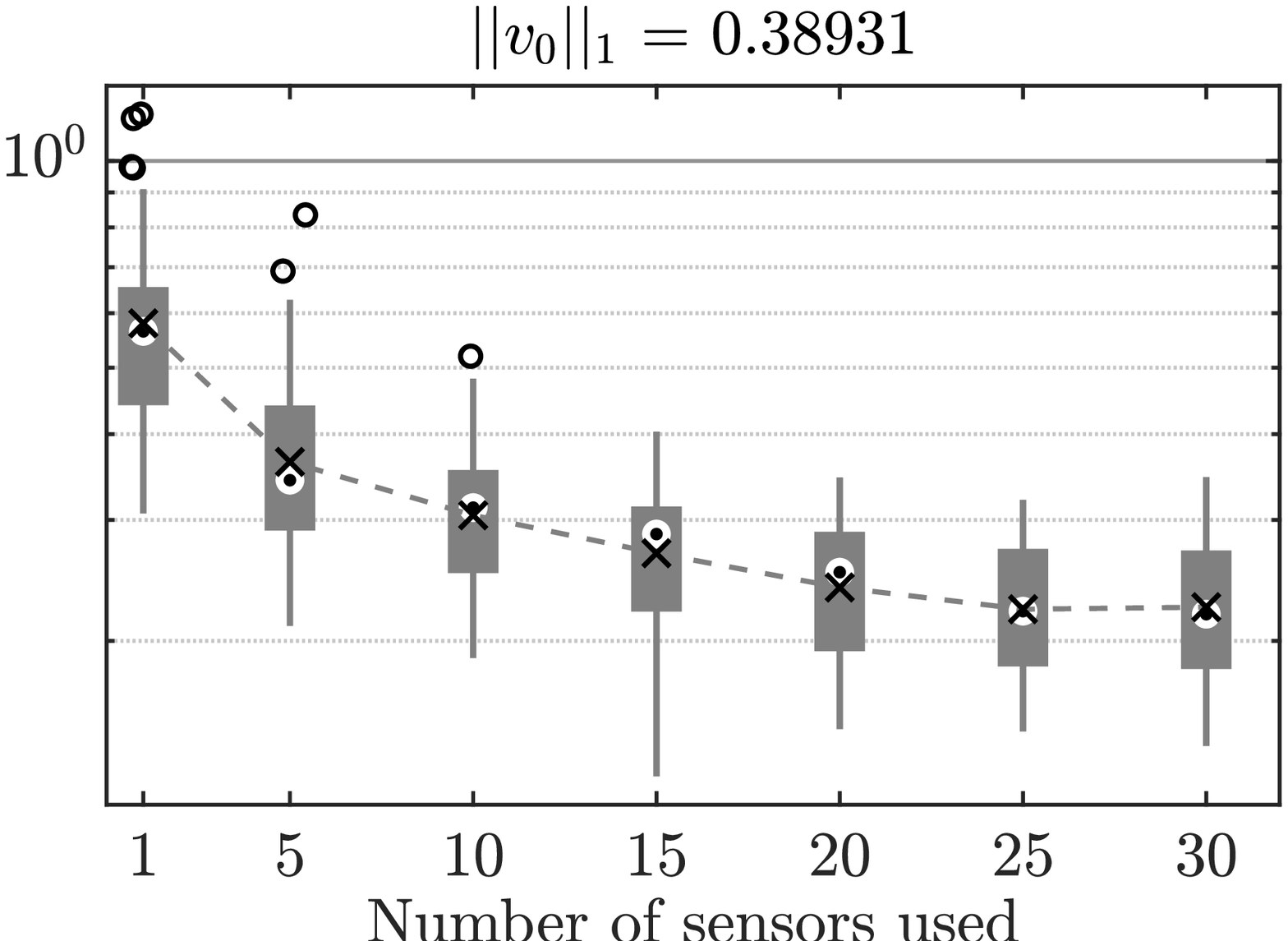}	
\caption{$L^1$ rel. error for $v_0$ (test case \#3)}
\label{fig_sensib_mix_L1_v}
\end{subfigure} %

\caption{Error plots for the test case \#3}
\label{sensib_3}
\end{figure}

\clearpage

%% file: main.bbl
\providecommand{\bysame}{\leavevmode\hbox to3em{\hrulefill}\thinspace}
\providecommand{\MR}{\relax\ifhmode\unskip\space\fi MR }
\providecommand{\MRhref}[2]{%
  \href{http://www.ams.org/mathscinet-getitem?mr=#1}{#2}
}
\providecommand{\href}[2]{#2}
\begin{thebibliography}{10}

\bibitem{adler2007}
R.~J. Adler and J.~E. Taylor, \emph{Random fields and geometry},
  Springer-Verlag New York, NY, 2007.

\bibitem{albert2020}
C.~G. Albert and K.~Rath, \emph{{G}aussian process regression for data
  fulfilling linear differential equations with localized sources}, Entropy
  \textbf{22} (2020), no.~2.

\bibitem{alvarado2014}
P.~A. Alvarado, M.~A. Alvarez, G.~Daza-Santacoloma, A.~Orozco, and
  G.~Castellanos-Dominguez, \emph{A latent force model for describing electric
  propagation in deep brain stimulation: A simulation study}, 2014 36th Annual
  International Conference of the IEEE Engineering in Medicine and Biology
  Society, 2014, pp.~2617--2620.

\bibitem{alvarez2013}
M.~{\'A}lvarez, D.~Luengo, and N.~Lawrence, \emph{Linear latent force models
  using {G}aussian processes}, IEEE Transactions on Pattern Analysis and
  Machine Intelligence \textbf{35} (2013), 2693--2705.

\bibitem{Ames1977NumericalMF}
W.~F. Ames, \emph{Numerical methods for partial differential equations (second
  edition)}, Academic Press, second edition ed., 1977.

\bibitem{ammari2012}
Habib Ammari (ed.), \emph{Mathematical modeling in biomedical imaging {II}.
  optical, ultrasound, and opto-acoustic tomographies}, Lecture Notes in
  Mathematics, vol. 2035, Springer, Berlin, Heidelberg, 2012.

\bibitem{Arkhipova2008NumberOL}
L.~Arkhipova, \emph{Number of lattice points in a sphere}, Moscow University
  Mathematics Bulletin \textbf{63} (2008), 214--215.

\bibitem{azais_level_2009}
Jean-Marc Aza{\"i}s and Mario Wschebor, \emph{{Level sets and extrema of random
  processes and fields}}, {Wiley \& Sons}, 2009.

\bibitem{agnan2004}
A.~Berlinet and C.~Thomas-Agnan, \emph{Reproducing kernel {H}ilbert spaces in
  probability and statistics}, Springer US, 2004.

\bibitem{bilbao2004}
S.~Bilbao, \emph{Wave and scattering methods for numerical simulations}, John
  Wiley \& Sons, Ltd, 2004.

\bibitem{ref_PML}
Jean-Pierre Bérenger, \emph{Perfectly matched layer ({PML}) for computational
  electromagnetics}, Synthesis Lectures on Computational Electromagnetics
  \textbf{2} (2007), no.~1, 1--117.

\bibitem{buhmann_2003}
Martin~D. Buhmann, \emph{Radial basis functions: Theory and implementations},
  Cambridge Monographs on Applied and Computational Mathematics, Cambridge
  University Press, 2003.

\bibitem{CHEN_owhadi_2021}
Yifan Chen, Bamdad Hosseini, Houman Owhadi, and Andrew~M. Stuart, \emph{Solving
  and learning nonlinear {PDE}s with {G}aussian processes}, Journal of
  Computational Physics \textbf{447} (2021), 110668.

\bibitem{cialenco2012approximation}
Igor Cialenco, Gregory~E Fasshauer, and Qi~Ye, \emph{Approximation of
  stochastic partial differential equations by a kernel-based collocation
  method}, International Journal of Computer Mathematics \textbf{89} (2012),
  no.~18, 2543--2561.

\bibitem{cockayne2017probabilistic}
J.~Cockayne, C.~Oates, T.~Sullivan, and M.~Girolami, \emph{Probabilistic
  numerical methods for partial differential equations and bayesian inverse
  problems}, ArXiv (2017).

\bibitem{cockayne_aip}
Jon Cockayne, Chris Oates, Tim Sullivan, and Mark Girolami, \emph{Probabilistic
  numerical methods for pde-constrained bayesian inverse problems}, AIP
  Conference Proceedings \textbf{1853} (2017), no.~1, 060001.

\bibitem{conrad2017statistical}
Patrick~R Conrad, Mark Girolami, Simo S{\"a}rkk{\"a}, Andrew Stuart, and
  Konstantinos Zygalakis, \emph{Statistical analysis of differential equations:
  introducing probability measures on numerical solutions}, Statistics and
  Computing \textbf{27} (2017), no.~4, 1065--1082.

\bibitem{dashti_stuart_pde}
S.~L. Cotter, M.~Dashti, and A.~M. Stuart, \emph{Approximation of {B}ayesian
  inverse problems for {PDE}s}, SIAM Journal on Numerical Analysis \textbf{48}
  (2010), no.~1, 322--345.

\bibitem{Dalang2009wave_3D}
Robert~C Dalang and Marta Sanz-Sol\'e Sol, \emph{{H}ölder-{S}obolev regularity
  of the solution to the stochastic wave equation in dimension three}, vol.
  199, Memoirs of the {A}merican Mathematical Society, no. 931, American
  Mathematical Society, 2009.

\bibitem{Dashti2017}
Masoumeh Dashti and Andrew~M. Stuart, \emph{The {B}ayesian approach to inverse
  problems}, Handbook of Uncertainty Quantification (Cham) (Roger Ghanem, David
  Higdon, and Houman Owhadi, eds.), Springer International Publishing, 2017,
  pp.~311--428.

\bibitem{citekergp}
Y.~Deville, D.~Ginsbourger, and O.~Roustant. Contributors:~Nicolas Durrande.,
  \emph{kergp: {G}aussian process laboratory}, 2021, R package version 0.5.5.

\bibitem{dudley2002}
R.~M. Dudley, \emph{Real analysis and probability}, 2 ed., Cambridge Studies in
  Advanced Mathematics, Cambridge University Press, 2002.

\bibitem{Duffy2015GreensFW}
Dean~G Duffy, \emph{{G}reen's functions with applications}, second edition ed.,
  Chapman and Hall/CRC, 2015.

\bibitem{Engquist1977AbsorbingBC}
Bjorn Engquist and Andrew Majda, \emph{Absorbing boundary conditions for the
  numerical simulation of waves}, Mathematics of Computation \textbf{31}
  (1977), no.~139, 629--651.

\bibitem{ESMAILZADEH201556}
S.~Esmailzadeh, Z.~Medina-Cetina, J.W. Kang, and L.F. Kallivokas, \emph{Varying
  dimensional {B}ayesian acoustic waveform inversion for 1d semi-infinite
  heterogeneous media}, Probabilistic Engineering Mechanics \textbf{39} (2015),
  56--68.

\bibitem{evans2018measure}
Lawrence~C Evans and Ronald~F Garzepy, \emph{Measure theory and fine properties
  of functions, revised edition (1st ed.)}, Chapman and Hall/CRC, 2015.

\bibitem{evans1998}
L.C. Evans, \emph{Partial differential equations}, Graduate studies in
  mathematics, American Mathematical Society, 1998.

\bibitem{fasshauer2007}
G.~Fasshauer, \emph{Meshfree approximation methods with {MATLAB}},
  Interdisciplinary Mathematical Sciences, 2007.

\bibitem{fasshauer2011reproducing}
Gregory~E Fasshauer and Qi~Ye, \emph{Reproducing kernels of generalized
  {S}obolev spaces via a {G}reen function approach with distributional
  operators}, Numerische Mathematik \textbf{119} (2011), no.~3, 585--611.

\bibitem{fasshauer2013reproducing}
\bysame, \emph{Reproducing kernels of {S}obolev spaces via a {G}reen kernel
  approach with differential operators and boundary operators}, Advances in
  Computational Mathematics \textbf{38} (2013), no.~4, 891--921.

\bibitem{ginsbourger2016}
D.~Ginsbourger, O.~Roustant, and N.~Durrande, \emph{On degeneracy and
  invariances of random fields paths with applications in {G}aussian process
  modelling}, Journal of Statistical Planning and Inference (2016), 170 :117
  – 128.

\bibitem{graepel2003}
T.~Graepel, \emph{{Solving Noisy Linear Operator Equations by {G}aussian
  Processes: Application to Ordinary and Partial Differential Equations}},
  {Proceedings of the 20th International Conference on Machine Learning}, {AAAI
  Press}, 2003, pp.~234--241.

\bibitem{gulian2022}
M.~Gulian, A.~Frankel, and L.~Swiler, \emph{{G}aussian process regression
  constrained by boundary value problems}, Computer Methods in Applied
  Mechanics and Engineering \textbf{388} (2022), 114117.

\bibitem{Hrmander1990TheAO}
Lars H{\"o}rmander, \emph{The analysis of linear partial differential operators
  {I}: Distribution theory and {F}ourier analysis}, Springer, 2015.

\bibitem{fu_tbc}
Fang~Q. Hu, \emph{Absorbing boundary conditions}, International Journal of
  Computational Fluid Dynamics \textbf{18} (2004), no.~6, 513--522.

\bibitem{janson_1997}
Svante Janson, \emph{Gaussian {H}ilbert spaces}, Cambridge Tracts in
  Mathematics, Cambridge University Press, 1997.

\bibitem{Jidling2018ProbabilisticMA}
C.~Jidling, J.~Hendriks, N.~Wahlstrom, A.~Gregg, T.~Schon, C.~Wensrich, and
  A.~Wills, \emph{Probabilistic modelling and reconstruction of strain},
  Nuclear Instruments \& Methods in Physics Research Section B-beam
  Interactions With Materials and Atoms \textbf{436} (2018), 141--155.

\bibitem{jidling2017}
C.~Jidling, N.~Wahlstr\"{o}m, A.~Wills, and T.~B. Sch\"{o}n, \emph{Linearly
  constrained {G}aussian processes}, Advances in Neural Information Processing
  Systems (I.~Guyon, U.~V. Luxburg, S.~Bengio, H.~Wallach, R.~Fergus,
  S.~Vishwanathan, and R.~Garnett, eds.), vol.~30, Curran Associates, Inc.,
  2017.

\bibitem{noble_kazakova}
Maria Kazakova and Pascal Noble, \emph{Discrete transparent boundary conditions
  for the linearized {G}reen--{N}aghdi system of equations}, SIAM Journal on
  Numerical Analysis \textbf{58} (2020), no.~1, 657--683.

\bibitem{kuchment2010}
P.~Kuchment and L.~Kunyansky, \emph{Mathematics of photoacoustic and
  thermoacoustic tomography}, Handbook of Mathematical Methods in Imaging
  \textbf{v. 2, Ch. 19} (2010), 817--866.

\bibitem{hegerman2018}
M.~Lange-Hegermann, \emph{Algorithmic linearly constrained {G}aussian
  processes}, Advances in Neural Information Processing Systems (S.~Bengio,
  H.~Wallach, H.~Larochelle, K.~Grauman, N.~Cesa-Bianchi, and R.~Garnett,
  eds.), vol.~31, Curran Associates, Inc., 2018.

\bibitem{hegermann2021LinearlyCG}
\bysame, \emph{Linearly constrained {G}aussian processes with boundary
  conditions}, Proceedings of The 24th International Conference on Artificial
  Intelligence and Statistics (A.~Banerjee and K.~Fukumizu, eds.), Proceedings
  of Machine Learning Research, vol. 130, PMLR, 13--15 Apr 2021,
  pp.~1090--1098.

\bibitem{lannes2D}
David Lannes and Philippe Bonneton, \emph{Derivation of asymptotic
  two-dimensional time-dependent equations for surface water wave propagation},
  Physics of Fluids \textbf{21} (2009), no.~1, 016601.

\bibitem{legall2013}
Jean-Fran{\c{c}}ois Le~Gall, \emph{Mouvement brownien, martingales et calcul
  stochastique}, Springer, 2013.

\bibitem{LpezLopera2021PhysicallyInspiredGP}
A.~F. L{\'o}pez-Lopera, N.~Durrande, and M.~{\'A}lvarez,
  \emph{Physically-inspired {G}aussian process models for post-transcriptional
  regulation in drosophila}, IEEE/ACM Transactions on Computational Biology and
  Bioinformatics \textbf{18} (2021), 656--666.

\bibitem{alvarez2009}
M.~Álvarez, D.~Luengo, and N.~D. Lawrence, \emph{Latent force models},
  Proceedings of the Twelth International Conference on Artificial Intelligence
  and Statistics (Hilton Clearwater Beach Resort, Clearwater Beach, Florida
  USA) (D.~van Dyk and M.~Welling, eds.), Proceedings of Machine Learning
  Research, vol.~5, PMLR, 16--18 Apr 2009, pp.~9--16.

\bibitem{MATLAB:2020a}
The Mathworks, Inc., Natick, Massachusetts, \emph{{{MATLAB} version
  9.8.0.1721703 (R2020a) Update 7}}, 2020.

\bibitem{mendes2012}
Fábio~Macêdo Mendes and Edson~Alves da~Costa~Júnior, \emph{{B}ayesian
  inference in the numerical solution of {L}aplace's equation}, AIP Conference
  Proceedings \textbf{1443} (2012), no.~1, 72--79.

\bibitem{Narcowich1994GeneralizedHI}
F.~J. Narcowich and J.~Ward, \emph{Generalized {H}ermite interpolation via
  matrix-valued conditionally positive definite functions}, Mathematics of
  Computation \textbf{63} (1994), 661--687.

\bibitem{NGUYEN_peraire}
N.C. Nguyen and J.~Peraire, \emph{{G}aussian functional regression for linear
  partial differential equations}, Computer Methods in Applied Mechanics and
  Engineering \textbf{287} (2015), 69--89.

\bibitem{Nikitin2021NonseparableSG}
A.~P. Nikitin, ST~John, A.~Solin, and S.~Kaski, \emph{Non-separable
  spatio-temporal graph kernels via {SPDE}s}, ArXiv (2021).

\bibitem{owhadi_bayes_homog}
Houman Owhadi, \emph{{B}ayesian numerical homogenization}, Multiscale Modeling
  \& Simulation \textbf{13} (2015), no.~3, 812--828.

\bibitem{alvarado2016}
M.~{\'A}lvarez P.~A.~Alvarado and A.~Orozco, \emph{A three spatial dimension
  wave latent force model for describing excitation sources and electric
  potentials produced by deep brain stimulation}, arXiv (2016).

\bibitem{Pazy1983SemigroupsOL}
A.~Pazy, \emph{Semigroups of linear operators and applications to partial
  differential equations}, Applied Mathematical Sciences, 1983.

\bibitem{2007numerical_recipes}
William~H Press, Saul~A Teukolsky, William~T Vetterling, and Brian~P Flannery,
  \emph{Numerical recipes 3rd edition: The art of scientific computing},
  Cambridge university press, 2007.

\bibitem{Purisha_2019}
Zenith Purisha, Carl Jidling, Niklas Wahlström, Thomas~B Schön, and Simo
  Särkkä, \emph{Probabilistic approach to limited-data computed tomography
  reconstruction}, Inverse Problems \textbf{35} (2019), no.~10, 105004.

\bibitem{citeR}
{R Core Team}, \emph{R: A language and environment for statistical computing},
  R Foundation for Statistical Computing, Vienna, Austria, 2020.

\bibitem{raissi2017}
M.~Raissi, P.~Perdikaris, and G.~E. Karniadakis, \emph{Machine learning of
  linear differential equations using {G}aussian processes}, Journal of
  Computational Physics \textbf{348} (2017), 683--693.

\bibitem{raissi2018numericalGP}
\bysame, \emph{Numerical {G}aussian processes for time-dependent and nonlinear
  partial differential equations}, SIAM Journal on Scientific Computing
  \textbf{40} (2018), no.~1, A172--A198.

\bibitem{gpml2006}
C.~E. Rasmussen and C.K.I. Williams, \emph{{G}aussian processes for machine
  learning}, the MIT Press, 2006.

\bibitem{ritter2007average}
K.~Ritter, \emph{Average-case analysis of numerical problems}, Lecture Notes in
  Mathematics, Springer Berlin Heidelberg, 2007.

\bibitem{rudin1991}
W.~Rudin, \emph{Functional analysis}, McGraw-Hill, New York, 1991.

\bibitem{Srkk2019GaussianPL}
S.~S{\"a}rkk{\"a}, M.~{\'A}lvarez, and N.~Lawrence, \emph{{G}aussian process
  latent force models for learning and stochastic control of physical systems},
  IEEE Transactions on Automatic Control \textbf{64} (2019), 2953--2960.

\bibitem{sarkka2011}
Simo S{\"a}rkk{\"a}, \emph{Linear operators and stochastic partial differential
  equations in {G}aussian process regression}, Artificial Neural Networks and
  Machine Learning -- ICANN 2011 (Berlin, Heidelberg) (Timo Honkela,
  W{\l}odzis{\l}aw Duch, Mark Girolami, and Samuel Kaski, eds.), Springer
  Berlin Heidelberg, 2011, pp.~151--158.

\bibitem{Schaback2009SolvingTL}
R.~Schaback, \emph{Solving the {L}aplace equation by meshless collocation using
  harmonic kernels}, Advances in Computational Mathematics \textbf{31} (2009),
  457--470.

\bibitem{schaefer1999}
H.H. Schaefer, \emph{Topological vector spaces}, Graduate Texts in Mathematics,
  Springer-Verlag New York, 1999.

\bibitem{scheuerer2012}
M.~Scheuerer and M.~Schlather, \emph{Covariance models for divergence-free and
  curl-free random vector fields}, Stochastic Models \textbf{28} (2012), 433 --
  451.

\bibitem{SCHEUERER2010}
Michael Scheuerer, \emph{Regularity of the sample paths of a general second
  order random field}, Stochastic Processes and their Applications \textbf{120}
  (2010), no.~10, 1879--1897.

\bibitem{Serre1999SystemsOC}
D.~Serre, \emph{Systems of conservation laws}, Cambridge University Press,
  1999.

\bibitem{lang1993}
S.Lang, \emph{Real and functional analysis}, Graduate Texts in Mathematics,
  Springer-Verlag New York, 1993.

\bibitem{Solin2020HilbertSM}
A.~Solin and Simo S{\"a}rkk{\"a}, \emph{{H}ilbert space methods for
  reduced-rank {G}aussian process regression}, Statistics and Computing
  \textbf{30} (2020), 419--446.

\bibitem{Solin2016}
Arno Solin, \emph{{Stochastic Differential Equation Methods for Spatio-Temporal
  {G}aussian Process Regression}}, Doctoral thesis, School of Science, 2016,
  pp.~68 + app. 72.

\bibitem{stuart_2010}
A.~M. Stuart, \emph{Inverse problems: A {B}ayesian perspective}, Acta Numerica
  \textbf{19} (2010), 451–559.

\bibitem{Swiler2020ASO}
Laura~P. Swiler, Mamikon Gulian, Ari~L. Frankel, Cosmin Safta, and John~D.
  Jakeman, \emph{A survey of constrained {G}aussian process regression:
  Approaches and implementation challenges}, Journal of Machine Learning for
  Modeling and Computing \textbf{1} (2020), no.~2, 119--156.

\bibitem{Tarantola2005}
A.~Tarantola, \emph{Inverse problem theory and methods for model parameter
  estimation}, SIAM, 2005.

\bibitem{treves2006topological}
F.~Treves, \emph{Topological vector spaces, distributions and kernels}, Dover
  books on mathematics, Dover Publications, 2006.

\bibitem{vergara2022general}
Ricardo~Carrizo Vergara, Denis Allard, and Nicolas Desassis, \emph{A general
  framework for {SPDE}-based stationary random fields}, Bernoulli \textbf{28}
  (2022), no.~1, 1--32.

\bibitem{vidossich1968}
G.~Vidossich, \emph{Characterization of separability for {LF}-spaces}, Annales
  de l’institut {F}ourier \textbf{tome 18} (1968), no.~2, p. 87--90.

\bibitem{wahlstrom2013}
N.~Wahlstrom, M.~Kok, T.~B. Sch{\"o}n, and F.~Gustafsson, \emph{Modeling
  magnetic fields using {G}aussian processes}, 2013 IEEE International
  Conference on Acoustics, Speech and Signal Processing (2013), 3522--3526.

\bibitem{wendland2004scattered}
Holger Wendland, \emph{Scattered data approximation}, vol.~17, Cambridge
  university press, 2004.

\end{thebibliography}
